\documentclass[11pt]{amsart}
\usepackage[margin=0.9in]{geometry}

\usepackage[dvipsnames]{xcolor}
\usepackage[centertableaux]{ytableau}
\usepackage{graphicx}

\usepackage{wrapfig}
\usepackage{eso-pic}
\usepackage{mathtools}
\usepackage{float} 
\usepackage{amsmath, amsthm, amsfonts, amssymb, amscd, bm}
\usepackage{mathrsfs}
\usepackage{multicol}
\usepackage{xcolor}
\definecolor{OxBlue}{HTML}{002147}

\usepackage[colorlinks]{hyperref}
\hypersetup{colorlinks = true, citecolor={OxBlue}, urlcolor  ={OxBlue}, linkcolor={OxBlue} }

\usepackage{perpage}
\MakePerPage{footnote}

\usepackage[normalem]{ulem}

\usepackage{enumerate}
\usepackage[capitalise, noabbrev]{cleveref}

\usepackage{tikz-cd} 
\usepackage{subfigure}

\usepackage[all,cmtip]{xy}
\usepackage[utf8]{inputenc}
\usepackage[T1]{fontenc}

\newtheorem{thm}{Theorem}[section]
\newtheorem{cor}[thm]{Corollary}

\newtheorem{prop}[thm]{Proposition}
\newtheorem{lm}[thm]{Lemma}

\newtheorem{con}[thm]{Conjecture}

\theoremstyle{definition}
\newtheorem{de}[thm]{Definition}
\newtheorem{ex}[thm]{Example}

\newtheorem{nota}[thm]{Notation}
\theoremstyle{remark}
\newtheorem{rmk}[thm]{Remark}

\renewcommand{\qedsymbol}{$\blacksquare$}

\newcommand\mapsfrom{\mathrel{\reflectbox{\ensuremath{\mapsto}}}}

\def\det{\text{det}}

\def\Hom{\text{Hom}}
\def\codim{\mathrm{codim}}
\def\rk{\text{rk}}

\def\ad{\text{ad}}

\def\z{\zeta}
\def\la{\langle}
\def\ra{\rangle}
 
\def\Q{\mathbb Q}
\def\R{\mathbb R}
\def\C{\mathbb C}
\def\N{\mathbb N}

\def\Z{\mathbb Z}

\def\H{\mathbb H}
\def\KH{K\"{a}hler }
\def\Kh{K\"{a}hler}
\def\P{{\mathbb P}}

\def\DD{\mathbb D}

\def\HKL{hyperk\"{a}hler}
\def\HK{hyperk\"{a}hler }

\def\eps{\varepsilon}
\def\a{\alpha}
\def\b{\beta}
\def\c{\beta}

\def\Fi{\varphi}

\def\W{\mathbb{W}}

\def\fun{\rightarrow}
\newcommand{\fja}[1]{\xrightarrow{#1}}

\def\dejstvo{\curvearrowright}

\def\BB{Bia\l{}ynicki-Birula }

\def\M{\mathfrak{M}}
\def\Aff{\mathrm{Aff}}
\def\CM0{\C[\M_0]}

\def\L{\mathfrak{L}}

\def\Fa{\mathfrak{F}_\alpha}

\def\F{\mathfrak{F}}

\def\Con1{Con_1(\M)}

\def\Hm{\mathcal{H}_m}
\def\Ymc{Y_{m,\beta}}

\def\sl{\mathfrak{sl}}

\def\B{\mathcal{B}}

\def\O{\mathcal{O}}

\def\FF{\mathscr{F}}

\def\FFFF{{\mathscr{F}_{\mathbb{B},\lambda}}}
\def\val{\mathrm{ini}}

\def\NN{\mathcal{N}}

\def\aa{\vec{a}}

\def\x{{\color{white}1}}

\def\MM0{\mathfrak{M}_{\tiny{(\zeta_{\mathbb{R}},0})}(Q,{\normalfont\textbf{v}},{\normalfont\textbf{w}})}

\def\MG0{\mathcal{M}_{0,\z_\C}(Q,{\normalfont\textbf{v}},{\normalfont\textbf{w}})}

\def\iso{\cong}

\def\Sym{\text{Sym}} %
\def\Hilb{\mathrm{Hilb}}

\def\GIT{/\!\!/}
\def\om{\omega}

\def\k{\mathbb{K}}

\def\J{I}

\def\ph{pseudoholomorphic}

\def\CC{$\C^*$}

\def\H{F}
\def\Fil{\FF}
\def\HL{F_\lambda}

\def\PartII{\cite{RZ2}}

\def\ku{\k [\![u]\!]}
\def\kuu{\k(\!(u)\!)}

\def\Fmin{\F_{\min}}
\def\m{k}
\def\tr{\mathrm{tr}}
\def\Id{\mathrm{Id}}

\def\Aff{\mathrm{Aff}}
\def\MM{\mathcal{M}}

\def\e{k}
\def\Y{Y}
\def\CP{\mathbb{C}P}

\def\omI{\om}
\def\aa{\min}
\def\yinfty{y_{0}}
\def\xinfty{x_{0}}
\def\iinfty{0}
\def\f{F}
\def\ff{y}
\def\ell{m}
\def\p{y}
\def\x{y}
\def\CoreY{\mathrm{Core}(Y)}

\def\PL{Poincaré--Lefschetz}
\def\AB{Atiyah--Bott}
\def\MB{Morse--Bott}

\def\MBF{Morse--Bott--Floer}

\def\HV{Hausel--Rodriguez-Villegas}
\def\RS{Robbin--Salamon}

\def\BPW{Braden--Proudfoot--Webster}

\def\G{\Gamma}
\def\Omin{\overline{\mathcal{O}_{min}}}

\makeatletter
\newcommand{\doublewidetilde}[1]{{%
  \mathpalette\double@widetilde{#1}%
}}
\newcommand{\double@widetilde}[2]{%
  \sbox\z@{$\m@th#1\widetilde{#2}$}%
  \ht\z@=.9\ht\z@
  \widetilde{\box\z@}%
}

\makeatother
\renewcommand{\arraystretch}{0.85}

\begin{document}

\title
[Filtrations on quantum cohomology]
{Filtrations on quantum cohomology \\
from the Floer theory of $\C^*$-actions}
\author{Alexander F. Ritter}
\address{A. F. Ritter, 
Mathematical Institute, University of Oxford, 
OX2 6GG, U.K.}
\email{ritter@maths.ox.ac.uk}
\author{Filip \v{Z}ivanovi\'{c}}
\address{F. T. \v{Z}ivanovi\'{c}, 
Simons Center for Geometry and Physics, 
Stony Brook, NY 11794-3636, U.S.A.}
\email{fzivanovic@scgp.stonybrook.edu} 

\begin{abstract} 
We construct a filtration by ideals on quantum cohomology for symplectic manifolds with a Hamiltonian $S^1$-action that extends to a {\ph } $\C^*$-action. These spaces include all Conical Symplectic Resolutions, in particular all Quiver Varieties. In particular, we obtain a family of filtrations on singular cohomology for any Conical Symplectic Resolution, that is sensitive to the choice of $\C^*$-action.
The symplectic form is rarely exact at infinity for these spaces, so substantial foundational work is carried out to rigorously define Floer theory, in particular symplectic cohomology. 
Using Floer theory, we construct a periodic persistence module, giving rise to a graded periodic barcode associated to the $\C^*$-action. This encodes birth-death phenomena of Floer invariants.  
Our filtrations can be viewed as a Floer-theoretic analogue of {\AB } filtrations, arising from stratifying a manifold by gradient flowlines of a 
{\MB } function, but they are distinct from those and they can detect non-topological properties of the quantum product.
\end{abstract}
\maketitle
\setcounter{secnumdepth}{3}
\setcounter{tocdepth}{1}
\vspace{-11mm}
\tableofcontents  %
\vspace{-15mm}
\section{Overview of the results}\label{Introduction}

\subsection{Motivation}
\label{Subsection intro motivation}

We will apply Floer-theoretic techniques to describe the topology and cohomology of a large new class of open symplectic manifolds. 
This class brings under one umbrella many interesting families of spaces: cotangent bundles of flag varieties,
negative complex vector bundles, Conical Symplectic Resolutions (CSRs),
Moduli spaces of Higgs bundles,
crepant resolutions of quotient singularities,\footnote{arising in the generalised McKay Correspondence.}
 and semiprojective toric manifolds. The family of CSRs itself includes:
ADE resolutions, hypertoric varieties, Nakajima quiver varieties, and Springer resolutions of Slodowy varieties. 

To illustrate some features, recall that
a \textbf{weight-$s$ CSR} is a projective $\C^*$-equivariant resolution $\pi:\M\to \M_0$ of a normal affine variety $\M_0$, whose $\C^*$-action contracts $\M_0$ to a point, and having a holomorphic symplectic structure $(\M,\om_\C)$ compatible with the $\C^*$-action: $t \cdot \omega_{\C}=t^s \omega_{\C}$.
These spaces are much studied in Geometric Representation Theory, whilst their symplectic topology is less investigated. All known examples of CSRs are hyperk\"{a}hler: the complex structures $J, K$ give rise to exact symplectic forms $\omega_J$, $\omega_K$, making these Liouville manifolds. These have special exact Lagrangian submanifolds called minimal Lagrangians \cite{vzivanovic2022exact} relevant to the A-side of Homological Mirror Symmetry.
In this paper we instead study directly the B-side: we consider the complex structure $I$ with its typically non-exact K\"{a}hler form $\omega_I$.
These spaces are almost never convex at infinity,\footnote{See \cref{Cor CSR is almost never convex}.} so they have been out of reach of Floer-theoretic study so far.
The above minimal Lagrangians now become $I$-holomorphic $\omega_I$-symplectic submanifolds. These spaces are very rich in $I$-holomorphic curves, whereas there are no non-constant closed $J$- or $K$-holomorphic curves as $\omega_J$, $\omega_K$ are exact. 

The approach of studying cohomological invariants directly on the B-side using Floer theory was explored in the context of the cohomological McKay correspondence by McLean--Ritter \cite{McLR18}. That paper studied crepant resolutions $Y\to \C^n/G$ of isolated quotient singularities, where a $\C^*$-action on $Y$ plays a prominent role, but it was still a setup where $Y$ is convex at infinity. Our original hope was that it should be fairly straightforward to adapt those constructions to quiver varieties, which are not convex at infinity, but still admit a very nice $\C^*$-action. As it turned out, the endeavour was much harder than expected, so we ventured on to establish this foundational paper, of wider scope. 

We call our class of spaces {\bf symplectic $\C^*$-manifolds} as they admit a {\ph } $\C^*$-action
$$
\Fi: \C^*\times Y \to Y,\; (t,y) \mapsto \Fi_t(y),
$$
whose $S^1$-part %
is a Hamiltonian action.
More precisely, we have a symplectic manifold $(Y,\omega)$ admitting an $\omega$-compatible almost complex structure $I$, with respect to which $\varphi$ is {\ph}.

The moment map for the $1$-periodic $S^1$-action is denoted
\begin{equation}
\label{Equation intro moment map}
H: Y \to \R.
\end{equation}
The subgroups $\R_+$ and $S^1$ of $\C^*$  yield commuting vector fields $X_{\R_+}=\nabla H$ and $X_{S^1}=X_H=IX_{\R_+}$.
To control the behaviour of $Y$ at infinity, we assume that $\Fi$ is {\bf contracting}: the $-\nabla H$ flow of any point eventually lands in a prescribed compact subset. This is equivalent to requiring that $H$ is bounded below and that the set $\F=\mathrm{Crit}(H)$ is compact. The subset $\F\subset Y$ is in fact the fixed locus of the action,
and its connected components $\F_\a$ are $I$-holomorphic and $\omega$-symplectic submanifolds of $Y.$
$$
\F = \mathrm{Crit}(H) = Y^{S^1} = Y^{\C^*} = \sqcup_\a \F_\alpha.
$$

There is a partition of $Y$ by the stable manifolds $U_\a$ of the $\F_\a$ under the $-\nabla H$ flow, whereas the unstable manifolds $D_\a$ partition a path-connected compact subset $\mathrm{Core}(Y) \subset Y$ called the \textbf{core}.
Unlike $\F$, the core can be singular. We prove that the inclusion recovers \v{C}ech cohomology,  $H^*(Y)\cong \check{H}^*(\mathrm{Core}(Y))$, and in many reasonable situations we show this is singular cohomology and that $Y$ deformation retracts onto $\mathrm{Core}(Y)$.
For weight-$1$ CSRs, $\mathrm{Core}(Y)$ is a singular $\omega_J$-Lagrangian subvariety, and the minimal $\omega_J$-Lagrangian $\F_{\min}:=\min H$ from \cite{vzivanovic2022exact} is a smooth irreducible component of it. 

Although the inspiration for this paper originated from a combination of \cite{McLR18,R10} with the goal of studying quiver varieties, the above level of generality owes much to the algebro-geometrical analogue of such spaces, called semiprojective varieties, 
which were first 
investigated in detail by {\HV} \cite{Hausel-Villegas} (building upon work by Simpson and Nakajima).

The topology of closed K\"{a}hler manifolds with Hamiltonian $S^1$-actions is well-understood, by Frankel \cite{Frankel59}, Atiyah--Bott \cite{AtB83} and Kirwan \cite{Ki84}. Using those methods we show that $H$ is a {\MB } function, whose {\MB } submanifolds are the $\F_\a$, and together with their (even) {\MB } indices $\mu_\a\geq 0$ they determine the 
cohomology of $Y$ as a vector space (working over a field),\footnote{Here $A[d]$ means we shift a graded group $A$ down by $d$, so $(A[d])_n=A_{n+d}$.}
\begin{equation}\label{EqnFrankelIntro}
H^*(Y)\cong \oplus H^*(\F_\a)[-\mu_\a].
\end{equation}
An intuitive explanation is that the unstable manifold of a cycle in $\F_\a$, which would normally have real codimension one boundary, has codimension two boundary due to the $S^1$-action, and can therefore be viewed as a cycle in $Y$. More generally, taking suitable unions of stable manifolds $U_\a$ of the $\F_\a$, the cohomology admits a filtration by subspaces \cite{AtB83}, the {\AB} filtration, %
arising from ordering the summands of \eqref{EqnFrankelIntro} essentially by the $H$-values of the $\F_\a$.
We may assume $Y$ is connected, so only one summand $H^*(\F_\a)$ in \eqref{EqnFrankelIntro} has $\mu_{\a}=0$: the one containing the unit $1\in H^0(Y)$, and it is $\F_{\min} =\min H$. 
It will be convenient to pick \eqref{Equation intro moment map} so that $\min H=0$ (by adding a constant to $H$), thus $H\geq 0$ on $Y$.

One goal of this paper is to construct a filtration $\Fil^{\Fi}_{\lambda}$ by ideals on quantum cohomology $QH^*(Y)$, ordered by $\lambda\in \R\cup \{\infty\}$. 
There is an on-going interest in Algebraic Geometry in filtrations on cohomology for such spaces, for example Higgs moduli admit the ``$P=W$'' perverse filtration \cite{de2012topology,hausel2022p,maulik2022p} which we compare with ours in \cite{RZ2}.
Our construction relies on Floer theory: symplectic cohomology $SH^*(Y,\Fi)$, whose chain level generators are loosely $S^1$-orbits; and its positive version $SH^*_+(Y,\Fi)$ which ignores constant $S^1$-orbits in $\F$. We anticipate the big picture relating them:

\begin{thm}\!\!\!\footnote{The technical assumption, needed for Floer theory and to make sense of $QH^*(Y)$, is that $Y$ satisfies a certain weak+ monotonicity property (\cref{Rmk technical symplectic assumptions on Y}). This includes for example all non-compact Calabi-Yau and non-compact Fano $Y$.}
\label{Prop vanishing of SH}
The canonical algebra homomorphism $c^*:QH^*(Y)\to SH^*(Y,\Fi)$ is surjective, equal to localisation at a  
Gromov--Witten invariant 
$Q_{\Fi}\in QH^{2\mu}(Y),$
where $\mu$ is the Maslov index of $\Fi$,
$$
SH^*(Y,\Fi) \cong QH^*(Y)/E_0(Q_{\Fi}) \cong QH^*(Y)_{Q_{\Fi}},
$$
where
$
E_0(Q_{\Fi})=\ker c^* \subset QH^*(Y)$ 
is the generalised $0$-eigenspace of quantum product by $Q_{\Fi}$. As a $\k$-module, $SH^{*-1}_+(Y,\Fi)\cong E_0(Q_{\Fi})$, yielding a Floer-theoretic presentation of $QH^*(Y)$ as a $\k$-module,
\begin{equation}\label{Equation Introduction QH is SH+}
QH^*(Y)\cong SH^{*-1}_+(Y,\Fi) \oplus SH^*(Y,\Fi).
\end{equation}
Moreover, for $N^+\in \R$ just above $N\in \N$, the continuation maps $c^*_{N^+}$ (whose direct limit is $c^*$), 
\begin{equation}\label{Equation cNplus maps intro}
c^*_{N^+}: QH^*(Y) \to HF^*(H_{N^+}),
\end{equation}
can be identified with quantum product $N$ times by $Q_{\Fi}$ on $QH^*(Y)$. In particular,
$$
SH^*(Y,\Fi)=0 \;\Longleftrightarrow  \;(\FF_{\lambda}^{\Fi}=QH^*(Y)\textrm{ for some }\lambda<\infty)
 \;\Longleftrightarrow \; 
(Q_{\Fi}\in QH^{2\mu}(Y)\textrm{ is nilpotent}),
$$
which occurs if $c_1(Y)\!=\!0$\footnote{See \cref{Intro vanishing of SH} and the paragraph before it.}\,(e.g.\;CSRs, crepant resolutions of quotient singularities, Higgs moduli spaces).
\end{thm}

The construction of the {\bf rotation class} $Q_{\Fi}$, and its usage to compute symplectic cohomology, goes back to Ritter \cite{R14}, see \cref{Theorem neg lb paper summary}; the localisation result generalises the theorem from \cite{R16}. This construction generalises the so-called \emph{Seidel elements}: invertible elements in quantum cohomology for \emph{closed} symplectic manifolds due to Seidel  \cite{Sei97}. The interesting feature of the non-compact setting is that the rotation class $Q_{\Fi}$ is usually not invertible in $QH^*(Y)$, but it becomes invertible in $SH^*(Y)$.

For CSRs, $QH^*(Y)\cong H^*(Y)$ as algebras (with suitable coefficients),
so we obtain a $\Fi$-dependent filtration on $H^*(Y)$ by ideals with respect to cup-product, and $H^*(Y)\cong SH^{*-1}_+(Y,\Fi)$.

In the sequel \cite{RZ2} we use the foundational results from this paper to construct a {\MBF } spectral sequence that converges to $QH^*(Y)$, whose $E_1$-page involves the cohomologies of {\MB } manifolds of $1$-orbits of the moment map \eqref{Equation intro moment map} intersected with higher and higher level sets of $H$.
This can be interpreted as a Floer-theoretic generalisation of \eqref{EqnFrankelIntro}, which instead arises from a 
{\MB } spectral sequence for ordinary {\MB } cohomology for the moment map.

In general, there is an obstacle to defining quantum cohomology, let alone Floer cohomology, due to the non-compactness of $Y$. The danger is the non-compactness of moduli spaces of PDE solutions which escape to infinity. The symplectic form $\omega$ is \textbf{very rarely exact at infinity} for these spaces, indeed for CSRs almost always\footnote{Exceptions being $\C^{2n},$ ADE resolutions and $T^*\CP^n$, see \cref{Preliminaries}.} one finds closed $I$-holomorphic curves appearing at infinity.

We, therefore, require one final condition, and in retrospect arriving at this definition was the decisive idea to open up this large class of examples to Floer-theoretic study going beyond just quiver varieties.
We say $Y$ is a symplectic $\C^*$-manifold {\textbf{over a convex base}} if outside of a compact subset $Y^{\mathrm{in}}$, so on $Y^{\mathrm{out}}:=Y\setminus \mathrm{int}(Y^{\mathrm{in}})$, there is a {\ph } proper map
\begin{equation}\label{Equation intro Psi}
\Psi: Y^{\mathrm{out}} \to B=\Sigma \times [R_0,\infty)
\end{equation}
to the positive symplectisation of a closed contact manifold $\Sigma$,
such that $X_{S^1}$ maps to the Reeb field,
\begin{equation}\label{Equation psi Xs1 is Reeb}
    \Psi_* X_{S^1} = \mathcal{R}_B.
\end{equation}
The existence of such a map $\Psi$ automatically ensures that the action $\Fi$ on $Y$ is contracting.
We say $\Psi$ is \textbf{globally defined} over a convex base $B$, if \eqref{Equation intro Psi} extends to a $\C^*$-equivariant {\ph } proper map
$
\Psi: Y \to B
$
to an open symplectic manifold $B$ which is convex at infinity (e.g.\;a Liouville manifold), admitting a $\C^*$-action whose $S^1$-part generates the Reeb flow at infinity. 
\begin{rmk}
We do not require any conditions on the dimension of $B$, and $\Psi_*$ is allowed to have a large kernel with varying ranks. We do not assume that $\Psi$ is surjective, so 
although the Reeb flow will be periodic on the image of $\Psi$, this need not hold everywhere on $B$.
One can allow a positive constant factor in \eqref{Equation psi Xs1 is Reeb}, but after a rescaling argument we may assume \eqref{Equation psi Xs1 is Reeb} holds as written. 
We caution the reader that level sets of $H$ are in general not preserved by the $\R_+$-action, so level sets of $H$ need not map into slices $\Sigma \times \{R\}$ via $\Psi$  (an exception
are cotangent bundles %
and negative vector bundles, Examples \ref{Example cotangent bundles intro} and \ref{Negative vector bundles intro}). In \eqref{Equation psi Xs1 is Reeb}, $\mathcal{R}_B=X_R$ is a Hamiltonian vector field for the radial coordinate $R\in [R_0,\infty)$.
In this paper, but not in \PartII, we can allow a slightly more general condition in \eqref{Equation psi Xs1 is Reeb}:
\begin{equation}\label{Equation psi Xs1 is Reeb 2}
    \strut\hspace{25ex} \Psi_* X_{S^1}
= X_{fR}%
\qquad \textrm{ for a Reeb-invariant function }f: \Sigma \to (0,\infty). 
\end{equation}

One can identify $Y^{\mathrm{out}}$ with the complement of a disc bundle of a complex line bundle over an orbifold 
$\{H=const\}/S^1$ 
(assuming $H$ is proper), e.g.\;viewing it as a subset of the symplectic cut, however one usually cannot hope to use this as $B$: the identification is rarely {\ph } \cite[Rmk.1.1]{lerman1995symplectic}. 
\end{rmk}

\subsection{Examples of symplectic \CC-manifolds over a convex base} \label{Intro Examples Section}

\begin{ex}[Equivariant projective morphisms]\label{Example intro CSR have global map} Let $X$ be an affine variety\footnote{\label{non-point affine}Which is not equal to a point. This ensures $Y$ is non-compact. We may assume $Y$ is connected.}
 with a contracting\footnote{The $\C^*$-action contracts $X$ to a single fixed point. Algebraically, the coordinate ring $\C[X]$ is $\N$-graded with $\C[X]_0=\C$.} 
 algebraic $\C^*$-action. Any %
equivariant projective morphism
$\pi: Y\fun X$ satisfies our assumptions, for a globally defined $\Psi$. Indeed, the coordinate ring $\C[X]$ is $\N$-graded by the $\C^*$-action; we choose homogeneous generators $(f_i)_{i=1}^N$ %
with
weights $w_i\geq 1$; and
we obtain a proper\footnote{$Y\fun X$ is proper because it is a projective morphism.} $I$-holomorphic map 
$$
\Psi: Y \to X \hookrightarrow \C^N \to \C^N,\quad y \mapsto \pi(y) \mapsto (f_1,\ldots,f_N)|_{\pi(y)} \mapsto (f_1^{w/w_1},\ldots,f_N^{w/w_N})|_{\pi(y)},
$$
that is $\C^*$-equivariant, using the diagonal action on $\C^N$ of weight $w:=\mathrm{lcm}_i\{w_i\}$. 
The quasi-projective variety $Y$  
admits\footnote{ %
The proof of these claims is the same as the proof written for CSRs in %
\cref{Subsection Explicit S1 invt Kahler form for CSR}.
} a $\C^*$-equivariant embedding $i:Y\hookrightarrow X \times \C P^m \hookrightarrow \C^n \times \C P^m$ for some $n,m$, and we pull back the standard $S^1$-invariant K\"{a}hler form from $\C^n \times \C P^m$ so that $S^1\subset \C^*$ defines
a Hamiltonian $S^1$-action on $Y$.
We study equivariant projective morphisms in detail in \cite{RZ4}.
Particular examples, arising as equivariant projective resolutions, are
CSRs
and crepant resolutions of quotient singularities.
\end{ex}

\begin{ex}
[\Kh/GIT quotients of $\C^n$]
\label{Example intro toric varieties}
An instance of \cref{Example intro CSR have global map} are
\KH reductions of unitary actions
$K \subset U(V)$ on a complex vector space $V$, or equivalently 
(by the Kempf--Ness theorem) GIT quotients of the complexified actions %
$K_\C=:G \subset GL(V)$.
By the GIT picture, they always come with a projective morphism
$\pi: Y= V\GIT_{\chi} G \fun V \GIT_0 G = X $ to the affine GIT quotient, where $\chi:G\fun \C^*$ is a character.
The scalar matrices $\C^* \iso S \leq GL(V)$ cannot be a subgroup of $G$, as otherwise
$X$ would be a single point,\footnote{Given any $z\in V,$ 
the closure of its orbit $G z \supset S z$ would contain the origin ($\lim_{s\fun 0} s \cdot z =0$), 
thus by definition
$[z]=[0]$ in the GIT quotient $X.$} which we do not allow (\cref{non-point affine}).
Hence, there is an induced $\pi$-equivariant action of $S$ on $Y$ and $X$, which is contracting.\footnote{$\lim_{s\fun 0} s \cdot [z]=
\lim_{s\fun 0} [s\cdot z]=[0].$}
Particular instances of this are 
\textbf{semiprojective toric manifolds}, for which $G$ is a torus \cite{HS02}. We discuss these spaces in detail in \cite{RZ4}.
\end{ex}

\begin{ex}[Submanifolds]\label{Example intro submfds}
Given any space $Y$ from our class of spaces, any $\C^*$-invariant properly embedded $I$-{\ph} submanifold $j: S\hookrightarrow Y$ will also belong to our class.\footnote{The $I$-compatible form $j^*\omega$ makes $S$ symplectic; the $S^1$-action is Hamiltonian with moment map $H\circ j$; and the $\Psi$-map for $S$ is $\Psi\circ j: S^{\mathrm{out}}:=S\cap j^{-1}(Y^{\mathrm{out}})\to \Sigma \times [R_0,\infty)$.}
\end{ex}

\begin{ex}[Blow-ups]
Our class of spaces is closed under blow-ups of compact subvarieties, provided that the $\C^*$-action lifts to the blow-up, since \eqref{Equation psi Xs1 is Reeb 2} is only a condition at infinity. In the toric case, \cite[Sec.3E-3F]{R16} described the effect of such blow-ups on quantum and symplectic cohomology.
\end{ex}

\begin{ex}[$A_2$-singularity]\label{Example running example of intro}
CSRs of the lowest dimension are ADE resolutions.\footnote{Minimal resolutions of quotient singularities $M\fun \C^2/\Gamma,$ where $\Gamma \leq SL(2,\C)$ is a finite subgroup.} The $A_2$ case is %
the minimal resolution $\pi:M \fun \C^2/(\Z/3)$ (third roots of unity $\zeta$ acting on $\C^2$ by $(x,y)\mapsto (\zeta x,\zeta^{-1}y)$). We can define $M $ as the blow up at $0$ of the image $V(XY-Z^3)\subset \C^3$ of the embedding $\C^2/(\Z/3)\hookrightarrow \C^3$, $[x,y]\mapsto (x^3,y^3,xy)$. The McKay $\C^*$-action on $M $ is the lift of the $\C^*$-action $(x,y)\mapsto (tx,ty)$. The $\C^*$-equivariant map $\Psi: M \to \C^3$ is $(X^2,Y^2,Z^3)$, for the weight $6$ diagonal $\C^*$-action on $\C^3$.
\end{ex}

\begin{ex}[Springer resolutions]\label{Example cotangent bundles intro}
A simple instance of \cref{Example intro CSR have global map}, and of a weight-1 CSR,
is
\begin{equation}\label{TCPn example resolution}
    T^*\CP^N \fun \Omin:=\{A\mid A^2=0, \ \rk(A)=1\}\subset \mathfrak{sl}_{N+1}.
\end{equation}
Here $\Fi$ is the standard $\C^*$-action that contracts fibres, and $\C^*$ acts by dilation on $\mathfrak{sl}_{N+1}.$ The zero section of $T^*\C P^N$ is $I$-holomorphic and $\omega$-symplectic, so $\omega$ is not the canonical exact form of $T^*\C P^N$.

A more general class of such examples are Springer resolutions, i.e.\;cotangent bundles of flag varieties. For example, the variety
$\mathcal{B}$ of all flags $0\subset F_1 \subset F_2 \subset \C^3$: 
there is a resolution of singularities
$\nu: T^*\mathcal{B}\to \NN=\{3\times 3\textrm{ nilpotent matrices in }\mathfrak{sl}_3\}$. 
The singular affine variety $\NN$ has three strata: (i) the point $\mathcal{O}_{1,1,1}:=0$ 
to which $\NN$ contracts to under the $\C^*$-action; (ii) the stratum $\mathcal{O}_{2,1}$ which is the adjoint orbit 
of the nilpotent Jordan normal form with blocks of sizes 2,1; and (iii) a generic stratum $\mathcal{O}_3$.
The orbit $\mathcal{O}_{2,1}$ is a singular stratum of $\NN$ that goes to infinity. Any transverse slice to $\mathcal{O}_{2,1}$ 
(entering from $\mathcal{O}_3$, avoiding $\mathcal{O}_{1,1,1}$) is an $A_2$-singularity; its preimage via $\nu$ is its resolution; so the fibre above the chosen point in $\O_{2,1}$ consists of two holomorphic $\C\P^1$'s intersecting transversely.
This illustrates how our spaces $Y$ may have
$I$-holomorphic spheres arbitrarily far at infinity, in fibres of \eqref{Equation intro Psi}.
\end{ex}

\begin{ex}[Negative Vector Bundles]\label{Negative vector bundles intro}
In \cref{Moment map is function of radial coordinate for TCPn} we explain why $T^*\C P^N$ in \cref{Example cotangent bundles intro} is a negative vector bundle. Other examples of negative vector bundles are the duals of ample vector bundles, e.g.\,conormal bundles of projective algebraic varieties $V\subset \C P^N$. One can also make a vector bundle negative by tensoring with a large power of a negative line bundle.
Following \cite[Sec.11]{R14},
for a negative complex vector bundle $Y=\mathrm{Tot}(\pi:E\fun B)$ with the standard contracting $\C^*$-action on fibres, we use the biholomorphism $\Psi:E\setminus 0 \to L\setminus 0$ where $L \to \P(E)$ is the tautological line bundle over the complex projectivisation of $E$. This $\Psi$ typically does not extend globally, and $\Psi$ is not symplectic when $\mathrm{rank}_{\C}E\geq 2$.\footnote{The natural symplectic form $\omega$ on $E$ is non-exact at infinity whereas for $L$ it is.} 
This map was used in \cite[Sec.11.2]{R14} to define quantum/Floer cohomology.
\end{ex} 

\begin{ex} (Trivial vector bundles). Given any\footnote{Assuming, as we do for general $Y$, the weak+ monotonicity condition, see \cref{Rmk technical symplectic assumptions on Y}.} 
closed symplectic $B$, a trivial vector bundle $Y= B \times \C^r$ 
satisfies our assumptions, with the weight-1 $\C^*$-action on fibres and projection $\Psi: Y \fun \C^r.$
\end{ex}

\begin{ex} (Hilbert schemes) Given a Riemann surface $\Sigma,$ the Hilbert scheme of a trivial bundle from the last example, $Y=\Hilb^n(\Sigma \times \C)$ satisfies our assumptions, with the action coming from weight-1 $\C^*$-action on $\C,$ and the projection 
$\Psi: Y\fun \Sym^n(\Sigma \times \C) \fun \Sym^n(\C)\iso \C^n$ being a composition of the Hilbert-Chow morphism, and the $\Sym^n$-functorial projection.
\end{ex}

\begin{ex}[Higgs moduli]\label{Example Intro Higgs moduli I} 
Another important class of examples are various
moduli spaces $\MM$ of Higgs bundles 
over a Riemann surface $\Sigma.$
Roughly, elements of these spaces are
(conjugacy classes) of stable pairs $(V,\Phi),$ where $V$ is a vector bundle over $\Sigma$ of fixed coprime rank and degree,
and $\Phi\in \Hom(V,V\otimes K_{\Sigma})$ 
is a Higgs field, which potentially can have some poles (here, $K_{\Sigma}$ is the canonical bundle of $\Sigma$).
The space $\mathcal{M}$ is a 
complete 
{\HK}manifold,
with a natural $I$-holomorphic \CC-action given by $t\cdot (V,\Phi)=(V,t \Phi),$ whose $S^1$-part is $\om_I$-Hamiltonian 
with a proper moment map. 
 It also admits the so-called Hitchin fibration, given as the characteristic polynomial of the Higgs field, 
 \begin{equation}\label{Intro Hitchin fibration}
\qquad \qquad \Psi: \mathcal{M}\to B\iso \C^{N}   \quad \textrm{ where }N:=\tfrac{1}{2} \dim_\C\MM,
 \end{equation}
which is a proper surjective $I$-holomorphic map, and it is $\C^*$-equivariant for a certain linear 
$\C^*$-action on $B.$ This makes $(\MM,\om_I)$ a symplectic \CC-manifold over a convex base, with $\Psi$ globally defined. 
\end{ex}

\begin{ex}[Crepant resolutions of quotient singularities]\label{Quotient singularities}
Let $\pi:Y\to \C^n/G$ be a crepant resolution,\footnote{a non-singular quasi-projective variety $Y$ together with a proper, birational morphism $\pi$ which is a
biholomorphism away from the singular locus. The crepant condition here is equivalent to $c_1(Y)=0$.} where $G\subset SL(n,\C)$ is a finite subgroup. 
The diagonal $\C^*$-action on $\C^n/G$ lifts\footnote{by Batyrev \cite[Prop.8.2]{Batyrev}, or \cite[Prop.3.6]{McLR18}.} to a $\C^*$-action on $Y$, 
so $\pi$ is a proper $\C^*$-equivariant $I$-holomorphic map.
These are much studied in Algebraic Geometry literature due to the generalised McKay Correspondence (see \cite{McLR18} for references).
A subclass of these, when $n=2k$ and $G \leq Sp(2k,\C)$, give rise to CSRs $Y$ (\cref{Rmk Refinement of the McKay correspondence}), studied in %
\cite{KaledinMcKay, %
BezKaledinMcKay}. Another subclass is when $G$ acts freely outside of $0\in \C^n$, so isolated singularities $\C^n/G$, see %
\cite{McLR18} 
(here $Y$ is a CSR only\footnote{see the paragraph before \cref{Cor CSR is almost never convex} for an explanation.} for $n=2$).
The isolatedness condition implies $Y$ is convex at infinity, i.e.\,\eqref{Equation intro Psi} is a symplectic isomorphism with $\Sigma=S^{2n-1}/G$ (conjugating $G$ so that $G\subset SU(n)$), which was needed to study $Y$ using Floer theory. 
In our new framework, we can apply Floer theory to all crepant resolutions, without the isolatedness condition, since they are a special case of \cref{Example intro CSR have global map}.
\end{ex}

\begin{ex}[Torsion submanifolds]\label{Intro Torsion submanifolds}
Given a symplectic \CC-manifold $Y$ over a convex base, and an integer $m\geq 2,$ the fixed locus of the subgroup $\Z/m\leq \C^*$
decomposes into finitely many connected $I$-{\ph} submanifolds called \textbf{torsion submanifolds} $\Ymc.$ 
They are symplectic $\C^*$-submanifolds of $Y$ over the same convex base, by restricting \eqref{Equation intro Psi} (where we could also replace the restricted $\C^*$-action by its $m$-th root).
Each $\Ymc$ contains a subcollection of the $\F_\a$, and has strata that converge to different $\F_\a$. 
Looking at one such point of convergence $y_0\in \F_\a$,
and viewing $T_{y_0} Y$ as a complex representation for the linearised $S^1$-action,
one obtains the weight decomposition 
\begin{equation}\label{Intro weight spaces}
T_{y_0} Y = \oplus_{k\in \Z} H_k.
\end{equation}
It follows that only finitely many weights $m$ can arise (the weights only depend on the component $\F_\a$, not the choice of $y_0\in \F_\a$),
and that locally near $\F_\a$ the subbundle $\oplus_{b\in \Z} H_{mb}$ parametrises $\Ymc$.
\end{ex}

\subsection{Floer theory is possible for symplectic $\C^*$-manifolds}
\label{Subsection Floer theory is possible for symplectic C-manifolds}
\begin{prop}\label{intro max princ}
For symplectic $\C^*$-manifolds over a convex base, a maximum principle holds on $Y^{\mathrm{out}}$ for the PDEs used in Gromov-Witten theory and in Floer theory, provided one uses the almost complex structure $I$ and the Hamiltonian vector field on $Y^{\mathrm{out}}$ is a positive constant multiple of $X_{S^1}$.   
\end{prop}

The hard part of this Proposition was not the proof, but rather finding the right statement (requiring the existence of the map $\Psi$).
Indeed we can summarise the proof: locally at infinity we have a PDE\vspace{-0.8mm}
\begin{equation}\label{Equation PDE upstairs}
\partial_s u + I(\partial_t u - k X_{S^1})=0,\vspace{-0.8mm}
\end{equation}
where $z=s+it$ is a local holomorphic coordinate for the domain of $u$. As $\Psi_*$ commutes with $I$ and satisfies \eqref{Equation psi Xs1 is Reeb 2}, abbreviating $H_B:=fR$, the projected curve $v=\Psi(u)$ in $B$ satisfies\vspace{-0.8mm}
\begin{equation}\label{Equation PDE downstairs}
\partial_s v + I(\partial_t v - k\,X_{H_B})=0.\vspace{-0.8mm}
\end{equation}
By the extended maximum principle due to \cite[App.C2-C4]{R16}, $R\circ v$ cannot have a local maximum. 
\begin{rmk}\label{Remark max principle for Fukaya cat}
The same argument holds for the PDE equations used to define Lagrangian Floer cohomology and the $A_{\infty}$-equations for the (wrapped) Fukaya category.
For the purpose of proving a maximum principle as above, one does not require a $\C^*$-action, one just requires \eqref{Equation intro Psi} and $\Psi_*X=X_{H_B}$ for the Hamiltonian vector field $X$ that one wishes to use on $Y$.
\end{rmk}

The moment map $H$ in \eqref{Equation intro moment map} plays a role similar to the radial coordinate $R$ for symplectic manifolds convex at infinity.\footnote{In the case of Liouville manifolds \cite{Sei08} and of symplectic manifolds convex at infinity \cite{R10}, one works with the class of Hamiltonians $H_{\lambda}:M \to \R$ which at infinity are ``radial'': they only depend on $R$ and they are eventually linear in $R$ with a generic slope $\lambda>0$. The chain-level generators of 
Hamiltonian Floer cohomology $HF^*(H_{\lambda})$ are the $1$-periodic orbits of $H_{\lambda}$. Due to a 1-to-1 correspondence between $T$-periodic Reeb orbits in $\Sigma=\{R=1\}$ and $1$-orbits of a radial Hamiltonian arising with slope value $T:=H_{\lambda}'(R)$, one has control over those generators. The generic slope $\lambda$ ensures that there are no $1$-orbits in the region at infinity where $H_{\lambda}$ is linear. The construction is motivated by cotangent bundles $M=T^*N$, where the Reeb flow is  the geodesic flow for $N$ on the sphere bundle $\Sigma\cong STN$; the slope $\lambda$ corresponds to the period of the geodesic \cite{Vi96}; and the limit of the $HF^*(H_{\lambda})$ recovers the homology of the free loop space of $N$.}
So it is natural to consider Hamiltonians which are functions $c(H)$ of $H$, that become linear at infinity. We have control over $1$-orbits arising at slope $c'(H)=\lambda$, as follows.
\begin{lm}
By considering the $1$-periodic $S^1$-flow, and the subgroup $G_{\lambda}:=\langle e^{2\pi i \lambda }\rangle \subset \C^*$, we have
$$
 (1\textrm{-orbits of }\lambda H) \stackrel{1:1}{\longleftrightarrow}  (\lambda\textrm{-periodic orbits of the }S^1\textrm{-flow})  \stackrel{1:1}{\longleftrightarrow}(G_{\lambda}\textrm{-fixed points in }Y).
$$
\end{lm}

Non-constant $1$-orbits can only arise if $\lambda=\tfrac{k}{m}$ is rational, for $k,m$ coprime, in which case  $G_{\lambda}\cong \Z/m$. 
They %
in fact arise in {\MB } families, intersecting the torsion submanifolds $Y_{m,\beta}$ from \cref{Intro Torsion submanifolds}:
\begin{equation}\label{Equation Bkm slices intro}
    B_{\frac{k}{m},\beta}\cong Y_{m,\c}\cap \{c'(H)=\tfrac{k}{m}\},
\end{equation}
so only finitely many $m\in \N$ can arise (it is understood that we choose $c''(H)>0$ when $c'(H)=\tfrac{k}{m}$; this ensures a certain {\MB}-property).
These are connected odd-dimensional smooth manifolds. Only the \emph{non-compact} $\Ymc$ arise in \eqref{Equation Bkm slices intro}, which we call {\bf outer torsion manifolds}, and we call the $c'(H)=\tfrac{k}{m}$ in \eqref{Equation Bkm slices intro} the {\bf outer $S^1$-periods}. Compact torsion manifolds $\Ymc$, which lie in $\mathrm{Core}(Y)$, do not arise in \eqref{Equation Bkm slices intro}, as we ensure that the only $1$-orbits of $c(H)$ near $\mathrm{Core}(Y)$ are the constant $1$-orbits at points of $\F=\sqcup\F_\a.$ Call $\lambda>0$ {\bf generic} if it is not an $S^1$-period (e.g.\,irrational $\lambda$). We write $H_{\lambda}:=c(H)$ to mean that $c'(H)=\lambda$ at infinity, for generic $\lambda$, so there are no $1$-orbits at infinity. 

So far, it was not essential that $H$ in \eqref{Equation intro moment map} is proper, but we now assume it (we show in \cref{Lemma making H proper} that one can tweak $\omega$ to ensure this). We construct $HF^*(H_{\lambda})$, which at chain level is generated by the $1$-orbits of $H_{\lambda}$, and a directed system of continuation maps $HF^*(H_{\lambda}) \to HF^*(H_{\lambda'})$ for generic $\lambda\leq \lambda'$.

\begin{cor}\label{intro SH for C* mflds corrolary}
There is a well-defined symplectic cohomology algebra $SH^*(Y,\Fi):=\varinjlim HF^*(H_{\lambda})$ admitting a canonical unital algebra homomorphism $c^*:QH^*(Y)\to SH^*(Y,\Fi)$.
\end{cor}

\begin{rmk}
The Corollary holds more generally if the $\C^*$-action $\Fi$ is only defined on $Y^{\mathrm{out}}$, as it only relies on using \eqref{Equation intro Psi} or \eqref{Equation psi Xs1 is Reeb 2} to achieve a maximum principle at infinity. 
However, in this generality, $Q_{\Fi}$ in \cref{Prop vanishing of SH} may no be defined and the vanishing result in
\cref{Intro vanishing of SH} may cease to hold.\footnote{E.g. this is the case for open Riemann surfaces \cite[Ex.3.3]{Sei08}.}
\end{rmk}

\begin{rmk} $SH^*(Y,\Fi)$ may depend on the choice of $\Fi.$ In \cite{RZ4} we show that commuting $\C^*$-actions over an affine variety yield isomorphic rings $SH^*(Y,\Fi)$; this includes almost all examples in \cref{Intro Examples Section}.
\end{rmk}
\begin{rmk}
Groman \cite{Gr15} defines a universal symplectic cohomology for symplectic manifolds $Y$ that are geometrically bounded. The philosophical idea is that rather than choosing a specific growth of Hamiltonians at infinity, one would like to use as many Hamiltonians as possible, and take a huge direct limit of Floer cohomologies. A stumbling block for us, however, was that it is not entirely straightforward to check all the conditions required for this construction, partly because the Riemannian geometry of our spaces $(Y,\omega)$ at infinity is not well understood. We, therefore, opted to choose a specific class of Hamiltonians, since our spaces $Y$ come with a natural choice of $H$ via \eqref{Equation intro moment map}. This allows us to obtain meaningful filtered invariants that are sensitive to the choice of $\C^*$-action.
\end{rmk}

\begin{prop}\label{Intro vanishing of SH}
When $c_1(Y)=0,$ the $\Fi$-symplectic cohomology vanishes: $SH^*(Y,\Fi)=0.$
\end{prop}

To prove this result, we show that the $\Z$-grading of generators of 
$HF^*(H_{\lambda})$ becomes arbitrarily negative 
as $\lambda\fun \infty$.
This phenomenon is typical for Hamiltonian $S^1$-actions, and 
generalises previous vanishing results, seen for $\C^n$ \cite{OanceaEnsaios}, ALE spaces \cite{R10}, many non-compact symplectic Calabi-Yau manifolds %
\cite{R14}, and crepant resolutions of isolated quotient singularities \cite{McLR18}. 

\cref{Intro vanishing of SH} may imply obstructions to existence of certain exact Lagrangians in the Liouville manifold $(Y,\om_J)$, when $Y$ is a \textbf{Higgs branch moduli space} (a \HK reduction of a unitary action on a flat vector space, e.g.\;quiver and hypertoric varieties). In the special case of ADE resolutions,
\cite{R10} used that vanishing to prove closed exact Lagrangians are 2-spheres. This suggests:
\begin{con}
Let $Y$ be a Higgs branch moduli space admitting a closed exact Lagrangian $L\subset (Y,\om_J)$. We conjecture that a deformation argument as in \cite[Lem.50]{R10} and \cite[Cor.1.6]{BeRit20} implies the vanishing of twisted symplectic cohomology $SH^*(Y,\omega_J,\underline{\Lambda}_{\tau \omega_I})=0$.
If this holds, then%
\footnote{The symplectic cohomology twisted by the transgression $\tau(\omega_I)\in H^1(\mathcal{L}Y)$ was constructed by Ritter \cite{ritter2009novikov}. By \cite[Cor.17]{R10}, $\tau(j^*\omega)\neq 0$ where $j:L \hookrightarrow Y$ (using $\mathrm{char}\,k=2$ in Floer theory if $L$ is non-orientable). Also, $\tau(j^*\omega_I)$ can be identified with a homomorphism $\pi_2(L)\to \R$, via $H^2(Y;\R) \to H^2(L;\R) \to \mathrm{Hom}(\pi_2(L),\R)\subset H^1(\mathcal{L}Y;\R)$, see \cite[Sec.4.1]{R10} or \cite[Sec.6.1]{ritter2009novikov}.
That the vanishing of twisted symplectic cohomology yields a bound on the HZ-capacity is due to Irie \cite[21, Cor.3.5]{irie2014hofer} and Albers--Frauenfelder--Oancea \cite[Sec.3.1--3.2]{albers2017local}.}
$H^2(L;\R)\neq 0$, $\pi_2(L)$ is infinite, and the Hofer-Zehnder capacity of $(Y,\om_J)$ is finite and bounded above by \eqref{Equation capacity}; in particular there are 
no Lagrangian spheres when $dim_\C Y>2.$ 
\end{con}

\subsection{Floer theory induces an $\R$-ordered filtration by ideals on quantum cohomology}\label{Subsection intro Floer theory induces filtration}

\begin{thm}\!\!\footnote{We discuss the field of coefficients $\k$ in \cref{Remark choice of coefficients}, and mild technical assumptions on $Y$ in \cref{Rmk technical symplectic assumptions on Y}.}\label{Cor intro filtration} %
There is a filtration by graded ideals of $QH^*(Y)$, ordered by  $p\in \R\cup\{\infty\}$,\footnote{When $c_1(Y)=0,$ $\FF_\lambda^{\Fi}=QH^*(Y)$ for large enough $\lambda,$ due to the vanishing from \cref{Intro vanishing of SH}.}
\vspace{-1mm}
\begin{equation}\label{DefinitionOfFiltration}
\FF^{\varphi}_p :=\bigcap_{\mathrm{generic}\,\lambda \geq p} \left(\ker c_\lambda^*:QH^*(Y)\to HF^*(H_{\lambda})\right), \qquad \Fil_{\infty}^{\Fi}:=QH^*(Y),\vspace{-1mm}
\end{equation} 
where $c_\lambda^*$ is a continuation map, a grading-preserving $QH^*(Y)$-module homomorphism.

As $\FF^{\varphi}_p$ are ideals, if $1\in \Fil_{p}^{\Fi}$ then $\Fil_{p}^{\Fi}=QH^*(Y)$ (``unity is the last to die'').

This filtration is sensitive to the choice of $\C^*$-action $\Fi$ (e.g.\,see Examples \ref{Example running example of intro 2}, \ref{Example running example of intro 3},
\ref{Example running example of intro 3.5}).

In {\PartII} we prove that the filtration satisfies the following {\bf stability property}:
\begin{equation}\label{Intro Stability property}
    \Fil^{\varphi}_\lambda  = \Fil^{\varphi}_{\lambda'} \textrm{ if there are no outer }S^1\textrm{-periods in the interval }(\lambda,\lambda'].
\end{equation}
\end{thm}

\begin{rmk}
$x\in \Fil_p^{\varphi}(QH^*(Y))\Leftrightarrow x$
can be represented as the boundary of a Floer chain $y$ involving non-constant $S^1$-orbits of period $\leq p.$  In {\PartII} we build an explicit map $SH^{*-1}_{(0,\lambda]}(Y,\Fi)\to QH^*(Y)$ with image $\Fil^{\varphi}_\lambda$, so that $[y]\in SH^{*-1}_{(0,\lambda]}(Y,\Fi)$ becomes a cycle representing $x$ via that map.

\noindent In \cref{Quotient singularities},
 the McKay correspondence says $\mathrm{rank}\,H^{2d}(Y) = \#($conjugacy classes in $G$ with ``age grading'' $d)$. An explicit bijection between conjugacy classes and a basis of $H^*(Y)$ is only known in special families (e.g.\;Kaledin \cite{KaledinMcKay}). 
  Our filtration sometimes refines the age grading  (\cref{Rmk Refinement of the McKay correspondence}). 

In algebraic geometry, and by different methods, Bellamy--Schedler \cite{BeSch18} constructed vector-space filtrations on cohomologies of certain CSRs. 
For $A_n$-singularities, our filtration agrees with theirs degree-wise using a certain $\C^*$-action (\cref{RemarkWhyFilrationsOnCohomologyInteresting}), but for other actions we get different filtrations. 
\end{rmk}

\begin{rmk}[Capacities]
    Our filtration encodes a large collection of invariants in an algebraic way. Among those, the following invariant measures when the unit first appears, $1\in \FF_{c(Y,\omega,\Fi)}^{\Fi}$,
\begin{equation}\label{Equation capacity}
c(Y,\omega,\Fi):=\min \{p: \FF_p^{\Fi}=QH^*(Y) \} \in (0,\infty],
\end{equation}
which is finite $\Leftrightarrow SH^*(Y,\Fi)=0$ by \cref{Prop vanishing of SH}. Such numbers, in various contexts, are known as ``symplectic capacity'' or ``spectral invariant''.
The idea of asking what lowest/largest action value is needed for a Floer-theoretic class to appear in an action-filtered subgroup, goes back to Viterbo \cite{viterbo1992symplectic}, \cite[Sec.5.3]{Vi99}. It appears, among others, in Oh \cite{oh1997symplectic}; Schwarz \cite[Sec.2.2]{schwarz2000action}; for (exact) cotangent bundles in Irie \cite[Sec.3]{irie2014hofer} and Albers--Frauenfelder--Oancea \cite[Sec.3.1-3.2]{albers2017local} (with relations to 
Hofer--Zehnder capacities);
for 4-dimensional Liouville domains in Hutchings's ECH capacities \cite{hutchings2011quantitative}, and for $S^1$-equivariant symplectic cohomology for domains in $\R^{2n}$ in Hutchings--Gutt \cite[Sec.4]{gutt2018symplectic}.
\end{rmk}

\subsection{Isomorphism invariance}\label{Subsection isomorphism invariance}

The choice of $\Psi$ in \eqref{Equation intro Psi} is auxiliary information: it does not affect $c^*_{\lambda}$ in \eqref{DefinitionOfFiltration}, it  is only needed to ensure quantum/ Floer cohomology are defined. An {\bf isomorphism}\footnote{One may wish to weaken the notion of isomorphism. The standard parametrised moduli space argument, that proves $QH^*(Y)$ is invariant under deformations of $\omega$ or of the almost complex structure $I$, applies in our setting provided that this is accompanied by a deformation of the map $\Psi$ in \eqref{Equation intro Psi}, satisfying \eqref{Equation psi Xs1 is Reeb} or \eqref{Equation psi Xs1 is Reeb 2}, so a maximum principle applies. Without that, one needs to justify why parametrised moduli spaces remain compact during the deformation.}
of symplectic $\C^*$-manifolds $j:Y \to Y'$ is a {\ph} $\C^*$-equivariant symplectomorphism (no conditions on $\Psi,\Psi'$).
In particular, if $j:Y \to Y'$ is a {\ph} $\C^*$-equivariant diffeomorphism then, after replacing $\omega_Y$ with $j^*\omega_{Y'}$ on $Y$, $j$ becomes an isomorphism of symplectic $\C^*$-manifolds.

\begin{prop}
    $\Fil_\lambda^{\varphi}(QH^*(Y))$
    is an invariant of  $Y$ up to isomorphism.
\end{prop}

\subsection{Periodic persistence module, and barcodes}\label{Subsection persistence}

The results of our paper can be rephrased in the modern language of persistence modules. We refer the reader to \cite{Polterovich} and references therein, for an overview of this language.
In our setup, we have $\k$-vector spaces
$$
V_0:=QH^*(Y), \qquad \textrm{ and }\qquad V_{p}:= HF^*(H_{p}),
$$
defined for $0<p \in \{\textrm{generic slopes}\}\subset \R$.
We extend this to non-generic slopes $p$ by the {\bf convention} 
$$V_{p}:=HF^*(H_{p^+}), \quad \textrm{ for }p^+>p\textrm{ sufficiently close to }p.$$
It does not matter which $p^+$ we choose, up to a continuation isomorphism.
We have continuation maps
$$
\psi_{p',p}: V_{p} \to V_{p}' \qquad \textrm{ for } \; p\leq p' \qquad (\textrm{ with }\psi_{p,p}=\mathrm{id}),
$$
with the convention that we replace $p$ by $p^+$ if $p$ is not a generic slope, and similarly for $p'$.
These maps are compatible: $\psi_{p'',p'}\circ \psi_{p',p}=\psi_{p'',p}$ for $p\leq p'\leq p''.$
Extend the definition by  
$$V_{\infty}:=SH^*(Y)=\varinjlim V_{p},$$  taking the direct limit over those directed maps.
We will explain later how we wish to define $V_{p}$ for $p<0$, but for now let us temporarily declare $V_{p}:=V_0$ for $p<0$ with $\psi_{p',p}=\mathrm{id}$ for $p\leq p'\leq 0$.

This defines a {\bf persistence module} $V$, in the weak sense of a functor $(\R\cup \{\infty\},\geq )\to \mathrm{Vect}_{\k}$ to the category of finite dimensional $\k$-vector spaces.
It satisfies {\bf properness}: $\psi_{p',p}$ is an isomorphism if $p,p'$ lie in the same connected component of $\R\setminus \{\textrm{critical slopes}\}$. {\bf Critical slopes} are the outer $S^1$-periods of the $S^1$-action; they form a closed discrete subset of $[0,\infty)$. 
{\bf Spectral points}, $\mathrm{Spec}(V)\subset \R\cup \{\infty\}$, are  those $p$ which don't admit an open neighbourhood on which the $\psi$-maps are isomorphisms, so $$\mathrm{Spec}(V)\subset \{\textrm{critical slopes}\}\cup\{\infty\}.$$

We have a {\bf semi-continuity property}: given $p$, the map $\psi_{p',p}$ is an isomorphism for any $p'>p$ sufficiently close to $p$. We differ here slightly from \cite{Polterovich}, as our semi-continuity is from above (due to our convention about using $p^+$, mentioned earlier).
We do not have the additional finite-type requirement from \cite{Polterovich} that $V_{p}=0$ for $p\ll 0$, but the properness condition is almost as good.

In practice, $V_{p}$ only changes when $p$ hits a critical slope $p$, so one could reduce to a sequence of vector spaces by evaluating only at critical slopes $p_0=0<p_1<p_2<\cdots \to \infty$. However, we keep the parameter $p\in \R\cup \{\infty\}$: it encodes non-trivial information about the $\C^*$-action (see examples later).

The {\bf $(p',p)$-persistent cohomology} is 
$$
P_{p',p}:=\mathrm{image}(\psi_{p',p}) \cong V_{p}/\ker \psi_{p',p}.
$$
In particular, much of our paper studies the filtration $\FF^{\Fi}_p\subset QH^*(Y)$ which determines
\begin{equation}\label{Equation Pp0}
P_{p,0}:=\mathrm{image}(c_p^*:QH^*(Y)\to HF^*(H_p)) \cong QH^*(Y)/\FF^{\Fi}_p
\end{equation}
(using $c_{p^+}^*$ and $H_{p^+}$ if $p$ is not a generic slope).
The $S^1$-action $\Fi$ gives rise to isomorphisms 
$$\mathcal{S}_{\Fi}:HF^*(H_{\lambda})\to HF^*(H_{\lambda-1})[2\mu],$$ 
called Seidel isomorphisms, where $\mu$ is the Maslov index of $\Fi$. The $\mathcal{S}_{\Fi}$ are compatible with $\psi_{\lambda',\lambda}$: 
\begin{equation}\label{Equation seidel compatibility}
\mathcal{S}_{\Fi} \circ \psi_{\lambda',\lambda}
=
\psi_{\lambda'-1,\lambda-1}\circ \mathcal{S}_{\Fi}.
\end{equation}
See for example the naturality diagram in \cref{Subsection Naturality under full rotations} as an instance of this compatibility.
\begin{cor}
    The persistence module is $1$-periodic: for $p'\geq p\geq 1$,
\begin{equation}\label{Equation shift persistence}
\mathcal{S}_{\Fi}: V_p \stackrel{\cong}{\longrightarrow} V_{p-1}[2\mu] \quad\textrm{ and }\quad
\mathcal{S}_{\Fi}: P_{p',p} \stackrel{\cong}{\longrightarrow} P_{p'-1,p-1}[2\mu].
\end{equation}
\end{cor}
In view of this, we wish to define $V_p$ for $p<0$ so that it preserves periodicity for all $p\in \R$. So we turn the isomorphisms in \eqref{Equation shift persistence} into a definition of $V_p$ for $p<0$, and we extend the definition of the $\psi$-maps by requiring that \eqref{Equation seidel compatibility} and compatibility both hold.\footnote{For example, $V_{-0.6}:=V_{0.4}[-2\mu]$ and $\mathcal{S}_{\Fi}:=\mathrm{id}:V_{0.4} \to V_{-0.6}[2\mu]$, and for example $\psi_{2.3,-0.6} := \mathcal{S}_{\Fi}\circ \psi_{3.3,0.4}$.}

It is not difficult to see that the data $(V_{\lambda})_{\lambda\in \R\cup \{\infty\}}$ is (non-canonically) isomorphic to a direct sum of {\bf basic persistence modules}. A basic persistence module is a one-dimensional vector space $W=\k$ {\bf supported} on some interval $\mathcal{I}\subset \R\cup\{\infty\}$, so: $W_p:=W$ for $p\in \mathcal{I}$, and $W_p:=\{0\}$ otherwise, with structure maps $\psi_{p',p}=\mathrm{id}:W\to W$ for $p,p'\in \mathcal{I}$, and $\psi_{p',p}=0$ otherwise. Such a decomposition is called the {\bf normal form} for $V$.
We then obtain the {\bf barcode} associated to $V$:
$$
(\mathcal{I}_j,m_j)
$$
where $\mathcal{I}_j\subset \R\cup \{\infty\}$ are intervals, and $m_j\in \N$ is the number of basic persistence modules $W$ in the decomposition of $V$ which are supported precisely on the interval $\mathcal{I}_j$.

In our case, the intervals are of the form $[a,b)$ or $(-\infty,b)$, for $a\in \R$ and $b\in \R\cup \{\infty\}$. When $Y$ is Calabi-Yau, or more generally when $SH^*(Y)=0$, the intervals are bounded. The capacity $c(Y,\omega,\Fi)$ from \eqref{Equation capacity} is the longest interval length $|\mathcal{I}_j|$ for which $0\in \mathcal{I}_j$, because the unit $1\in QH^*(Y)$ is the last to die (if it dies) so it persists the longest.

Hamiltonian Floer cohomology is always $\Z/2$-graded, but this can often be upgraded to $\Z/2N$ or $\Z$ depending on assumptions on $Y$; for example $c_1(Y)=0$ ensures a $\Z$-grading. The continuation maps $\psi_{b',b}$ preserve this grading. We can refine the barcode by $(\mathcal{I}_j,m_j,d_j)$ by making $m_j$ only count the dimensions of $\mathcal{I}_j$-supported summands $W$ that live in grading $d_j$ (we can ensure the $W$ above live entirely in some grading). We call this a {\bf graded} barcode.

\begin{cor}
    The $\mathcal{S}_{\Fi}$ map corresponds to a translation symmetry on the barcode: the number of intervals $[a,b)$ equals the number of intervals $[a-1,b-1)$; and analogously for  $(-\infty,b)$ and $(-\infty,b-1)$. If we grade the barcode, then that translation shifts the grading by $d_j \mapsto d_j + 2\mu$.
\end{cor}

We can also define a filtration of $V_{q}$ for any $q\in \R$ by
$$
\FF^{\Fi}_{p,q}:=\ker \psi_{p,q} \;\quad \textrm{ where }p\in \R\cup \{\infty\},
$$
with the convention $\FF^{\Fi}_{p,q}:=V_q$ when $p< q$.
For example, $\FF^{\Fi}_{p}=\FF^{\Fi}_{p,0}$  is the filtration in \eqref{DefinitionOfFiltration} for $QH^*(Y),$
and it determines \eqref{Equation Pp0}, and therefore also $P_{p+k,k}\cong P_{p,0}$ for all $k\in \N$. In general, we have  
$$P_{p,q}\cong V_q/\FF^{\Fi}_{p,q} \qquad \textrm{ and }1\textrm{-periodicity:} \qquad \mathcal{S}_{\Fi}:\FF^{\Fi}_{p,q}\cong \FF^{\Fi}_{p-1,q-1}[2\mu].$$
The case when $H^*(Y)$ lies in even degrees, described in \cref{Theorem what happens in even degrees}, is particularly nice: in that case, all $V_{p}$ have the same dimension, and therefore the pictorial representation of the above barcode means that the death of bars is matched by births of bars (in possibly different gradings), since the sum of all $m_j$ for which $p\in \mathcal{I}_j$ must equal the dimension of $V_p.$

We remark that there is a notion of distance between persistence modules, called interleaving distance, which coincides with a notion of distance for barcodes, called bottleneck distance. Thus, the above assigns a notion of {\bf distance} between two $\C^*$-actions on $Y$ (we will see in \cref{Subsection Different fi actions} that $Y$ often admits discrete families of $\C^*$-actions). This will be explored in future work.

\subsection{Specialisation %
to a filtration %
on $H^*(Y)$ by cup product ideals}\label{Subsection specialisation intro}
\strut\\[0mm]\indent
The Novikov field $\k$ is a vector space over a chosen base field $\mathbb{B}$ (see Remarks \ref{Rmk about coeffs Novikov} and \ref{Remark choice of coefficients}). As a $\k$-module, $QH^*(Y)=H^*(Y;\k)\cong H^*(Y;\mathbb{B})\otimes_{\mathbb{B}} \k$. To induce a filtration on $H^*(Y;\mathbb{B})$, the naive intersection $H^*(Y;\mathbb{B})\cap \FF_{\lambda}^{\Fi}$ is not well-behaved: its rank over $\mathbb{B}$ is unrelated\footnote{E.g. $c_{\lambda}^*(x)=x$, $c_{\lambda}^*(y)=Tx$, for $x,y\in H^*(Y;\mathbb{B})$, then $y-Tx\in \ker c_{\lambda}^*$. Over $\mathbb{B}$ the map is injective on $\mathrm{span}_{\mathbb{B}}(x,y)$.} to $\mathrm{rank}_{\k}\FF_{\lambda}^{\Fi}$ and it does not respect cup product. Instead, we define the {\bf specialised $\Fi$-filtration} 
$$\FFFF^{\Fi}(H^*(Y;\mathbb{B}))=\FFFF^{\Fi}:=\val(\FF_{\lambda}^{\Fi}) \subset H^*(Y;\mathbb{B}),$$
where $\mathrm{ini}(x)$ is the initial term\footnote{For $x\neq 0\in \k$, $x=n_0 T^{a_0} + n_1 T^{a_1} + \cdots$ for $n_j\in \mathbb{B}$, $n_0\neq 0$, where $a_0,a_1,\ldots$ are an increasing unbounded sequence in $\R$. The $T$-valuation $\mathrm{val}_T(x)=a_0 \in \R$ defines the ``specialisation''
$
\textstyle
\val(x):= (T^{-\mathrm{val}_T(x)}x)|_{T=0} = 
n_0 \in \mathbb{B};
$
and $\mathrm{ini}(0)=0$.
This naturally extends to $\mathrm{ini}:QH^*(Y)=H^*(Y;\mathbb{B})\otimes_{\mathbb{B}} \k \to H^*(Y;\mathbb{B})$. For
$S\subset QH^*(Y)$, $\val(S):=\{\val(s):s\in S\}$.}
$n_0 \in H^*(Y;\mathbb{B})$ of the formal Laurent ``series'' $x:=n_0 T^{a_0} + (T^{>a_0}\textrm{-terms})$.

\begin{cor}\label{Cor intro about B filtration}
        $\FFFF^{\Fi}$ %
        is a filtration by cup-ideals on $H^*(Y,\mathbb{B})$ %
        with $
    \mathrm{rank}_{\mathbb{B}}\, \FFFF^{\Fi}
    = 
    \mathrm{rank}_{\k}\, \FF_{\lambda}^{\Fi}.
    $
\end{cor} 

    \subsection{Effective methods to compute the $\Fi$-filtration $\Fil^{\Fi}_{\lambda}$} %
Floer invariants are notoriously difficult to compute. Nevertheless we developed two efficient tools to compute the $\Fi$-filtration, which we compare in an explicit example in \cref{Example Slodowy S32 using mulambda Falpha}: the Springer resolution of the Slodowy variety $\mathcal{S}_{32}$.

One tool is developed in \PartII: a {\MBF } spectral sequence, whose $E_1$-page essentially consists of $H^*(Y)$ together with ordinary cohomology $H^*(B_{k/m})[-\mu(B_{p,\c})]$ of the slices  from \cref{Equation Bkm slices intro}, i.e.\,the {\MB } manifolds of period-$k/m$ $S^1$-orbits.
This spectral sequence is induced by the period-filtration we construct on Floer chain complexes. 
The columns of the spectral sequence are labelled by the slope values $\lambda=k/m$, so the $\Fi$-filtration value of a class in $H^*(Y)$ is the smallest slope needed so that the columns up to that slope have gathered enough cohomology to kill the given class in the spectral sequence.
This method relies on knowing information about the torsion $m$-submanifolds from \cref{Intro Torsion submanifolds}, and an analysis of the {\MBF } indices $\mu(B_{p,\c})$.

The other tool, developed in this paper, decomposes and computes the continuation map $c_{\lambda}^*$ from \cref{Cor intro filtration} for Hamiltonians of type $H_{\lambda}=\lambda H$. This involves knowing information about the fixed loci $\F_\a$ and analysing certain indices $\mu_{\lambda}(\F_\a)$ which are a generalisation of the {\MB } indices $\mu_\a$,  that are dependent on the slope $\lambda.$ The precise definition of these indices in terms of the weight decompositions \eqref{Intro weight spaces}, and the study of their properties, is carried out in \cref{Section RS indeces}. 

\subsection{Lower bounds for the ranks of $\Fil^{\Fi}_{\lambda}$} 
To simplify the discussion, suppose $Y$ has no odd degree cohomology. This holds for all CSRs, and all crepant resolutions of quotient singularities. 

\begin{thm}\label{Theorem what happens in even degrees}
Suppose $H^*(Y)$ lies in even degrees, then (over Novikov field coefficients) 
\begin{equation}\label{Equation intro HF lambda H splits}
    HF^*(H_{\lambda})\cong \oplus H^*(\F_\a)[-\mu_\lambda(\F_\a)].
\end{equation}
Thus,
using \eqref{EqnFrankelIntro}, the continuation map in \eqref{DefinitionOfFiltration} for generic $\lambda>0$ becomes
\begin{equation}\label{Equation clambda map intro}
\textstyle
c_{\lambda}^*:QH^*(Y)\cong \bigoplus_{\a} H^*(\F_\a)[-\mu_\a]\longrightarrow \bigoplus_{\a} H^*(\F_\a)[-\mu_\lambda(\F_\a)].
\end{equation}
\end{thm}
Abbreviate the rank of a $\k$-module by $|V|:=\mathrm{rank}_{\k}\,V.$ For generic $\lambda>0$, define
\begin{equation}\label{Equation difference of HF ranks}
\textstyle
\delta_{\lambda}^k:
=
|QH^k(Y)|-|HF^*(H_{\lambda})|
=
\sum_{\a} |H^{\e-\mu_\a}(\F_\a)|-|H^{\e-\mu_{\lambda}(\F_\a)}(\F_\a)|
.
\end{equation}
From \eqref{Equation clambda map intro} we obtain lower bounds on the ranks of $\Fil_{\lambda}^{\Fi}$; it holds also without the above simplification: 
\begin{cor}\label{Cor estimating filtration using degrees}\label{Equation lower bound rank of filtration}\label{estimates on filtration by drops}
Assume $c_1(Y)=0$. Let $t=\max \{d: H^d(Y)\neq 0\}$,
so $0\leq t\leq 2\dim_{\C}Y-2$. Then
\begin{equation*}
\begin{array}{rclrcl}
|\Fil^{\Fi}_{\gamma}\cap QH^{\e}(Y)| & \geq & \displaystyle\max_{0<\lambda<\gamma^+}\delta_{\lambda}^k \qquad\qquad\textrm{ and }
&
|\Fil^{\Fi}_{N}\cap QH^{\e}(Y)|  & \geq &
|H^{\e}(Y)|-|H^{\e+2N\mu}(Y)|,
 \\[3mm]
 \Fil^{\Fi}_{N}\cap QH^{\e}(Y) & = &
QH^\e(Y)  \;\textrm{ for }\e\geq t+1-2N\mu,
& \textrm{ and } \qquad
 \Fil^{\Fi}_{\lambda} & = &
QH^*(Y)  \;\textrm{ for }\lambda \geq  \lceil\tfrac{t+1}{2\mu} \rceil,
\end{array}
\end{equation*}
where $N\in \N$. If the action $\Fi$ is the $m$-th power of a $\C^*$-action, the above also holds for $N\in \tfrac{1}{m}\Z_{> 0}.$
\end{cor}
The result for $N$ follows from the general one, as $\mu_{N^+}(\F_\a) = \mu_\a - 2N\mu$ (where $N^+$ means $N+\delta$ for small $\delta>0$), thus the continuation map for ``$N$ full rotations'' (compare \cref{Equation cNplus maps intro}) is
$$
c_{N^+}^*: QH^*(Y)\cong HF^*(H_{\delta}) \to HF^*(H_{N^+}) \cong QH^*(Y)[2N\mu].
$$

\begin{ex}[CSRs]\label{Example CSR filtration Intro 1} For weight-$s$ CSRs, $c_1(Y)=0.$\footnote{as they are holomorphic symplectic.}
Also, $QH^*(Y)=H^*(Y)$ is ordinary cohomology,\footnote{Since by \cite{Nam08}, $(Y,I)$ can be deformed to an affine algebraic variety.
} 
which lies in even degrees up to $\dim_{\C}Y$, and $2\mu=s\cdot\dim_{\C}Y$. Thus \eqref{estimates on filtration by drops} yields
\begin{equation}\label{Intro simple estimate CSR}
    \FF_1^{\Fi} = H^*(Y) \text{ for } s\geq 2, \ \ \text{ whereas for } s=1: \ \ \FF_1^{\Fi} \supset H^{\geq 2}(Y) \textrm{ and } \ \FF_2^{\Fi}=H^*(Y).
\end{equation}
More generally:\footnote{
 \cref{Estimates of 1/s+ time indeces for weight-s CSRs} shows
$\mu_{({1}/{s})^+}(\F_\a)\leq -\dim_\R \F_\a$, 
$\mu_{({2}/{s})^+}(\F_\a)< -\dim_\R \F_\a.$ Then use \cref{Equation lower bound rank of filtration}.} 
 $\FF_{{1}/{s}}^{\Fi} \supset H^{\geq 2}(Y)$ and $\FF_{{2}/{s}}^{\Fi} = H^{*}(Y)$ for any $s\geq 1$.
\end{ex}

\begin{ex}[$A_2$-singularity]\label{Example running example of intro 2}
In the minimal resolution $M\to \C^2/(\Z/3)$ from \cref{Example running example of intro}, the core $\mathfrak{L}=\pi^{-1}(0)=S^2_1 \cup S^2_2$ consists of two copies of $S^2$ intersecting transversely at a point $p$.
There are three natural $\C^*$-actions, obtained by lifting via $\pi$ the following actions on $V(XY-Z^3)\subset \C^3$:\\
{\small
\strut\quad\qquad (a)\quad %
$(t^3 X, t^3Y, t^2Z)$;\,\quad $\F=p_1 \sqcup p \sqcup p_2$ \;(3 points),\quad where $p=\F_{\min}$ and $p_i\in S^2_i$.\;\; [CSR of weight 2]
\\
\strut\quad\qquad (b)\quad %
$(tX,t^2Y,tZ)$;\quad\quad $\F=S_1^2 \sqcup p_2$,\qquad \qquad \qquad \;\quad where $S_1^2=\F_{\min}$.\qquad\qquad\quad\;\, [CSR of weight 1]
\\
\strut\quad\qquad (c)\quad %
$(t^2X,tY,tZ)$;\,\quad\quad $\F=p_1\sqcup S^2_2$,\;\qquad \qquad\quad \qquad where $S_2^2=\F_{\min}$.\qquad\qquad\quad\;\, [CSR of weight 1]
}
\\
The map $\Psi: M \to \C^3$ from \eqref{Equation intro Psi} is respectively:\footnote{using a diagonal $\C^*$-action on $\C^3$ of weight respectively $6,2,$ and $2$.} $(X^2,Y^2,Z^3)$,
$(X^2,Y,Z^2)$,
$(X,Y^2,Z^2)$.
The weight decomposition  \eqref{Intro weight spaces} has $1$-dimensional summands.\footnote{Weights:
(a) $(1,1)$ at $\F_{\min}=p$,
and $(3,-1)$ at $p_1$ and $p_2$; 
(b-c) $(0,1)$ at the sphere $\F_{\min}$, and $(2,-1)$ at $p_i.$} %
Below is the {\AB} presentation \eqref{EqnFrankelIntro} and ours from \eqref{Equation intro HF lambda H splits}, where $\lambda^+$ denotes slopes just above $\lambda$, and $\k_m:=\k[-m]$ (the field $\k$ in degree $m$).
 \begin{center}
 \small
 \begin{tabular}[t]{rlcrl}
    (a)\; $H^*(M):$ & $\k_0\oplus \k_2\oplus \k_2$
    & \strut
    & (b-c)\; $H^*(M):$ & $H^*(S^2)\oplus \k_2$
    \\ 
    $HF^*(\tfrac{1}{3}^+H):$ & $\k_0\oplus \k_0\oplus \k_0$
& \strut
    & 
    $HF^*(\tfrac{1}{2}^+H):$ & $H^*(S^2)\oplus \k_0$
    \\
    $HF^*(\tfrac{2}{3}^+H):$ & $\k_0\oplus \k_{-2}\oplus \k_{-2}$
    & \strut
    & 
    $HF^*(1^+H):$ & $H^*(S^2)[2]\oplus \k_0 =H^*(M)[2] $;
    \\
    $HF^*(1^+H):$ & $\k_{-4}\oplus \k_{-2}\oplus \k_{-2} = H^*(M)[4] $
    & \strut
    & 
    $HF^*(2^+H):$ & $H^*(S^2)[4]\oplus \k_{-2}  = H^*(M)[4] $
    \end{tabular}     
 \end{center}  
The maps \eqref{Equation clambda map intro} are composites of grading-preserving vertical downward maps between the above groups (they need not preserve summands). We deduce the following information:
\begin{center}
 \begin{tabular}[t]{rclll}
    (a)\;&\strut & $\Fil^{\Fi}_{1/3} \supset H^2(M),\;$&$\Fil^{\Fi}_{1} = H^*(M).$
    \\
     (b-c)\;&\strut & 
     $|\Fil^{\Fi}_{1/2}\cap H^2(M)| \geq 1,\;$&$\Fil^{\Fi}_{1}\supset H^2(M),\;$&$\Fil^{\Fi}_{2}= H^*(M)
     .$
    \end{tabular}     
 \end{center}

Leter we show that $\Fil^{\Fi}_{\lambda}$ distinguishes each case; moreover $\Fil^{\Fi}_1\neq H^*(M)$ in (b-c) so our filtration is different from the {\AB} filtrations of $H^*(M)$ which do not separate the unit from $H^2(S^2)$.%
\end{ex}

\subsection{Computing the $\Fi$-filtration for full rotations: the rotation class $Q_{\Fi}$}\label{Subsection Computing the continuation map for full rotations: the QFi invariant}
\strut\\[0mm]
\noindent Computing the map $c_{\lambda}^*$ in Floer theory is usually impossible. An exception is the full rotation $c^*_{N^+}$ in \eqref{Equation cNplus maps intro}. In \cref{Subsection computation of the continuation maps} we identify $c_{N^+}^*$ with $N$-fold quantum product %
by 
$Q_{\varphi}\in QH^{2\mu}(Y),$ so
\begin{equation}\label{Equation FN is Qpower N}
\Fil^{\Fi}_{N} = \ker \left(Q_{\varphi}^{\star N}\, \star \, \cdot:QH^*(Y)\to QH^*(Y)[2N\mu]\right). 
\end{equation}
The Maslov index
$\mu>0$ of the $S^1$-action is easily computable from \eqref{Intro weight spaces} (\cref{Section RS indeces}). 
If $\Fi$ is {\bf semi-free}\footnote{$\Fi$ acts freely outside of the fixed locus.} then \eqref{Equation FN is Qpower N} completely determines the filtration by the stability property \eqref{Intro Stability property} (cf.\;Examples \ref{Remark intro c1 nonzero case}, \ref{Remark intro c1 nonzero case 2}).

Let $q_{\Fi} = \max \{ q\in \N: x^q$ divides the minimal polynomial of $Q_{\Fi}\,\star: QH^*(Y)\to QH^*(Y)\}$. Then:\vspace{-2mm}
$$E_0=\ker Q_{\Fi}^{\star q_{\Fi}} = \FF_{q_{\Fi}}^{\Fi}=\FF_{\lambda\geq q_{\Fi}}^{\Fi},\;\;\;
0 
\to
\FF_{N}^{\Fi}
\stackrel{\mathrm{incl}\,\,}{\longrightarrow}
E_0
\stackrel{Q_{\Fi}^{\star N}}{\longrightarrow}
Q_{\Fi}^{\star N}\star E_0
\to
0
\textrm{ is exact,}
\;
\textrm{ and }\;
Q_{\Fi}^{\star N}\star E_0 \subset \FF_{q_{\Fi}-N}^{\Fi}.
$$
The last inclusion can be an equality (\cref{Remark intro c1 nonzero case}) but it can be strict (\cref{Remark intro c1 nonzero case 2}).
Note that
$\FF_{q_{\Fi}}^{\Fi}=QH^*(Y)\Leftrightarrow SH^*(Y)=0$; e.g.\;it holds  if $c_1(Y)=0$; but for monotone $Y$ it only holds for nilpotent $Q_{\Fi}$, e.g.\;$Q_{\Fi}$ is \emph{not} nilpotent for monotone toric negative line bundles \cite[Cor.1.9]{R16}.

Suppressing the $\star$ symbol, by linear algebra $E_0:=E_0(Q_{\Fi})$ from \cref{Prop vanishing of SH} satisfies strict inclusions $
0\subsetneq 
\
Q_{\Fi}^{q_{\Fi}-1}E_0
\
\subsetneq 
\
Q_{\Fi}^{q_{\Fi}-2}E_0
\
\subsetneq \cdots 
\subsetneq 
\
Q_{\Fi}E_0
 \
 \subsetneq
 \Fil^{\Fi}_{q_{\Fi}}=E_0 \subseteq \Fil_{\infty}^{\Fi}=QH^*(Y),
$
yielding strict inclusions:
$$
0=\Fil^{\Fi}_0\subsetneq 
\
\Fil^{\Fi}_1
\
\subsetneq 
\
\Fil^{\Fi}_2
\
\subsetneq \cdots 
\subsetneq 
\
 \Fil^{\Fi}_{q_{\Fi}-1}
 \
 \subsetneq
 \Fil^{\Fi}_{q_{\Fi}}\subseteq \Fil_{\infty}^{\Fi}=QH^*(Y),$$
 whose ranks over $\k$ satisfy $|\Fil^{\Fi}_{N}| = |E_0| - |Q_{\Fi}^N E_0|$ by the isomorphism $Q_{\Fi}^N: E_0/\Fil^{\Fi}_{N} \to Q_{\Fi}^N E_0.$

\begin{prop}\label{Intro prop about Qfi (non) vanishing}
If $c_1(Y)=0$ and the intersection product $H_{2\mu}(Y)\otimes H_{2\dim_{\C}Y-2\mu}(Y)\to \k$ vanishes, then $Q_\Fi=0$ and $\Fil^{\Fi}_{1}=QH^*(Y).$ 

If $Y$ is K\"{a}hler Calabi-Yau or %
Fano; $\F_{\min}$ only has weights $0$ and $1$; and the Euler class $e_{\min}$ of the normal bundle of $\F_{\min}\subset Y$ is non-zero;\footnote{here PD stands for the {\PL} dual of the locally-finite cycle $[\F_{\min}]\in H^{lf}_{2\dim_{\C}Y-2\mu}(Y)$.} then $\FF_1^{\Fi}\neq QH^*(Y)$ and 
\begin{equation}\label{Equation intro Qfi nonzero} Q_{\varphi}= PD[\F_{\min}] %
+ (T^{>0}\textrm{-terms}) \neq 0\in QH^{2\mu}(Y). %
\end{equation}
\end{prop}

\begin{ex}[CSRs]\label{Example CSR filtration intro} 
For a weight-$s$ CSR,
$Q_{\Fi}=0$ for $s\geq 2$, for degree reasons, so $\Fil^{\Fi}_{1}=H^*(Y)$, as expected from
\cref{Example CSR filtration Intro 1}.
For weight-$1$ CSRs, 
$Q_{\Fi}\neq 0 \in H^{\dim_{\C}Y}(Y)$ by \cref{Intro prop about Qfi (non) vanishing}, so
$$ \Fil^{\Fi}_{1} = H^{\geq 2}(Y) \quad\textrm{ and }\quad \Fil^{\Fi}_{2}=H^*(Y).$$
So the filtration separates the unit $1\in H^0(Y)$. In \cref{Example CSR unit death intro} we show $\Fil^{\Fi}_{\lambda}\neq H^{\geq 2}(Y)$ for all $\lambda<1$.
\end{ex}

\begin{ex}[Higgs moduli]\label{Example Intro Higgs moduli II}
Unlike CSRs, Higgs moduli (\cref{Example Intro Higgs moduli I}) can have odd-cohomology 
and the intersection form is degenerate. 
Let $\MM:=\MM_G(d,g)$ be ordinary\footnote{Meaning that the Higgs field does not have poles.} $G$-Higgs bundles 
of degree $d$ coprime to $n$ over a genus $g$ Riemann surface for $G\in \{GL_n,SL_n\}$.
Then the intersection form vanishes,\footnote{due to \cite{hausel_1998} for $SL_n$ and 
 \cite{heinloth2016intersection} for $GL_n$.} so $Q_\Fi=0$, $\FF_1=H^*(\MM)$. The minimal $\omega_J$-Lagrangian $\Fmin$ is the stable bundles moduli. 
\end{ex}

\begin{ex}[First example with $c_1(Y)\neq 0$]\label{Remark intro c1 nonzero case}
If $c_1(Y)\neq 0$, $SH^*(Y,\Fi)$ is not $\Z$-graded.
If $Y$ is non-compact Fano (or monotone, see \cref{Rmk technical symplectic assumptions on Y}) one gets a $\Z$-grading by grading the formal variable $T$ of $\k$. 
The total space $Y$ of the negative line bundle $\pi:\mathcal{O}(-k) \to \C\P^m$ for $1\leq k \leq m$ is monotone, $c_1(Y)=(1+m-k)x\neq 0$ where $x:=\pi^*[\omega_{\C\P^m}]$ (also, $\omega$ on $Y$ satisfies $[\omega]=x\in H^2(Y;\R)$), and $|T|=2(1+m-k)$. 
For simplicity, here we use
$\k:=\Q(\!(T)\!)=\{\text{Laurent series in }T\}$. Abbreviate
$y:=x^{1+m-k}-(-k)^kT$.
By \cite[Cor.4.14]{R16}, for the standard $\C^*$-action on fibres,
$$QH^*(Y)=\k[x]/(x^k y),\quad Q_{\varphi}=-kx,\quad E_0(Q_{\varphi})=\langle y \rangle \subset QH^*(Y),\quad \textrm{ and }\quad SH^*(Y)\cong \k[x]/(y).$$
As $\Fi$ is semi-free, the filtration is determined by \eqref{Equation FN is Qpower N}; we get\footnote{For negative line bundles $Y=\mathrm{Tot}(\pi:E\to X)$, 
refining symplectic cohomology via Archimedean norms as in Venkatesh \cite[Thm.1]{venkatesh2021quantitative}, one might get a filtration of other sums of generalised eigensummands of $\pi^*c_1(E)\in QH^*(Y)$.} $q_{\Fi}=k$ and 
$Q_{\Fi}^{\star N}\star E_0 = \FF_{k-N}^{\Fi}$:
\begin{equation}
\label{filtrVecBund}
0\subset 
\langle x^{k-1}y \rangle 
\subset 
\langle x^{k-2}y \rangle 
\subset
\cdots
\subset
\langle x y \rangle
\subset 
E_0=\langle y \rangle 
\subset QH^*(Y).
\end{equation}
These are not cup-ideals: %
in $H^*(Y)=\k[x]/(x^{1+m})$, ordinary cup-product gives 
$x\cdot (x^{k-1}y)= -(-k)^kTx^k$, which is not in the ideal $\FF_{1}^{\Fi}=\langle x^{k-1}y \rangle=\k x^{k-1}y$ of $QH^*(Y)$. 
\end{ex}

Let $Y$ be any negative complex vector bundle, $\mathrm{Tot}(\pi:E\fun B)$, over a closed symplectic manifold. 
Assume that $Y$ is monotone or Calabi--Yau. Let $\Fi$ be the fiberwise $\C^*$-action.
Let $c:=\pi^*c_{\mathrm{top}}(E)$ (e.g.\;$\pi^*c_1(E)$ for line bundles). Then $Q_{\Fi}=c$ up to rescaling by $\k\setminus\{0\}$,\footnote{By \cite[Thm.72]{R14}.} and by \cref{Subsection The filtration for negative vector bundles}: %

\begin{prop}[Negative vector bundles]\label{Proposition filtration for negative line bundles}
The filtration jumps precisely at $\lambda=N\in \{1,2,\ldots,q_{\Fi}\}$, 
$$\FF_{N}^{\Fi}=\ker c^{\star N} \qquad \textrm{ and } \qquad
\FF_{\mathbb{B},N}^{\Fi} \subset \ker c^N,
$$
If $\mathrm{Im}(c^{\star j}\star )\subset \mathrm{Im}(c^j\,\cup)$ in $QH^*(Y)$, then
$\FF_{\mathbb{B},j}^{\Fi} =\ker (c^j\,\cup)$. If that holds for all $1\leq j\leq q_{\Fi}$, then $\FF_{\mathbb{B},N}^{\Fi} =\ker c^N$ and the ranks of $\FF_{\lambda}^\Fi$ (but not the ideals) are determined topologically by \cref{Cor intro about B filtration}.
\end{prop}

\begin{ex}\label{Remark intro c1 nonzero case 2}
The final hypothesis above occurs for $\mathcal{O}(-1,-1)\to \C P^1 \times \C P^1$. Let $x_j:=\pi_j^*\omega_{\C P^1}$, via the two projections.
Let $X:=x_1+x_2$, so $c=-X$. 
We use
$\k:=\Q(\!(T)\!)$. %
By \cite[Cor.4.16]{R16},\footnote{In that paper, there are two typos: $E=\mathcal{O}(-1,-1)$ not $\mathcal{O}(-1,1)$, and $SH^*(E)\cong \Lambda$ as $x_1,x_2$ both get identified with $2T$ (there $t=T$, and $\Lambda=\k$).}
$$
QH^*(Y) \cong \k[x_1,x_2]/(x_1^2+TX,x_2^2+TX),\quad |T|=2\quad\textrm{ and }\quad SH^*(Y) \cong QH^*(Y)_{-X} \cong \k.
$$
The minimal polynomial for $X\star$ 
is $X^2(X+4T)$, so $q_{\Fi}=2$ and 
$E_0 = \ker c^{\star 2} = \langle X+4T \rangle \subset QH^*(Y)$ is of dimension $3$ (codimension $1$). 
The filtration, showing the precise slopes where it jumps, is
$$0\subset \FF_1^{\Fi} = \la X(X+4T), x_1-x_2 \ra 
\subset 
\FF_2^{\Fi} = E_0= \la X+4T \ra
\subset \FF_{\infty}^{\Fi} =  QH^*(Y).$$
Note that\footnote{The last equality is due to 
$X(X+4T)x_1=X(X+4T)x_2=0$ (by direct computation in $QH^*(Y)$).} 
$Q_{\Fi}\star E_0 = \la X(X+4T) \ra =\k X(X+4T) \subsetneq \FF_1^{\Fi}$ is a strict inclusion (compare \cref{Remark intro c1 nonzero case}).
\end{ex}

\begin{ex}\label{Example CP1 times C}
    The trivial complex line bundle $Y=\C P^1 \times \C\to \C P^1$ is a non-compact toric manifold, but there are closed holomorphic curves at infinity: $\C P^1 \times \{\textrm{constant}\}$. Thus, $Y$ cannot be made convex at infinity. It is however a Fano symplectic $\C^*$-manifold for the natural symplectic form, using the fiberwise $\C^*$-action (so acting only on the $\C$ factor, by weight one),  
    and by \cite{RZ4} we get:
    $$QH^*(Y)\cong \k[x]/(x^2-T^2)\cong QH^*(\C P^1),\;\;\; Q_{\Fi}=0,\;\;\; SH^*(Y)=QH^*(Y)_{Q_{\Fi}}=0.$$
    This generalises to $Y=\C P^m \times \C^n$, 
    giving a family of examples with $SH^*(Y)=0$ with\footnote{Compare with \cref{Intro vanishing of SH}. We clarify that $c_1(Y)$ refers as usual to $c_1(TY)$, not to $c_1($line bundle$)=0$.} $c_1(Y) \neq 0$.
\end{ex}

\subsection{Lowest order approximations of continuation maps}\label{Subsection Lowest order approximations of continuation maps}
We now try to determine the ideals $\FF^{\lambda}_{\lambda}$ rather than just estimate ranks. We try to describe the lowest order $T$-terms of $c_{\lambda}^*$.\footnote{Quantum/Floer solutions are counted with a factor $T^{\textrm{ energy}\geq 0}$. For continuation maps, under suitable conditions the factor is also $T^{\geq 0}$ but there is a subtle issue that constant solutions may not involve the lowest power (\cref{Subsection about what powers of T are used in continuation solutions}).}
We preface that continuation maps factorise well: for generic slopes $\lambda<\gamma$ there is a continuation map $\psi_{\gamma,\lambda}$ with
\begin{align}
c_{\gamma}^* =\psi_{\gamma,\lambda}\circ c_{\lambda}^*&:QH^*(Y)\to HF^*(\lambda H) \to HF^*(\gamma H),\label{Equation 1 cont intro} %
\\
\label{Equation 2 cont intro}
\textrm{If }H^*(Y)\textrm{ lies in even degrees, } \psi_{\gamma,\lambda}&:\oplus_\a H^*(\F_\a)[-\mu_{\lambda}(\F_\a)]
\to \oplus_\a H^*(\F_\a)[-\mu_{\gamma}(\F_\a)].
\end{align}

\begin{prop}\!\!\!\footnote{See 
\cref{Prop filtration precise information}
for a more general statement.}\label{Prop intro about continuation map}
Suppose $H^*(Y)$ lies in even degrees, and 
$(\lambda,\gamma)\cap \tfrac{1}{\m}\Z=\emptyset$
for weights $\m$ of $\F_{\min}$. 
Then $\mu_{\lambda}(\F_{\min})=\mu_{\gamma}(\F_{\min})$ and $\psi_{\gamma,\lambda}=\mathrm{id}_{H^*(\F_{\min})[-\mu_{\lambda}(\F_{\min})]}+(T^{>0}\textrm{-terms})$ (so $\psi_{\gamma,\lambda}|_{H^*(\F_{\min})}$ is injective).
\end{prop}

\begin{cor}\label{Cor intro about Fmin surviving}
Let $\lambda_{\min}:=1/(\textrm{maximal absolute weight of }\F_{\min})$.
If $H^*(Y)$ lies in even degrees, and $\lambda<\lambda_{\min}$, then $c_{\lambda}^*:H^*(\F_{\min})\to HF^*(H_{\lambda})$ is injective, so $\FF_{\lambda}^{\Fi}\cap H^*(\F_{\min})=\{0\}$, and in \cref{Cor intro about B filtration}:
$$
\FFFF^{\Fi} \subset \oplus_{\a\neq \min} H^*(\F_{\a};\mathbb{B})[-\mu_{\lambda}(\F_{\a})]. 
$$
\end{cor}
\begin{ex}(CSRs)\label{Example CSR unit death intro}
For weight-1 CSRs,
abbreviating $d=\dim_{\C}Y$,
we deduce\footnote{%
For weight-1 CSRs, the weights at $\F_{\min}$ are $0$ or $1$. 
In the $d$-case, \cref{Lemma Fma is jump in grading}\eqref{Shifts just before the integer time} implies $\mu_{1^-}(\F_\a) = 0$ (whereas the original shift was $\mu_\a\geq 2$ for $\a\neq \min$).
For $d\neq \dim_{\C}Y$ the statement fails in examples \cite{RZ2}. 
}
\begin{equation}\label{Intro CSR weight 1 time 1- filtration upper bound estimate}
\mathscr{F}_{\mathbb{B},1^-}^{\Fi} \subset \oplus_{\a\neq \min} H^*(\F_{\a};\mathbb{B})[-\mu_{1^-}(\F_{\a})],
\quad \textrm{in particular} \quad
\mathscr{F}_{\mathbb{B},1^-}^{\Fi}\cap H^d(Y)=\oplus_{\a\neq \min} H^{d-\mu_\a }(\F_\a),
\end{equation}
so $H^d(\F_{\min})\subset H^d(Y)$ dies last.
For weight-2 CSRs\footnote{We combine \eqref{Intro CSR weight 1 time 1- filtration upper bound estimate} and
\cref{Example CSR filtration Intro 1}. All weights of $\F_{\min}$ are equal to 1, due to a certain $\om_\C$-duality of weight spaces $H_k\leftrightarrow H_{2-k}$, and we use $H^*(\F_{\min})=H^0(Y)$.} with $\F_{\min}=\{\mathrm{point}\}$, $\Fil^{\Fi}_{{1}/{2}}=\Fil^{\Fi}_{1^-}= H^{\geq 2}(Y)$.
\end{ex}

\begin{ex}[$A_2$-singularity]\label{Example running example of intro 3}
Continuing \cref{Example running example of intro 2}, \cref{Cor intro about Fmin surviving} yields sharper results:\footnote{we also use \cref{Example CSR filtration intro} to get $\Fil^{\Fi}_{1}$ for (b-c), and we use the stability result in \cref{Intro Stability property}.}
\begin{center}
 \begin{tabular}[t]{rcllll}
    (a)\;&\strut & 
    $\Fil^{\Fi}_{1/3^-}=0,\;$&$
    \Fil^{\Fi}_{1/3} = H^2(M) = \Fil^{\Fi}_{1^-},\;$&$\Fil^{\Fi}_{1} = H^*(M).$
    \\
     (b-c)\;&\strut & 
     $\Fil^{\Fi}_{1/2^-}=0,\;$&$\Fil^{\Fi}_{1/2}=\k p_i =  \Fil^{\Fi}_{1^-},
     \;$&$\Fil^{\Fi}_{1}= H^{\geq 2}(M),
     \;$&$\Fil^{\Fi}_{2}= H^*(M)
     ,$
    \end{tabular}     
 \end{center} 
 where %
 $i=1,2$ for (c),(b),  respectively, and $p_i$
 denotes the summand 
 $H^*(p_i)[-2]\cong \k_2$ in \eqref{EqnFrankelIntro}.
 
Moreover,\footnote{ $H^*(M)=H^*(S^2_1)\oplus H^*(p_2)[-2]$ in (b). By {\AB }, $H^*(p_2)[-2]$ arises from the unstable cell of $p_2$ 
which yields  $H^2(S^2_2)$.
So $H^*(M)=H^*(S^2_1)\oplus H^2(S^2_2)$, using pull-backs to identify classes. For (c), $H^*(M)=H^2(S^2_1)\oplus H^*(S^2_2)$.
}
(b-c) are distinguished:
$\Fil^{\Fi_{(b)}}_{1/2}=H^2(S^2_2)$ but
$\Fil^{\Fi_{(c)}}_{1/2}=H^2(S^2_1)$.
\end{ex}
Call $\F_\a$ {\bf $p$-stable} if $\tfrac{1}{p}\Z_{\neq 0}\cap ($weights($\F_\a))=\emptyset$.
Call $\F_\a$ {\bf $m$-minimal} if $-m\N_{\neq 0}\cap ($weights($\F_\a))=\emptyset$ (then $\F_\a=\min H|_{\Ymc}$ for the torsion $m$-submanifold $\Ymc$ containing $\F_\a$). We seek generalisations of \cref{Prop intro about continuation map}
when passing a `non-critical period'' $p$, or a ``critical period'' $\tfrac{k}{m}$ where $m\in \mathrm{weights}(\F_\a)$:
\begin{prop}\!\!%
\label{Prop injectivity of c map on stable Fa} For $p$-stable $\F_\a$ the two indices $\mu_{p^{\pm}}(\F_\a)$ agree. The continuation $\psi_{p^-,p^+}$ in \eqref{Equation 2 cont intro} is injective on $\oplus_{(p\textrm{-stable }\F_\a)} H^*(\F_\a)[-\mu_{p^-}(\F_\a)]$.
In particular, for the smallest period $p>0$ for which $S^1$-orbits exist, the total rank $\|\FF^{\Fi}_{p}\|\leq \|H^*(Y)\|-\sum_{(p\textrm{-stable }\F_\a )} \|H^*(\F_\a)\|.$ (See also \cref{Prop filtration precise information 2}.)
\end{prop}

\begin{con}\!\!\!%
\label{Introduction conjecture about cup with Euler class}
For $\F_\a$ $m$-minimal, if $(m\lambda, m\gamma)\cap \Z=\{k\}$ is coprime to $m$,
then the lowest order $T$ term of $H^*(\F_\a)[-\mu_{\lambda}(\F_\a)]
\to H^*(\F_\a)[-\mu_{\gamma}(\F_\a)]$ in \eqref{Equation 2 cont intro} is cup-product by the Euler class of $\Ymc$, viewed as a local 
complex vector bundle over $\F_\a$. (See also \cref{Remark explaining conj about Euler class in intro}.)
\end{con}

\subsection{Naturality under full rotations}\label{Subsection Naturality under full rotations}
\begin{prop}
    For each integer $N\in \N$, there is a commutative diagram 
    $$
    \xymatrix@C=45pt@R=18pt{
    QH^*(Y)\cong HF^*(0^+H) 
    \ar@{->}[r]^-{c_{\lambda^+}} 
    \ar@{->}[d]_-{\mathcal{S}_{\Fi}^{-N}}^-{\cong} 
    &
    HF^*(\lambda^+ H)
    \ar@{->}[d]^-{\mathcal{S}_{\Fi}^{-N}}_-{\cong}  
    \\
    HF^*(N^+H)[-2N\mu]
    \ar@{->}[r]^-{\psi_{(N+\lambda)^+,N^+}} 
    &
    HF^*((N+\lambda)^+ H)[-2N\mu]
    }
    $$
    Horizontal arrows are continuations, vertical arrows are chain isomorphisms \cite[Thm.18]{R14}. Thus
    $$
    \Fil^{\Fi}_{N+\lambda} = 
    \ker c^*_{(N+\lambda)^+} = 
    \ker \left(c^*_{\lambda^+} \circ (Q_{\Fi}^{\star N} \star \cdot): QH^*(Y) \to HF^*(\lambda^+H)[2N\mu]\right).
    $$
\end{prop}

\begin{ex}\label{Example running example of intro 3.1}
For weight-$1$ CSRs, naturality implies\footnote{$\ker c^*_{2^-}=\ker (c^*_{1^-}\circ (Q_{\Fi}\star \cdot))$ by naturality, $Q_{\Fi}=\mathrm{PD}[\F_{\min}]+(T^{>0}\textrm{-terms})$ by \cref{Intro prop about Qfi (non) vanishing}, and $c^*_{1^-}(\mathrm{PD}[\F_{\min}])=\mathrm{PD}[\F_{\min}]+(T^{>0}\textrm{-terms})$ by \cref{Cor intro about Fmin surviving}. So $c^*_{1^-}(Q_{\Fi} \star 1)=\mathrm{PD}[\F_{\min}]+(\textrm{terms with }T^{>0}) \neq 0$. Now use \cref{Example CSR filtration intro}.} $\Fil^{\Fi}_{1}=\Fil^{\Fi}_{2^-}=H^{\geq 2}(Y)$, and $\Fil^{\Fi}_{2}=H^*(Y)$.
    This applies to (b-c) in \cref{Example running example of intro 3}, so the filtration is now fully determined also for those two cases.
\end{ex}

\subsection{Existence of periodic orbits, and upper bounds for the ranks of $\Fil^{\Fi}_{\lambda}$} \label{Subsection existence of orbits and upper bounds for ranks}
\strut\\[0mm]
\indent The Weinstein conjecture states that any closed contact manifold admits a closed Reeb orbit.
For Liouville manifolds $M$, if $c^*:H^*(M)\to SH^*(M)$ is not an isomorphism, the Weinstein conjecture holds for the contact hypersurface used to model $M$ at infinity, for generic contact forms \cite{Vi96} (e.g.\;for $M=T^*N$, it implies the existence of non-constant closed geodesics in $N$). This relies on deformation invariance results for $SH^*(M)$ (see \cite{BeRit20} and references therein).
It is beyond our current scope to prove an invariance statement stronger than the one in \cref{Subsection isomorphism invariance}, but the expectation of this leads one to ask what cohomological constraints the filtration imposes on the existence of periodic orbits. 

By \cref{Subsection Floer theory is possible for symplectic C-manifolds}, $HF^*(H_{\lambda})$ at chain level involves {\MB} manifolds $\F_\a$ of constant orbits and {\MB} manifolds $B_{p,\beta}$ of non-constant orbits for various $S^1$-period values $p<\lambda$. We build a filtration on the chain complex in \cref{filtrationFloer} (with technical details relegated to \cite{RZ2}), so that 
$$A:=CF^*(H_{\lambda})\subset CF^*(H_{\gamma})=:A\oplus B$$ 
is a subcomplex for any generic slopes $0<\lambda<\gamma$, where $B$ is generated by the orbits with period-values in $(\lambda,\gamma)$. On cohomology, $A\subset A\oplus B$ induces the continuation $\psi_{\lambda,\mu}:HF^*(H_{\lambda})\to HF^*(H_{\gamma})$. So $HF^*(H_{\gamma})$ is the homology of the cone of a map $\nu:B \to A[1]$.
Passing from $HF^*(H_{\lambda})$ to $HF^*(H_{\gamma})$, the $B$ kills some old classes of $A$ and creates some new classes. 
If $H^*(Y)$ lies in even degrees, then  $\|\ker \psi_{\gamma,\lambda}\|
 = \tfrac{1}{2} \|H_*(B)\|$ (abbreviating total ranks by $\|V_*\|:=\sum |V_n|$), and
\begin{equation}\label{Equation Intro lower bounds on ranks of F}
\|\FF^{\Fi}_{\gamma}\|\leq \|\FF^{\Fi}_{\lambda}\|+\tfrac{1}{2}\|H_*(B)\|.
\end{equation}
It is clear however from \cite{RZ2}
that the best approach to this problem is the spectral sequence method, rather than investigating the uncomputable $\psi_{\lambda,\mu}$ and $c^*_{\lambda}$.
When $H^*(Y)$ lies in even degrees, we get\footnote{The first inequality follows from the long exact sequence for the cone, the second inequality is a consequence of \cite{RZ2}. If we omit the middle expression, the inequality is a consequence of Poincar\'{e} duality and gradings: only odd classes of the $B_{p,\beta}$ can kill $H^*(Y)$-classes in the spectral sequence.}
\begin{equation}\label{Equation Intro upper bound on F ranks via B}
\|\FF^{\Fi}_{\gamma}\| \leq \tfrac{1}{2} \|H_*(B)\| \leq \tfrac{1}{2}\sum_{0<p<\gamma^+} \sum_{\beta} \|H^*(B_{p,\beta})\|.
\end{equation}
Thus the $\FF^{\Fi}_{\lambda}$-filtration imposes cohomological constraints on the {\MB} manifolds of non-constant $S^1$-orbits, and vice-versa in \cite{RZ2} we use the knowledge of the $H^*(B_{p,\beta})$ to find the ranks of $\FF^{\Fi}_{\lambda}$.

In \cref{Subsection rank estimates via H of Fa} we show that the $\F_\a$ and their weights give lower bounds for the ranks of $H_*(B)$. With the exception of a phenomenon we call ``vertical differentials'', it predicts the spectral sequences that we computed (by finding all $H^*(B_{p,\beta})$) in \cite{RZ2}. We illustrate this in an example in \cref{Example Intro Slodowy S32 two tables}.
\subsection{Different $\Fi$-actions and their filtrations}\label{Subsection Different fi actions}
\strut\\[0mm]\indent
In summary, given any symplectic \CC-manifold satisfying \eqref{Equation intro Psi}, and \eqref{Equation psi Xs1 is Reeb} or \eqref{Equation psi Xs1 is Reeb 2}, we have built a map
\begin{equation}
\label{filtration functor}
\{\text{contracting } \C^*\textrm{-actions on }Y\} \to \{\R_{\infty}\textrm{-ordered filtrations of }QH^*(Y)\}, \ \ \Fi \mapsto \FF_\lambda^{\Fi}. 
\end{equation}
Powers $\Fi^N$ cause slope-dilations: $\Fil^{\Fi^N}_{\lambda}=\Fil^{\Fi}_{N\lambda}$.
In \cref{Example running example of intro 3}
 we distinguished 
actions via fixed loci or the $\Fil_{\lambda}$ ideals. Even when the latter coincide, sometimes the $\R_{\infty}$-ordering distinguishes actions:
\begin{ex}($A_1$-singularity) \label{Example running example of intro 3.5}
Consider actions %
$\Fi_1:=(t^3 X, t Y, t^2 Z)$ and $\Fi_2:=(t^5 X, t Y, t^3 Z)$
on the  minimal resolution of $A_1$ singularity 
$\pi: T^*\CP^1 \fun V(XY+Z^2)\subset \C^3$.
They have the same fixed locus
$\F=p_1 \sqcup p_2$ (two fixed points on the $\CP^1$),
and the same {\MB} indices.\footnote{The weight decompositions of $p_1, p_2$ are $(3,-1),(1,1)$ for 
$\Fi_1$ and  $(5,-2),(2,1)$
for $\Fi_2$.}
We find: 
$$
\FF_{1/3}^{\Fi_1}=\k p_1 , \
\FF_{1}^{\Fi_1}=\k p_1 + \k p_2
\;\;
\text{   whereas  } \; \ \FF_{1/5}^{\Fi_2}= \k p_1, \ \FF_{3/5}^{\Fi_2}=\k p_1 + \k p_2$$ 
with no filtered levels in between, using
Corollaries \ref{Cor estimating filtration using degrees}, \ref{Cor intro about Fmin surviving} and stability property\footnote{Stability is needed for $\FF_{3/5}^{\Fi_2}=\k p_1 + \k p_2$, as other tools only imply that
$\FF_{(1/2)^-}^{\Fi_2} = \k p_1$
(\cref{Cor intro about Fmin surviving}) and 
$\FF_{3/5}^{\Fi_2}=\k p_1 + \k p_2$ (\cref{Cor estimating filtration using degrees}), without determining that the filtration changes exactly at $\lambda=3/5$.
} 
\eqref{Intro Stability property}.
So the $\R$-ordering distinguishes the actions (the {\AB } filtration does not distinguish them).
\end{ex}

Often $Y$ admits several $\C^*$-actions $\Fi$, inducing different filtrations $\Fil^{\Fi}_{\lambda}$ on $QH^*(Y)$. 
Algebraically, two filtrations by ideals $I_i,J_j$ of $QH^*(Y)$, for ordered indices $i,j$, yield a bigraded filtration $K_{ij}=I_i\cap J_j$ of ideals of $QH^*(Y)$, so $K_{ij}\subset K_{i'j'}$ for $i\leq i'$, $j\leq j'$, which leads to a collection of filtrations of $QH^*(Y)$ by ideals.
Geometrically, comparing actions is difficult: asking that two $\C^*$-actions $\Fi,\psi$ must satisfy \eqref{Equation psi Xs1 is Reeb 2} for the \emph{same} $\Psi$ is too restrictive. 
In \cite{RZ3} we compare filtrations induced by $1$-parameter subgroups of algebraic torus actions (i.e.\;filtrations for commuting $\C^*$-actions), even though the $\Psi$-maps change.

\begin{ex}[$A_2$-singularity]\label{Example running example of intro 4}
In \cref{Example running example of intro 3} call $A,B,C$ the three actions in $(a),(b),(c)$, and note that $A=B\circ C$. In \cite{RZ3} we explain why the filtrations for $B,C$ are naturally sub-ideals of the filtration for $A$, after rescaling periods by $2/3$. Here we merely wish to illustrate an example of how the filtrations relate:  we see that sometimes these two sub-ideals generate the ideal for $A$, 
$$
\Fil^{B}_{1/2} + \Fil^{C}_{1/2}
=
\k p_2 + \k p_1
= H^2(M) =  \Fil^{A}_{1/3}, \quad \qquad
\Fil^{B}_{1} + \Fil^{C}_{1} = H^{\geq 2}(M) = \Fil^{A}_{2/3},
$$
and sometimes they do not: $\Fil^{B}_{3/2} + \Fil^{C}_{3/2} = H^{\geq 2}(M)\subsetneq H^*(M) = \Fil^{A}_{1}$ (using \cref{Example running example of intro 3.1}).
\end{ex}
\begin{ex}(CSRs) A CSR $Y$ typically contains many \CC-actions, arising from compositions $\Fi:=\phi^k \circ G$ of a 
given contracting \CC-action $\phi$ and any $1$-parameter \CC-subgroup $G$ %
of a maximal torus $T \leq \mathrm{Symp}_{\phi}(Y,\om_\C)$ of symplectomorphisms commuting with $\phi$,
which is usually non-trivial.\footnote{This holds for the main examples of CSRs: quiver varieties, hypertoric varieties, resolutions of Slodowy varieties.} %
Subgroups $G$ for which $\Fi$ is contracting, for fixed $k$, constitute a convex subset $K_k$ %
of the lattice $\Lambda$ of all \CC-subgroups of $T.$ 
Via \eqref{filtration functor} we get a $K_k$-labelled family of filtrations on $H^*(Y)$ by cup-product ideals,
where $\cup_{k\in\N} K_k=\Lambda$ exhausts the whole lattice. 
We discuss this for $Y=T^*\CP^n$ in \cite{RZ2}.
\end{ex}

\subsection{$S^1$-equivariant symplectic cohomology}\label{Subsection intro S1 equiv SH}
\strut\\[0mm]\indent
There is an $S^1\times S^1$-action in Floer theory for $Y$: the $S^1$-action on $Y$ via $\Fi$ and the $S^1$-action by loop-reparametrisation.\footnote{i.e.\,acting by translation on the domains of the Hamiltonian $1$-orbits $x=x(t)$ that generate Hamiltonian Floer cohomology, with some chosen weight $b\in \Z$, so $(e^{i\theta}\cdot x)(t)=x(t-b\theta)$. In our current discussion we take $b=1$.} 
In this paper we use only the latter; we consider  the general action and various equivariant theories in \cite{RZ3}. Using only loop-reparameterisation, constant orbits give rise to ordinary cohomology rather than equivariant cohomology. In the conventions of \cite{McLR18}, $S^1$-equivariant symplectic cohomology $ESH^*(Y,\Fi)$ is a $\ku$-module (not a ring) with a $\ku$-module homomorphism $$Ec^*:H^*(Y)\otimes_{\k}\mathbb{F} \to ESH^*(Y,\Fi),$$ 
where $\mathbb{F}:=\kuu /u\k [\![u]\!]\cong H_{-*}(\C\P^{\infty})$, and $u$ is the equivariant parameter in degree two. At chain level, each $1$-orbit gives a copy of the $\ku$-module $\mathbb{F}$. When $c_1(Y)=0$, we have\footnote{For the same grading reasons that prove $SH^*(Y,\Fi)=0$ in \cref{Intro vanishing of SH}.} $ESH^*(Y,\Fi)=0.$

\begin{thm}\label{Theorem curlyE filtration for ESH}
There is an $\R_{\infty}$-ordered filtration by graded $\ku$-submodules of $H^*(Y)\otimes_{\k}\mathbb{F}$,
\begin{equation*}%
E\FF^{\varphi}_p :=\bigcap_{\mathrm{generic}\,\lambda>p} \ker \left(Ec_\lambda^*:H^*(Y)\otimes_{\k}\mathbb{F}\to EHF^*(H_{\lambda})\right), \qquad E\Fil_{\infty}^{\Fi}:=H^*(Y) \otimes_{\k} \mathbb{F},
\end{equation*} 
where $Ec_\lambda^*$ is the equivariant continuation map, a grading-preserving $\ku$-linear map.

In general, $\Fil^{\Fi}_\lambda \subset E\Fil^{\Fi}_\lambda$. If $H^*(Y)$ lies in even degrees (e.g.\;CSRs), then $\Fil^{\Fi}_\lambda = H^*(Y)\cap E\Fil^{\Fi}_\lambda.$
\end{thm}

\begin{rmk}
The reason $\Fil^{\Fi}_\lambda \subset E\Fil^{\Fi}_\lambda$ may be larger is that the image $c_{\lambda}^*(x)$ may lie in the $\k$-span of the images of the higher-order maps $\delta_j$, $j\geq 1$, involved in the equivariant Floer differential $d=\delta_0 + \sum u^{j}\delta_j $ (for example $\delta_1=\Delta$ induces the BV-operator on $SH^*(Y,\Fi)$, see \cite{Sei08}). In that case, $Ec_{\lambda}^*(x)=0$.
\end{rmk}

We show in \cite{RZ2} that the $S^1$-equivariant spectral sequence for $ESH^*_+(Y,\Fi)$ often collapses on the $E_1$-page, making this approach computationally easier, at the cost of analysing the presence of repeated copies of generators due to the $u^{-j}$ parameters.
We illustrate an example in \cref{Example Slodowy S32 using mulambda Falpha}.

\subsection{Filtrations in the context of Fukaya categories}
We conclude with a brief comment about a categorical version of the filtration.
Let us make the non-trivial\footnote{A starting  point is to mimic constructions from \cite{ritter2017monotone}, by replacing the radial coordinate $R$ by the moment map $H$, and using \cref{Remark max principle for Fukaya cat} to ensure that Floer solutions do not escape to infinity.} assumption that a compact Fukaya category $\mathcal{F}(Y,\Fi)$ and a wrapped Fukaya category $\mathcal{W}(Y,\Fi)$ are defined for $Y$. 
By \cite{ritter2017monotone}, one expects the Hochschild (co)homologies of $\mathcal{F}(Y,\Fi)$ and $\mathcal{W}(Y,\Fi)$ to be modules respectively over $QH^*(Y)$ and $SH^*(Y,\Fi)$. This would imply that $\mathcal{W}(Y,\Fi)$ (but not necessarily $\mathcal{F}(Y,\Fi)$) is homologically trivial\footnote{For Liouville manifolds $M$, $SH^*(M)=0$ implies that there are no closed \textit{exact} Lagrangian submanifolds in $M$. However, for our spaces $Y$, $\omega$ is typically \textit{not} exact, so there is no notion of \emph{exact} Lagrangian, and no known obstructions to the existence of arbitrary closed Lagrangians $j:L\hookrightarrow M$ (beyond the homological self-intersection, $j_*[L]\cdot j_*[L]=-\chi(L)$).} when $SH^*(Y,\Fi)=0$, so we are interested in slope-filtration invariants analogous to the $\ker c^*_{\lambda}\subset QH^*(Y)$.

For the compact Fukaya category, a filtration by $QH^*(Y)$-submodules on Hochschild homology $\mathrm{HH}_*(\mathcal{F}(Y))$ arises from the open-closed string map, so considering kernels of the composite $c_{\lambda}^*\circ \mathcal{OC}:\mathrm{HH}_*(\mathcal{F}(Y))\to QH^*(Y) \to HF^*(H_{\lambda}).$
In the wrapped setup, one needs to decide which non-compact Lagrangian submanifolds $L\subset (Y,\omega)$ to allow.\footnote{for example including the requirement that $L$ at infinity is $\nabla H$-invariant (by comparison: in the Liouville manifold setup, one requires $L$ to be ``conical'' at infinity \cite{ritter2017monotone}).} 
For such $L,L'$, consider the acceleration map $$f^*_{\lambda}:HF^*(L,L') \to HW^*(L,L';H_{\lambda}),$$
where the latter module at chain level, $CW^*(L,L';\lambda H)$, is generated by time $1$-orbits for $H_{\lambda}$ from $L$ to $L'$ rather than intersections points $L\cap L'$.
By \cite{ritter2017monotone}, one expects this to be a $QH^*(Y)$-module map, so $\ker f^*_{\lambda}$ filters $HF^*(L,L')$ in a way analogous to how $\ker c^*_{\lambda}$ filtered $QH^*(Y)$. 
As $\varinjlim HW^*(L,L';H_{\lambda})$ is expected to be a module over $SH^*(Y,\Fi)$, if $SH^*(Y,\Fi)=0$ then $f_{\lambda}^*$ vanishes for large enough $\lambda$.
\begin{ex}
The construction of the $SH^*(Y)$-action on $HH_*(\mathcal{W}(Y,\Fi))$, 
at its most basic level, shows that $q \in QH^*(Y)$ acting by $c_{\lambda}^*(q)\in HF^*(H_{\lambda})$ on $HF^*(L,L')$, with output in $HW^*(L,L';H_{\lambda})$, agrees with composing $f_{\lambda}^*$ with the action $q\cdot$ on $HF^*(L,L')$. Thus, $\FF^{\Fi}_{\lambda}\cdot HF^*(L,L') \subset \ker f_{\lambda}^*$.

A simple illustration of this, for $Y=\C P^1\times \C$ as in \cref{Example CP1 times C}, letting $L=S^1 \times \R$ and denoting $L_{\lambda}$ its image under the time $1$-flow of $H_{\lambda}$: $HF^*(L,L):=HF^*(L,L_{0^+})\cong H^*(S^1)\subset \ker f_{1^+}^*$, as $c_{1^+}^*=0$.
\end{ex}

To upgrade the above to categorical level, for compact Lagrangians, we would proceed as follows.
Given a Hochschild cycle $\gamma=[y]\in \mathrm{HH}_*(\mathcal{F}(Y,\Fi))$, consider the acceleration functor 
\vspace{-1mm}
$$f:\mathrm{HH}_*(\mathcal{F}(Y,\Fi)) \to \mathrm{HH}_*(\mathcal{W}(Y,\Fi)).\vspace{-1mm}$$
Suppose $f(y)$ is exact (e.g.\,always if $\mathrm{HH}_*(\mathcal{W}(Y,\Fi))=0$), so at chain level the bar differential $b(z)=y$ for a Hochschild chain $z$. A Hochschild chain is a finite sum of words of type $\underline{x_{n+1}}\otimes x_n \otimes \cdots \otimes x_1$, for various lengths $n$, where each $x_i\in CW^*(L_{i-1},L_{i};\lambda_i H)$ lies in a chain complex as described above (and $L_{n+1}=L_0$). We may therefore associate to $z$ a slope-value $\lambda(z):=\max \sum \lambda_i$, maximising over the finitely many words in $z$ the total slope value $\sum \lambda_i$ of each word. We then define a filtration of Hochschild homology by $p\in (0,\infty]$ by measuring what slope is needed to witness the triviality of $f(y)$:
$$
\mathcal{F}^{p}:=\bigcap_{\lambda>p}\{ [y]\in \mathrm{HH}_*(\mathcal{F}(Y,\Fi)):
\exists z\in \mathrm{CC}_*(\mathcal{W}(Y,\Fi))\textrm{ with } f(y)=b(z) \textrm{ and } \lambda(z)\leq \lambda\}.
$$
This filters  $\mathrm{HH}_*(\mathcal{F}(Y,\Fi))$ by $QH^*(Y)$-submodules,\footnote{By inspecting the construction of the chain map $\psi_q$ that defines the $QH^*(Y)$-action on the Hochschild chain complex by $q\in QH^*(Y)$ \cite[Sec.8]{ritter2017monotone}, one sees that the total slope $\sum \lambda_i$ of words will not increase. Thus $f(\psi_q (y))=\psi_q(f(y))=\psi_q(b(z))=b(\psi_q(z))$ and $\lambda(\psi_q(z))\leq \lambda(z)\leq \lambda$.} with $p\in \R\cup \{\infty\}$, letting $\mathcal{F}^{\infty}:=\mathrm{HH}_*(\mathcal{F}(Y,\Fi))$.\\[2mm]
\noindent \textbf{Acknowledgements.} 
We thank Paul Seidel for very valuable conversations, particularly in the initial stages of this project. We thank Mohammed Abouzaid, %
Dylan Cant, Ivan Smith, Jack Smith, and Nicholas Wilkins 
for helpful conversations. 
We thank Leonid Polterovich for encouraging us to rephrase our results in terms of persistence modules (\cref{Subsection persistence}).
The first author is grateful to the Mathematics Department at Stanford for their hospitality during the author's sabbatical year.
Early stages of this work are in the second author’s DPhil thesis \cite{FZ20}, and he acknowledges support from Oxford University, St Catherine's College, University of Edinburgh and ERC Starting Grant 850713 – HMS.
\section{Example: the Slodowy variety $\mathcal{S}_{32}$}\label{Example Slodowy S32 using mulambda Falpha}  

\subsection{Definition of $\mathcal{S}_{32}$}
We mention for illustration one of many examples from {\PartII }.
The partition $(3,2)$ of $n=5$ defines a standard nilpotent $5\times 5$ Jordan canonical form $e$, with Jordan block sizes $3$ and $2$. This determines a standard $\mathfrak{sl}_2$-triple $(e,f,h)$ in $\mathfrak{sl}_5$, where $h$ is a diagonal matrix with eigenvalues $h_i$, in our case $2,0,-2,1,-1$ \cite[Sec.5.2.1]{FZ20}.\footnote{There is a typo in that thesis: the diagonal entries of $h_k$ are $h_0^k,\ldots,h_{k-1}^k$ starting the numbering at $i=0$.}
This determines a Slodowy slice
$S_{e}=e + \text{ker} (\ad f) \subset \sl_5.$
The nilpotent cone $\mathcal{N}=\{5\times 5 \textrm{ nilpotent matrices} \}\subset \mathfrak{sl}_5$ determines the Slodowy variety $\mathcal{S}_e:=S_e\cap \mathcal{N}$.
Now consider the Springer resolution $\nu: \widetilde{\mathcal{N}}\to \mathcal{N}$, whose fibre $\B^x$ %
over $x\in \mathcal{N}$ consists of all complete flags $\f=\{0\subset F_1 \subset \cdots \subset F_4 \subset \C^5\}$ satisfying $x  F_i\subset F_{i-1}$.
In fact, $\widetilde{\mathcal{N}}\cong T^*\mathcal{B}$ where $\mathcal{B}$ is the complete flag variety for $\C^5.$
The Springer resolution restricts to a resolution $Y:=\widetilde{S}_e=\nu^{-1}(\mathcal{S}_e)\stackrel{\nu}{\to} \mathcal{S}_e,$
with $\dim_{\C} Y=\dim_{\C}\mathcal{S}_e= 4$.
It admits the Kazhdan $\C^*$-action: $t\cdot (x,\f)=(t^2\mathrm{Ad}(t^{-h})x,t^{-h}\f)$, where $t^{-h}$ is the diagonal matrix with entries $t^{-h_i}$, thus $t^2\mathrm{Ad}(t^{-h})x$ is explicitly $t^{2+h_j-h_i}x_{ij}$ on the entries $x_{ij}$ of $x$. 
It turns out that $Y$ is a weight-2 CSR, so the  Maslov index $\mu=\dim_{\C}Y=4$.

\subsection{The fixed loci $\F_\a$, their weights, and their indices $\mu_{\lambda}(\F_\a)$}
In {\PartII } we show that the fixed components $\F_\a$ are all points in this case (thus the moment map $H$ of the $S^1\subset \C^*$ action is a Morse function) %
and we compute their weights:
$\f_{big}$: $(3,\!3,\!-1,\!-1)$, 
$\f_p,\f_w$: $(5,\!3,\!-3,\!-1)$, 
$\f_j',\f_y'$: $(3,\!3,\!-1,\!-1),$ 
$\f_j^3,\f_y^3,\f_j^1,\f_y^1$: $(3,\!1,\!1,\!-1)$, 
$\f_{min}$: $(1,\!1,\!1,\!1).$
We calculate the indices $\mu_{\lambda}(\F_\a)$ at generic slopes $\lambda= T^+$ slightly above $T=0,$ $1/5$, $1/3$, $2/5$, $3/5$, $2/3$, $4/5$ and $1$. These are the only periods when the filtration can possibly change, due to  \cref{Prop filtration is stable}. The following table shows how these indices vary.

{\small
\begin{center}
 \begin{tabular}[t]{rcccccccccc}
 \strut & 
 $\f_{big}$
 &
 $\f_p$
 &
 $\f_w$
 &
 $\f_j'$
 &
 $\f_y'$
 &
$\f_j^3$
 &
 $\f_y^3$
 &
 $\f_j^1$
 &
 $\f_y^1$
 &
$\f_{min}$
 \\
    $H^*(Y):$ 
    &$4$
    &$4$
    &$4$
    &$4$
    &$4$
    &
    $2$
    &$2$
    &$2$
    &$2$
    &$0$    
    \\ 
    $HF^*(\tfrac{1}{5}^+H):$ 
    &$  4$
    &$  2$
    &$  2$
    &$  4$
    &$  4$
    &
    $  2$
    &$  2$
    &$  2$
    &$  2$
    &$  0$    
     \\ 
    $HF^*(\tfrac{1}{3}^+H):$ 
    &$  0$
    &$  2$
    &$  2$
    &$  0$
    &$  0$
    &
    $  0$
    &$  0$
    &$  0$
    &$  0$
    &$  0$ 
     \\ 
    $HF^*(\tfrac{2}{5}^+H):$ 
    &$  0$
    &$  0$
    &$  0$
    &$  0$
    &$  0$
    &
    $  0$
    &$  0$
    &$  0$
    &$  0$
    &$  0$ 
     \\ 
    $HF^*(\tfrac{3}{5}^+H):$ 
    &$  0$
    &$  {-2}$
    &$  {-2}$
    &$  0$
    &$  0$
    &
    $  0$
    &$  0$
    &$  0$
    &$  0$
    &$  0$ 
     \\ 
    $HF^*(\tfrac{2}{3}^+H):$ 
    &$  -4$
    &$  {-2}$
    &$  {-2}$
    &$  {-4}$
    &$  {-4}$
    &
    $  {-2}$
    &$  {-2}$
    &$  {-2}$
    &$  {-2}$
    &$  0$ 
     \\ 
    $HF^*(\tfrac{4}{5}^+H):$ 
    &$  -4$
    &$  {-4}$
    &$  {-4}$
    &$  {-4}$
    &$  {-4}$
    &
    $  {-2}$
    &$  {-2}$
    &$  {-2}$
    &$  {-2}$
    &$  0$ 
     \\ 
    $HF^*(1^+H)\cong H^*(Y)[8]:$ 
    &$  -4$
    &$  {-4}$
    &$  {-4}$
    &$  {-4}$
    &$  {-4}$
    &
    $  {-6}$
    &$  {-6}$
    &$  {-6}$
    &$  {-6}$
    &$  {-8}$ 
    \end{tabular}     
 \end{center}  
 }

In the table, each number indicates a copy of $\k$ placed in the indicated degree. At the start we get the {\MB } indices $\mu_\a$ at the 10 fixed points, confirming \eqref{EqnFrankel}: $H^*(Y)= \k_4^5 \oplus \k_2^4 \oplus \k_0 $.

\subsection{The ranks of the $\FF^{\Fi}_{\lambda}$ filtration}
Rank considerations, using \cref{Cor estimating filtration using degrees}, imply the following:
 \begin{equation*}
|\Fil^{\Fi}_{1/5}\cap H^4(Y)|\geq 2,
\;
\Fil^{\Fi}_{1/3}\supset H^4(Y),
\;
|\Fil^{\Fi}_{1/3}\cap H^2(Y)|\geq 2,
\;
\Fil^{\Fi}_{2/5}\supset 
H^2(Y)\oplus H^4(Y),
\;
\Fil^{\Fi}_{1}=
H^*(Y).
\end{equation*}			
\cref{Cor CSR when unit dies} refines this: 
$\k \f_{min}$
survives until time $1^-$.
All fixed loci except for $F_p,F_w$ are $\tfrac{1}{5}$-stable, so $|\ker c^*_{(1/5)^+}|\leq 2$ by \cref{Prop filtration precise information 2}.
Altogether:\footnote{Although we conjecture that $\Fil^{\Fi}_{1/5}\cap H^4(Y)$ is spanned by $F_p,F_w$, this does not follow from \cref{Prop filtration precise information 2} because the isomorphism $c_{(1/5)^-}^*$ may not preserve the degree $2$ classes.}
$$
|\Fil^{\Fi}_{1/5}|=|\Fil^{\Fi}_{1/5}\cap H^4(Y)| = 2,
\quad 
\Fil^{\Fi}_{1/3}=\k_2^r\oplus  H^4(Y),
\quad 
\Fil_{2/5}^{\Fi}=\Fil^{\Fi}_{1^-} =
H^2(Y)\oplus H^4(Y),
\quad
\Fil^{\Fi}_{1}=
H^*(Y).
$$

The summand $\k_2^r$ of rank $r\in\{2,3,4\}$ is unknown because higher order $T$ contributions to $c^*_{1/5^+}$ and $c^*_{1/3^+}$ may allow $\f_j^3,\f_y^3,\f_j^1,\f_y^1$ to have non-trivial images in the summand $(H^*(\f_p)\oplus H^*(\f_w))[-2]$. 

\subsection{An approximate prediction of the spectral sequence}
\label{Example Intro Slodowy S32 two tables}
We now illustrate the method from
\cref{Subsection existence of orbits and upper bounds for ranks} for $Y=\mathcal{S}_{32}$. On the left, we will tabulate the rank of $HF^*(H_{\lambda})\cong \oplus H^*(\F_\a)[-\mu_{\lambda}(F_\a)]$ by degree ($\mathrm{rk}\,H^*(\F_\a)=1$ as $\F_\a\cong \mathrm{point}$). On the right, we show lower bounds for the ranks of the $E_1$-page spectral sequence for $SH^*(Y,\Fi)$ that can be deduced combinatorially from the left table (without any knowledge of the Morse-Bott manifolds $B_p=\cup_{\beta} B_{p,\beta}$). The column for slope $\lambda=0^+$ is $H^*(Y)$.
\[
{\small
\renewcommand{\arraystretch}{1.1}
\begin{array}{|c||c|c|c|c|c|c|}
\hline
\textrm{degree} 
&
\lambda=0^+
& \tfrac{1}{5}^+
& \tfrac{1}{3}^+
& \tfrac{2}{5}^+
& \cdots
\\[0.2mm]
\hline
\hline
4 & 5 & 3 & 0 & 0 %
& \cdots 
\\
\hline
3 & 0 & 0 & 0 & 0 %
&  \cdots
\\
\hline
2 & 4 & 6 & 2 & 0 %
& \cdots 
\\
\hline
1 & 0 & 0 & 0 & 0 %
&  \cdots
\\
\hline
0 & 1 & 1 & 8 & 10 %
& \cdots
\\
\hline
\cdots & \cdots & \cdots & \cdots & \cdots & \cdots %
\\
\hline
\end{array}
\qquad 
\begin{array}{|c||c|c|c|c|c|c|}
\hline
\textrm{degree} 
&
\lambda=0^+
& \tfrac{1}{5}^+
& \tfrac{1}{3}^+
& \tfrac{2}{5}^+
& \cdots
\\[0.2mm]
\hline
\hline
4 & 5 & 0 & 0 & 0 %
& \cdots 
\\
\hline
3 & 0 & 2 & 3 & 0 %
& \cdots 
\\
\hline
2 & 4 & 2 & {\bf 0} & 0 %
& \cdots 
\\
\hline
1 & 0 & 0 & {\bf 4} & 2 %
& \cdots 
\\
\hline
0 & 1 & 0 & 7 & 2 %
& \cdots 
\\
\hline
\cdots & \cdots & \cdots & \cdots & \cdots & \cdots %
\\
\hline
\end{array}
\qquad 
\renewcommand{\arraystretch}{1.0}
}
\]
We illustrate how to go from the left table to the right. Consider the increase in slope from $0^+$ to $\tfrac{1}{5}^+$: in degree four, 2 fixed points drop from degree 4 to 2. Thus, the $\lambda=\tfrac{1}{5}^+$ column must have 2 generators in degree 3 (to kill the old classes) and 2 generators in degree 2 (to create the new classes). Similarly, one inductively builds each column in the second table. For $\lambda=\tfrac{1}{3}^+$, the boldface ranks 0 and 4 (in degrees 2 and 1) are not known even by spectral sequence methods: the Morse cohomology approximation gives a certain sum $\oplus H^*(B_{p,\beta})$ with ranks 2 and 6 respectively, but we do not know whether the local Floer cohomology is further corrected by Floer solutions that are not approximated by Morse trajectories (this discrepancy causes the second inequality in \eqref{Equation Intro upper bound on F ranks via B}). We abusively refer to this phenomenon as ``vertical differentials'', even though such differentials belong to the $E_0$-page, not the $E_1$-page. 
Those boldface ranks (0,4) could therefore also be (1,5) or (2,6), on the actual $E_1$-page.

This $E_1$-page approximation is so accurate because all odd classes of a column must kill classes strictly to the left on some $E_{\geq 1}$-page, since columns with slope values $<\lambda$ give a spectral sequence converging to $HF^*(H_{\lambda})$ which lies in even degrees by \eqref{Equation intro HF lambda H splits}. 
Newly created even degree classes cannot kill classes to the left, as there are no odd classes there.
The only ambiguity in the approximate $E_1$-page arises if a column simultaneously kills $r$ degree $2d$ classes and creates $r$ degree $2d$ classes: this contributes $r$ generators to the column in degrees $2d-1$ and $2d$ respectively. E.g.\,$(0,4)$ above is $(0,4)+r(1,1)$ for some $r\in \{0,1,2\}$ on the actual $E_1$-page.
We can bound how bad this ambiguity may be: $r$ is bounded above by the total number of classes to the left in degree $2d$, by rank considerations.

The lower bounds for $|\FF^{\Fi}_{\lambda}|$ in this approach coincide with \cref{Cor estimating filtration using degrees}. We emphasize that effective upper bounds only arise thanks to our second paper \cite{RZ2} which yields the right-most bound in \eqref{Equation Intro upper bound on F ranks via B}. E.g.\,even the simplest top entries (0,2) of the $\tfrac{1}{5}^+$ column in the right-table is $(0,2)+r(1,1)$ for $r\in \{0,1,2,3\}$ on the actual $E_1$-page; the natural way to see $r=0$ is to verify that $\sqcup B_{1/5,\beta}=S^1\sqcup S^1$.
\subsection{The {\MBF} spectral sequences}
For the sake of comparison, we now show the two $E_1$-pages of the {\MBF} spectral sequences that converge to $SH^*(Y,\Fi)=0$ and to $S^1$-equivariant symplectic cohomology $ESH^*(Y,\Fi)=0$;
{\PartII} will discuss these in more detail. 

\begin{center}
\includegraphics[scale=0.7]{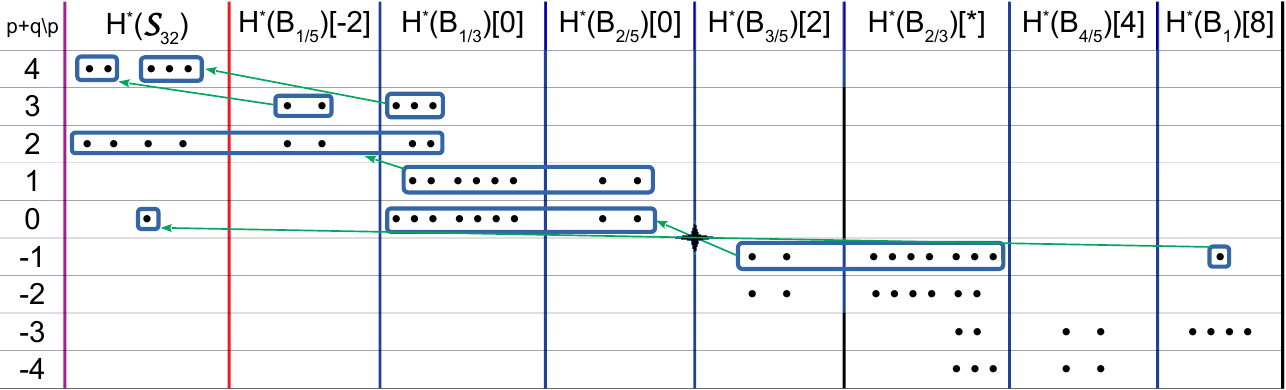}
\\[2mm]
\includegraphics[scale=0.66]{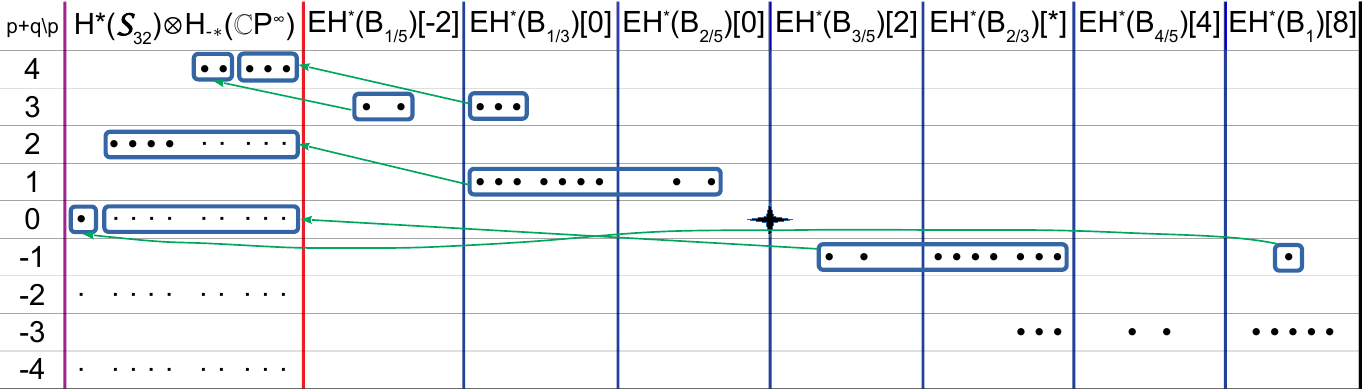}
\end{center}

In the main columns, each dot contributes $1$  to the rank. %
In the first picture, the arrows indicate how edge differentials on the $E_1$ and higher pages must kill $H^*(Y)$. In the second picture again all classes must cancel, so arrows go from odd classes to even classes in degree one higher.
The $0$-th column $H^*(Y;\k)\otimes_{\k} \mathbb{F}$ consists of copies of the $\k[u]$-module $\mathbb{F}$ mentioned in \cref{Subsection intro S1 equiv SH 2}, so the smaller dots indicate $u^{-j}\cdot$\,(generator) for $j\geq 1$.
The other columns
are substantially different in the equivariant case because the $S^1$-reparametrisation action on $1$-orbits means that instead of the ordinary cohomology $H^*(B_{k/m})[-\mu(B_{p,\c})]$ of the slices \eqref{Equation Bkm slices intro} ({\MB } manifolds of period-$k/m$ $S^1$-orbits), we have
\begin{equation}\label{EH is in odd degree intro}
EH^*(B_{p,\c})[-\mu(B_{p,\c})] \cong
H^*\left(B_{p,\c}/S^1\right)[-1-\mu(B_{p,\c})].
\end{equation}
We explain in \cite{RZ2} that in many examples, including the current example, 
\eqref{EH is in odd degree intro} lies in odd degrees
and $EH^*(Y)$ lies is even degrees;
so the spectral sequence
for $ESH^*_+(Y,\Fi)$ collapsed
and we can read off $H^*(Y)\otimes \mathbb{F}$ from it, in fact one can also recover $SH^*_+(Y)\cong \ker (u:ESH^*_+(Y)\to ESH^*_+(Y))$. 

The spectral sequences 
show clearly that
$|\Fil^{\Fi}_{1/5}| = 2$: precisely 2 classes in the 0-th column get killed by the 1-st column.
However, %
in both cases, for different reasons, one cannot determine the rank $r\in\{2,3,4\}$ above. In the first spectral sequence, this is because of the possible presence of residual $E_0$-page ``vertical differentials'' in the $B_{1/3}$-column arising from an energy-spectral sequence. In the second spectral sequence, this is because the spectral sequence forgets the original $u$-action: we conjecturally expect the last four dots in the $B_{1/3}$-column to kill the first four dots in the $H^2(Y)$ slot, but we cannot exclude that the $u$-action on the $B_{1/3}$-column might hit the $B_{1/5}$-column.
\section{Symplectic $\C^*$-manifolds: definition and properties}%
 \label{Sec mfds with holo C* acn}
\subsection{Moment map, fixed locus, convergence points, and contracting actions}
\begin{de}\label{Def:CstarManifold}
A \textbf{symplectic $\C^*$-manifold} $(Y,\omega,\J,\Fi)$ consists of a connected\footnote{This condition is inessential, but we assume it for convenience.} symplectic manifold $(Y,\omega)$, a choice of $\omega$-compatible almost complex structure $\J$ on $Y$ (so $g(\cdot, \cdot):=\omega(\cdot,\J\cdot)$ is a Riemannian metric), and a {\bf \ph }\footnote{{\ph}ity of $\Fi$ means that the differential of $\varphi:\C^*\times Y \to Y$ is $((i,\J),\J)$-holomorphic.}
$\C^*$-action $\Fi$ on $(Y,\J)$, so that its $S^1$-part is Hamiltonian.
A \textbf{symplectic $\C^*$-submanifold} $W\subset Y$ is a connected $\C^*$-invariant submanifold which is $I$-\ph, 
so $\omega$-symplectic.\footnote{$\om(v,\J v)=g(v,v)\neq 0$ for $v\neq 0$.} Thus $(W,\omega|_{W},\J|_{W},\Fi|_{W})$ is a symplectic $\C^*$-manifold. 

\end{de} 

\begin{rmk}\label{Rmk about S1 actions}
If we only had an $\J$-{\ph }\footnote{meaning $(\psi_t)_*\circ \J = \J \circ (\psi_t)_*$ for $t\in \R$.} $S^1$-action, 
$\psi_t:=\varphi_{e^{2\pi it}}:Y \to Y$, then this locally extends to a $\C^*$-action. The Lie derivative of its vector field $X_{S^1}$ satisfies $\mathcal{L}_{X_{S^1}}(\J)=0$, so $X_{S^1}$ and $X_{\R_+}:=-\J X_{S^1}$ commute.\footnote{$[X_{S^1},\J X_{S^1}]=\mathcal{L}_{X_{S^1}}(\J X_{S^1})=\J \mathcal{L}_{X_{S^1}}(X_{S^1})=0$.} So we get a partially defined {\ph } map $\varphi:\C^*\times Y \to Y$, $\varphi_{e^{2\pi(s+it)}} = \mathrm{Flow}_{X_{\R_+}}^s\circ \psi_t$. If  $X_{\R_+}$ is integrable then this $\varphi$ becomes a globally defined $\C^*$-action.
\end{rmk}

Instead of ``{\ph},'' we will abusively just use the term ``holomorphic'' from now on.
If $H^1(Y)=0$, asking for the $S^1$-action to be Hamiltonian is equivalent to it being symplectic.\footnote{This is a standard fact that follows from the Cartan formula, see e.g. \cite[footnote 2 in Sec.2.6]{McLR18}.}

Let $\Fi^y:\C^* \to Y$, $\tau \mapsto \tau \cdot y$, denote the $\C^*$-action acting on $y\in Y$. 
By definition, $\varphi^y$ is $(i,\J)$-holomorphic.
We call $\R_+$ the subgroup $\R\hookrightarrow \C^*$, $s\mapsto e^{2\pi s}$, and $S^1$ the subgroup $\R/\Z \hookrightarrow \C^*$, $s\mapsto e^{2\pi i s}\subset \C^*$. The flows of the {\bf $\R_+$-part} and {\bf $S^1$-part} of $\Fi$ are given by the vector fields %
\begin{equation}\label{eqn XR+ and XS1}
X_{\R_+}(y) = \partial_s|_{s=0} \,\Fi^{y}(e^{2\pi s}) 
  \quad \textrm{ and }\quad
X_{S^1}(y)= \partial_s|_{s=0}
\,\Fi^{y}(e^{2\pi is}).
\end{equation}
As we parameterise $S^1$ by $s\in \R/\Z,$ the $S^1$-flow is $1$-periodic. The {\bf moment map} $H$ is 
the smooth function \eqref{Equation intro moment map} whose Hamiltonian vector field $X_H=X_{S^1}$ generates the $S^1$-action (where $\omega(\cdot,X_H)=dH$).

\begin{lm}\label{Lemma:XR+ is nabla h}
$X_{\R_+}=\nabla H$ and $X_{S^1}=X_H=\J\nabla H = \J X_{\R_+}$.
\end{lm}
\begin{proof}
By definition, $\omega(\cdot,\J\nabla H)=g(\cdot, \nabla H)=dH=\omega(\cdot,X_H)$ so $X_H=\J\nabla H$.
As $\varphi^{y}:\C^*\to Y$ is holomorphic, 
$X_{S^1}(y)=d\varphi^{y}\circ 2\pi i = \J (d\varphi^{y} \circ 2\pi) = \J X_{\R_+}.$
\end{proof}

\begin{lm}\label{Lemma Fixed loci agree}
The fixed locus of the $\C^*$-action equals the fixed locus of the $S^1$-action, 
$$
\F:=\mathrm{Fix}(\Fi) = \mathrm{Fix}(\Fi|_{S^1})= \mathrm{Crit}(H) = \mathrm{Zeros}(X_{\R_+}).  %
$$
This is a closed subset of $Y$, and if the $-X_{\R_+}$ flow from $y\in Y$ converges then its limit %
lies in $\F$. 
\end{lm}
\begin{proof}
Since $X_{\R_+}=\J X_{S^1}$, a fixed point of the $S^1$-action is automatically a fixed point of the $\C^*$-action, so the two fixed loci agree and coincide with $\mathrm{Crit}(H)=\mathrm{Zeros}(X_{S^1})=\mathrm{Zeros}(X_{\R_+})$ since $X_{S^1}=X_H$.
\end{proof}

\begin{lm}\label{LemmaSec2FixedLocus}
$\F$ is %
a symplectic $\J$-{\ph } submanifold, whose connected components $\F_\a$ %
are the %
critical submanifolds of the {\MB } function $H$. The {\MB} indices and {\MB} coindices of the $\F_\a$ are all even.
\end{lm}
\begin{proof}
By Kobayashi \cite{Ko58}, $\mathrm{Fix}(\Fi|_{S^1})$ is a smooth oriented even-dimensional submanifold: by averaging, we can pick an $S^1$-invariant Riemannian metric, so the $S^1$-action yields a Killing vector field as required in \cite{Ko58}. We recall the short proof. The $S^1$-invariance of the metric ensures that at a fixed point $\p$, a neighbourhood of $\p$ can be $S^1$-equivariantly identified with a neighbourhood of $0\in T_\p Y$, using the linearised $S^1$-action on $T_\p Y$ \cite[Ch.VI, Thm.2.2]{Bre72}.
The fixed locus near $\p$ is parametrised via $\exp_p$ by the vector subspace $H_0$ of all $v\in T_\p Y$ with $d_\p\varphi\cdot v = 0$. The holomorphicity assumption on $\varphi$ ensures that $\J H_0=H_0$. Viewing $T_\p Y$ as an $S^1$-representation via the action of $d_\p \varphi$, the orthogonal complement of $H_0\subset T_\p Y$ decomposes orthogonally into a direct sum of two-dimensional planes $T_i$ which are rotated with speed $w_i\in \Z.$
Thus, as $Y$ is even-dimensional and oriented, so is $H_0$.
Although we do not assume that $Y$ is \Kh, Frankel's argument that the Hamiltonian is {\MB } in \cite[Lem.1]{Frankel59} nevertheless applies in our case, by letting  $\phi$ (in his notation) be our Hamiltonian $H$, $X$ be our $S^1$-vector field, and using our holomorphicity assumption for $\varphi$ to get the equation $\J S=S\J $ needed in his proof. 
By definition, the {\MB } index $\mu_\a$ of a component $\F_\a\subset \F$ equals the sum of the real dimensions of the planes $T_i$ 
which are rotated with negative speeds $w_i<0$. The coindex $\dim_{\R}Y - \dim_{\R}\F_a - \mu_\a$ is also even as $Y,\Fa$ are symplectic, hence have even real dimension.
\end{proof}

\begin{rmk}\label{Rmk hyperbolic fixed points}
   The critical points of the {\MB } function $H$ are hyperbolic (the flow on the $T_j$ complex line above is $e^{2\pi i w_j}$ which contributes two eigenvalues of the differential of the flow of modulus equal to $2\pi |w|\in 2\pi \Z$, so not $1$), 
so the unstable/stable manifolds are non-properly embedded Euclidean spaces (a classical theorem by Hadamard--Perron and Hartman--Grobman). 
The non-properness occurs if the flowlines converge to a critical point that does not belong to that unstable/stable manifold, or if a flowline goes to infinity in finite time (which can only occur if $H$ is not proper).
\end{rmk}

\begin{de}
If $\varphi_t(y)$ has a limit as $\C^* \ni t\to 0$, we say $y\in Y$ has \textbf{convergence point} $y_0$, where  $$\displaystyle \yinfty:=\lim_{\C^* \ni t\fun 0} t\cdot \x \;\in \F.$$
\end{de}

\begin{lm}
A $y\in Y$ has a convergence point if and only if the $-\nabla H$ flowline from $y$ converges. 
The subset of points $y\in Y$ with a given convergence point $y_\iinfty\in \F$ is the stable manifold $W^s_{-\nabla H}(y_\iinfty)$. 
\end{lm}
\begin{proof}
This essentially follows from \cref{Lemma Fixed loci agree}.
The $-\nabla H=-X_{\R_+}$ flowline $\gamma(s)$ with $\gamma(0)=y$ corresponds to the action $\gamma(s)=e^{-2\pi s}\cdot y$ by $-s\in \R_+$. So the convergence assumption on the flowline is equivalent to assuming $\lim (t\cdot y)$ exists for $\R_+\ni t \to 0$. If that limit exits then,
by continuity of the $S^1$-action, $e^{2\pi i\theta}\lim (t\cdot y) = \lim ((e^{2\pi i\theta} t)\cdot y)$ as $\R_+\ni t \to 0$. It follows that the whole circle $S^1\cdot y$ will uniformly converge to the same limit point under the action of $\C^*\ni t \to 0.$
\end{proof}

\begin{de}\label{Def:Contracting CstarManifold}
For a symplectic $\C^*$-manifold $(Y,\omega,\J,\Fi)$,
the $\C^*$-action $\Fi$ is {\bf contracting} if there is a compact subdomain $Y^{\mathrm{in}}\subset Y$ such that the $-X_{\R_+}$ flow starting from any point $y\in Y$ will eventually enter and stay in $Y^{\mathrm{in}}$.
In particular, $\F\subset Y^{\mathrm{in}}$.
\end{de}

\begin{lm}\label{Lemma H bdd below}\strut The following are equivalent:
\begin{enumerate}
\item $\varphi$ is contracting.
    \item  %
    $H$ is bounded below and $\F$ is compact.
    \item %
    \label{Semiprojective}
    $\F$ is compact and all points have convergence points.
\end{enumerate}

\end{lm}
\begin{proof}
If $\F$ is compact, we can pick a compact neighbourhood $Y^{\mathrm{in}}$ of $\F$.
The flow of $-\nabla H/\|\nabla H\|^2$ starting from $y\notin \F=\mathrm{Crit}(H)$ will decrease $H$ at the rate of $1$ in time $1$ unless the flowline reaches $\F$. So if $H$ is bounded below then this flowline must reach $Y^{\mathrm{in}}$ in finite time. Thus $\varphi$ is contracting. 
 Conversely, suppose $\varphi$ is contracting. As $\F\subset Y^{\mathrm{in}}$ is closed and $Y^{\mathrm{in}}$ is compact, also $\F$ is compact. There are no zeros of $X_{\R_+}$ outside of $Y^{\mathrm{in}}$ by the contracting assumption, so
$H$ achieves a global minimum inside $Y^{\mathrm{in}}$ as $-\nabla H$-flowlines are eventually in $Y^{\mathrm{in}}$. So $H$ is bounded below. Finally, the assumption on convergence points in (\ref{Semiprojective}) is equivalent to saying $-\nabla H$-flowlines converge to points in $\F$, which is equivalent to saying the absolute minima of $H$ lie in $\F$ (assuming $\F$ is compact).
\end{proof}

{\bf Assume henceforth that $(Y,\omega,\J,\Fi)$ is a symplectic $\C^*$-manifold with contracting $\Fi.$}
By condition (\ref{Semiprojective}), these are symplectic generalisations of smooth semiprojective varieties \cite{Hausel-Villegas}.

\subsection{The topology of the core}\label{Subsection topology of the core}

For any symplectic $\C^*$-manifold $Y$ with contracting $\Fi$, define %
$$\text{Core}(Y) :=\{y\in Y \mid \lim_{\C^* \ni t\fun \infty} t\cdot y \text{ exists}\}=\{y\in Y \mid \textrm{the }+\nabla H \text{-flowline from }y\text{ converges}\}.$$ %
Thus the forward $X_{\R_+}$-flow of a core point $y$ converges to some $y_\infty\in\F,$ %
and the unstable manifold of the $-\nabla H$ flow %
of a point $y_\infty\in \F$ consists of those $y\in \text{Core}(Y)$ that limit to it. %

If $Y$ is compact, then $Y=\mathrm{Core}(Y)$. For non-compact $Y$, we now prove $\mathrm{Core}(Y)$ is compact. 
We first need a technical lemma about $X_{\R_+}$-flowlines, which arise as
$\gamma(t)=e^{2\pi t} \cdot \gamma(0)$ 
with flow-time $t\in \R$.

\begin{lm}\label{Lemma Cannot keep leaving all compacts}
Given any compact subset $K\subset Y$, there cannot be $-X_{\R_+}$-flowlines $\gamma_n:[a_n,b_n]\to Y$ with endpoints $\gamma_n(a_n),\gamma(b_n)\in K$, with $c_n\in [a_n,b_n]$ such that\footnote{i.e.\,given any compact subset $K'\subset Y$, for all large enough $n$ the $\gamma_n$ leave $K'$ at some time $c_n$.} $\gamma_n(c_n)\to \infty$.

If $\gamma$ is an $X_{\R_+}$-flowline with $\gamma(0)=p$ and $\gamma(c_n)\to \infty$ for some $c_n\to \infty$, then there is a neighbourhood $p\in U_p$ such that for every compact $K\subset Y$ there is a time $\tau_K>0$ such that $\varphi_t(U_p)\cap K=\emptyset$ for all $\tau_K<t\in \R_+$ (i.e.\,``flowlines uniformly diverge to infinity near a divergent flowline'').
\end{lm}
\begin{proof}
Consider an exhaustion\footnote{By convention, our manifolds are assumed to be second-countable, so $Y$ can be covered by a sequence compact subsets $B_m$ homeomorphic to a compact ball in Euclidean space, so we could take $K_n=\cup\{B_m: m\leq n\}$.} 
of $Y$ by a sequence of compact subsets $\emptyset=K_{-1}\subset K_0\subset K_1\subset \cdots$. 
By redefining the $K_n$, we may assume that $Y^{\mathrm{in}}\subset K_0=K$ and that $K_n\subset \mathrm{Int}(K_{n+1})$ (note that any compact $K'\subset Y$ lies inside some $K_n$, otherwise $(\mathrm{Int}(K_{n+2})\setminus K_n)\cap K'$ is an open cover of $K'$ with no finite subcover).
Suppose the $\gamma_n$ are as in the claim. By $\R$-reparametrisation (so $\gamma_n(\cdot + c_n)$ and translating $[a_n,b_n]$ by $-c_n$) we may assume that $p_n':=\gamma_n(0)\in K_n\setminus K_{n-1}$. Let $\tau_n\geq 0$ be the smallest time for which $p_n:=\gamma_n(-\tau_n)\in K_1$. After passing to a subsequence, we may assume that $p_n\to p$ in $K_1$. Let $\gamma:\R \to Y$ be the (unique) $-\nabla H$ flowline with $\gamma(0)=p$. By $\R$-reparametrisation, we may now assume $p_n=\gamma_n(0)$ and $p_n'=\gamma(\tau_n)$. By construction, $\gamma_n|_{[0,\tau_n]}\subset Y\setminus K_0 \subset Y\setminus Y^{\mathrm{in}}$. By smooth dependence of ODE solutions on initial conditions, given a compact interval $[0, T] \subset \R$, for large enough $n$ the restrictions of $\gamma,\gamma_n$ to $[0, T]$ are arbitrarily close.
If the $\tau_n$ are unbounded, then it follows that $\gamma(t) \subset Y\setminus Y^{\mathrm{in}}$ for $t\geq 0$, contradicting the contracting assumption on $\Fi$ as $\gamma(0)$ does not flow into $Y^{\mathrm{in}}$ via $-X_{\R_+}=-\nabla H$.
By passing to a subsequence, we may assume $\tau_n \to \tau>0$. Then the image $C$ of $\gamma:[0,\tau]\to Y$ is compact, but also enters every compact subset $K_n$, meaning that $(\mathrm{Int}(K_{n+2})\setminus K_n)\cap C$ is an open cover of $C$ with no finite subcover, contradiction.

Before proving the second claim, we show $\gamma(t)\notin K$ for all $t\geq T_K$, for some $T_K\geq 0$ depending on $K$. In the notation above, we may assume $\gamma(c_n)\in K_n\setminus K_{n-1}$, and by contradiction $\gamma(b_n)\in K$ for some $b_n>c_n$. Then  $\gamma_n:=\gamma|_{[0,b_n]}$ contradicts the first claim (for the compact subset $K\cup\{p\}$). 
Now suppose the second claim fails, so there is a compact subset $K$ of $Y$ and there are $p_n\to p$, such that the flowlines $\gamma_n$ of $X_{\R_+}$ starting at $\gamma_n(0)=p_n$ satisfy $\gamma_n(b_n) \in K$ for some $b_n\to \infty$. By smooth dependence of ODE solutions on initial conditions, given $T>0$, the $\gamma_n,\gamma$ restricted to $[0, T]$ are arbitrarily close for sufficiently large $n$. Thus, after passing to a subsequence of the flowlines $\gamma_n$ and renumbering, we may assume $\gamma,\gamma_n$ are very close on $[0,c_n]$, thus: $\gamma_n(c_n),\gamma(c_n)$ are both in $K_n\setminus K_{n-1}$ and both $\gamma_n(t),\gamma(t)\notin K$ for $t\in [T_K,c_n]$.
As $b_n\to \infty$ and $\gamma_n(b_n)\in K$, this implies that $b_n>c_n$ for sufficiently large $n$.
Thus $\gamma_n|_{[0,b_n]}$ contradicts the first claim (for the compact subset $K\cup \{p\} \cup \{p_n: n\geq 1\}$).
\end{proof}

\begin{lm}\label{Lemma core is compact}
The core is a $\C^*$-invariant compact subset of $Y$. %
It is the union of all unstable manifolds,
$$
\mathrm{Core}(Y) = \bigcup_{\ff\in \F} W^u_{-\nabla H}({\ff})=\bigcup_{\a} W^u_{-\nabla H}(\Fa).
$$
The complement $Y':=Y\setminus \mathrm{Core}(Y)$ is a trivial real line bundle over the smooth compact manifold $\Sigma:=Y'/\R_+$, whose $\R$-fibres parameterise the $X_{\R_+}$-flowlines in $Y'$.
We can identify $\Sigma$ with a smooth real hypersurface in $Y'$ along which $X_{\R_+}$ is everywhere strictly pointing outward (to infinity). 
\end{lm}
\begin{proof}
Recall the convergence point $y_{\iinfty}=\lim_{t\to 0} t \cdot y \in \F$, and clearly $\F\subset \mathrm{Core}(Y)$. So the core is the union of all $-\nabla H$ flowlines that converge at both ends,\footnote{Convergence at the negative end is automatic by the contracting assumption.} including asymptotics.
Such flowlines must be trapped in some large enough compact subset of $Y$, otherwise we could build a sequence of them that contradict \cref{Lemma Cannot keep leaving all compacts}.
As they are all trapped in a compact subset, standard arguments about breakings of families of negative gradient flowlines (or arguing directly, by smooth dependence on initial conditions of ODE solutions of $\gamma'(t)=-\nabla H$) imply that $\CoreY$ is compact.

As %
$\F\subset \mathrm{Core}(Y)$ is disjoint from $Y',$ the $\R_+$-action on $Y'$ is free. 
To show that $Y'/\R_+$ is a smooth manifold and $Y'\to Y'/\R_+$ is a smooth principal $\R_+$-bundle, it is therefore enough to show that this action is also proper.\footnote{This is a standard theorem, see e.g. \cite[Thm.1.11.4.]{DK00}.} %
One characterisation of properness\footnote{Given e.g. in \cite[Prop.21.5(b)]{lee2013smooth}.}
is that whenever we have convergent sequences $y_n\fun p,$ $r_n \cdot y_n \fun q$ in $Y',$ %
then $r_n\in\R_+$ 
needs to have a convergent subsequence. %
Assuming the contrary, there is a subsequence $r_n\to \infty$ (switching the roles of $y_n$ and $r_n \cdot y_n$ changes $r_n$ to $1/r_n$, so this also deals with the case $r_n \to 0$). By the second claim in \cref{Lemma Cannot keep leaving all compacts}, 
$r_n \cdot y_n \fun \infty$ since %
$r_n \cdot p \to \infty$ as $r_n\to \infty$. This contradicts that $r_n \cdot y_n\to q$.

Hence we have a smooth principal $\R_+$-bundle $Y'\to Y'/\R_+$, which is thus trivial and has a section.\footnote{Principal $\R_+$-fibre bundles are classified up to isomorphism by homotopy classes of maps $Y'/\R_+ \to B\R_+$ into the classifying space. 
But we can take $E\R_+=\R_+$ with the standard $\R_+$-action, so $B\R_+=(E\R_+)/\R_+=\mathrm{point}$.} %
Identifying $\R\cong \R_+$ by the exponential map identifies %
it with a real line bundle whose zero section is the image of the above section. The claim now follows immediately, up to the compactness of $\Sigma:=Y'/\R_+.$
Since $\mathrm{Core}(Y)$ is compact, by enlarging $Y^{\mathrm{in}}$ we may assume that its interior contains $\mathrm{Core}(Y).$ %
Then $\Sigma$ is compact %
being the image of a surjective quotient map $\partial Y^{\mathrm{in}} \to Y'/\R_+$ 
from the compact %
$\partial Y^{\mathrm{in}}$.
\end{proof}

\begin{lm}\label{Lm nice subdomain}
We may always assume 
$Y^{\mathrm{in}}$ contains $\mathrm{Core}(Y)$ in its interior, $\Sigma:=\partial Y^{\mathrm{in}}$ is a smooth $S^1$-invariant hypersurface along which $X_{\R_+}$ points strictly outwards, and there is an isomorphism $$\Psi:Y^{\mathrm{out}}=Y\setminus \mathrm{int}(Y^{\mathrm{in}})\cong \Sigma \times [1,\infty),\quad
{\rho} \cdot y \mapsfrom (y,\rho).
$$
In particular, $\rho$ is an $S^1$-invariant function.
\end{lm}
\begin{proof}
If we ignore the $S^1$-invariance claims, then the claim follows immediately from \cref{Lemma core is compact}.
As $S^1$ is compact, we can ensure that $S^1\cdot Y^{\mathrm{in}}$ does not intersect $\Sigma \times [\rho_1,\infty)$
for some finite $\rho_1$, and we redefine $Y^{\mathrm{out}}$ to be the latter set.
The  $S^1$-invariant function $\widetilde{H}(\x)=\int_{S^1} \rho(e^{2\pi is} \cdot \x)\ ds$ is an exhausting (i.e.\,proper and bounded below)
smooth function on $Y$, where we extend $\rho \equiv 1$ on $Y^{\mathrm{in}}.$ %
By computing $d\widetilde{H}(\nabla H)$, using that $\nabla H=X_{\R_+}$, $d\varphi_{e^{2\pi is}}\cdot X_{\R_+}=X_{\R_+}$, and $d\rho(X_{\R_+})>0$ by construction, we deduce that $d\widetilde{H}(\nabla H)>0$. 
Thus the preimage of a large regular value $M$ of $\widetilde{H}$ is a smooth  $S^1$-invariant compact hypersurface that avoids the compact set $S^1\cdot \partial Y^{\mathrm{out}}$, and $X_{\R_+}$ is strictly outward pointing along it. This hypersurface can be taken as the new definition of $\Sigma:=\widetilde{H}^{-1}(M)=\partial Y^{\mathrm{in}}$ (so the new $Y^{\mathrm{out}}$ is the $X_{\R^+}$ forward flow of this $\Sigma$).
The final claim follows as the $S^1$ and $\R_+$ actions commute.
\end{proof}

We now show that the core captures the topology of $Y$ in all reasonable cases.

\begin{prop}\label{Prop def retr to core analytic case}
If $(Y,\mathrm{Core}(Y))$ is a CW-pair,\footnote{More generally, it suffices that $(Y,\mathrm{Core}(Y))$ is a neighbourhood deformation retract pair, in other words the inclusion $\mathrm{Core}(Y)\to Y$ is a cofibration \cite[Sec.6.4]{May}.
} then $Y$ deformation retracts onto $\mathrm{Core}(Y)$.
This assumption 
holds when $I$ is integrable and $\mathrm{Core}(Y)\subset Y$ is cut out by analytic equations (e.g.\,all CSRs).
\end{prop}
\begin{proof}
First we observe that $Y$ can be deformation retracted onto  an arbitrarily small neighbourhood $U$ of $A:=\mathrm{Core}(Y)$. Namely, by rescaling $-X_{\R_+}$ by a cut-off function vanishing near the core,\footnote{We remark that the map defined by taking the limits of the $-X_{\R_+}$ flow is usually a discontinuous map.} the flow of the resulting vector field can be used to define that deformation. 
Secondly, recall that for any CW pair $(Y,A)$,  $A$ is a deformation retract of some neighbourhood $N_A$ of $A\subset Y$ \cite[Prop.A.5]{Hatcher}.

Now combine the two steps: deform $Y$ onto an open neighbourhood $U$ with $A\subset U\subset N_A$, then restrict to $U$ the deformation of $N_A$ onto $A$ to get a deformation retraction of $U$ onto $A$.

The CW assumption on $(Y,A)$ holds for any compact analytic subset $A$ of any analytic manifold $Y$: by Giesecke and 
Lojasiewicz \cite{Giesecke, Lojasiewicz} the pair $(Y,A)$ can even be triangulated.\footnote{by analytic simplices, so the homeomorphisms from the simplices are analytic on the interiors of the simplices.}
This applies to CSRs $Y=\M$: 
the (compact) core is the preimage of an analytic map $\pi: \M \to \M_0$ of the unique $\C^*$-fixed point in $\M_0$, so it is a compact analytic subset of the analytic manifold $\M$.
\end{proof}

\begin{rmk}[{\bf {\AB} filtration}]\label{Rmk Atiyah-Bott}
Unlike the fixed locus $\F\subset Y,$ %
the core is in general singular.
As in Atiyah--Bott's discussion \cite[Prop.1.19]{AtB83}, $Y$ has a $\C^*$-equivariant finite Morse stratification: $Y$ is the disjoint union of the locally closed submanifolds given by the stable manifolds $U_\a:=W^s_{-\nabla H}(\F_\a)$ of points whose $-\nabla H$ flowlines converge to the component $\F_\a\subset \F=\mathrm{Crit}(H)$.
There is a partial order on the indices: $\a<\beta$ if a sequence of $-\nabla H$-trajectories $\gamma_n:Y \to \R$, with $\gamma_n(0)\in U_\a$, converges to a broken trajectory, and one of the breakings occurs at a point in $\F_\beta$. 
This gives the closure condition:
\begin{equation}\label{Equation partial order AB}\overline{U_\a}\subset \sqcup_{\a\leq \beta} U_{\beta}.\end{equation}
 The partial order is constructed inductively by starting with a local minimum $\F_{\min}$, which will be minimal w.r.t.\,$<$, and investigating how trajectories to $\F_{\min}$ can break (one must then take the transitive closure of the resulting partial pre-order, as we may not be in a Morse--Smale setup, which would imply transitivity by the classical gluing theorem for Morse trajectories).
In particular, $\a<\beta$ implies $H(\F_\a)<H(\F_\beta)$, but the partial order may be ``coarser'' than the 
partial order by $H$-values. To get a total order $\leq$ by $H$-values, we replace $\F_\a$ by the disjoint union of all $\F_\a$ with the same $H(\F_\a)$ value so that antisymmetry holds. Taking unions of subcollections of the $\F_\a$, as prescribed by the total order, one obtains a filtration of $Y$ by subspaces $\emptyset = W_0\subset W_1\subset \cdots \subset Y$.
There are also other,\footnote{Example: view the partial order as a directed graph, with vertices labelled $\F_\a$. Associate integer values to the vertices, with repetitions, given by the maximal length of paths in the graph from $\F_\a$ to a minimum. Take the disjoint union of $\F_\a$ with equal integer values. Then totally order these unions via their integer values.} albeit more arbitrary, ways to produce such a filtration by subspaces, again by refining the partial order to a total order.
 A filtration $\emptyset = W_0\subset W_1\subset \cdots \subset Y$ induces a filtration of $H^*(Y)$ by cup-product ideals by considering $K_i:=\ker(H^*(Y)\to H^*(W_i))$. These ideals, $H^*(Y)=K_0\supset K_1 \supset \cdots \supset \{0\}$, arise as sums of sub-collections of summands from \eqref{EqnFrankel}, e.g.\,those above a given $H$-value if we order by $H$-values.

We also have a stratification of $\mathrm{Core}(Y)$ by unstable manifolds
$D_\a:=W^u_{-\nabla H}(\F_\a)$; it stratifies $Y$ if we also artificially allow as a ``stable manifold'' the open subset $Y\setminus \mathrm{Core}(Y)$ of points whose $+\nabla H$ flow goes to infinity. So $\mathrm{Core}(Y)$ is a compact stratified subspace of $Y$: the union of the $W^{u}_{-\nabla H}(\F_\a)=W^{s}_{+\nabla H}(\F_\a)$ by \cref{Lemma core is compact}.
To reduce to the previous discussion, we would switch the sign to $+\nabla H$, so if we keep the original partial order $<$ then \eqref{Equation partial order AB} changes inequality-direction:
\begin{equation}\label{Equation partial order AB 2}\overline{D}_{\a}\subset \sqcup_{\a\geq \beta} D_{\beta}.\end{equation}

If the stratification by unstable manifolds $D_\a$ was a Whitney stratification (thus a Thom--Mather stratification), then $(Y,\mathrm{Core}(Y))$ would be a CW pair by Goresky \cite{Goresky}.
If the Morse--Smale condition\footnote{$W^s_{-\nabla H}(p)$ is transverse to $W^u_{-\nabla H}(q)$ for any critical points $p,q\in \mathrm{Crit}(H)$.} holds then unstable manifolds determine a Whitney stratification by Nicolaescu \cite[Ch.4]{Nicolaescu},\footnote{And proved originally by Laudenbach \cite{Laudenbach}, under an additional hypothesis.} 
so \cref{Prop def retr to core analytic case} applies.
We are unsure whether $(Y,\mathrm{Core}(Y))$ is a CW pair in complete generality.

{\AB} call a subset $J$ of the indices {\bf open} if $\lambda\in J$ implies that any index $\a<\lambda$ is also in $J$ (and an index-subset is {\bf closed} if the complement is open). Then $U_J=\cup_{\a\in J} U_{\a}$ is open if and only if $J$ is open (and analogously for the closed case). For $J$ open, $U_J\subset Y$ is an open symplectic $\C^*$-submanifold   with $\mathrm{Core}(U_J)=D_J=\cup_{\a\in J}D_\a$, but if $Y$ lives over a convex base (\cref{Def:KahlerMfdWithProjection}), this does not immediately imply that $U_J$ lies over a convex base because the restriction $\Psi|_{U_J}$ is typically not proper.
In the subsequent discussion we will use $$J^+:=J\cup \{\lambda\},$$ where $\lambda$ is a minimal index in the complement $J'$ of $J$, so $U_{\lambda}\subset U_J$ is closed, and $D_{\lambda}\subset D_J$ is open.

One can also filter $\mathrm{Core}(Y)$ by ``stable subspaces'', 
$$U_{\a}^C:=U_{\a}\cap \mathrm{Core}(Y),$$
so that we obtain a closure condition like in
\eqref{Equation partial order AB}, $\overline{U}_{\a}^{C}\subset \sqcup_{\a\leq \beta} U_{\beta}^C.$
In particular, $U_{\lambda}^C\subset U_J^C:=\cap_{\a\in J}U_{\a}^C$ is closed (for open $J$). However, the Thom isomorphism argument of
\cite[Eq.(1.16)]{AtB83}
need not apply to the $U_J^C$ because in general there is no clear way to deformation retract a neighbourhood $U_{\lambda}^{tub,C}$ of $U_{\lambda}\subset U_{J^+}$ onto $\F_{\lambda}^{tub}:=U_{\lambda}^{tub,C}\cap D_{\lambda}$ \emph{while staying within} $\mathrm{Core}(Y)$, unlike the easy case of a neighbourhood $U_{\lambda}^{tub}$ of $U_{\lambda}\subset U_{J^+}$ deformation retracting to $\F_{\lambda}^{tub}$ within $Y$ by combining a rescaling of the $-\nabla H$ flow and then exploiting $\exp_p$ for an $S^1$-invariant metric, for $p\in \F_{\lambda}$ \cite[Sec.5.1 p.54]{Nak99}.
\end{rmk}

\begin{de}\label{Definition HLC}
    A space $C$ is {\bf HLC} (homologically locally connected) if for each neighbourhood $U\subset C$ of any $c\in C$, $c$ has a neighbourhood $V\subset U$ making $\widetilde{H}_*(V)\to \widetilde{H}_*(U)$ vanish in reduced singular homology. 
    A space is {\bf HLC+} if it is HLC, paracompact\footnote{A useful fact is that a locally compact, Hausdorff, second countable space is paracompact.} and Hausdorff. 
    CW complexes are HLC+ spaces \cite[Prop.A.4]{Hatcher},
    so this includes all topological manifolds (possibly with boundary) and all real or complex algebraic (or semi-algebraic) varieties \cite{Lojasiewicz, Hironaka75}.
    Examples of HLC spaces are locally contractible spaces and any space homotopy equivalent to a CW complex.
    For paracompact Hausdorff spaces $X$, Alexander--Spanier cohomology $H^*_{AS}(X)$, and \v{C}ech cohomology $\check{H}^*(X)$ are isomorphic \cite[Cor.6.8.8]{Spanier} (isomorphic to sheaf cohomology of $X$ for the sheaf of locally constant coefficients).
    For $X$ HLC+ these are isomorphic to singular cohomology $H^*(X)$ \cite[Cor.6.9.5]{Spanier}.
\end{de}

\begin{prop}\label{Lemma Coh of core is coh of ambient via Alexander-Spanier coh}
For any symplectic $\C^*$-manifold,
$$
H^*(Y)\cong H_{AS}^*(\mathrm{Core}(Y))\cong \check{H}^*(\mathrm{Core}(Y)).
$$
If $\mathrm{Core}(Y)$ is HLC, then these are isomorphic to $H^*(\mathrm{Core}(Y))$.\\
$\mathrm{Core}(Y)$ is HLC if and only if each $\F_\a$ admits an HLC neighbourhood in $\mathrm{Core}(Y)$.
\end{prop}
\begin{proof}
Abbreviate $A:=\mathrm{Core}(Y)$.
 By \cite[Cor.6.9.9]{Spanier}, $H_{AS}^*(A)\cong \varinjlim H^*(U)$ over restrictions for all neighbourhoods $U$ of $A$. We can pick a cofinal sequence of such $U$ using the $\C^*$-action, then it is a direct limit over isomorphisms, so $H_{AS}^*(A)\cong H^*(U)\cong H^*(Y)$. Note $A$ is paracompact (since $A$ is closed and $Y$ is paracompact) and Hausdorff (since $Y$ is).

 For the last claim, only one direction is non-trivial. Given a neighbourhood $c\in U\subset \mathrm{Core}(Y)$, $\varphi_t(c)$ converges towards some $\F_\a$ as $t\to \infty$. So $c':=\varphi_H^t(c)$ lies in an HLC neighbourhood of $\F_\a$ for large $t\in \R_+$. Pick neighbourhoods $c'\in V' \subset W'$ with vanishing $\widetilde{H}_*(V')\to \widetilde{H}_*(W')$ and $W:=\varphi_H^{-t}(W')\subset U$. Then $c\in V:=\varphi_H^{-t}(V')\subset U$ and $\widetilde{H}_*(V)\to \widetilde{H}_*(W) \to \widetilde{H}_*(U)$ vanishes as the first map vanishes.
\end{proof}

\begin{cor}\label{Cor Core is connected}
    $\mathrm{Core}(Y)$ is connected, indeed path-connected.
\end{cor}
\begin{proof}
 $H^0(Y)\cong H_{AS}^0(\mathrm{Core}(Y))$ has rank one as $Y$ is connected, and Alexander--Spanier cohomology in degree $0$ detects connectedness \cite[Cor.6.4.7]{Spanier}. Path-connectedness will follow from \cref{Lemma Frankel properness} and \cref{CorollaryGradientTrajBecomeSpheres}: any $\F_\a$ can be path-connected to $\F_{\min}:=\min H$ (which is a connected manifold and thus path-connected) by a succession of $\C^*$-invariant holomorphic spheres in $\mathrm{Core}(Y)$ (see \cref{CorollaryGradientTrajBecomeSpheres}), 
 and any point of $\mathrm{Core}(Y)$ lies on such a sphere.
\end{proof}

\begin{rmk}
For smooth semiprojective varieties, {\HV} \cite[Thm.1.3.1 and Cor.1.3.6]{Hausel-Villegas} 
studied the \BB decomposition of the variety, corresponding to our stable/unstable manifolds above. 
They first prove that $H^*(Y)\cong H^*(\mathrm{Core}(Y))$; secondly, they prove that $\pi_*(\mathrm{Core}(Y))\cong \pi_*(Y)$; and %
in the end, they use that $\mathrm{Core}(Y)$ is a CW complex in their setting to conclude that the inclusion $\mathrm{Core}(Y)\to Y$ is a homotopy equivalence by Whitehead's theorem.
Their proof, as written, only
relies on classical algebraic topology involving the $U_\a, D_\a$ from \cref{Rmk Atiyah-Bott} 
and using the Thom isomorphism from \cite[Eq.(1.16)]{AtB83}. 
We believe there is an imprecision in their proof.
The second Thom isomorphism in \cite[p.119]{Hausel-Villegas} is claimed for 
$\mathrm{Core}(Y)$ using the filtration by $D_\a$, but this involves the different closure condition 
\eqref{Equation partial order AB 2}.
Perhaps the misapprehension was assuming that  $D_{\lambda}\subset D_{J^+}$ is closed: this would explain why in \cite[Eq.(1.3.9)]{Hausel-Villegas} 
they claim to use an embedding $F_{\lambda}^{tub}\cap (U_J\cap D_{J^+})\to D_J$ which does not in fact exist%
\footnote{If $D_J$ was a typo for $U_J\cap D_{J^+}$, one still needs to prove that $\pi_1(U_J\cap D_{J^+})\cong \pi_1(D_J)$, which is not obvious as the question of whether $U_J\cap D_{J^+}$ deformation retracts onto $D_J$ is not clear without appealing to CW complex structures.
One can salvage the proof if there existed a map $\pi_1(U_J\cap D_J) \to \pi_1(D_J)$ that can be inserted in the second diagram of \cite[Eq.(1.3.9)]{Hausel-Villegas}, so $\pi_1(\F_{\lambda}^{tub}\cap(U_J\cap D_{J^+}))\to \pi_1(U_J\cap D_{J^+}) \to \pi_1(D_J)$, making the diagram commute.
} (e.g.\,in \cref{Example running example of intro}
using $D_{\lambda}=S^2\setminus p$, for the minimum point $D_J=\{p\}=\min H$).
One can salvage both proofs of \cite[Thm.1.3.1 and Cor.1.3.6]{Hausel-Villegas} (respectively the statement that the inclusion induces an isomorphism $H^*(\mathrm{Core}(Y))\cong H^*(Y)$ and isomorphisms $\pi_*(\mathrm{Core}(Y))\cong \pi_*(Y)$) for any symplectic $\C^*$-manifold for which a neighbourhood $V:=D_{J^+}\cap \{H< H(\F_{\lambda})-\textrm{(small constant)}\}$ of $D_J$ in $D_{J^+}$ is homotopy equivalent to $D_J$ (e.g.\,this holds if a neighbourhood of $D_J$ in $\mathrm{Core}(Y)$ deformation retracts to $D_J$). Indeed, one can then excise $D_J$ from $V$ so: $H^*(D_{J^+},D_J)\cong
H^*(D_{J^+},V)\cong H^*(D_{J^+}-D_J,V-D_J)\cong H^{*-\mu_\a}(\F_{\lambda})$, using the Thom isomorphism by viewing $D_{\lambda}$ as a vector bundle over $\F_{\lambda}$, and the rest of the proof of \cite[Thm.1.3.1]{Hausel-Villegas} would hold. In the proof of \cite[Cor.1.3.6]{Hausel-Villegas}, one makes the following replacements in \cite[Eq.(1.3.9)]{Hausel-Villegas}: $\pi_1(D_{\lambda}^{tub})$ by $\pi_1(\F_{\lambda}^{tub})$; $\pi_1(D_J)$ by $\pi_1(U_J\cap D_{J^+})$; then one uses a suitable rescaling of $-\nabla H$ to obtain a flow inducing $\pi_1(U_J\cap D_{J^+})\cong \pi_1(V)$, and the assumption on $V$ gives $\pi_1(V)\cong \pi_1(D_J)$.
In the context of smooth semiprojective varieties, we are dealing with analytic subvarieties $D_J\subset D_{J^+}$, so one could argue the existence of $V$ as in the proof of \cref{Prop def retr to core analytic case}; and for the cohomological claim it is enough to use the proof of \cref{Lemma Coh of core is coh of ambient via Alexander-Spanier coh} to argue $H^*(D_{J^+},D_J)\cong \varinjlim H^*(D_{J^+},V_n)$ over a sequence of $V_n$ which shrink to $D_J$, obtained from $V$ by flowing with $-\nabla H$. However, if one already uses those proofs, one might as well bypass the {\AB} filtration.
\end{rmk}

\subsection{Topology via the moment map, and 
dealing 
with its possible non-properness} %

\begin{lm}\label{Lemma making H proper MB}
There is an exhausting 
(i.e.\,proper and bounded from below)
{\MB } function $f: Y \to \R$ which equals $H$ on $Y^{\mathrm{in}}$, with $df(X_{\R_+})>0$ on $Y^{\mathrm{out}}$.
One can choose $\Sigma$ in \cref{Lm nice subdomain} to be a (sufficiently large) level set of $f$, and
 $\nabla f$ is strictly outward pointing on each hypersurface $\Sigma \times \{\rho\}$.
The $\mathrm{Core}(Y)$ is unaffected, and the {\MB } cohomologies of $H$ and $f$ coincide at chain level.
\end{lm}
\begin{proof}
$H$ is bounded below, but we may need to modify it on $Y^{\mathrm{out}}$ to make it proper.
In the proof of \cref{Lm nice subdomain}, 
we built $\widetilde{H}$ on $Y^{\mathrm{out}}$.
We interpolate $H$ with $\widetilde{H}$ on $\Sigma \times [1,\infty)$, $f:=(1-\psi(\rho))H + \psi(\rho)\widetilde{H}$, where $\psi:[1,\infty)\to [0,1]$ is a cut-off function such that $\psi$ increases from $\psi=0$ near $\rho=1$ to $\psi = 1$ for $\rho\geq 2$.
We may assume that $\widetilde{H}\geq H$ in the region $\rho\in [1,2]$ (by rescaling $\widetilde{H}$ by a large constant). A simple calculation shows $df(X_{\R_+})>0$ on $Y^{\mathrm{out}}$ (using that $dH,d\widetilde{H},d\rho$ are strictly positive on $X_{\R_+}$ on $Y^{\mathrm{out}}$ and $\psi'\geq 0$). 
Regarding $\Sigma$: in the proof of 
 \cref{Lm nice subdomain}, we picked $\Sigma$ to be a level set of $\widetilde{H}$, and we may choose this level set to lie in the region $\rho\geq 2$ where $f=\widetilde{H}$.
\end{proof}

\begin{cor} \label{FibersMomentMapConnected} 
Recall we always assume $Y$ is connected.
If $H$ is proper, then the fibres of $H: Y\to \R$ are connected.
If $H$ is not proper,\footnote{e.g.\,this often occurs naturally for the open subsets $U_J\subset Y$ from \cref{Rmk Atiyah-Bott}, due to the removal of $U_{\alpha}$'s from $Y$.}
then the minimum of $H$ is connected and so are all level sets $H=c$ with\footnote{
For any sequence of points $p_n\to \infty\in Y$ (meaning $p_n$ leaves any $K_i$ of a given exhaustion of $Y$ by compact subsets $K_i$), let $H_p:=\liminf H(p_n)$. Then $\liminf H:=\inf \{H_p:$ all sequences $p_n\to \infty\}$. So $H\leq c$ is compact for $c<\liminf H$.
}
$c<\liminf H$. 
Moreover, the $\Sigma$ in \cref{Lm nice subdomain} is connected.
\end{cor}
\begin{proof}
The moment map of a Hamiltonian $S^1$-action on a connected compact symplectic manifold has connected fibres \cite[Lemmas 2.1 and 2.2]{At82}. The same proof holds for connected
open symplectic manifolds if the moment map is exhausting. Suppose now $H$ is not proper. We instead use the {\MB } function $f$ in \cref{Lemma making H proper MB}. The connectivity lemma in Nicolaescu \cite[Lem.3.62]{Nicolaescu} proves the claim for {\MB } functions on closed connected manifolds whose indices and co-indices are even. %
We can adjust his proof to work for exhausting {\MB } functions as follows. Suppose $f\geq c$ is disconnected, for large $c>0$. 
By the same handle-attachment argument as in Nicolaescu's proof, $Y$ is obtained from $f\geq c$ by handle-attachments which do not change the number of connected components, contradicting connectedness of $Y$. 
So $f\geq c$ is connected. The flow of $\tfrac{\nabla f}{\|\nabla f\|^2}$ is defined for $f\geq c$ (there are no critical points), yielding diffeomorphisms $\{f\geq c\} \iso \{f=c\}\times [c,\infty),$ the second coordinate being the value of $f$. So all level sets of $f$ above $c$ are connected (so in particular $\Sigma$ is connected).
The rest of Nicolaescu's proof applies, after replacing any occurrences of ``$f=f_{\max}$'' by $f=c$.
The second claim now follows by ensuring that $H=f$ on a sufficiently large compact set that includes $H\leq c$ in the interior (and we know that outside of this set, $df(X_{\R_+})>0$ by \cref{Lemma making H proper MB} so $f>c$ there).
\end{proof}

\begin{lm}\label{Lemma Frankel properness}
Over any field of coefficients,
we have a module isomorphism
\begin{equation}\label{EqnFrankel}
H^*(Y)\iso \bigoplus_{\a} H^*(\F_\a)[-\mu_\a].
\end{equation}
Over $\Z$ or a field, $H^*(Y)$ lies in even degrees if and only if all $H^*(\F_\a)$ do, and $H^*(Y)$ is torsion-free if and only if all $H^*(\F_\a)$ are.
As $Y$ is connected, there is a unique  {\bf minimal component} $\F_{\textrm{min}}:=\F_{\a}$ with $\mu_{\a}=0$ generating $1\in H^0(Y)$ in \eqref{EqnFrankel}; also $\F_{\textrm{min}}=\min H$ (there are no other local minima).
\end{lm}
\begin{proof}
By Frankel \cite[Sec.4]{Frankel59},\footnote{\cite{Frankel59} deals with compact \KH manifolds, but we imposed on $Y$ the conditions needed in the proof.
Alternatively, Kirwan \cite{Ki84} and Nakajima \cite[Sec.5.1]{Nak99} prove it for Hamiltonian torus actions on symplectic manifolds.
}
a Hamiltonian $S^1$-action whose moment map $H$ is  proper and bounded below is a perfect {\MB } function, and its fixed loci $\F_\a$ are symplectic submanifolds satisfying \eqref{EqnFrankel} (a spectral sequence argument, whose Floer analogue is Corollary \ref{pureHam}, yields \eqref{EqnFrankel}).
If $H$ is not proper, we use $f$ as in \cref{Lemma making H proper MB} instead, which does not affect the argument. Then \eqref{EqnFrankel} follows by using the filtration by values of $H$ to run the argument \cite[Sec.4]{Frankel59} (it only relies on $H$ being the moment map of an $S^1$-action in the vicinity of the critical locus $\F$). 
The claim about odd cohomology follows from the fact that {\MB } indices are even.
The statement about odd cohomology and torsion follows by universal coefficients \cite[Cor.1]{Frankel59}.
\end{proof}

\section{Torsion, periods, holomorphic spheres, and the attraction graph}%
 \label{Section Torsion sbmfds attraction graph}

\subsection{Weight spaces}\label{Subsection weight spaces}

In the proof of \cref{LemmaSec2FixedLocus} we obtained a unitary\footnote{Using the Hermitian form $\la \cdot,\cdot \ra = g(\cdot,\cdot) + i \om(\cdot,\cdot)$ on $T_p Y$, where $g(\cdot,\cdot)=\om(\cdot,\J \cdot)$.} decomposition
\begin{equation}\label{Eqn weight spaces Ti}
    T_\p Y = \oplus_i T_i \;\textrm{ at } \p\in \F
\end{equation}
 into two-planes $T_i$ of {\bf weight}\footnote{Explicitly, $v\neq 0 \in T_\p Y$ has {\bf weight} $w\in \Z$ means $d_\p \varphi_{t} \cdot v = t^{w} v$ for $t=e^{2\pi i s}$ and $s\in \R/ \Z$.}  $w_i\in \Z.$ %
Collecting summands of equal weights, 
\begin{equation}\label{Eqn weight spaces Hk}
    T_\p Y = \oplus_{k\in\Z} H_{k} \;\textrm{ at } \p\in \F, 
\end{equation}
where the complex vector subspace $H_k= \oplus \{ T_i: w_i = k \} \subset T_\p Y$ is the {\bf weight space} for weight $k\in \Z$. %

\begin{lm}\label{Lemma omega gives perfect pairing}
    The $H_k$ are symplectically 
    orthogonal symplectic subspaces for $\omega$.
\end{lm}
\begin{proof}
    Let $t\in S^1$, then $\varphi_t^*\omega =\omega$ since $\varphi_t$ is symplectic. So if $a\in H_k$, $b\in H_{k'}$, then $$\omega(a,b)=
    \omega(d\varphi_t\cdot a,d\varphi_t\cdot b)= \omega(t^k\, a,t^{k'}\,b)
    =
    \cos ((k'-k)\theta)
    \omega(a,b)+
    \sin ((k'-k)\theta)
    \omega(a,Ib),
    $$
    where $t=e^{i\theta}$. As this holds for all $\theta\in \R$, this implies $\omega(a,b)=0$ and $\omega(a,Ib)=0$ unless $k=k'$.
    Non-degeneracy of $\omega$ implies the rest of the claim.
\end{proof}

\begin{lm}\label{Lemma constant weights on F}
The set of all weights of $T_\p Y$ is constant on a connected component $\F_\a$ of $\F,$ and \eqref{Eqn weight spaces Hk} induces a bundle decomposition $T_{\F_\a} Y = \oplus_k H_k$ over $\F_\a.$ This yields well-defined numbers
\begin{equation}\label{notation h_k^alpha for dimension of wt spaces}
    h_k^{\alpha}:=\dim_{\C} H_k.
\end{equation}
\end{lm}
\begin{proof}
Using a weight decomposition of $T_\p Y$, $d_\p \varphi_{\tau}$ can be represented by a diagonal matrix with diagonal entries $\tau^{w_i}$ where $w_i$ are the (possibly repeated) weights $w \in \Z$. Along any path in $\F$, the eigenvalues of $d_\p \varphi_{\tau}$ vary continuously in $\p$.
Our eigenvalues $\tau^{w_i}$ depend on a discrete parameter $w_i\in \Z$, so the $w_i$ are constant on (path-)connected components of $\F$. 
The bundle decomposition follows, since $H_k$ is determined by the property of being the $\tau^{k}$-eigenspace of the endomorphism $d \varphi_{\tau}$ of $T_{\F_\a} Y$.
\end{proof}
\begin{rmk}\label{Remark H+}
The zero weight subspace at $y\in \F_{\a}$  corresponds to the fixed component     
$H_0=T_\ff\F_{\a}$. The negative weight subspace $H_-=\oplus_{k<0}H_k$ corresponds (via the exponential map) to $-\nabla H$ flowlines coming out of $\F_{\a}$ so they lie in $\mathrm{Core}(Y)$. For the positive weight subspace $H_+=\oplus_{k>0} H_k$ it is more complicated. Consider a small sphere subbundle $S_+$ of $H_+$, and let $p=\exp_{y}(v)$ for $v\in S_+$. For a subset $P_{in}$ of $v\in S_+$, the $+\nabla H$-flow of $p$ stays in $\mathrm{Core}(Y)$ (i.e. flows to a finite point), whereas for $P_{out}=S_+\setminus P_{in}$ it flows to infinity. By \cref{Lemma Cannot keep leaving all compacts}, $P_{out}$ is an open subset of $S_+$ (a point $p'$ sufficiently close to $p$ will also flow out of $Y^{in}$ and thus not belong to the core). So $P_{in}=S_+\setminus P_{out}$ is closed. In general, the subsets $P_{in}, P_{out}$ can be horrible, as we do not assume the Morse-Smale property for $H$. 
We call \textbf{outer} weight of the action $\Fi$
the weights occurring in the subset $P_{out}\subset H_+.$ 
\end{rmk}
\subsection{Torsion points and torsion submanifolds $\Ymc$}

\begin{de}\label{TorsionPtsTorsionSubmflds}
Identify $\Z/m$ with the subgroup 
$\Z/m\hookrightarrow S^1\subset \C^*$, $k\mapsto e^{2\pi i k/m}$.
For $m\geq 2$, a {\bf $\Z/\ell$-torsion point} is a $\Z/m$-fixed point. 
Such points define a $\C^*$-invariant smooth submanifold\footnote{a fixed locus of a compact Lie group action on a smooth manifold is a smooth submanifold \cite[p.108]{DK00}.} $Y_{\ell}\subset Y$.
We can decompose $Y_m=\sqcup_\c \Ymc$ into connected components $Y_{\ell,\beta}$ called {\bf $\Z/\ell$-torsion submanifolds}.
We show below that each $Y_{m,\b}$ converges via the $\C^*$-flow to (possibly several) fixed components $\F_\a$, and that each $\Ymc$ contains a subcollection of the $\F_\a$. 
\end{de}

\begin{lm}\strut
\begin{enumerate}
    \item $Y_m$ contains all $Y_{mb}$ for integers $b\geq 1$;
    \item $Y_m\subset Y$ is a closed subset, with a relatively open dense stratum $Y_m\setminus \cup_{b\geq 2} Y_{mb}$;
    \item each $\Ymc$ is a symplectic $\C^*$-submanifold of $Y$, and its $\C^*$-action admits an $m$-th root;
    \item if $Y$ is a symplectic $\C^*$-submanifold over a convex base (\cref{Def:KahlerMfdWithProjection}), then so is each $\Ymc$ with either action from (3).
\end{enumerate}
\end{lm}
\begin{proof}
    By continuity, a sequence of $\Z/m$-fixed points converges to a $\Z/m$-fixed point (which may also be a $\Z/mb$-fixed point for $b\geq 2$). So (1) and (2) are immediate.
    The linearised $\Z/m$ action on $TY|_{Y_m}$ decomposes it $I$-holomorphically into into $\Z/m$-weight spaces, and the zero weight space is $TY_m$, so $Y_m$ is $I$-{\ph}, and thus $\om$-symplectic. So (3) follows (the $m$-th root of the $\C^*$-action is well-defined as $m$-th roots of unity act as the identity on $Y_m$).
    Claim (4) follows by restricting the map 
    from \cref{Equation intro Psi} to $\Psi|_{\Ymc}:\Ymc \to B=\Sigma \times [R_0,\infty)$. If we use the $m$-th root of the $\C^*$-action, we just need to rescale the data $\alpha,R,R_0$ on $B$ from \cref{Def:KahlerMfdWithProjection} to $m\alpha$, $R/m$, $R_0/m$ (this rescales the Reeb field $\mathcal{R}_B$ by $1/m$: note the definition does not require $B$ to admit an $S^1$-action on all of $B$).
\end{proof}

\begin{lm}\label{Lemma torsion submanifolds}
At $p\in \F_\a$ the tangent space $T_p Y_{\ell}$ is the $\Z/\ell$-fixed locus of the linearised action, so
\begin{equation}\label{TangentSpaceOfTorsionMfd}
T_p Y_{\ell}=\oplus_{b\in \Z}%
H_{\ell b}\subset T_p Y.	
\end{equation}
In a sufficiently small neighbourhood of $\F_\a$, $Y_{\ell}$ is the image
$Y_{\a,\ell}^\mathrm{loc}$ via $\exp_{\F_\a}$ %
of a small neighbourhood of the zero section in $\oplus_{b}%
H_{\ell b}\subset T_{\F_\a} Y$. Globally $Y_{\ell}=\cup_{\a} (\C^*\cdot Y_{\a,\ell}^\mathrm{loc})$.
Thus, the $\C^*$-fixed locus of each $Y_{m,\beta}$ consists of a union of a subcollection of the $\F_\a.$
\end{lm}
\begin{proof}
The possible proper isotropy groups (i.e.\,stabilisers) of the $S^1$-action which contain $\Z/\ell$ are $\Z/(\ell b)$ for $b\in \N$.
The claim therefore follows by the proof of Lemma \ref{LemmaSec2FixedLocus}, since near $\F_\a$ the points whose isotropy is $\Z/(\ell b)$ corresponds via the exponential map to tangent vectors in $TY|_{\F_\a}$ fixed by $\Z/(\ell b)$ (the case $b=0$ gives $H_{\ell b}=H_0=T_p\F_\a$, corresponding to isotropy $S^1$, but that case is already dealt by $x=p=\exp_p(0)$).
The final claim also follows, noting that the summand $H_{\ell b}$ for $b<0$ corresponds to points with isotropy $\Z/(|b|\ell)$ but which flow to $\F_\a$ via $+X_{\R_+}$ instead of $-X_{\R_+}$.
\end{proof}

\begin{cor}\label{TorsionMfdIsSymplectic}
There is a finite number of $Y_{m,\b},$ each of which is a union of various $\C^*\cdot Y_{\a,\ell}^{\mathrm{loc}}$.
\end{cor}
\begin{proof}
By \cref{Lemma torsion submanifolds}, locally near $\F_\a$ there is a unique $\Ymc$ (if it exists), so for fixed $m$ we have $\#\{\Ymc\}\leq \#\{\F_\a\}$, which is finite as $\F=\sqcup \F_\a$ is compact. The possible torsion groups $\Z/m$ is also finite, as each $m$ is the absolute value of a weight of some $\F_\a.$ 
\end{proof}

\begin{lm}\label{Lemma Umin is dense}
The stable manifold $U_{\min}:=W^s_{-\nabla H}(\F_{\min})$ of the minimal component $\F_{\min}:=\min H$ is open, connected and dense.
Moreover, $U_{\min}$ is diffeomorphic to the normal bundle of $\F_{\min}.$
\end{lm}
\begin{proof}
$U_{\min}=\R_+\cdot V$ for any neighbourhood $V\subset Y$ of $\F_{\min}$. So $U_{\min}\subset Y$ is open and connected.
Any $\F_{\gamma}\neq \F_{\min}$ has some negative weight by \cref{Lemma Frankel properness}, so its stable manifold has real codimension$\geq 2$. As $Y$ is the disjoint union of the finitely many stable manifolds, it follows that $U_{\min}$ is dense.
\end{proof}

\begin{de}\label{Definition generic points of Ymc}\label{Lemma generic points of Ymc}
Viewing $\Ymc$ as a symplectic $\C^*$-manifold, by \cref{Lemma Frankel properness} it contains a unique minimal component $\min (H|_{\Ymc}:\Ymc\to \R)$ which must arise as some component $\F_\a$ of $Y$. 
By minimality, it is the only $\F_{\alpha}\subset \Ymc$ with the property that all weights $mb$ in \eqref{TangentSpaceOfTorsionMfd} are non-negative. 

Call $\F_\a$ {\bf $m$-minimal} if it is the minimal component of some $Y_{m,\b}$,
equivalently if $\F_\a$ has at least one non-zero weight divisible by $m$ 
and all such weights are positive ($T_{p}\Ymc=T_p\F_\a\oplus \oplus_{b\geq 1} H_{mb}$ in \eqref{TangentSpaceOfTorsionMfd}). 

The $U_{\min}$-locus of $\Ymc$ is $U_\a \cap \Ymc$ where $U_\a=W^s_{-\nabla H}(\F_\a)\subset Y$, where $\F_\a$ is the minimal component of $\Ymc$. We call $U_{\min}$ the {\bf generic locus} or locus of {\bf generic points} of $\Ymc$.
By \cref{Lemma Umin is dense}, the generic $y\in \Ymc$ are precisely those points admitting a neighbourhood in $Y_{m,\b}$ which converges to the same fixed component in $Y$ (which must be the minimal component $\F_\a$ of $\Ymc$). 
\end{de}

\begin{cor}\label{Cor Ymc hits core in path connected subset}
  The intersection $Y_{m,\b}\cap \mathrm{Core}(Y)$ is path-connected.  
\end{cor}
\begin{proof}
This follows by \cref{Cor Core is connected}, as $\mathrm{Core}(Y_{m,\b})=\Ymc\cap \mathrm{Core}(Y)$.
\end{proof}

\begin{prop}\label{Lemma Torsion bundle is a bundle} If $Y_{m,\b}$ contains a single $\F_\a$, then $\Ymc$ is diffeomorphic to the weight $m$-part 
$\oplus_{b} H_{mb}=\oplus_{b\geq 0} %
H_{mb}$ of the normal bundle of $\F_\a$, and $\Ymc\cap \mathrm{Core}(Y)=\F_\a$.
\end{prop}
\begin{proof}
$\F_\a=\min H|_{\Ymc}$, so
$\Ymc$ equals the $U_{\min}$-set of $\Ymc$ viewed as a symplectic $\C^*$-manifold. 
Let $p\in \Ymc\cap \mathrm{Core}(Y)$. Then $\lim_{t\to \infty} t\cdot p\in \Ymc \cap \F_\a$.
But $H(t\cdot p)$ is non-decreasing in $t\in \R_+$, so $H(t\cdot p) \leq H(\F_a)$, which forces $H(t\cdot p)=\min H|_{\Ymc}$, so $t\cdot p \in \F_\a$ for all $t$. Thus $\Ymc\cap \mathrm{Core}(Y)=\F_\a$.
\end{proof}
\begin{de}\label{Defn torsion bundle}
If $Y_{m,\b}$ converges to a single $\F_\a$ we call it a {\bf torsion bundle}, $\Hm\fun\F_\a.$ 
\end{de}

\subsection{$S^1$-orbits, holomorphic spheres, Hamiltonian periods, and generic slopes}
An {\bf $S^1$-orbit} is a closed orbit $x:S^1 \to Y$ of $X_{S^1}$. As the $S^1$-flow has period $1$, non-constant orbits have minimal period $1/m$ for some positive $m\in \N$.
These consist of fixed points of the $\Z/m$ action, however the $\Z/m$ action also fixes orbits of minimal period $1/mb$ for any $b\in \N$.

\begin{cor}\label{CorSec2OrbitsOfellH}
$Y_{\ell}$ is the image of all $1$-periodic Hamiltonian orbits of $\tfrac{1}{m} H:Y\to \R$.
\end{cor}
\begin{proof}
     The $1$-orbits of $\lambda H$ for $\lambda=1/m$ have period $1/m$, so minimal period $1/mb$ for some $b\in \N$.
     Near $p\in \F_\a$ the $S^1$-invariant submanifold obtained by evaluating $\exp_p$ on a small neighbourhood of the zero section of $\oplus_{b\in \Z} H_{mb}$ yields an $S^1$-invariant submanifold near $p$, consisting precisely of the points near $p$ with period\footnote{if we wanted \emph{minimal} period $1/m$, we would need $\exp_p(v)$ with $v\in \oplus_{b\in \Z} H_{mb}$ having non-zero entry in $H_{-m}\oplus H_m$.} $1/m$, which are precisely the $\Z/m$-fixed points.
\end{proof}

\begin{lm}\label{LemmaCanonicalFillingDiscs}
Any $y\in Y$ yields a {\ph } disc
$\psi_y: \DD \to Y$, $ \psi_\x(z)=\varphi_z(y)$, $\psi_y(0)=\yinfty$.
A unitary basis $v_i$ for $T_{\yinfty}Y$ induces a canonical unitary (so symplectic) trivialisation $v_i(z)$ of $\psi_y^*TY$ with $v_i(z)=v_i$. The trivialisation is $S^1$-equivariant in $y$ in the sense that $v_i(tz)$ is the canonical trivialisation of $\psi_{ty}^*TY$ induced by $v_i$, for any $t\in S^1$. 

If the $S^1$-orbit of $y$ has minimal period $1/m$, for $m\in \N$, then $\psi_\x(z)$ is an $m$-fold cover of a {\ph } disc $\hat{\psi}_\x:\DD \to Y$, and $v_i(z)$ is induced by a canonical trivialisation of  $\hat{\psi}_{y}^*TY$, in particular it is $\Z/m$-equivariant: $v_i(\zeta z)=v_i(z)$ whenever $\zeta^m=1$.
\end{lm}
\begin{proof}
Define $\psi_y: \{z\in \C: 0<|z|\leq 1\} \to Y$ by $\psi_\x(z)=\varphi_z(y).$
Observe that $\psi_y$ extends continuously over $0$ via the convergence point $\psi_y(0):=y_{\iinfty}$.  
If this were a complex holomorphic setup (i.e.\,for an integrable complex structure), we could argue that in local coordinates $\psi$ extends holomorphically over $0$ in each coordinate by the classical removable singularity theorem.
In the almost complex setup, Gromov's removable singularity theorem (e.g.\,see \cite[Thm.4.1.2]{McDuff-Salamon}) implies the same result provided we show that the energy of $\psi_y$ is bounded. As $y_{\iinfty}$ is a hyperbolic fixed point (\cref{Rmk hyperbolic fixed points}), the flow converges exponentially to $y_{\iinfty}$. Thus $\psi_y$, viewed as a {\ph } cylinder, converges exponentially near $z=0$, and therefore it has bounded energy as required.

The bundle $\psi_y^*TY$ is a complex vector bundle over $\DD$. Taking the $S^1$-invariant almost \KH metric $g(\cdot,\cdot)=\om(\cdot,\J\cdot),$
we can trivialise $\psi_y^*TY$ unitarily by parallel transporting a unitary basis $v_i$ of $T_{y_{\iinfty}}Y$ radially outwards from the centre of the disc.\footnote{Recall that parallel transport for an almost \KH manifold $(Y,g,\om,\J)$ is unitary.}
This trivialisation is canonical up to a choice of unitary basis for $T_{y_{\iinfty}}Y$, and it is also symplectic with respect to $\om$ since $(Y,g,\om,\J)$ is almost \Kh. The final two claims follow by the construction (noting that $\hat{\psi}_{y}(z^m)=\psi_\x(z)$).
\end{proof}

\begin{cor}
Any $S^1$-orbit $x=x(t)$ with minimal period $\lambda$ yields a holomorphic map 
$c=\psi_{x(0)}: \C \to Y$, $c(e^{2\pi(s+it)})=\Fi_{e^{2\pi s}}(x(t))$,
converging to a fixed point $p=c(0)=x(0)_{\iinfty}\in \F$, and $1/\lambda$ is a weight of $T_p Y$. In particular, the minimal periods of all %
orbits of $X_{S^1}$ form a finite subset of $S^1$.
\end{cor}
\begin{proof}
The tangent space to the filling disc $c$ in the claim lies $TY_m$ where $m=1/\lambda$. 
By Lemma \ref{Lemma constant weights on F}, there are finitely many weights for each of the finitely many components of $\F$.
\end{proof}

\begin{cor}\label{CorollaryGradientTrajBecomeSpheres}
$\mathrm{Core}(Y)$ is covered by copies of $\C\P^1$ arising as the closures of $\C^*$-orbits. 
 The non-constant spheres are embedded except possibly at the two points where they meet $\F$ (where several different $\C\P^1$ may meet).
The $\C^*$-orbit closure of any $y\in \mathrm{Core}(Y)$ yields a {\ph } sphere 
$$
u_y: \C\P^1 \to Y, \;\; u_y([1:t]) = \varphi_t(y),\;\; u_y([1:0])=y_0,\;\; u_y([0:1])=y_{\infty},
$$
where $y_0$ is the convergence point of $y$, and $y_{\infty}$ is the limit of $\varphi_t(y)$ as $|t|\to +\infty$. In particular,
$$
\C\cong \mathrm{im}\, d u_y|_{[1:0]} \subset T_{y_0}Y\;\;\mathrm{and}\;\; \C\cong \mathrm{im}\, d u_y|_{[0:1]} \subset T_{y_{\infty}} Y
$$
are weight subspaces of opposite weights $k$ and $-k$ respectively, for some integer $k\geq 0$ (with $k=0$ precisely if $u_y\equiv y \in \F$ is constantly equal to a fixed point).\\ \indent The energy of that $I$-holomorphic sphere is $E(u)=\int u^*\omega=H(y_{\infty})-H(y_0)$.
\end{cor}
\begin{proof}
The first part follows by \cref{LemmaCanonicalFillingDiscs}, where we can use the inverse action by $s=t^{-1}\in \C^*$ to deal with the extension over $y_{\infty}$ (which exists as $y\in \mathrm{Core}(Y)$). The statement about weight subspaces follows from the observation that $y_0,y_{\infty}$ are fixed points and the image of $u_y$ is a subspace preserved by the $\C^*$-action, so the tangent spaces to those subspaces at $y_0,y_{\infty}$ are $\C^*$-invariant.
The fact that the weights are opposite is due to the above-mentioned change of local variables $s=t^{-1}$ on $\P^1$. That $k\geq 0$ is because $t^k$ must be contracting on $\mathrm{im}\, d u_y|_{[1:0]}$ as $y_0$ is the convergence point.
The final claim follows from: $u^*\omega= \omega(\partial_s u,\partial_t u) ds\wedge dt$ and $\omega(\partial_s u,\partial_t u)= \omega(\partial_s u,X_{S^1})= \omega(\partial_s u,X_H)=\partial_s (H\circ u)$.
\end{proof}

\begin{de}\label{Definition outer S1 period}
We call $\lambda\in [0,\infty)$ a {\bf ($\Fi$-)generic slope} if %
$\lambda$ is not equal to a period of an $S^1$-orbit; equivalently 
$\lambda\notin\cup_i \Z\cdot \tfrac{1}{w_i}\subset \Q$, 
where $w_i$ are the weights of $\Fi.$ 
Such $\lambda$ define a generic set in $[0,\infty)$ as there are only finitely many possible weights. %
We call $\lambda>0$ an {\bf outer $S^1$-period} if $\lambda = \tfrac{k}{m}$, for coprime positive integers $k,m$, such that $Y_m$ is non-compact, or equivalently $m$ is the outer weight of some $\F_\a$ (see \cref{Remark H+}).\footnote{Only the outer $S^1$-periods arise when we consider the {\MB} manifolds $B_{k/m,\beta}$ (see the introduction). One can allow a Hamiltonian $H_{\lambda}$ to have non-generic slope $\lambda$, if $\lambda$ is not an outer $S^1$-period. Our definition of generic $\lambda$ is the most suitable one for stating results about the indices in \cref{Section RS indeces}. If $\lambda$ is an $S^1$-period which is not an outer $S^1$-period, one must be cautious that the Hamiltonian $\lambda H$ will detect the $1$-orbits in $\mathrm{Core}(Y)$ of ``inner'' $S^1$-period $\lambda$.}
\end{de}
\begin{cor}\label{CorSec2OrbitsOfLambdaH}
For generic slopes $\lambda$, the only $1$-periodic orbits of the rescaled Hamiltonian $\lambda\cdot H$ are the constant orbits at points $x \in \F$, in particular there are no
$1$-periodic orbits in $Y^{\mathrm{out}}=Y\setminus Y^{\mathrm{in}}$. 
\end{cor}
\subsection{Attraction graphs}
\label{Subsection Attraction graphs}
The geometry of the $\C^*$-flow within the core 
is depicted as follows: 
\begin{de} The \textbf{attraction graph} $\G_{\Fi}$ %
is a directed graph whose vertices %
represent the $\F_\a$, and there are $N$ edges from
	 $\a_1$ to $\a_2$ if the space of $\C^*$-flowlines from $\F_{\a_1}$ to $\F_{a_2}$ (as $t\to \infty$, so $H$ increases) has $N$ connected components. 
	The \textbf{leaves} of $\G_{\Fi}$ are vertices with no outgoing edges, i.e.\,if $\F_\a$ is a local maximum of $H|_{\mathrm{Core}(Y)}$.
A leaf is {\bf $m$-minimal} if the corresponding $\F_\a$ is $m$-minimal (\cref{Definition generic points of Ymc}).
\end{de}

\begin{lm} \label{ThereAreLeaves}
    $\G_{\Fi}$ is a connected directed acyclic graph.
\end{lm}
\begin{proof}
$H$ is strictly increasing along the $\R_+$-action as $X_{\R_+}=\nabla H$, so there is no directed cycle.
$\G_{\Fi}$ is connected as there is a path of edges from $\F_{\min}$ to any $\F_\a$ (by \cref{Lemma Frankel properness} and \cref{CorollaryGradientTrajBecomeSpheres}).
\end{proof}
\begin{prop}\label{THERE_ARE_ISOTROPIES}
Each $m$-minimal leaf for $m\geq 2$ has an $m$-torsion bundle converging to it (thus the action $\Fi$ is not free  outside of $\mathrm{Core}(Y)$).

For a CSR with $\#\textrm{Vertices}(\G_{\Fi})\geq 2$, 
every leaf $\a$ is $m_\a$-minimal for the largest weight $m_\a\geq 2$ of $\F_\a$.

If $c_1(Y)=0$, $\#\textrm{Vertices}(\G_{\Fi})\geq 2$ and $\mathrm{Core}(Y)\subset Y$ is equicodimensional near $\F_{\min}$ and near some leaf, then $\varphi$ does not act freely outside of $\mathrm{Core}(Y)$.

\end{prop}
\begin{proof}
Let $\a$ be any leaf, so $TY|_{\F_\a}=T\F_\a \oplus H_-\oplus H_+$, and $\exp|_{\F_\a}$ maps a neighbourhood of the zero section of $T\F_\a \oplus H_-$
into $\mathrm{Core}(Y)$, whereas for $H_+\setminus \{0\}$ it maps into $Y\setminus \mathrm{Core}(Y)$.

Suppose a leaf $\a$ is $m$-minimal, and let $\Ymc$ be the torsion submanifold containing $\F_\a$. Near $\F_\a$, $\Ymc$ only intersects $\mathrm{Core}(Y)$ at $\F_\a$ (since there are no negative weights $mb$ for $b<0$ and all positive weight directions point out of the core at a leaf). By \cref{Cor Ymc hits core in path connected subset}, $\Ymc$ contains a single fixed component, $\F_\a$, so it is a torsion bundle over $\F_\a$.

For the second claim we use \eqref{NonDegenPairingByOmC}: the
$\omega_{\C}$-duality isomorphism $H_{s-(-m)}=H_{s+m}\cong H_{-m}$ for all $m$, where $s\geq 1$ is the weight of the CSR.
For $m_{\a}$ as in the claim, there cannot be a weight $-m_{\a}b$ for $b\geq 1$ as $\omega_{\C}$-duality would imply there is a weight $s-(-m_{\a}b)=s+m_{\a}b>m_{\a}$. 

For the third claim, we prove more generally that if $\Fi$ is free outside of $\mathrm{Core}(Y)$, and $c_1(Y)=0$, then the Maslov index $\mu$ (\cref{Subsection discussion of ma and muFalapha}) satisfies the following at leaves $\a$, for $H_+$ as in \cref{Remark H+},
$$
\codim_{\C} \F_{\min} \leq \mu \leq c_\a:=\dim_{\C}H_+ = (\textrm{local complex codimension of }\mathrm{Core}(Y)\textrm{ near }\F_\a),
$$
 and the second inequality is strict if $\#\textrm{Vertices}(\G_{\Fi})\geq 2$. 
By the freeness assumption,  all positive weights at $\F_\a$ must be $+1$ (as any $m$-torsion submanifold (for $m\geq 2$) would contradict freeness). So $\mu \leq c_\a$ by definition, with strict inequality if $\F_\a$ has a negative weight, equivalently when $\a$ is not an isolated leaf (equivalently, by  connectedness in \cref{ThereAreLeaves},  $\#\textrm{Vertices}(\G_{\Fi})\geq 2$).
At $\F_{\min}$ all non-zero weights are positive, so the Maslov index satisfies $\mu\geq \codim_{\C}\F_{\min}$. The third claim follows.
\end{proof}
\begin{de}\label{defnExtendedAttractionGraph}
	The \textbf{extended attraction graph}
 $\widetilde{\G}_{\Fi}$
 decorates ${\G_{\Fi}}$ %
 with an outward-pointing arrow at a vertex $\a$ for each torsion bundle $\Hm\fun \F_\a$. %
\end{de}
\begin{ex}\label{ExtendedAtractionGraphA4}
Let $Y$ be 
the minimal resolution of the $A_4$-singularity $\pi: X_{\Z/5}\fun \C^2/(\Z/5).$
The standard weight-2 $\C^*$-action on this CSR is induced from the weight-1 diagonal action on $\C^2.$ The core  
$\pi^{-1}(0)$ consists of a Dynkin $A_4$ tree of spheres that intersect transversely. There are five \CC-fixed points, three of which are intersections of spheres. The other two are the leaves of $\widetilde{\G}_{\Fi}$ in Figure \ref{attractionGraphA4}. Each leaf has one torsion bundle attached to it, so $\widetilde{\G}_{\Fi}$ has two outward  arrows labelled with the weight of the torsion bundle (in blue).
Each directed edge of $\G_{\Fi}$ is labelled by the positive outgoing weight and negative incoming weight, which are opposite by 
\cref{CorollaryGradientTrajBecomeSpheres}. That pairs of edges at each vertex have weights summing to $2$ is due to this CSR having weight $s=2$, and using \eqref{NonDegenPairingByOmC}.
\begin{figure}[H]
		\begin{center}
			\includegraphics[scale=0.25]{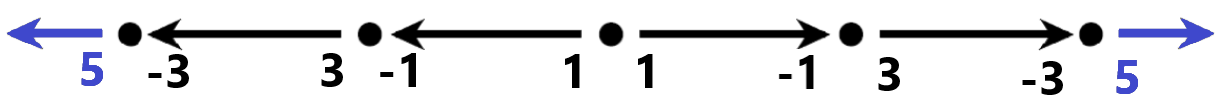}
			\caption{Extended attraction graph for the minimal resolution of the $A_4$-singularity} %
			\label{attractionGraphA4}
		\end{center}
\end{figure}
\end{ex}
\section{Computation of the Robbin--Salamon and Maslov indices}\label{Section RS indeces}

\subsection{Maslov index}\label{Subsection discussion of ma and muFalapha}
Let $(Y,\omega,\J,\Fi)$  be a
symplectic $\C^*$-manifold with non-trivial $\C^*$-action.
We now assume $c_1(Y)=0$, so Hamiltonian Floer cohomology (and thus $SH^*(Y,\Fi)$) can be $\Z$-graded by making a choice of trivialisation of the canonical bundle %
$\Lambda^{\mathrm{top}}_{\C}T^*Y$ (see \cref{AppendixFloer} and \cite[Sec.3.6]{R13}).
The Hamiltonian $S^1$-action $\varphi$ admits a \textbf{Maslov index}\footnote{In general, this involves a certain choice of lift of the action to the cover formed by capping discs, but we will always choose the lift which preserves the constant disc at a (hence any, as $c_1(Y)[\pi_2(Y)]=0$) fixed point of the $S^1$-action.} $\mu=\mu(\Fi)$ \cite{Sei97}.
At a fixed point of the $S^1$-action, its linearisation is a loop of unitary matrices $U_t$ and the Maslov index $\mu$ equals\footnote{compare \cite[Lemmas 48 and 71]{R14}.} the degree in $\Z$ of the loop of determinants $\det\, U_t \subset S^1$. Equivalently,\footnote{compare \cite[Thm.48]{R10}.} $\varphi_t^*$ acts on $\Lambda^{\mathrm{top}}_{\C}T^*Y$ by rotation with speed $\mu$. %
For the definition of {\RS} indices we refer to  \cref{AppendixFloer}.

\begin{lm}\label{LemmaSec2MaslovIndex}
$\mu=\sum w_i>0$ for the weights $w_i$ of the $S^1$-action at any fixed point $x\in \F\subset Y$.
\end{lm}
\begin{proof}
This follows from \eqref{Eqn weight spaces Ti}.
At $p\in \F_{\min}$ non-zero weights are positive, and there must be a positive weight (otherwise $Y=\F_{\min}$, $H\equiv 0$, and the $\C^*$-action is trivial), indeed $\mu \geq \codim_{\C}\F_\a\geq 1$.
\end{proof}

\begin{prop}\label{PropSec2RSIndices}
For a generic slope $\lambda$, the only $1-$periodic orbits of the Hamiltonian $\lambda H$ are the constant orbits, i.e.\,the fixed points $x\in \F$, and their {\RS} indices satisfy (uniformly in $x$)
\begin{equation}\label{indecesgotoinfinity}
\displaystyle \lim_{\lambda\fun +\infty} RS(x,\lambda H)=+\infty.
\end{equation}
\end{prop}
\begin{proof}
By \cref{CorSec2OrbitsOfLambdaH} we just need to consider $x\in \F$.
By \eqref{Eqn weight spaces Ti}, on $T_x Y = \oplus T_i$ the linearisation of the flow of $\lambda H$ is rotation in $T_i$ with speed $\lambda w_i$. The latter rotation contributes $\W(\lambda w_i)$ to the {\RS} index, in the notation of \cref{RSProperties}. Using Equation \eqref{Equation Appendix W properties}:
$$\textstyle
RS(x,\lambda H)=\sum_i  \W(\lambda w_i)
\geq \sum_i (2 \lambda w_i -1 )
=2\lambda \sum_i w_i -\dim_{\C} Y =  2\lambda \mu - \dim_{\C} Y,
$$
(we did not assume genericity of $\lambda$ here).
Since $\mu>0$, this diverges to $\infty$ uniformly in $x$ as $\lambda \to \infty$.
\end{proof}

\begin{rmk}\label{rmk CZ to infty always}
For a non-constant $1$-orbit $x$ of $\lambda H$ for non-generic $\lambda$, by homotopy invariance of {\RS} indices we can first push $x$ towards the core by using the $\R_+$-action. Thus the index of $x$ equals the index of the convergence point $\xinfty$, whose index we computed above. If $x$ arises as a $1$-orbit of a Hamiltonian $F$ in a region where $F=c(H)$, for a smooth function $c$, the {\RS} index of $x$ agrees with that computed for $\lambda H$, where $\lambda = c'(H(x))$, up to a shift in the index by $\pm \tfrac{1}{2}$ depending on the sign of $c''(H(x))$, due to a shear \cite[Sec.3.3]{OanceaEnsaios}.
Thus the {\RS} indices also diverge to $+\infty$ for $1$-orbits $x$ of $c(H)$ arising with higher and higher slopes $c'(H(x))$.
\end{rmk}

\subsection{The indices $\mu_{\lambda}(\F_\a)$}
\label{GradingConventions}
	Our {\bf grading convention} for Floer cohomology follows \cite{McLR18}, namely a $1$-periodic orbit $x$ of a Hamiltonian $F$  has grading $$|x|:=\dim_\C Y - RS(x,F).$$
For %
$C^2$-small Morse Hamiltonians, this %
agrees with the Morse index grading. %

For the Hamiltonian $\lambda H,$ the \textbf{grading} of a connected component $\F_\a$ of $\F$ is defined as
$$\boxed{\mu_\lambda(\F_\a):=\dim_\C\, Y - \dim_\C\, \F_\a- RS(x,\lambda H) }$$ which is independent of the choice of $x\in\F_\a$. This $\mu_\lambda(\F_\a)$ is the Floer grading of $\F_\a$ seen as a {\MB } manifold of Hamiltonian $1$-orbits of $\lambda H$ (see \cite[Appendix A]{RZ2}). 

 \begin{lm}\label{LemmaSec2GradingOfFa}
The $\mu_\lambda(\F_\a)$
are even integers for positive $\lambda \notin \cup_i  \Z\cdot \frac{1}{w_i}\subset \Q,$ 
where $w_i$ are weights of $\F_\a.$ 
\end{lm}
\begin{proof}
As in \cref{PropSec2RSIndices}, the 
RS-index contribution of $H_k$ is $\W(\lambda k) \cdot \dim_{\C} H_k$. For $k\neq 0$, the assumption on $\lambda$ implies $\lambda k\not\in \Z$ for the weights $k$ of $\F_\a$, thus $\W(\lambda k)$ is odd, and for $k=0$ we have $\W(\lambda k)=\W(0)=0$ (see Equation \eqref{Equation Appendix W properties}). Thus,
$$\mu_\lambda(\F_\a) = \dim_\C Y - \dim_\C \F_\a - \sum_{k\neq 0} \dim_{\C}(H_k) \W(\lambda k) \stackrel{\mathrm{mod}\,2}{\equiv} \dim_\C Y - \dim_{\C} H_0- \sum_{k\neq 0} \dim_{\C}(H_k)=0. \eqno\qed$$
\renewcommand{\qedsymbol}{}
\vspace{-\baselineskip}
\end{proof}

By definition and the proof of Lemma \ref{LemmaSec2FixedLocus}, the {\MB } index $\mu_\a$ of $\F_\a$ is twice the number of two-planes $T_i$ in 
\eqref{Eqn weight spaces Ti}
on which the linearised $S^1$-action rotates clockwise, so it is automatically even. In the notation $h_k^{\alpha}=|H_k|$ of \eqref{notation h_k^alpha for dimension of wt spaces}, where $|V|:=\dim_{\C}\,V$, we can rewrite $\mu_\a$ and the Maslov index $\mu$:
\begin{equation}\label{mu a and maslov index equations}
  \mu_\a=2 \sum_{k<0} h_k^{\alpha} \quad \textrm{ and }  \quad 
\mu = \sum_{k} k\cdot h_k^{\alpha} = \sum_{k>0} k(h_k^{\alpha}-h_{-k}^\a).   
\end{equation}

\begin{cor}\label{mu_lambda go to -infty}
$\mu_\lambda(\F_\a)\leq 2|Y|- |\F_\a| - 
2\lambda \mu$ for any $\lambda>0$, so $\mu_\lambda(\F_\a)\to-\infty$ as $\lambda\to \infty$, and
\begin{equation}\label{mu_lambda_calculation}
\mu_\lambda(\F_\a) = |Y| - |\F_\a| - \sum_k \mathbb{W}(\lambda k)\,h_k^{\alpha} 
 = \sum_{k\neq 0} (1-\mathbb{W}(\lambda k))\,h_k^{\alpha}
 = \!\!\!\!\!\sum_{w_i\neq 0\textrm{ in }\eqref{Eqn weight spaces Ti}}\!\!\!\!\!\! (1-\mathbb{W}(\lambda w_i)).
\end{equation}
\end{cor}
\begin{proof}
    We use $|\F_\a| = h_0^{\alpha}$, $|Y|=\sum h_k^{\alpha}$,  $\mathbb{W}(0)=0$. The first claim uses the proof of \cref{PropSec2RSIndices}.
\end{proof}

\begin{lm}\label{Lemma difference between muFalpha and mua}
    $\mu_\a - \mu_{\lambda}(\F_\a) = \sum_{k>0} (\mathbb{W}(\lambda k)-1)(h_{k}^{\alpha} - h_{-k}^{\alpha})$.
\end{lm}
\begin{proof}
The claim follows by using the property $\mathbb{W}(-a)=-\mathbb{W}(a)$ on the last expression below:
\begin{align}
\begin{split}
\label{CZCZ}
\mu_\a - \mu_{\lambda}(\F_\a) & = 
 \sum_{k<0} h_k^{\alpha}  -
\sum_{k>0} h_k^{\alpha} + \sum_k \mathbb{W}(\lambda k)\,h_k^{\alpha}
=
 \sum_{k<0} (1+\mathbb{W}(\lambda k))\,h_k^{\alpha} -
\sum_{k>0} (1-\mathbb{W}(\lambda k))\,h_k^{\alpha} 
\\
& 
=  \sum_{k>0} [(1+\mathbb{W}(-\lambda k))\,h_{-k}^{\alpha} - (1-\mathbb{W}(\lambda k))\,h_k^{\alpha}].
\qedhere
\end{split}
\end{align}
\end{proof}

\begin{cor}\label{For Lambda small Floer Shift Equal To Morse}
Let $\lambda_{\alpha}:=\min\{\tfrac{1}{|k|}: h_k^{\alpha}\neq 0 \textrm{ for }k\in \Z\setminus \{0\}\}=1/(\textrm{maximal absolute weight of }\F_\a)$.
\hspace{-0.5cm}
\begin{itemize}
 \item  $\mu_\lambda(\F_\a)=\mu_\a$ for all $0<\lambda < \lambda_{\alpha}$.
\item $\mu_\lambda(\F_\a)=\mu_\a-2N\mu$ for all $ N <\lambda < N+\lambda_{\alpha}$, where $N\in \N$.
\end{itemize}
\end{cor}
\begin{proof}
The first claim follows
by \cref{Lemma difference between muFalpha and mua}, as $\mathbb{W}(\lambda k)=1$ for $k>0$ and $\lambda<\lambda_{\alpha}$.
The second follows because 
$\mathbb{W}((N + \lambda)k)=\mathbb{W}(Nk + \lambda k)=2Nk+\mathbb{W}(\lambda k)$, and $\sum 2k(h_k^\a-h_{-k}^\a)=2\mu$.
\end{proof}

\subsection{Behaviour of $\mu_{\lambda}(\F_\a)$ as $\lambda$ increases}
\label{Subsection weight spaces Hk}

\begin{de}\label{Remark local rank}\label{Definition critical times}
Let $\tau_p:=\tfrac{k_0}{m}$, for $k_0,m$ coprime positive integers, so $\tau_p k \in  \Z$ precisely when $k=mb$ for some $b\in \Z$.  
Call $\tau_p$ a {\bf critical time} if some $h_{mb}^{\alpha}\neq 0$ for some $b\in \Z$. 
Call $\tau_p$ an {\bf $\alpha$-critical time} if 
some $mb$ is a weight of $\F_\a$ (so $\lambda_\a$ is the first $\alpha$-critical time).
Note that $\mu_{\lambda}(\F_\a)$
can only change when $\lambda$ is an $\alpha$-critical time (that is when $\mathbb{W}(\lambda k)$ can jump/drop).
We abusively call $0<\lambda<\min_{\alpha}\lambda_\a$ all {\bf 0-th critical times}, and $0<\lambda<\lambda_\a$ all 0-th critical $\alpha$-times (for those, $\mu_\a=\mu_{\lambda}(\F_\a)$).

Note: if both $\F_\a,\F_\gamma\subset Y_{m,\c}$ then $\mu_{\lambda}(\F_\a)=\mu_{\lambda}(\F_\gamma)$ for the critical time $\lambda=\tfrac{k_0}{m}$, as those indices equal the (homotopy-invariant) Floer grading of the $1$-orbits of $\lambda H$ whose initial points lie in $Y_{m,\c}$.

For a torsion submanifold $\Ymc$, with minimal component $\F_\a$ (\cref{Definition generic points of Ymc}),
call
$$\textstyle  \mathrm{rk}(Y_{m,\c})=|Y_{m,\c}|-|\F_\a|=\sum_{b\geq 1} h_{mb}^{\alpha}$$
 the {\bf rank} of $\Ymc$. When $Y_{m,\c}$ is a torsion bundle, this is the complex rank of the bundle. %
\end{de}
\begin{ex}\label{Example how the indices mulambda Fa jump}
We illustrate how the indices vary for an example of a $\C^*$-action on the minimal resolution of the $A_3$-singularity 
$\{XY-Z^4=0\}\subset \C^3$, obtained by lifting the $\C^*$-action $$t\cdot(X,Y,Z)=(t X,t^3 Y, t Z).$$
The fixed locus consists of $\C P^1:=\F_\beta$ and two points 
$\F_\a$, $\F_\gamma$, 
with weight decompositions $H_3\oplus H_{-2}$, $H_2\oplus H_{-1}$ where all $|H_i|=1$.
Then $\mu_{\lambda}(F_\a)= 2 - \mathbb{W}(3\lambda) + \mathbb{W}(2\lambda)$ and $\mu_{\lambda}(F_\gamma)=2-\mathbb{W}(2\lambda)+\mathbb{W}(\lambda)$ (using $\mathbb{W}(-2\lambda)=-\mathbb{W}(2\lambda)$).
The 0-th critical times are the $\lambda<  \tfrac{1}{3}$, the other critical times are $  \tfrac{1}{3}$, $  \tfrac{1}{2}$, $  \tfrac{2}{3}$, $1$, etc. At critical times $\mu_{\lambda}(F_\a)$ is $(\mu_\a=2)$, $1$, $1$, $1$, $0$, etc., and $\mu_{\lambda}(F_\gamma)$ is $(\mu_\gamma=2)$, $2$, $1$, $0$, $0$, etc.
If we allow values of $\lambda$ in between critical times, then the first sequence fails to be non-increasing as $\F_\a$ is not $2$-minimal (\cref{Lemma generic points of Ymc}): $\mu_{\lambda}(F_\a)$ is 
$(\mu_\a=2)$, $1$, $\bf 0$, $1$, $\bf 2$, $1$, $\bf 0$, $0$, etc. (in bold the values in between critical times). 
Whereas $\mu_{\lambda}(F_\gamma)$ is $(\mu_\gamma=2)$, $2$, $\bf 2$, $1$, $\bf 0$, $0$, $\bf 0$, $0$, etc.
There is a torsion submanifold $Y_{3}$ containing $\F_\a$, $Y_{2}'$ containing $\F_\a$, $Y_2$ containing $\F_{\gamma}$.
Then $\mathrm{rk}(Y_m)=1$ for $m=2,3$.
There is a flowline from $\F_\a$ to $\F_\gamma$ following the $H_{-2}$ direction while staying within $Y_{2}'$ (\cref{Lemma generic points of Ymc}), so in fact $Y_2=Y_{2}'$ are the $\Z/2$-torsion points of $Y$. 
Note $\mu_{\lambda}(F_\a)=\mu_{\lambda}(F_\gamma)$ for $\lambda=\tfrac{1}{2}$ as $Y_2=Y_{2}'$ connects $\F_\a$, $\F_\gamma$.
This fails for general $\lambda$ as we cannot connect $\F_\a$ to $\F_\b$ by $1$-orbits of $\lambda H$.
    \end{ex}

We seek a reasonable condition (which holds for all CSRs) that ensures that $\mu_{\lambda}(\F_\a)$ is bounded above by $\mu_\a$.
At $\F_\a$, order the non-zero-weight spaces by decreasing absolute weight: $(H_{k_1}\oplus H_{-k_1})\oplus (H_{k_2}\oplus H_{-k_2})\oplus \cdots$ where $k_1>k_2>\cdots > 0$ (we allow the possibility that one of $H_{k_j},H_{-k_j}$ is zero). For a general $Y$, we cannot exclude that $h_{-k_1}^\a>h_{k_1}^{\a}$, which would cause $\mu_{\lambda}(\F_\a)$ to jump above $\mu_a$ at the critical time $\tfrac{1}{k_1}$. So the condition involves the behaviour of the differences $h_{k_j}^\a - h_{-k_j}^\a$.

\begin{de}\label{Definition compatibly weighted}
We call $Y$ {\bf compatibly-weighted} if $\delta_r^\a\geq 0$ for all $r$ and all $\a$, where
    $$
    \delta_r^\a:=\sum_{j=1}^r (h_{k_j}^\a - h_{-k_j}^\a).
    $$
\end{de}

\begin{ex}\label{Exercise CSRs are compatibly weighted} 
All CSRs are compatibly-weighted. We just need the fact \eqref{NonDegenPairingByOmC} about CSRs: the
$\omega_{\C}$-duality isomorphism $H_{s-(-m)}=H_{s+m}\cong H_{-m}$ for all $m$, where $s\geq 1$ is the weight of the CSR. Thus
    $
    h_{s+m}^{\a} - h_{-m}^\a = 0.
    $
    It easily follows that $\delta_r^\a\geq 0$ since if $h_{-m}^\a$ appears in $\delta_r^a$ for $m\geq 0$, then so must $h_{s+m}^\a$ since $s+m\geq |-m|$.
    In \cref{Example how the indices mulambda Fa jump} we had a CSR $Y$ with $s=1$, and the weight $s+m=3$ ensured that $-m=-2$ (causing the problematic $\mathbb{W}(-2\lambda)$) did not make $\mu_{\lambda}(\F_\a)$ jump beyond $\mu_\a$.
\end{ex}

\begin{prop}\label{Lemma muFa is smaller than MB index mua}
If $Y$ is compatibly-weighted (\cref{Definition compatibly weighted}) then
    $\mu_{\lambda}(\F_\a)\leq \mu_\a$ for all $\lambda>0.$
\end{prop}
\begin{proof}
Recall $\mu_\a-\mu_{\lambda}(\F_\a)=\sum_{k>0} (\mathbb{W}(\lambda k)-1)(h_k^{\a}-h_{-k}^\a)$.
We will use the property that $\mathbb{W}(a)\geq \mathbb{W}(b)$ for $a\geq b$ repeatedly, and that $\mathbb{W}(\lambda k_j)-1\geq 0$ since $k_j\geq 1$.
We claim by induction on $r\geq 1$ that
$$\textstyle
\sum_{j=1}^r (\mathbb{W}(\lambda k_j)-1)(h_{k_j}^{\a}-h_{-k_j}^\a)
\geq 
(\mathbb{W}(\lambda k_r)-1)\delta_r^\a.
$$
This implies the claim: taking $r$ to be the final value, we get $\mu_\a-\mu_{\lambda}(\F_\a)\geq (\mathbb{W}(\lambda k_r)-1)\delta_r^\a\geq 0$.
For $r=1$ the inductive claim is an equality.
Now the inductive step, using $\delta_r^\a\geq 0$:
\begin{align*}
\begin{split}
\textstyle
\sum_{j=1}^{r+1} (\mathbb{W}(\lambda k_j)-1)(h_{k_j}^{\a}-h_{-k_j}^\a)
& \geq 
(\mathbb{W}(\lambda k_r)-1)\delta_r^\a + (\mathbb{W}(\lambda k_{r+1})-1)(h_{k_{r+1}}^{\a}-h_{-k_{r+1}}^\a)\\
& \geq 
(\mathbb{W}(\lambda k_{r+1})-1)\delta_r^\a + (\mathbb{W}(\lambda k_{r+1})-1)(h_{k_{r+1}}^{\a}-h_{-k_{r+1}}^\a)\\
& = (\mathbb{W}(\lambda k_{r+1})-1)\delta_{r+1}^\a. \qedhere
\end{split}
\end{align*}
\end{proof}
\begin{nota}
We write $\lambda^+$ to mean a non-critical time just above a given critical time $\lambda=\tfrac{k_0}{m}$ ($k_0,m$ coprime), with no critical times between $\lambda$ and $\lambda^+$. Similarly for $\lambda^-$ below $\lambda$. 
Abbreviate
\begin{align*}
f_m^\a &:=  (h_{m}^\a-h_{-m}^\a),\quad \textrm{ and }
\quad  
F_m^\a:=  \sum_{b\geq 1} f_{mb}^\a \qquad (\textrm{thus }\delta_r^{\a}=\sum_{m\geq k_r} f_m^{\a}),\\
N_m\mathcal{I}&:= \textrm{double-count\footnotemark{} of the number of integers in the interval }m\cdot \mathcal{I},\\
C_m\mathcal{I}&:= \textrm{double-count of the number of integers in the interval }m\cdot \mathcal{I}\textrm{ coprime to }m.
\end{align*}\end{nota}\footnotetext{if it arises in the interior of the interval we count it twice, but it only contributes one if it arises at a boundary point. Example: $N_6([\tfrac{1}{2},1])=1+2+2+1=6$ due to $3,4,5,6 \in [3,6]=6\cdot [\tfrac{1}{2},1]$, whereas $N_6((\tfrac{1}{2},1])=2+2+1=5$.}
\begin{lm}\label{Lemma Fma is jump in grading}
\strut
\strut\begin{enumerate}
    \item \label{mu lambda before and after crit value difference}
     At a critical time $\lambda=\tfrac{k_0}{m}$, with $\gcd(k_0,m)=1$,
     we have
     $$F_m^\a = 
\mu_{\lambda^-}(\F_\a) - \mu_{\lambda}(\F_\a)= \mu_{\lambda}(\F_\a) - \mu_{\lambda^+}(\F_\a)=\tfrac{1}{2} (\mu_{\lambda^-}(\F_\a)-\mu_{\lambda^+}(\F_\a)).$$
\item \label{m-minimal Fma equal to rk Ymb} If $\F_\a$ is $m$-minimal (\cref{Definition generic points of Ymc}), then $F_m^\a=\mathrm{rk}(Y_{m,\c})$.
\item \label{Shifts just before the integer time} In particular, for $H_{\pm}$  as in \cref{Remark H+}, and $N\in \N$, $$\mu_{N^-}(\F_\a)=\mu_\a-2N\mu+2(|H_+|-|H_-|).$$ 
E.g.\ for weight-1 CSRs, $\mu_{1^-}(\F_\a)=0$. %
\item \label{mu lambda via Nm and Cm} For all $\lambda>0$,
\begin{equation}\label{Equation mulambda Fa formula general}
\mu_{\lambda}(\F_\a)  \quad=\quad \mu_{\a}- \sum
N_m(0,\lambda] \cdot f_m^{\a} \quad = \quad\mu_{\a}- \sum 
C_m(0,\lambda] \cdot F_m^{\a}.
\end{equation}
\end{enumerate}
\end{lm}
\begin{proof} Part \eqref{mu lambda before and after crit value difference} follows from $\mathbb{W}(\lambda mb)-\mathbb{W}(\lambda^{-} mb) = 1 =\mathbb{W}(\lambda^+ mb)- \mathbb{W}(\lambda mb)$, using \eqref{WfunctionForRSAppendix}.
In \eqref{m-minimal Fma equal to rk Ymb}, $h_{-mb}^\a=0$ so $F_m^\a=\mathrm{rk}(Y_{m,\c})$ by \cref{Definition generic points of Ymc}.
The first part of
\eqref{Shifts just before the integer time} follows from \cref{For Lambda small Floer Shift Equal To Morse}, applying \eqref{mu lambda before and after crit value difference} in the case $\tfrac{k_0}{m}=\tfrac{N}{1}$, using $F_1^{\a}=|H_+|-|H_-|.$
The claim for weight-1 CSRs follows from \cref{GradientIsRPlus} (so $2\mu=\dim_{\C}Y$) and \cref{NonDegenPairingByOmC} (which implies that $|H_+|-|H_-|=h_1^{\a}=h_0^{\a}=\dim_{\C}\F_\a$), so $\mu_{N^-}(\F_\a)=2\dim_{\C}\F_\a+\mu_\a 
-N\dim_{\C}Y$, then use \cref{LemmaMomentMapMorseBott}\eqref{fa vs top cohomology weight-1 CSR}.
Part \eqref{mu lambda via Nm and Cm} follows from \eqref{mu lambda before and after crit value difference}, and the observation that fractions $\tfrac{\textrm{integer}}{b}$ in an interval $\mathcal{I}$ correspond (by rescaling by $b$) to integers in $b\cdot \mathcal{I}$.
\end{proof}
\begin{ex}
We illustrate \eqref{Equation mulambda Fa formula general} for $k_1=11$, $k_2=10$, $k_3=7$, $k_4=6$ with $h_{11}^\a=3$, $h_{-10}^\a=3$, $h_{7}^\a=1$, $h_{-6}^\a=1$ (all other $h_{\pm k}^\a=0$). 
Thus $F_{6}^{\a}=F_{3}^{\a}=-1$, $F_7^{\a}=1$, $F_{10}^{\a}=F_5^{\a}=-3$, $F_{11}^{\a}=3$. 
Suppose we already knew $\mu_{2/7}(\F_\a)=1$, and we want to compute $\mu_{3/7}(\F_\a)=\mu_{2/7}(\F_\a)-\sum C_m[2/7,3/7]\,F_m^{\a}$. The critical times in $\mathcal{I}:=[\tfrac{2}{7},\tfrac{3}{7}]$ are ${\bf \tfrac{2}{7}}  < \tfrac{3}{10} <  \tfrac{2}{6} <\tfrac{4}{11} < \tfrac{4}{10} < {\bf \tfrac{3}{7}}$. By definition 
$C_3\mathcal{I}=1$ due to $\tfrac{2}{6}=\tfrac{1}{3}$,
$C_5\mathcal{I}=1$ due to $\tfrac{4}{10}$,
$C_6\mathcal{I}=0$ (the $\tfrac{2}{6}$ is already accounted for in $C_3\mathcal{I}$),
$C_7\mathcal{I}=1$ due to $\tfrac{2}{7}$, $\tfrac{3}{7}$ being boundary points,
$C_{10}\mathcal{I}=1$ due to $\tfrac{3}{10}$ (not $\tfrac{4}{10}$),
$C_{11}\mathcal{I}=1$. Thus, as expected, $\sum C_m[2/7,3/7]F_m^{\a}=2 F_6^{\a}+2F_{7}^{\a}+4 F_{10}^{\a}+2F_{11}^{\a}
=-6.$
So
$\mu_{3/7}(\F_\a)=\mu_{2/7}(\F_\a)+6=7$ jumps up, because unexpectedly 
there were more fractions in $[\tfrac{2}{7},\tfrac{3}{7}]$ with the lower denominator 10 than with 11. In the following Proposition,  
it is more frequent to observe ``drops'' $f(p)\geq f(p+1)$, nevertheless even for positively-weighted $Y$ a jump $f(p)<f(p+1)$ is possible, and several consecutive jumps are possible. 
\end{ex}

\begin{prop}
Fix $\a$. Let $\lambda_p= \tfrac{p}{m}$ be a sequence of $\a$-critical times for fixed $m\geq 1$, for $p=1,2,3,4,\ldots$ not necessarily coprime to $m$. Let $f(p)$ denote any of the following three functions,
$$p\mapsto \mu_{\lambda_p^-}(\F_\a), \;\;\; p\mapsto \mu_{\lambda_p}(\F_\a),  \;\;\; p\mapsto \mu_{\lambda_p^+}(\F_\a),
$$
letting $f(0):=\mu_\a$. 
Then $K_j:=N_{k_j}[\lambda_p,\lambda_{p+1}]\in \N$ is the double-count of divisors of $m$ in the list of consecutive integers $pk_j,pk_j+1,\ldots,(p+1)k_j$, and 
\begin{equation}\label{Equation drop in f(p)}
f(p)-f(p+1)= 
\sum_{j=1}^r K_j \cdot f_{k_j}^\a
=
K_r \delta_r^\a + (K_{r-1}-K_r) \delta_{r-1}^\a + \cdots + (K_1-K_2)\delta_1^\a,
\end{equation}
where $K_m=2$; $K_j\in \{0,1,2\}$ for $k_j<m;$ 
 and the brackets $(K_{j-1}-K_j)\in \{-1,0,1,2,\ldots\}$, where $-1$ cannot occur if $mb+m-2\geq k_{j-1}>k_j\geq mb$ fails for all $b\in \N$ (e.g.\,it fails if $k_{j-1}-k_j\geq m-1$).

If $Y$ is compatibly-weighted, and $\delta_{j-1}^{\a}= 0$ when $K_{j-1}-K_j= -1$, then $f$ is non-increasing.
\end{prop}
\begin{proof}
The fractions $\tfrac{\textrm{integer}}{b}$ in $[\tfrac{p}{m},\tfrac{p+1}{m}]$ correspond (by rescaling by $mb$) to the integers in $[pb,(p+1)b]$  divisible by $m$.
By \eqref{Equation mulambda Fa formula general},
$
\mu_{\lambda_{p}}(\F_\a)-\mu_{\lambda_{p+1}}(\F_\a)
=\sum_{j=1}^r N_{k_j}[\lambda_p,\lambda_{p+1}]\cdot f_{k_j}^\a
$
so \eqref{Equation drop in f(p)} holds for $f(p)=\mu_{\lambda_p}(\F_\a)$. The other two cases follow by \cref{Lemma Fma is jump in grading}.(1)-(2).
 The number of divisors of $m$ in $N$ consecutive integers is either $\lfloor \tfrac{N}{m} \rfloor$ or $\lceil  \tfrac{N}{m} \rceil$ (these equal when $m$ divides $N$).
If $k_{j-1}=a>k_{j}=c$ but $K_{j-1}<K_j$,
then
$\lfloor  \tfrac{1+c}{m} \rfloor\leq \lfloor  \tfrac{1+a}{m} \rfloor<\lceil  \tfrac{1+c}{m} \rceil$, so 
$K_{j-1}-K_j=-1$. 
That condition means $m(b+1)>1+a>1+c>mb$ for some $b\in \N$. Rearranging gives the claim about $-1$ occurrences. The other claims are easy.
\end{proof}

\begin{lm}\label{Estimates of 1/s+ time indeces for weight-s CSRs} For any weight-$s$ CSR, $\mu_{{(i/s)}^+}(\F_\a)$ is non-increasing in $i\in \N$, given by \eqref{Equation mu i over s calculation}, with
$$\mu_{{(1/s)}^+}(\F_\a)= -\dim_\R \F_\a \quad\textrm{ and }\quad
\mu_{{(2/s)}^+}(\F_\a)=-\dim_{\R} \F_{\a}-\dim_{\C} Y - h_{s/2}^{\a}<\dim_{\R}\F_{\a}.$$
For $s\geq 3,$ $\mu_{(1/(s-1))^+}(\F_\a) = 
  -\dim_{\R} \F_\a - 2 h_1^{\a} - 2\sum_{j\geq 2} h_{(s-1)j}^{\a}$.
\end{lm}
\begin{proof}
By \cref{mu_lambda go to -infty},  $\mu_{(i/s)^+}(\F_\a) 
 = \sum_{k\neq 0} f(i,k)\,h_k^{\a}$ where $f(i,k):=1-\mathbb{W}((\tfrac{i}{s})^+ k)$.
 For $k>s$,
 $$
 (\tfrac{i}{s})^+ (s-k)=i - (\tfrac{i}{s})^+k,
 $$
 where we may write $i$ instead of $i^+$ because $[(\tfrac{i}{s})^+-\tfrac{i}{s}]k>[(\tfrac{i}{s})^+-\tfrac{i}{s}]s$ when $k>s$.
Thus 
$$
f(i,s-k)=1-\mathbb{W}((\tfrac{i}{s})^+ (s-k))=
1-2i+\mathbb{W}((\tfrac{i}{s})^+k)
=-(2i-2)-f(i,k).
$$
Now we pair the $k>s$ case with the $s-k<0$ case, using that $h^{\a}_k=h^{\a}_{s-k}$ by \eqref{NonDegenPairingByOmC}:
\begin{equation}\label{Equation pairing indices trick}
f(i,k)\,h_k^{\a}
+
f(i,s-k)\,h_{s-k}^{\a}
=
[f(i,k)-(2i-2)-f(i,k)]\,h_k^{\a}
=
-(2i-2)h_k^{\a}.
\end{equation}
The $k\neq 0$ values not paired are $k=1,2,\ldots,s$. For $k>0$, $f(i,k)\leq 0$ and it is non-increasing in $i$, so 
$$f_s(i):=f(i,1)+f(i,2)+\cdots+f(i,s-1)\leq 0,$$
and it is non-increasing in $i$.
As $(\tfrac{i}{s})^+s=i^+$, $f(i,s)=-2i$.
As $2h_s^{\a}=2h_0^{\a}= \dim_{\R}\F_\a$ (using  \eqref{NonDegenPairingByOmC}),%
\begin{equation}\label{Equation mu i over s calculation}
\mu_{(i/s)^+}(\F_\a) =
f_s(i)
-i\cdot \dim_{\R} \F_a
-
(2i-2)\sum_{k>s} h_k^{\a}.
\end{equation}
For $i=1$, $f_s(i)=0$ as $(\tfrac{1}{s})^+ k\in (0,1)$ for $1\leq k <s$, so $\mu_{(1/s)^+}(\F_\a)=- \dim_{\R} \F_a$. For $i=2$, $f(2,k)=0$ for $k<\tfrac{s}{2}$, and $f(2,k)=-2$ for $\tfrac{s}{2}\leq k<s$. So $\mu_{(2/s)^+}(\F_\a)=-\dim_{\R} \F_a -2 \sum_{k\geq s/2} h_k^{\a}$, as required since
$\dim_{\C}Y = \sum_k h_k^{\a} = -h_{s/2}^{\a} + 2 \sum_{k\geq s/2} h_k^{\a}$ (using $h_k^{\a}=h_{s-k}^{\a}$),
where $h_{s/2}^{\a}:=0$ for odd $s$. 

Finally, we have $(\tfrac{1}{s-1})^+ (s-k)=(\tfrac{1}{s-1})^+ (s-1+1-k)=1-(\tfrac{1}{s-1})^+(k-1)$ for $k>s$, so 
$$
1-\W((\tfrac{1}{s-1})^+ (s-k))=-1+\W((\tfrac{1}{s-1})^+(k-1))\leq -1+\W((\tfrac{1}{s-1})^+k).
$$
The analogue of \cref{Equation pairing indices trick} becomes
\begin{equation}\label{Equation pairing indices trick 2}
g(s,k):=[-\W((\tfrac{1}{s-1})^+k)+\W((\tfrac{1}{s-1})^+(k-1))]\ h_k^{\a}
\leq 0,
\end{equation}
moreover $g(s,k)$ is either $-2$ or $0$: the $-2$ case occurs precisely when $k= (s-1)j$ ($>k-1$) for some $j\geq 2\in \N$. 
For $k=1,\ldots,s-2$ we have $1-\W((\tfrac{1}{s-1})^+k)=0$. For $k=s-1,s$ (and $s\geq 3$) we have $1-\W((\tfrac{1}{s-1})^+k)=-2$. 
Using $h_{s-1}^{\a}=h_1^{\a}$ and $2h_s^{\a}=2h_0^{\a}=\dim_{\R}\F_\a$ (by \eqref{NonDegenPairingByOmC}), the analogue of \eqref{Equation mu i over s calculation} is
\begin{equation}\label{Equation mu i over s-1 calculation}
\mu_{(1/(s-1))^+}(\F_\a) = -2h_{s-1}^{\a} - 2h_s^{\a} + \sum_{k>s} g(s,k)
= -\dim_{\R} \F_\a - 2 h_1^{\a} - 2\sum_{j\geq 2} h_{(s-1)j}^{\a}. \qedhere
\end{equation}
\end{proof}
\subsection{Relating weight spaces of $\C^*$-related fixed loci}\label{Subsection relating weight spaces via Cstar action}
Each $\Y_{m,\c}$ is a symplectic $\C^*$-submanifold, so we can associate a Maslov index $\mu_{m,\c}$ to the $S^1$-action as in Seidel \cite[Lem.2.6]{Sei97} (see also \cite[Sec.7.8]{R14}), despite often having $c_1(\Y_{m,\c})\neq 0$ even when $c_1(Y)=0$.
In practice, at $p\in Y_{m,\c}$, $\mu_{m,\c}(o_p,c_p)$ depends not just on the $S^1$-orbit $o_p$ of period $\tfrac{1}{m}$, with $o_p(0)=p$, but also on a choice of capping $c_p:\mathbb{D}\to Y_{m,\c}$, $c_p|_{S^1}=o_p$. The $\mu_{m,\c}(o_p,c_p)$ is constant if one varies $(o_p,c_p)$ continuously.

Suppose $Y_{m,\c}$ contains $\F_\a$.
For the constant $p\equiv o_p:[0,1/m] \to Y$ with constant $c_p\equiv p$,
$$\textstyle
\mu_{m,\c}(p,p)=\sum_{b} mbf_{mb}^{\a}=\sum_{b} mb(h_{mb}^{\a}-h_{-{mb}}^{\a}).
$$
Suppose that
there is a $+\nabla H$ flowline in $Y_{m,\c}$ from $\F_\a$ to $\F_{\gamma}$ (so also $\F_{\gamma}\subset Y_{m,\c}$).
By \cref{CorollaryGradientTrajBecomeSpheres}, there is a $\C^*$-invariant pseudo-holomorphic sphere
$S_q$ in $Y_{m,\c}$ from $p\in \F_\a$ to $q\in \F_{\gamma}$. %
To simplify the discussion, 
suppose $\F_\a$ and $\F_{\gamma}$ do not have weights $\pm bm$ for $b>1$.
Then, using \cite[Lem.2.6]{Sei97},
$$
mf_m^{\a}
=
\mu_{m,\c}(p,p)
=
\mu_{m,\c}(q,S_q)
=
mf_m^{\gamma} +c_1(Y_{m,\c})[S_q].
$$
Abbreviate $|X|:=\dim_{\C}X$. Using $|\F_\a|+h_m^{\a}+h_{-m}^{\a}=|Y_{m,\c}|=|\F_{\gamma}|+h_m^{\gamma}+h_{-m}^{\gamma}$, we deduce
\begin{align*}
h_m^{\a} &= h_m^{\gamma}+\tfrac{1}{2}(|\F_{\gamma}|-|\F_{\a}|) +\tfrac{1}{2m}c_1(Y_{m,\c})[S_q],
\\
h_{-m}^{\a} &=  h_{-m}^{\gamma}+\tfrac{1}{2}(|\F_{\gamma}|-|\F_{\a}|)  - \tfrac{1}{2m}c_1(Y_{m,\c})[S_q].
\end{align*}
Now inductively proceed with $Y_{k,\beta}\supset Y_{m,\beta}$ for a divisor $k|m$, assuming by induction that we already related $h_{\pm k'}^{\a}$ with $h_{\pm k'}^{\beta}$ for $k'>k$.
This method determines all $h_{\pm k}^{\a}$ in terms of $h_{\pm k}^{\gamma}$ and $c_1(Y_{k,\beta})[S_q]$ for all divisors $k|m$.
The inductive equations, summing over divisors $n|m$ with $n> k$, are
\begin{align*}
 h_k^{\a} &=  h_{k}^{\gamma}+\tfrac{1}{2}(|\F_{\gamma}|-|\F_{\a}|) +\tfrac{1}{2k}\sum ((k+n)(h_n^{\gamma}-h_n^{\a})+(k-n)(h_{-n}^{\gamma}-h_{-n}^{\a})+c_1(Y_{n,\c})[S_q]),
\\
 h_{-k}^{\a} &=   h_{-k}^{\gamma}+\tfrac{1}{2}(|\F_{\gamma}|-|\F_{\a}|)  +\tfrac{1}{2k}\sum ((k-n)(h_n^{\gamma}-h_n^{\a})+(k+n)(h_{-n}^{\gamma}-h_{-n}^{\a})- c_1(Y_{n,\c})[S_q]).
\end{align*}
    
\section{Symplectic cohomology associated to a Hamiltonian $S^1$-action} \label{SectionSHassociatedToHamS1Action}

\subsection{Symplectic $\C^*$-manifolds over a convex base}\label{Subsection Kahler mfds admitting C*action}\label{Subsec overview}
We will assume the reader is familiar with Hamiltonian Floer theory and symplectic cohomology (e.g.\,\cite{Sa97, Sei08, R10}), in particular we use the notation, terminology and conventions from \cite{R10} unless stated otherwise.

\begin{de}\label{Def:KahlerMfdWithProjection}
 $(Y,\omega,\J ,\Fi)$ is a \textbf{symplectic $\C^*$-manifold over a convex base} 
if there is an $(I,I_B)$-{\ph } proper\footnote{meaning that preimages via $\Psi$ of compact subsets in $B$ are compact in $Y$.} map 
$$
\Psi: Y^{\mathrm{out}} = Y \setminus \mathrm{int}(Y^{\mathrm{in}}) \to B^{\mathrm{out}}=\Sigma \times [R_0,\infty),
$$
where $R_0 \in \R$ is any constant; $(\Sigma,\alpha)$ is a closed contact manifold; $I_B$ is a $d(R\alpha)$-compatible almost complex structure on $B^{\mathrm{out}}$ of contact type such that 
\begin{equation}\label{Equation Psi of XS1 is XfR main}
\Psi_* X_{S^1} = X_{fR}
\end{equation}
where $f: \Sigma \to (0,\infty)$ is a Reeb-invariant function (i.e.\,$df(\mathcal{R}_B)=0$).
Here $I$ and $I_B$ are \textit{almost} complex structures (but we often abusively say $\Psi$ is holomorphic).

We also assume that $B^{\mathrm{out}}$ is geometrically bounded at infinity (see \cref{Rmk weight w of Reeb} for explanations).

We often abusively write $\Psi: Y \to B$ even though the map $\Psi$ is only defined at infinity, and we sometimes write $B^{\mathrm{out}}$ instead of $B$ even though 
$B^{\mathrm{out}}$ is not required to have a filling $B$.

A stronger condition is to be a symplectic $\C^*$-manifold %
\textbf{globally defined over a convex base}:
we mean there is a {\ph } proper $S^1$-equivariant map
$\Psi:(Y,I) \to (B,I_B)$
defined on all of $Y$, whose target is a symplectic manifold
$(B,\omega_B,\J_B)$ convex at infinity, with a Hamiltonian $S^1$-action, whose Reeb flow at infinity agrees with the $S^1$-action. It is understood that $\J_B$ is an $\omega_B$-compatible almost complex structure, of contact type at infinity.
The definition in fact implies that $\Psi$ is also $\C^*$-equivariant (see \cref{Rmk weight w of Reeb}). 
\end{de}

\begin{rmk}\label{Rmk weight w of Reeb}
$B$ being \textbf{convex at infinity} means there is a compact subdomain $B^{\mathrm{in}}\subset B$ outside of which we have a \textbf{conical end} $B^{\mathrm{out}}:=B\setminus\mathrm{int}(B^{\mathrm{in}}) \cong \Sigma \times [R_0,\infty)$ such that the symplectic form becomes $\omega_B=d(R\alpha)$. 
The radial coordinate $R\in [R_0,\infty)$ yields the Reeb vector field $\mathcal{R}_B$
for the contact hypersurface $(\Sigma,\alpha)$, $\Sigma:=\{R=R_0\}$ (defined by $d\alpha(\mathcal{R}_B,\cdot)=0$ and $\alpha(\mathcal{R}_B)=1$). So $\mathcal{R}_B=X_R$ is the Hamiltonian vector field for the function $R$. 
After increasing $R_0$ if necessary, we can always assume that $I_B$ is $\omega_B$-compatible and of {\bf contact type} on $B^{\mathrm{out}}$, meaning%
\footnote{Equivalently $dR= R\alpha\circ I_B$, so $I_B$ preserves the contact distribution $\xi=\ker \alpha \subset T\Sigma$. The $\omega_B$-compatibility condition ensures that $d\alpha$ is an $I_B$-compatible symplectic form on $\xi$.
By \cite[Lemma C.9]{R16}, it suffices to assume $a(R)dR= R\alpha\circ I_B$ for a positive smooth function $a$, equivalently
$I_B Z_B = a(R)\mathcal{R}_B$. 
} 
$I_B Z_B = \mathcal{R}_B$, where $Z_B = R \partial_R$ is the {\bf Liouville vector field} defined by $\omega_B(Z_B,\cdot)=R\alpha$ on $B^{\mathrm{out}}$. If $I_B$ does not depend on $R$,
then clearly $B^{\mathrm{out}}$ is geometrically bounded at infinity due to the radial symmetry. 
We allow non-radially invariant $I_B$ as it imposes fewer constraints on the {\ph}ity assumption on $\Psi$.
If $I_B$ depends on $R$ (on the $\xi=\ker d\alpha$ orthogonal summand of $TB^{\mathrm{out}}=\xi\oplus \R Z_B\oplus \R \mathcal{R}_B$),
then it is desirable to require
that $B^{\mathrm{out}}$ is geometrically bounded at infinity. This assumption is needed to prove that  Floer solutions ``consume $F$-filtration'' if they go far out at infinity on a long region on which $c'(H)$ is linear (the $F$-filtration is constructed in \cref{filtrationFloer}, but this property is proved in \PartII). This property is needed in \cref{Prop filtration is stable}, in the construction of the $Q_{\Fi}$ class  (see the footnote to \cref{Theorem neg lb paper summary}), and it is used in {\PartII} to ensure a certain consistency between {\MBF } spectral sequences so that we can take the direct limit over slopes $\lambda$.

If $f\equiv w>0$ is a constant in \eqref{Equation psi Xs1 is Reeb 2}, one may as well assume \cref{Equation psi Xs1 is Reeb}  by rescaling
$R,\alpha,R_0$ to $wR,\alpha/w,wR_0$ (leaving $\omega_B=d(R\alpha)$ and the Liouville form $\theta=R\alpha$ unchanged on $B^{\mathrm{out}}$). 

If \cref{Equation psi Xs1 is Reeb} holds, then it follows that the Reeb flow on the image of $\Psi$ is an $S^1$-action and the map $\Psi: Y^{\mathrm{out}} \to \mathrm{Im}(\Psi)$ is $S^1$-equivariant. In the more general case of \eqref{Equation psi Xs1 is Reeb 2}, $$\Psi_* X_{S^1} = X_{fR} = f\mathcal{R}_B + R X_f$$ generates an $S^1$-action on $\mathrm{Im}(\Psi)$.
As we do not assume that $\Psi$ is surjective, in both cases it is not necessary for those flows to arise from an $S^1$-action defined on all of $B$.
On the image $\mathrm{Im}(\Psi)\subset B^{\mathrm{out}}$ the $S^1$-action in fact extends to a (partially defined) $\C^*$-action. Indeed, in $Y$ the vector fields $X_{\R_+}$, $X_{S^1}$ commute, therefore the same holds for their $\Psi_*$-pushforwards 
$\Psi_*X_{\R_+}$, $\Psi_*X_{S^1}$ on $\mathrm{Im}(\Psi)$, where 
$$\Psi_*X_{\R_+} = \Psi_*(-IX_{S^1})=-I_B\Psi_*X_{S^1} = -I_BX_{fR} = \nabla (fR).$$
Also, the integrability of $X_{\R^+},X_{S^1}$ combined with the $\Psi$-projection of their flow implies the integrability of those  $\Psi_*$-pushforwards 
on $\mathrm{Im}(\Psi)$.
If \eqref{Equation psi Xs1 is Reeb} holds, then $\Psi_*X_{\R_+}=\nabla R = Z_B$ is the Liouville field.
\end{rmk} 

\begin{rmk}[Examples]
When a symplectic $\C^*$-manifold $Y$ is a Liouville manifold whose Reeb flow at infinity is the $S^1$-action, it lies over a convex base with $\Psi$ being the restriction to $Y^{\mathrm{out}}$ of the identity map $\mathrm{id}:Y \to B=Y$ (or of a suitable Liouville flow map if $\Psi_*X_{S^1}=w\mathcal{R}_B$ and one wants to get rid of the constant $w>0$ by the rescaling trick).
In this case, symplectic cohomology $SH^*(Y,\varphi)$ (defined later) agrees with the usual symplectic cohomology $SH^*(Y)$ \cite{Vi96,Sei08}.
The analogous statement holds when a symplectic $\C^*$-manifold $Y$ is convex at infinity. Examples of non-Liouville but convex examples are negative complex line bundles (e.g.\,see \cite{R14}).
Examples of non-convex symplectic $\C^*$-manifolds $Y$ globally defined over non-Liouville convex bases are negative complex vector bundles $E \to M$ over closed symplectic manifolds $M$, using the natural $\C^*$-action on the fibres \cite[Sec.11.2]{R14}.
\end{rmk}

\begin{lm}
The action $\Fi$ on a symplectic $\C^*$-manifold over a convex base is contracting.
\end{lm}
\begin{proof} This follows by the properness of $\Psi$ and that on $B^{\mathrm{out}}$ we have $\Psi_*X_{\R_+}=\nabla(fR)$, with $f>0$.
\end{proof}
Note that $H_B:=fR$ is a Hamiltonian for the $S^1$-flow on $\mathrm{Im}(\Psi)$.
 By Definition \ref{Def:KahlerMfdWithProjection}, using $X_{S^1}=X_H$ on $Y$, and $\mathcal{R}_B=X_R$ on $B^{\mathrm{out}}$, we deduce that on $B^{\mathrm{out}}$ %
\begin{equation}\label{Eq:ProjectionOfHamVF}
\Psi_*  X_H = X_{fR} = f\mathcal{R}_B + RX_f = X_{H_B}.
\end{equation}
However, there may be no relationship between $\omega$ and $\Psi^*\omega_B$, so $H$ need not be related to $H_B\circ \Psi.$

We now show that one can ``twist'' the symplectic structure on $Y$ without affecting the class $[\om]\in H^2(Y;\R)$ in order to get a \emph{proper} moment map. This will be useful in later sections.

\begin{lm}\label{Lemma making H proper}
Let $\Psi: Y^{\mathrm{out}} \fun B^{\mathrm{out}}=\Sigma \times [R_0,\infty)$ be a symplectic $\C^*$-manifold over a convex base, and $\phi:[R_0,\infty)\to [0,\infty)$ a non-decreasing unbounded smooth function vanishing near $R=R_0$. Then
$$\om_{\phi}:=\om+d(\Psi^*(\phi(R)\alpha))$$ 
is symplectic, cohomologous to $\om$, and the $S^1$-action is Hamiltonian for $\om_{\phi}$, with proper moment map.

If $\Psi: Y \fun B$ is a symplectic $\C^*$-manifold globally defined over a convex base, $\om+\Psi^*\om_B$ is a symplectic form for which the $S^1$-action is Hamiltonian and has a proper moment map.

\end{lm}
\begin{proof}
We first prove the second statement: by holomorphicity, $(\Psi^*\omega_B)(v,\J v)=\omega_B(\Psi_* v, \J_B \Psi_* v)\geq 0$ as $\J_B$ is $\omega_B$-compatible. Similarly, $(\Psi^*\omega_B)(v_1,\J v_2)=(\Psi^*\omega_B)(v_2,\J v_1)$. So $\omega+\Psi^*\omega_B$ is $\J$-compatible,
thus symplectic.
As $\Psi$ is $S^1$-equivariant and $X_{S^1,B}=X_{H_B}$ is Hamiltonian (with $H_B=R$ at infinity),
$$(\Psi^*\omega_B)(\cdot,X_{S^1})=\omega_B(\Psi_*\cdot,X_{H_B}) = \Psi^*dH_B = d(H_B\circ \Psi).$$
Hence $H+H_B\circ \Psi$ is the Hamiltonian for $X_{S^1}$ on $(Y,\omega+\Psi^*\omega_B)$ outside of a compact subset of $Y$.
As $H$ is bounded below, on the sublevel set $Y_{\leq C}=\{H+H_B\circ \Psi \leq C\}$ (for $C\in \R$) the function $H_B\circ \Psi$ is bounded, so $Y_{\leq C}$ lies in the $\Psi$-preimage of a sublevel set of $H_B$ in $B$. As $H_B$ is proper, since $H_B=R$ at infinity, and $\Psi$ is proper, it follows that $Y_{\leq C}$ is compact. Properness of $H+H_B\circ \Psi$ follows.

The form $\om_{\phi}=\om+\Psi^*(d(\phi(R)\alpha))$ is well-defined on $Y$ even though $\Psi$ is only defined on $Y^{\mathrm{out}}$, because $\phi$ vanishes near $R=R_0$ so $\Psi^*(d(\phi(R)\alpha))$ extends by zero over $Y^{\mathrm{in}}$.
Also $\om_\phi$ 
is $\J$-compatible because $d(\phi(R)\alpha)=\phi'(R)\,dR\wedge \alpha + \phi(R) d\alpha$ satisfies $(d(\phi(R)\alpha))(v,\J_B v)\geq 0$
(using that $\J_B$ is of contact type and $\phi'\geq 0$). 
The $S^1$-flow on $Y$ is symplectic for $\om+\Psi^*(d(\phi(R)\alpha))$ because $(d(\phi(R)\alpha))(\cdot,X_R)=\phi'(R)\,dR=d(\phi(R))$ (using that $X_R$ is the Reeb vector field).
The Hamiltonian is now $H+\phi(R\circ \Psi)$, and the proof of properness is analogous (as $\phi$ is proper), using that
 $Y^{\mathrm{in}}$ is compact so $H|_{Y^{\mathrm{in}}}$ is proper.
\end{proof}

\begin{rmk}\label{Rmk weak convexity} [Weak convexity condition at infinity]
At the cost of some intuition,  \cref{Def:KahlerMfdWithProjection} need not make any reference to $B$ or $(\Sigma,\alpha)$ as follows. We pull-back all data via $\Psi$ to $Y$,
$$
\Theta := \Psi^*(R\alpha), \qquad \Omega := \Psi^*\omega_B=d\Theta, \qquad \rho:=R\circ \Psi, \qquad F:=f\circ \Psi.
$$

By \cite[Lemma C.6]{R16}, the Hamiltonians $H_B=fR$, with $f$ as in \cref{Equation psi Xs1 is Reeb 2}, are characterised by the conditions $H_B\geq 0$, $dR(X_{H_B})=0$, $R\alpha(X_{H_B})=H_B$.
The proof in \cite[Theorem C2, Lemma C7]{R16} of the extended maximum principle for Floer solutions in $B=\Sigma \times [R_0,\infty)$ for such Hamiltonians only requires the contact type condition 
$R\alpha = -dR\circ I_B$.
We can rephrase the above conditions on $Y$ as: $\Theta=-d^c \rho$ 
where $d^c \rho := d\rho\circ \J$, so $\Omega=-dd^c \rho$; the condition $d\rho(X_{S^1})=0$, equivalently $S^1$-invariance of $\rho$; the relation $d\Theta(\cdot,X_{S^1})=d(F\rho)$ corresponds in $B$ to $\omega_B(\cdot,X_{H_B})=d(fR)$; and finally 
$\Theta(X_{S^1})=F\rho$.

The condition $d\Theta(\cdot,X_{S^1})=d(F\rho)$ can be rewritten as a Lie derivative condition $\mathcal{L}_{X_{S^1}} \Theta =0$: using $\Theta(X_S^{1})=F\rho$ in Cartan's formula,
$
d\Theta(\cdot,X_{S^1})=-i_{X_{S^1}} d\Theta =  d(i_{X_{S^1}} \Theta) -\mathcal{L}_{X_{S^1}} \Theta=
d(F\rho) - \mathcal{L}_{X_{S^1}} \Theta.
$

Thus, we propose the definition: a symplectic manifold $Y$ with a Hamiltonian $S^1$-action is {\bf weakly convex at infinity} if outside of a compact subset there is an exhausting $S^1$-invariant function
$$
\rho: Y^{\mathrm{out}} \to \R,
$$
giving rise to a semi-positive $(1,1)$-form $-dd^c \rho$, such that
$$
\Theta (X_{S^1})=F\rho \qquad \mathrm{and} \quad \mathcal{L}_{X_{S^1}} \Theta=0, 
$$
where $\Theta:=-d^c\rho$, and $F: Y^{\mathrm{out}}\to [0,\infty)$ is some smooth function.

Note $-dd^c\rho=d\Theta$ need not be symplectic on $Y^{\mathrm{out}}$. It is closed and semi-positive: $-dd^c\rho(v,\J v)\geq 0$ for all $v\in TY^{\mathrm{out}}$.
The class of Hamiltonians to use for Floer theory on $Y$ should\footnote{In the presence of the map $\Psi: Y^{\mathrm{out}}\to B$, a local Floer solution $u$ in $Y$ for such a Hamiltonian (using $I$)
maps to a local Floer solution $v=\Psi\circ u$ in $B$ for $(\lambda H_B,I_B)$, for which the extended maximum principle applies at infinity.} equal %
$\lambda H+\textrm{constant}$ at infinity, for a constant $\lambda>0$.
\end{rmk}

\subsection{Maximum principle for admissible Hamiltonians}\label{Subsec max principle}

\begin{rmk}[Technical symplectic assumptions on $Y$]\label{Rmk technical symplectic assumptions on Y}
We always tacitly assume that our symplectic manifold $Y$ is {\bf weakly-monotone}
so that transversality arguments in Floer theory can be dealt with by the methods of Hofer--Salamon \cite{HS95}. This means one of the following holds: \begin{enumerate} 
\item $c_1(Y)(A)=0$ when we evaluate on any spherical class $A\in \pi_2(Y)$, or 
\item $\omega(Y)(A)=0$ when we evaluate on any spherical class $A\in \pi_2(Y)$, or 
\item for some $k>0$ we have $c_1(Y)(A)=k \cdot \omega(A)$ for all $A\in \pi_2(Y)$, or 
\item the smallest positive value $c_1(Y)(A)\geq n-2$ where $\dim_{\R}Y=2n$.\end{enumerate} 

Case (2) holds if $\omega$ is exact; (3) is the {\bf monotone} case.
In \cref{Subsection computation of the continuation maps} we use {\bf weak+ monotonicity} \cite[Sec.2.2]{R14} which means the same as above except $n-2$ in (4) becomes $n-1$.

We do not use the Floer cohomology of the base $B$, so we do not require that $B$ is weakly-monotone when $Y$ is globally defined over a convex base (of course $B^{\mathrm{out}}$ is exact).
\end{rmk}

\begin{de}\label{Def:ClassOfHamiltonians}\label{Rmk monotone homotopies admissible}
For $\Psi:Y\fun B$ a symplectic $\C^*$-manifold over a convex base, $\mathcal{H}(Y,\Fi)$ is the class of \textbf{$\Fi$-admissible Hamiltonians}: any smooth function $F:Y \fun \R$ which at infinity is a linear function $\lambda\cdot H$ of the moment map for a generic slope $\lambda>0$.
A {\bf $\Fi$-admissible homotopy} $F_s$ means
\begin{enumerate}
    \item $F_s=F_-$ for $s \ll 0$ and $F_s=F_+$ for $s\gg 0$, where $F_{\pm} \in  \mathcal{H}(Y,\Fi)$;
    \item $F_s=\lambda_s\cdot H$ outside of a compact subset of $Y$ independent of $s$, with $\lambda_s>0$ possibly not generic;
    \item $\partial_s \lambda_s \leq 0$.
\end{enumerate}
\end{de}

The class of $\Fi$-admissible Hamiltonians is natural in the sense that an isomorphism of symplectic $\C^*$-manifolds (see \cref{Subsection intro Floer theory induces filtration}) yields an isomorphism of the corresponding admissible Hamiltonians (via composition with the isomorphism). The class depends on the choice of $\Fi$, which determines $H$.

\begin{thm}\label{Thm:SHForFi}
For any symplectic $\C^*$-manifold $Y$ over a convex base, Hamiltonian Floer cohomology $HF^*(F)$ for $F\in \mathcal{H}(Y,\Fi)$ is well-defined. They form a directed system via Floer continuation maps $HF^*(F_+) \to HF^*(F_-)$ induced by $\Fi$-admissible homotopies $F_s$ (which exist %
 whenever $F_+ \leq F_-$ at infinity, equivalently for slopes $\lambda_+\leq \lambda_-$).
The continuation map is independent of the choice of $F_s$; for slopes $\lambda_-=\lambda_+$ there are continuation maps in both directions and they are inverse to each other; so the groups $HF^*(F)$ up to continuation isomorphism only depend on the slope $\lambda$ at infinity of $F$, and we write $
\mathbf{HF^*(\HL)}
$
for that isomorphism class without specifying $F_{\lambda}$ except for its generic slope $\lambda$.
\end{thm}
\begin{proof}
Let $x_{\pm}$ be any given $1$-periodic Hamiltonian orbits for a given $F\in \mathcal{H}(Y,\Fi)$. We will prove that there is a compact subset $C\subset Y,$ depending only on $x_{\pm}$ and $F$, such that all Floer trajectories for $F$ converging to $x_{\pm}$ are contained in $C$. This ensures that compactness and transversality arguments required to construct $HF^*(F)$ can be dealt with just as in the compact setting of Hofer--Salamon \cite{HS95}, using that $Y$ is weakly monotone (\cref{Rmk technical symplectic assumptions on Y}).

Let $u$ be such a Floer trajectory, i.e.\;a smooth map $u:\R\times S^1 \fun Y$ with $u(-\infty,t)=x_-(t)$, $u(+\infty,t)=x_+(t),$ satisfying the Floer equation 
\begin{equation}\label{Eqn Floer eq for F and J}
\partial_s u + \J (\partial_t u -X_F )=0.
\end{equation}
By admissibility, $X_F = \lambda X_H$ outside of a compact subset. Using \eqref{Equation Psi of XS1 is XfR main}, letting $H_B:=fR$,
\begin{equation}\label{Eq:pi* of Ham is Ham}
\Psi_* X_F = \lambda\, X_{H_B}.
\end{equation}
By holomorphicity of $\Psi$, outside a compact subset of $B$, $v=\Psi\circ u$ satisfies the Floer equation for $\lambda H_{B}$,
 \begin{equation}\label{Eq: projected v continuation}
\partial_s v + \J_B(\partial_t v -\lambda X_{H_B} )=0.
\end{equation}
We apply the extended maximum principle for $v$ in $B$ \cite[Theorem C2, Lemma C7]{R16}, which prohibits $v$ from leaving a compact subset of $B$ which is determined by $\Psi\circ x_{\pm}$.
When we assume instead \eqref{Equation psi Xs1 is Reeb}, the simpler maximum principle from \cite[App.D]{R13} can be applied. 
Finally, as $\Psi$ is proper, this implies that $u$ also lies in a compact subset of $Y$ determined by $x_{\pm}$, as required.

To achieve transversality for moduli spaces of Floer solutions, one may need to perturb $I$ in regions that those solutions cross, which would ruin \eqref{Eq: projected v continuation}. However, the above argument showed that Floer solutions do not enter a neighbourhood at infinity, so we do not need to perturb $I$ at infinity.

For the Floer continuation maps, one replaces $X_F$ in equation \eqref{Eqn Floer eq for F and J} by $X_{F_s}$, 
in particular $x_{\pm}$ are now $1$-periodic orbits of $X_{F_{\pm}}$ respectively. 
The projection $v=\Psi\circ u$ outside of a compact subset of $B$ satisfies equation \eqref{Eq: projected v continuation} with $\lambda$ replaced by $\lambda_s$. The decreasing condition $\partial_s \lambda_s \leq 0$ is precisely what ensures that the maximum principle still holds by \cite[Thm.C.11]{R16} (or the simpler \cite[App.D]{R13} when \eqref{Equation psi Xs1 is Reeb} holds).
The other properties in the claim follow by standard Floer theory arguments.
\end{proof}

\begin{rmk}\label{Rmk about coeffs Novikov}
We will not review in detail the chain-level construction of $HF^*(F)$ (see e.g.\,\cite{R10}), but we remind the reader that the chain complex $CF^*(F)$ is a module over a certain Novikov field $\mathbb{K}.$ 
When $c_1(Y)=0$, we will in fact work over the {\bf Novikov field}
\begin{equation}\label{EqnSec2NovikovField}
\textstyle \mathbb{K} = \{\sum n_j T^{a_j}: a_j\in \R, a_j\to \infty, n_j\in \mathbb{B} \},
\end{equation}
where $T$ is a formal variable in grading zero, and $\mathbb{B}$ is any choice of base field. In the monotone case, the same Novikov field can be used but $T$ will have a non-zero grading \cite[Sec.2A]{R16}. In other situations, e.g.\,the weakly-monotone setup, the Novikov field is more complicated \cite[Sec.5B]{R16}.
In general, the chain complex $CF^*(F)$ is a free $\mathbb{K}$-module
generated by the $1$-periodic orbits of a generic $C^2$-small time-dependent perturbation of $F$ supported away from the region at infinity where $F$ is linear.
The perturbation of $F$ ensures, among other things, that the chain complex is finitely generated (our condition on generic slopes at infinity implies that no generators exist at infinity, so no perturbation is required there). As the perturbation does not change the slopes at infinity, up to a continuation isomorphism the group $HF^*(F)$ does not depend on the choice of perturbation of $F$. The {\MBF} approach which avoids perturbing $F$ is explained in \cite[Appendix A]{RZ2}.

The key trick was that \eqref{Eq:pi* of Ham is Ham} yields $\Psi_* X_F = \lambda X_{H_B}=X_{\lambda H_B}$, a Hamiltonian vector field on $B$. If $F=c(H)$ depended on $H$, then $\lambda=c'(H(u(s,t)))$ would depend on the domain coordinates $(s,t)\in \R\times S^1$. The maximum principle for such domain-dependent Hamiltonians would require $\partial_s \lambda\leq 0$ (see \cite[Thm.C.11]{R16} and \cite[App.D]{R13}).
Even assuming \eqref{Equation psi Xs1 is Reeb}, instances where that maximum principle must fail can be constructed \cite[Rmk.6.2.5]{FZ20} for the minimal resolution $\M=X_{\Z/5}$ of the Du Val singularity $\M_0=\C^2/(\Z/5) \iso V(XY-Z^5)\subset \C^3$ (Example \ref{ExtendedAtractionGraphA4}).
\end{rmk}

\subsection{Symplectic cohomology $SH^*(Y,\Fi)$ for a $\C^*$-action $\Fi$}

\begin{de}\label{Def:SHForFi}
By Theorem \ref{Thm:SHForFi}, to any symplectic $\C^*$-manifold $(Y,\omega,\J,\Fi)$ over a convex base we may associate the \textbf{$\Fi$-symplectic cohomology}, given by the direct limit
\begin{equation}\label{SHforGeneral}
SH^*(Y,\om,\J,\Fi):=\lim_{F \in \mathcal{H}(Y,\varphi)} HF^*(F).
\end{equation}
This is a vector space over the Novikov field $\mathbb{K}$ (the Floer continuation maps are $\mathbb{K}$-linear maps). When $c_1(Y)=0,$ there is a well-defined $\Z$-grading by the {\RS} index on $HF^*(F)$ (by Lemma \ref{ZGradingsOnComplSympl}, using that Floer continuation maps are grading-preserving). 
The above construction holds more generally if we just assume that $Y$ is weakly convex at infinity with $S^1$-action $\Fi$ as in Remark \ref{Rmk weak convexity}.
\end{de}

\begin{prop} 
The $\Fi$-symplectic cohomology $SH^*(Y,\om,\J,\Fi)$ is a unital $\mathbb{K}$-algebra admitting a unital $\mathbb{K}$-algebra homomorphism from the quantum cohomology of $Y$,%
$$
c^*:QH^*(Y,\J) \to SH^*(Y,\om,\J,\Fi).
$$
The product is the {\bf pair-of-pants product}, which arises as the direct limit of the pair-of-pants products $$HF^*(F_{\lambda_1})\otimes HF^*(F_{\lambda_2}) \to HF^*(F_{\lambda_1+\lambda_2}).$$
If $c_1(Y)=0$, those cohomology groups are $\Z$-graded and the morphisms are compatible with the grading. 
\end{prop}
\begin{proof}
For a detailed discussion of the pair-of-pants product on symplectic cohomology and of gradings we refer to \cite{R13,R14}. The new ingredient here is to explain why pair-of-pants solutions $u:S \to Y$ do not escape to infinity. This follows by the same projection trick we used in the proof of Theorem \ref{Thm:SHForFi}, by using the maximum principle for pair-of-pants solutions $v=\Psi\circ u$ in symplectic manifolds that are convex at infinity such as $B$
\cite[Lem.C.7 and the comments under Rmk.C.10]{R16}.
Explicitly (as explained in \cite[App.D]{R13}) in a complex coordinate $z=s+\sqrt{-1}t$ on the pair-of-pants surface $S$ the Floer equations are: $\partial_t u = X_F\beta_t + \J\partial_s u - \J X_F \beta_s$ and $\partial_s u = X_F \beta_s - \J \partial_t u + \J X_F \beta_t$, where $\beta$ is a certain auxiliary one-form on $S$.
So the projection trick still applies: at infinity, the projected solution $v=\Psi\circ u$ in $B$ satisfies the same equations with $X_F$ and $\J$ replaced respectively by $\lambda X_{H_B}$ and $\J_B$. 
When we assume instead \eqref{Equation psi Xs1 is Reeb}, the simpler maximum principle from \cite[App.D]{R13} applies.
\end{proof}

\subsection{Vanishing of symplectic cohomology when $c_1(Y)=0$}\label{Sec2VanishingOfSH} Let $(Y,\omega,\J,\Fi)$ be a symplectic $\C^*$-manifold over a convex base.

\begin{prop}\label{PropSec2VanishingTheorem}
If $c_1(Y)=0$,  then
$
SH^*(Y,\omega,\J,\Fi)=0.
$
\end{prop}
\begin{proof}
The following mimics \cite[Thm.48]{R10} and \cite[Sec.2.6]{McLR18}.
We compute $SH^*=\varinjlim HF^*(\lambda H)$ as a direct limit using the cofinal sequence of admissible Hamiltonians%
\footnote{Technical remark: note that the $1$-periodic orbits of $\lambda H$ are typically degenerate since they are not isolated, so by convention $HF^*(\lambda H)$ actually means that $\lambda H$ is first perturbed (usually in a time-dependent way), or that a {\MB } model is used. In our case, the orbits are constant, so an autonomous perturbation suffices.
By a standard continuation argument, the choice of (compactly supported) perturbation does not matter up to an isomorphism on Floer cohomology.}
$\lambda H$ for generic slopes $\lambda\to \infty$. As $c_1(Y)=0$, the maps in the direct limit are $\Z$-grading preserving maps. The claim follows by \cref{PropSec2RSIndices}, as $HF^*(\lambda H)$ is supported in arbitrarily negative degrees\footnote{A technical remark: in a {\MB } model for Floer cohomology, the generators are graded by an index that can differ from 
the index we calculated for our {\MB } manifolds $\F_\a$ of orbits by up to $\dim_{\R} \F_\a\leq 2\dim_{\C} Y$, so the indices still diverge as $\lambda \to \infty$. There is also a generic perturbation of the Hamiltonian whose $1$-periodic orbits are graded as in the {\MB } model, so the indices also diverge in a perturbation model for Floer cohomology (\cref{PropRSIndicesGoToInfty2}\eqref{MorseItem4}).} for large $\lambda$. 
\end{proof}
\section{Filtration on quantum cohomology}
\label{SectionFiltration}

\subsection{PSS-morphism into symplectic cohomology}
\label{Subsection PSS morph into SH}
For $(Y,\omega,\J,\Fi)$
 a symplectic $\C^*$-manifold over a convex base, and $\HL$ any admissible Hamiltonian of slope $\lambda$ in $H$ at infinity,
 \cite{PSS,R13} yields:
\begin{prop}\label{SmallHam} When $\delta>0$ is smaller than any period of $X_H,$ %
we have an isomorphism of $\k$-algebras $$HF^*(\H_\delta)\iso QH^*(Y).$$ 
Note $HF^*(\H_\delta)\iso H^*(Y;\k)$ as
$\k$-vector spaces. When $c_1(Y)=0$, all those isomorphisms are $\Z$-graded.
\end{prop}

Composing the above isomorphism with the continuation map from $\H_{\delta}$ to $\H_{\lambda},$ for generic $\lambda>\delta$: %
\begin{equation}\label{EquationCanonicalClambda}
c_{\lambda}^*: QH^*(Y) \to HF^*(\HL).
\end{equation}

\subsection{Filtration of $QH^*(Y)$ by ideals}\label{FiltrationOnSingCohomology}

Now suppose $SH^*(Y,\Fi)=0$ (e.g.\,when $c_1(Y)=0$, by 
\cref{PropSec2VanishingTheorem}).
As $SH^*(Y,\Fi)=0$ is a direct limit of continuation maps, every class in $QH^*(Y)\cong HF^*(\H_\delta)$ must map to zero in $HF^*(\HL)$ via \eqref{EquationCanonicalClambda} for large enough $\lambda$. This determines a filtration:
\begin{de}\label{Definition fi-filtration on QH}
The $\varphi$-\textbf{filtration} is $\Fil^{\varphi}_{\lambda}:=\cap\{\ker c_{\mu}^*:\mu>\lambda\textrm{ is generic}\}$, and $\Fil^{\varphi}_{\infty}:=QH^*(Y)$.
\end{de}

\begin{rmk}
    The above does not strictly require $SH^*(Y,\Fi)=0$. The same construction yields a filtration on $\ker(c^*=\varinjlim c^*_{\lambda}:QH^*(Y)\to SH^*(Y))$, with the same properties as those we explain  below. One can then extend this to a filtration of $QH^*(Y)$ via $\ker c^*\subset QH^*(Y)=:\Fil^{\Fi}_{\infty}.$  
\end{rmk}

Observe the following basic properties:
\begin{enumerate}
    \item $\Fil^{\varphi}_\lambda=0$ for $\lambda\leq 0.$ 
    \item $\Fil^{\varphi}_\lambda  \subset \Fil^{\varphi}_{\lambda'}$ for $\lambda\leq \lambda'$.
    \item $\Fil_{\lambda}^{\varphi}\subset QH^*(Y)$ is a graded\footnote{When $c_1(Y)=0$, we refer to a choice of $\Z$-grading. In the monotone setting, gradings can be taken in a certain finite group. In general, however, there is only a $\Z/2$-grading.} $\k$-vector subspace (since $c_{\mu}^*$ is grading-preserving).
\end{enumerate}

\begin{prop}\label{Prop filtration is stable}
$\Fil^{\varphi}_\lambda  = \Fil^{\varphi}_{\lambda'} $ if there is no outer\footnote{recall \cref{Definition outer S1 period}} $S^1$-period in $(\lambda,\lambda'].$ 
\end{prop}
\begin{proof}
This is not immediate: it is a consequence of the $F$-filtration defined in \cref{filtrationFloer}, whose detailed construction is carried out in \PartII.
The Hamiltonians $H_{\lambda},H_{\lambda'}$ can be constructed to have the same $1$-orbits, by modifying $H_{\lambda}$ in the region at infinity where it has slope $\lambda$, and increasing the slope to $\lambda'$. In \cite{RZ2} we prove they admit a continuation map $\psi_{\lambda',\lambda}:HF^*(H_{\lambda}) \to HF^*(H_{\lambda}')$ which is an inclusion at the level of complexes. Indeed Floer solutions $u$ which enter the region where $H_{\lambda}\neq H_{\lambda'}$ must cross a large region where the homotopy $H_s$ has slope $\lambda$, forcing the drop in $F$-filtration $-\int dF(\partial_s u)\,ds$ to be larger than the priori bound $F(x_-)-F(x_+)$ at asymptotics (compare \cref{H_lambdaIsOneDirected} and \cref{Lemma Floer continuation maps are ok}). Once Floer continuations solutions $u$ are trapped in the region where $H_s=H_{\lambda}=H_{\lambda'}$, they cannot be rigid due to $s$-translation symmetry, so the $u$ are $s$-independent and the continuation map is an inclusion. A similar argument shows that the Floer trajectories counted by the Floer differential for $H_{\lambda'}$ are also trapped in the region where $H_{\lambda}=H_{\lambda'}$, so the complexes $CF^*(H_{\lambda})$,$CF^*(H_{\lambda'})$ are identified. Finally, continuation maps compose compatibly on cohomology, so $c^*_{\lambda'}=\psi_{\lambda',\lambda}\circ c^*_{\lambda}$, thus $$\Fil^{\Fi}_{\lambda'}=\ker c^*_{\lambda'} = \ker \psi_{\lambda',\lambda}\circ c^*_{\lambda} = \ker c^*_{\lambda} = \Fil^{\Fi}_{\lambda},$$
since we showed that $\psi_{\lambda',\lambda}$ is an isomorphism.
\end{proof}

\begin{prop}\label{filtrationByIdeals} The $\Fi$-filtration is a filtration by ideals on the %
$\k$-algebra $QH^*(Y).$ %
In particular, if the unit $1\in \Fil_{\lambda}^{\Fi}$ then $\Fil_{\lambda}^{\Fi}=QH^*(Y)$ (``unity is the last to die'').
\end{prop}
\begin{proof} 
By compatibility of continuation maps with pair-of-pants products \cite{R13}, and \cref{SmallHam},
\begin{equation} %
\xymatrix@C=45pt{
QH^*(Y)\otimes QH^*(Y) \ar@{->}[r]^-{c_{\delta}\otimes c_{k}} \ar@{->}[d]_-{\textrm{Quantum cup-product}} &
HF^*(\H_{\delta}) \otimes HF^*(\H_{k})
\ar@{->}^-{\textrm{POP-product}}[d] \\
QH^*(Y)
\ar@{->}[r]^-{c_{\delta+k}} & HF^*(\H_{\delta+k})}
\end{equation}
commutes.
Let  $q\in QH^*(Y)$ and $x \in \Fil^{\Fi}_\lambda$. Then $x\in \ker c^*_{\lambda+\delta}$ for any $\delta>0$, so taking 
$k=\lambda+\delta$ in the diagram we deduce $q \star x \in \ker c^*_{\lambda+2\delta}$. As $\delta>0$ was arbitrarily small, 
we deduce $q \star x \in  \Fil^{\varphi}_\lambda$.
\end{proof}

\begin{rmk}[{\bf Choice of coefficients}]\label{Remark choice of coefficients}
The Novikov field $\k$ from \eqref{EqnSec2NovikovField} is a $\mathbb{B}$-vector space, after making a choice of the base field $\mathbb{B}$. It is flat over $\mathbb{B}$, so $H^*(Y;\k)\cong H^*(Y;\mathbb{B})\otimes_{\mathbb{B}} \k$. This induces an ($\omega$- and $\Fi$-dependent) filtration on $H^*(Y;\mathbb{B})\subset H^*(Y;\mathbb{B})\otimes_{\mathbb{B}} \k \cong QH^*(Y)$ by $\mathbb{B}$-vector subspaces for any field $\mathbb{B}$ (a better behaved filtration can be defined by specialisation, as in \cref{Subsection specialisation argument}). Floer/quantum cohomology is also defined over a Novikov ring $\k$ using any underlying ring $\mathbb{B}$ \cite{HS95} (by \cref{Rmk technical symplectic assumptions on Y}), and \cref{filtrationByIdeals} still holds.
If one forgoes the multiplicative structure by ideals, one can more generally replace $\mathbb{B}$ by any abelian group, yielding a filtration on $H^*(Y;\mathbb{B})$ by abelian subgroups.
\end{rmk}

\subsection{Specialisation to a filtration of $H^*(Y)$ by cup-ideals}\label{Subsection specialisation argument}

In \cref{Subsection specialisation intro} we defined the {\bf initial part} %
of a formal Laurent ``series'' in $T$ to be the coefficient of the lowest power of $T$.
For any subspace $S\subset H^*(Y;\k) \cong H^*(Y;\mathbb{B})\otimes_{\mathbb{B}} \k$, this gives rise to $\val(S)\subset H^*(Y;\mathbb{B})$. In general it is not sufficient to apply $\val$ to a basis over $\k$ of $S$, but it is possible to apply a row-reduction like argument over $\k$ to such a basis so that this holds as follows. 

Let $\k_{>0}\subset \k$ be the ideal consisting of elements involving only strictly positive $T$-powers.

\begin{lm}\label{Lemma ini valuation}
Let $V$ be a $\mathbb{B}$-vector space.
For any finite dimensional $\k$-vector subspace $S\subset \k \otimes_{\mathbb{B}} V$, $\val(S)$ is a $\mathbb{B}$-vector space with $\mathrm{rank}_{\mathbb{B}}(\val(S))=\mathrm{rank}_{\k}(S)$. There is a $\k$-basis $s_i$ 
such that $\val(s_i)$ is a $\mathbb{B}$-basis for $\val(S)$, yielding a $\k$-linear isomorphism $$\val(S)\otimes_{\mathbb{B}}\k \to S \subset \k\cdot(\val(S)+\k_{>0} \otimes_{\mathbb{B}}V), \; \val(s_i)\mapsto s_i.$$
\end{lm}
\begin{proof}
Given any $\k$-basis $s_1,\ldots,s_d$ of $S$, rescale by suitable powers of $T$ so that $s_i = b_{i}T^0 + (T^{>0}\textrm{-terms})$ for $b_{i}\neq 0 \in \mathbb{B}$. We prove by induction that we may assume that the $b_i$ are $\mathbb{B}$-linearly independent. Inductively assume this holds for $b_1,\ldots,b_k$. If $b_{k+1}$ is a $\mathbb{B}$-linear combination of $b_1,\ldots,b_k$, then we can redefine $s_{k+1}$ by subtracting the same $\mathbb{B}$-linear combination but replacing $b_1,\ldots,b_k$ by $s_1,\ldots,s_k$. The new $b_{k+1}$ becomes zero. Rescaling $s_{k+1}$ by a suitable negative power of $T$, redefines $b_{k+1}\neq 0$. Repeat the argument. This process must stop, otherwise we would deduce that $s_{k+1}$ can be written as a (convergent) Laurent ``series'' in $s_1,\ldots,s_k$. So after finitely many steps we may assume $b_1,\ldots,b_{k+1}$ are $\mathbb{B}$-indepenent, as required. We obtain an isomorphism $\k^d\cong \k \otimes_{\mathbb{B}} \mathbb{B}^d \to S$, $1\otimes e_ig \mapsto s_i$ for the standard basis $e_i$ of $\mathbb{B}^d$. It follows\footnote{$b_i=\val(s_i)\in \val(S)$, and the lowest order $T$-term of a general $\k$-linear combination of $s_i$ is a non-trivial $\mathbb{B}$-linear combination of the $b_i$ which cannot vanish by $\mathbb{B}$-independence.} that 
$\val(S)=\mathrm{span}_{\mathbb{B}}(b_1,\ldots,b_d)\cong \mathbb{B}^d$. 
\end{proof}

\begin{cor}
    $\FFFF^{\Fi}:=\val(\FF_{\lambda}^{\Fi}) \subset H^*(Y;\mathbb{B})$ is a filtration by cup-product ideals, and $
    \mathrm{rank}_{\mathbb{B}} \FFFF^{\Fi}
    = 
    \mathrm{rank}_{\k} \FF_{\lambda}^{\Fi}.
    $
    There is a $\k$-linear isomorphism
    $\FFFF^{\Fi}\otimes_{\mathbb{B}}\k\to \FF_{\lambda}^{\Fi} \subset \k \cdot ( \FFFF^{\Fi} + H^*(Y;\k_{> 0}))$.
\end{cor}
\begin{proof}
Suppose $x=x_0 T^{a_0} + (T^{>a_0}\textrm{-terms})\in \FF_{\lambda}^{\Fi}$.
Let $r\in H^*(Y;\mathbb{B})$.
By definition, quantum product
$$
r\star x = r\cup x + (\textrm{higher }T\textrm{ terms})
= T^{a_0}(r\cup x_0)  + (T^{>a_0}\textrm{-terms})
$$
is cup product with higher corrections.
As $\FF_{\lambda}^{\Fi}$ is an ideal with respect to quantum product, $r\star x\in \FF_{\lambda}^{\Fi}$.
So $r\cup x_0 = \val(r\star x) \in \FFFF^{\Fi}$, as required.
The claim about the ranks and isomorphism follows from \cref{Lemma ini valuation} (using $S:=\FF_{\lambda}^{\Fi}$ and $V:= H^*(Y;\mathbb{B})$).
\end{proof}

\subsection{Description of the map $c_{\lambda}^*:QH^*(Y) \to HF^*(F_{\lambda})$}
\label{Sec2PerturbationModel}

\begin{prop}\label{PropRSIndicesGoToInfty2}
Let $Y$ be a symplectic $\C^*$-manifold over a convex base.
Let $f_\a: \F_\a \to \R$ be any Morse function on $\F_\a$.
One can construct an admissible Hamiltonian $\widetilde{F}:Y \to \R$ of slope $\lambda$ by making an autonomous perturbation  of $F=\lambda H$ supported in disjoint neighbourhoods of the $\F_\a$, such that:
\begin{enumerate}[(1)]
\item\label{MorseItem1} $\widetilde{F}$ is Morse and its critical points are precisely the critical points $p$ of the $f_\a$.
\item\label{MorseItem2} The $1$-periodic orbits of $\widetilde{F}$ are the (isolated) constant orbits $x_{\a,p}$ at the $p\in \mathrm{Crit}(f_\a)$.
\item\label{MorseItem3} $|x_{\a,p}|=
\mu_{f_\a}(p)+\mu_{\lambda} (\F_\a)
$ is their grading\footnote{A  $\Z$-grading if $c_1(Y)=0$; a grading in a certain finite group if $Y$ is monotone; a $\Z/2$-grading otherwise. The $\mu_{\lambda} (\F_\a)$ are the indices from \cref{GradingConventions}.} in $HF^*(\tilde{F})$, 
where $\mu_{f_\a}(p)$ is the Morse index of $p$.
\item\label{MorseItem4} $|x_{\a,p}|\to -\infty$ as $\lambda\fun +\infty$.
\item\label{MorseItem5} For generic choices of $f_\a$, the function $\widetilde{F}$ is Morse-Smale, in particular the Morse trajectories of $-\varepsilon_\a\nabla f_\a$ in $\F_\a$ are regular Morse trajectories of $-\nabla\widetilde{F}$ in $Y$.
\item\label{MorseItem6} There is a continuation isomorphism
$$
BHF^*(F;f_\a)\cong
HF^*(\widetilde{F}),
$$
where $BHF^*(F;f_\a)$ is the {\MBF } cohomology of $F:Y\to \R$ using the auxiliary Morse functions $f_\a$ on $\mathrm{Crit}(F)=\F=\sqcup_\a \F_\a$ $($see \cite[Appendix A.1]{RZ2}$)$.
\item\label{MorseItem7} There is a convergent spectral sequence
\begin{equation}\label{Equation spectral sequence for Falpha}
\bigoplus_{\a} H^*(\F_\a;\k)[-\mu_\lambda(\F_\a)]
\Rightarrow BHF^*(F;f_\a)\cong HF^*(\widetilde{F}),
\end{equation}
where $H^*(\F_\a;\k)[-\mu_\lambda(\F_\a)]$ arises from the direct product of the Morse complexes for $f_a:\F_\a\to \R$ (shifted in grading) arising as the low-energy local {\MBF } cohomology $($involving only simple cascades in the {\MB } submanifolds $\F_\a$, see \cite[Appendix A.1]{RZ2}$)$. Equivalently, it arises as the low-energy local Floer cohomology of $\widetilde{F}$ near $\F_\a$.
\end{enumerate}
\end{prop}
\begin{proof}

Following \cite{BH13}, we pick bump functions $\rho_\a$ supported near the $\F_\a$ and we define
\begin{equation}\label{EqnFtildeMB}
    \widetilde{F}=F+ \sum \varepsilon_\a \rho_\a \widetilde{f}_\a,
\end{equation}
where $\widetilde{f}_\a$ is an extension of $f_\a$, constant in normal directions to $\F_\a$ (after parametrising a tubular neighbourhood of $\F_\a$ by its normal bundle in Y %
via the exponential map).
Claims (1)-(2) are now a standard perturbation argument. Since $F$ has only constant $1$-periodic orbits (as $\lambda$ is generic), the {\MB } property of $F$  ensures that the Floer action functional of any sufficiently small autonomous perturbation $\widetilde{F}$ of $F$ is still {\MB } and its $1$-periodic orbits are still constant orbits at the critical points of $\widetilde{F}$.
So \eqref{MorseItem1} follows for sufficiently small constants $\varepsilon_\a>0$, and
\eqref{MorseItem2} follows from \eqref{MorseItem1}. 
We have $RS(x_{\a,p},\widetilde{F})=RS(p,F)+ { \frac{1}{2} } \dim_{\R} \F_\a - \mu_{f_\a}(p)$ by \cite[Sec.3.3]{OanceaEnsaios}, so we get \eqref{MorseItem3} by our grading conventions (\cref{GradingConventions}). Claim \eqref{MorseItem4} follows by \cref{mu_lambda go to -infty}.
Claim \eqref{MorseItem5} is a standard transversality result, using that $F$ is {\MB } for $\F_\a$.
Claim \eqref{MorseItem6} follows by a {\MBF} continuation argument (see \cite[Appendix A.2]{RZ2}).
Claim \eqref{MorseItem7} follows by an energy-spectral sequence argument (see \cite[Appendix B.2]{RZ2}, where in our current setting 
we do not need $c_1(Y)=0$ as we can use constant cappings at constant orbits, and
the triviality of the local system follows from the complex-linearity of the linearised $S^1$-flow).
One can alternatively prove this by using \cite[Theorem C.9]{RZ2} (without the need to use $p_t$, so take $p_t:=\mathrm{id}$), to show that for small $\varepsilon_\a>0$ the low-energy local Floer cohomology of $\widetilde{F}$ near $\F$ is the direct product of the
Morse cohomologies of the $f_\a: \F_\a\to \R$ with degree shifts by $\mu_{\lambda}(\F_\a)$; and this in turn has an energy-spectral sequence converging to $HF^*(\widetilde{F})$.
\end{proof}

\begin{rmk}
By the proof of Theorem \ref{Thm:SHForFi} all Floer trajectories for $F=\lambda H$ are trapped in a compact region of $Y$, and those for $\widetilde{F}$ in \eqref{EqnFtildeMB} are arbitrarily close to that compact region if we make the support of the $\rho_{\alpha}$  sufficiently close to the $\F_\a$.
In the Floer differential, a Floer trajectory $u$ is counted with a Novikov weight $T^{E(u)}$, where $E(u)$ is the energy. When $u$ converges at the ends to critical points $p_{\pm}$ of the relevant Hamiltonian $G:Y\to \R$, then this energy is in fact determined by the ends and the spherical class $A\in \pi_2(Y)$ of $u:\R \times S^1\to Y$ extended continuously at $\pm \infty$ using its asymptotics $p_{\pm}$,
	$$
 \textstyle
	E(u) := \int_{\R\times S^1} \|\partial_s u\|^2 \, ds\,dt =
	\om[A] + G(p_-)-G(p_+).
	$$
 \end{rmk}

From now on, ordinary cohomology is understood to be taken with coefficients in $\k$.

\begin{cor}\label{pureHam} 
Suppose $Y$ (equivalently, each $\F_\a$) has no odd cohomology. Then for generic $\lambda>0$ the vector space $HF^*(\lambda H)$ is supported in even degrees, and 
the $\k$-linear map from \eqref{EquationCanonicalClambda} becomes:
\begin{equation}\label{contmap}
c_{\lambda}^*:QH^*(Y)\cong \bigoplus_{\a} H^*(\F_\a)[-\mu_\a]\fun \bigoplus_{\a} H^*(\F_\a)[-\mu_\lambda(\F_\a)] \cong HF^*(\lambda H).
\end{equation}
For small $\lambda=\delta>0$, $\mu_\a=\mu_\delta(\F_\a)$ and the isomorphism becomes \eqref{EqnFrankel} so reformulates \cref{SmallHam}. 
\end{cor}
\begin{proof}
In \cref{PropRSIndicesGoToInfty2}.\eqref{MorseItem7} the $E_1$-page is concentrated in even total degrees because $H^*(\F_\a)$ lives in even degrees by assumption, and by %
\cref{LemmaSec2GradingOfFa} 
$\mu_\lambda(\F_\a)$ is even for generic $\lambda$.
The differentials in the spectral sequence increase the total grading by one, so all differentials vanish after the $E_1$-page, so the spectral sequence has already converged.
The claim about $\mu_\a=\mu_{\delta}(\F_\a)$ is \cref{For Lambda small Floer Shift Equal To Morse}.
\end{proof}

\begin{rmk}\label{Rmk upper bound on grading for lambda H} 
Suppose $\mu_{\lambda}(\F_\a)\leq \mu_\a$ for all $\lambda>0$, $\a$ (e.g.\,see \cref{Lemma muFa is smaller than MB index mua}).
Then even if $Y$ had odd cohomology, the above argument would imply that $HF^*(\lambda H)$ is supported in degrees 
$$*\in \left[\min_{\a}\mu_{\lambda}(\F_\a),\max_{\a} \left(\mu_{\lambda}(\F_\a)+2| \F_\a|\right)\right] \subset 
\left[\min_{\a}\mu_{\lambda}(\F_\a),\max_{\a} \left(\mu_\a+2|\F_\a|\right)\right],
$$
where $|V|:=\dim_{\C}V$.
So $HF^*(\lambda H)$ has grading $\leq $ the top supported grading of $H^*(Y)$, by \cref{Lemma Frankel properness}.
\end{rmk}

\begin{rmk}[Rank drops of $\mu_{\lambda}(\F_\a)$ yield lower bounds on the filtration]\label{Remark how rank drops give lower bounds on filtration}
Observe that the total rank of the vector spaces $QH^*(Y)\cong \oplus_\a H^*(\F_\a)[-\mu_\a]$ and $HF^*(\lambda H)\iso \oplus_\a H^*(\F_\a)[-\mu_{\lambda}(\F_\a)]$ are the same, but the map $c_{\lambda}^*$ in \eqref{contmap} is grading-preserving. Thus the more $\mu_\a-\mu_{\lambda}(\F_\a)$ drops in $\lambda$ (without a compensating jump from a different value of $\a$), the better the lower bound on the rank of $\ker c_{\lambda}^*$ (this also works without the odd-cohomology vanishing assumption of \cref{pureHam}, by considering
\cref{PropRSIndicesGoToInfty2}.\eqref{MorseItem7}). This yields \cref{Equation lower bound rank of filtration} (the maximum is due to $\FF^{\Fi}_{\lambda}\subset \FF^{\Fi}_{\gamma}$ for $\lambda<\gamma$).
\end{rmk}

Considering the last remark (or equally, using \cref{Equation lower bound rank of filtration}), together with \cref{Estimates of 1/s+ time indeces for weight-s CSRs} we get:
\begin{cor} Given a weight-$s$ CSR $Y,$ 
$\FF_{{1}/{s}}^{\Fi} \supset H^{\geq 2}(Y),  \ \FF_{{2}/{s}}^{\Fi} = H^{*}(Y).$\\ 
In particular, $\FF_1^{\Fi}=H^*(Y)$ for $s=2$; and 
$\FF_1^{\Fi} \supset H^{\geq 2}(Y), \ \FF_2^{\Fi}=H^*(Y)$ for $s=1.$
\end{cor}
\subsection{Computation of the continuation maps}\label{Subsection first order approximation of continuation maps}
\label{Subsection computation of the continuation maps}
\label{Subsection about what powers of T are used in continuation solutions}

Following \cite[Sec.3.2-3.3]{R10}, in situations where the Floer action functional $A_H$ is well-defined, for continuations solutions $v$ there is a difference between the (geometric) energy $E(v):=\int |\partial_s v|^2\, ds \wedge dt = A_{H_-}(x_-)-A_{H_+}(x_+)+\int_{\R \times S^1} \partial_s H_s(v)\, ds \wedge dt$ and the topological energy $E_0(v)=A_{H_-}(x_-)-A_{H_+}(x_+)$, where $x_{\pm}(t)=v(\pm\infty,t)$. Continuation maps count solutions with Novikov weight $T^{E_0(v)}$. One can write this without reference to $A_H$,
$$\textstyle
E_0(v)=E(v)-\int_{\R \times S^1} \partial_s H_s(v)\, ds \wedge dt = \int v^*(\omega - dK \wedge dt),
$$
abbreviating $K: \R \times Y \to \R$, $K(s,y) := H_s(y)$.
If $\partial_s H_s\leq 0$, then $E_0(v)\geq E(v)\geq 0$.
If a $1$-orbit $x$ is in a region where $\partial_s H_s = 0$, then $x$ is a $1$-orbit of $H_{\pm}=H_{\pm \infty}$, and
$E_0(v_x)=0$ for the ``constant'' continuation solution $v_x(s,t)=x(t)$.

Consider homotopies of type $H_s = \lambda_s H$ with $\lambda_s$ non-increasing. Add a constant to $H$ if necessary, so that $H\geq 0$, then $\partial_s H_s\leq 0$. There can be ``constant'' continuation solutions with $E_0(v_x)>0$: for homotopies $H_s = \lambda_s H$ with $\lambda_s$ non-increasing, $H\geq 0$, and $x\in \F_\a\subset \mathrm{Crit}(H)$, let $\lambda_{\pm}:=\lambda_{\pm \infty}$, then 
$$E_0(v_x)=H_{\a}\cdot (\lambda_- - \lambda_+),$$
where $H_{\a}:=H(\F_{\a})$.
Despite $\partial_s H_s\leq 0$, one cannot exclude the possibility of non-constant continuation solutions $v$ of the same (or possibly lower) $E_0(v)$ value than $E_0(v_x)$. Indeed, all continuation solutions in a small tubular neighbourhood of $x\in \F_\a$ will have the same $E_0$ value.\footnote{$\omega=d\theta$ is exact in such a neighbourhood, so $A_H(x):=-\int x^*\theta + \int H(x)\,dt$ and $A_{H_-}(x_-)-A_{H_+}(x_+)=H_{\a}\cdot (\lambda_- - \lambda_+)$.}

It seems that the contribution of a regular constant continuation solution might be unexpectedly cancelled in the continuation map count. We can exclude this for $\F_{\min}$: as\footnote{Alternatively, replacing $\lambda_s H$ by $H_s:=\lambda_s(H-H_{\min})$ does not affect $X_{H_s}$ (the Floer continuation solutions are the same, but we count them with shifted weights). The minimisers of $E_0(v)$ are then constant continuation solutions $v_x$ at $x\in \F_{\min}$, which have $E_0(v_x)=0$. Finally, adding a constant $C$ to $H$ corresponds to a continuation map $T^{C}\cdot \mathrm{id}$. We cannot do this for general $\F_{\a}\neq \F_{\min}$ as $\partial_s(\lambda_s(H-H_{\a}))$ is not $\leq 0$ everywhere.} $H\geq H_{\min},$ 
\begin{equation}
\label{Equation estimate energy continuation} 
-\int \partial_s H_s (v)\,ds\wedge dt \geq - H_{\min} \cdot \int \partial_s \lambda_s\, ds = (\lambda_- - \lambda_+) H_{\min}=:\delta_{\min}.
\end{equation}
So constant continuation solutions are counted with weight $T^{\delta_{\min}}$; any other with $T^{>\delta_{\min}}.$ So it is convenient to choose $H_{\min}=0$, thus $\delta_{\min}=0$, by adding a constant to $H$.
As in \cref{For Lambda small Floer Shift Equal To Morse}, let $$\lambda_{\alpha}:=\min\{\tfrac{1}{|k|}: h_k^{\alpha}\neq 0 \textrm{ for }k\in \Z\setminus \{0\}\}=1/(\textrm{maximal absolute weight of }\F_\a).$$

Similarily to \cref{pureHam}, when $H^*(Y)$ lies in even degrees, Floer continuation maps become:
\begin{align}
\label{Equation 2 cont}
\psi_{\gamma,\lambda}&:HF^*(\lambda H)\cong \oplus_\b H^*(\F_\b)[-\mu_{\lambda}(\F_\b)]
\to \oplus_\b H^*(\F_\b)[-\mu_{\gamma}(\F_\b)] \cong H^*(\gamma H).
\end{align}

\begin{prop}\label{Prop filtration precise information}
Let $\F_{\a}:=\F_{\min}$ here. Suppose $H^*(Y)$ lies in even degrees. 
If for each weight $\m$ of $\F_\a$ in \eqref{Eqn weight spaces Hk} there are no integers in the interval $(|\m|\lambda, |\m|\gamma)$, then $\mu_{\lambda}(\F_\a)=\mu_{\gamma}(\F_\a)$ and the matrix for 
the part of 
\eqref{Equation 2 cont} given by $H^*(\F_\a)[-\mu_{\lambda}(\F_\a)]
\to \oplus_{\beta} H^*(\F_{\beta})[-\mu_{\gamma}(\F_{\beta})]$, 
has the form
\begin{equation}\label{Equation continuation estimate map id plus higher}
T^{\delta_{\a}}\cdot \left(\mathrm{id}_{\alpha}+T^{>0}\textrm{--terms}\right),\, \textrm{where }\delta_{\a}:=(\gamma-\lambda)\cdot H_{\a},
\end{equation}
and where $\mathrm{id}_{\alpha}$ denotes the identity map on $H^*(\F_\a)$-classes. In particular, in the notation of \cref{Subsection specialisation argument}, $$\mathrm{ini}(\ker \psi_{\gamma,\lambda})\subset \oplus_{\beta\neq \a}H^*(\F_{\beta};\mathbb{B})[-\mu_{\lambda}(\F_{\beta})].$$
\end{prop}
\begin{proof}
We will use the {\MB} model for Floer cohomology, which counts ``cascades'': this is explained in detail in the Appendix in \PartII. In particular, a basis of cycles for $H^*(\F_{\beta};\mathbb{B})$ arises from Morse cohomology over $\mathbb{B}$ for auxiliary Morse functions $f_{\beta}:\F_{\beta}\to \R$. 

Continuation maps in the {\MB} model count cascades in which a Floer continuation cylinder must be present.
The homotopy of Hamiltonians used is $\lambda_s H$ where $\lambda_s=\gamma$ for $s\ll 0$; $\lambda_s = \lambda$ for $s\gg 0$; and $\partial_s \lambda_s \leq 0$.
Thus $\partial_s H_s\leq 0$, so
continuation cascades will be counted with a factor given by a non-negative power of $T$.
By \eqref{Equation estimate energy continuation}, only the constant cascades with asymptotes in $\F_\a$ are counted with power $T^{\delta_{\a}}$, any other cascade is counted with a power $T^{>\delta_{\a}}$. 
Thus, to prove the claim, it suffices to explain why a constant cascade is regular. We consider the constant continuation cylinder at each constant $1$-orbit at $x\in \mathrm{Crit}(f_\a)\subset \F_\a$, where $f_\a: \F_\a \to \R$ is the auxiliary Morse function used in the {\MBF} model. 
The local model near $x$ is described by the weight decomposition:
$$
\C^d \oplus \bigoplus_i \C_{w_i}
$$
where $d=\dim_{\C} \F_\a$, and $\C_{w_i}$ denotes a copy of $\C$ with the weight $w_i\neq 0$ action. Thus
$$\textstyle
H_s = \pi \lambda_s \sum  w_i |z_i|^2.
$$
The Floer equation therefore decouples, and we reduce to considering the continuation map in $\C$,
$$
HF^*(\C;\pi w_i \lambda |z_i|^2) \to HF^*(\C, \pi w_i \gamma |z_i|^2)
$$
for $\pi w_i \lambda_s |z_i|^2$. One approach is to verify the regularity of the constant Floer continuation solution at $0$ directly. Indirectly, we just need to show that map, viewed as local low-energy Floer cohomologies, is an isomorphism.
But this is known:\footnote{\label{Footnote about continuation map being iso} The continuation map only depends on the two slopes at infinity. Suppose $w_i>0$. We use instead a Hamiltonian $H_{w_i\gamma}$ on $\C$ obtained from $H_{w_i\lambda}:=\pi w_i\lambda|z_i|^2$ by increasing the slope \emph{away from a neighbourhood of $0$}. A monotone homotopy $H_s$ from $H_{w_i\lambda}$ to $H_{w_i\gamma}$ satisfies the maximum principle, so continuation solutions lie in the region where $H_s=H_{w_i\gamma}=H_{w_i\gamma}$ is $s$-independent, as $H_{w_i\gamma}$ does not have any new $1$-orbits compared to $H_{w_i\lambda}$. Those solutions cannot be rigid due to $s$-reparametrisation unless they are constant at $0$. Finally the constant $0$ is regular for $H_{w_i\lambda}$ (which equals $H_s$ near $0$). For $w_i<0$, $H_{w_i\gamma}$ is obtained by decreasing the negative slope of $H_{w_i \lambda}$ at infinity, so the homotopy $H_s$ is not monotone, but we are only interested in the local low-energy Floer continuation solutions (which is meaningful by using the monotonicity lemma \cite[Lem.33]{R14}), and the same argument shows that the constant solution at $0$ is regular.
Indeed, the local low-energy continuation map is inverse to the continuation map $HF^*(H_{w_i\gamma}) \to HF^*(H_{w_i\lambda})$ in the reverse direction which involves a monotone homotopy and (by the previous argument) just counts the solution at $0$.} 
the map is an isomorphism precisely if the $1$-periodic $S^1$-action on $\C$ does not have periodic orbits with period inside $(w_i\lambda,w_i\gamma)$, equivalently $(w_i\lambda,w_i\gamma)\cap \Z=\emptyset$. The latter holds by the assumption.
The final claim now follows:
if the initial part (cf.\,\cref{Subsection specialisation argument}) of a class in $\oplus_{\beta} H^*(\F_{\beta})[-\mu_{\lambda}(\F_{\beta})]$ is in $H^*(\F_\a)[-\mu_{\lambda}(\F_\a)]$, then its $c_{\lambda}^*$-image  cannot be cancelled by other contributions, as those arise with higher order $T$-terms.
\end{proof}

\begin{cor}\label{Corrolary on until when F_min = Id under continuation}

If $H^*(Y)$ lies in even degrees, and $\lambda<\lambda_{\min}$, then $c_{\lambda}^*:H^*(\F_{\min})\to HF^*(H_{\lambda})$ is injective, so $\FF_{\lambda}^{\Fi}\cap H^*(\F_{\min})=\{0\}$.
Therefore, in the notation of \cref{Subsection specialisation argument},
$$
\FFFF^{\Fi} \subset \oplus_{\beta\neq \min} H^*(\F_{\beta};\mathbb{B})[-\mu_\beta] \qquad \textrm{and }\qquad
\FFFF^{\Fi}\otimes_{\mathbb{B}}\k\cong \FF_{\lambda}^{\Fi} \subset \k \cdot ( \FFFF^{\Fi} + \oplus_{\beta} H^*(\F_{\beta};\k_{> 0})). 
$$
Without assumptions on $H^*(Y)$, if $c_1(Y)=0$, $\lambda< \lambda_{\min}$ and $\mu_{\lambda}(\F_\beta)\geq 0$ for all $\beta$, then $1\notin \Fil^{\Fi}_{\lambda}$.
\end{cor}
\begin{proof}
The first part is immediate from \cref{Prop filtration precise information}:
if the initial part (in the sense of \cref{Subsection specialisation argument}) of a class in  $\oplus_{\beta} H^*(\F_{\beta})[-\mu_{\beta}]$ involves a non-zero entry in $H^*(\F_{\min};\mathbb{B})$, then the image under $c_{\lambda}^*$ of that entry cannot be cancelled by other contributions, as those arise with higher order $T$-terms. 

For the second part, let $1$ be the unit cycle in $QH^*(Y)$. Its image $c^*_{\lambda}(1)=T^{\delta_{\min}}\cdot (1+T^{>0}\textrm{-terms})\in CF^*(\lambda H)$ (by \cref{Prop filtration precise information}) is a cycle but we need to ensure it is not a boundary. The conditions $\mu_{\lambda}(\F_\beta)\geq 0$ ensure no chains in $CF^*(\lambda H)$ have negative grading, in particular grading $-1$ (which could give rise to a non-trivial differential killing the unit due to high-energy Floer trajectories). 
\end{proof}

\begin{rmk}
    We expected \cref{Prop filtration precise information} to hold for all $\F_\a$, and the proof as written would work if we knew that non-constant Floer continuation solutions/cascades for $H_s$ with one end on $\F_{\a}$ had $E_0$-value strictly greater than $\delta_{\a}$ (currently we obtain this only for $\F_{\min}$ via \cref{Equation estimate energy continuation}).
\end{rmk}

\begin{de}
    $\F_\a$ is {\bf $p$-stable} if $\tfrac{1}{p}\Z_{\neq 0}\cap ($weights($\F_\a))=\emptyset$. (So:\,$p\neq \tfrac{m}{k}$, for $k\in $\,weights($\F_\a$), $m\in \Z$).
\end{de}

\begin{prop}\label{Prop filtration precise information 2}
Let $H^*(Y)$ lie only in even degrees. 
Suppose that 
for each weight $m$ of $\F_\a$ %
there are no integers in the interval $(|m|\lambda, |m|\gamma)$, 
so $\mu_{\lambda}(\F_\a)=\mu_{\gamma}(\F_\a)$. 
If $\gamma-\lambda$ is sufficiently small, then the part of the map $\psi_{\gamma,\lambda}$ in \eqref{Equation 2 cont} given by $H^*(\F_\a)[-\mu_{\lambda}(\F_\a)]
\to \oplus_{\beta} H^*(\F_{\beta})[-\mu_{\gamma}(\F_{\beta})]$ equals \eqref{Equation continuation estimate map id plus higher}.

For $\lambda=p^-$, $\gamma=p^+$, and $p$-stable $\F_\a$ the two indices $\mu_{p^{\pm}}(\F_\a)$ are equal, and $\psi_{p^+,p^-}$ restricted to $\oplus_{p\textrm{-stable }\F_\a} H^*(\F_\a)[-\mu_{p^-}(\F_\a)]$ is injective.
For the smallest period $p>0$ for which $S^1$-orbits exist, the total rank $\|\FF^{\Fi}_{p}\|\leq \|H^*(Y)\|-\sum_{(p\textrm{-stable }\F_\a )} \|H^*(\F_\a)\|.$ 
\end{prop}
\begin{proof}
As $\partial H_s\leq 0$, a monotonicity lemma argument (e.g. \cite[Lem.33]{R14}) implies that given a tubular neighbourhood $N_{\a}$ of $\F_\a$, there is an $\epsilon>0$, so that continuation solutions with at least one end on $\F_\a$, and not entirely contained in $N_{\a}$, have energy $E(v)\geq \epsilon >0.$ If $\gamma-\lambda$ is small enough so that $\delta_{\a}$ in \eqref{Equation continuation estimate map id plus higher} satisfies $\delta_{\a}<\epsilon$, then $E_0(v)>\delta_{\a}$ for all cascades with ends on $\F_\a$ except possibly for those in $N_{\a}$.
Cascades in $N_{\a}$, for small $N_{\a}$ and small energy, involve the low-energy local Floer cohomology of $\F_\a$ (see the Appendix in \PartII). But the assumption implies that those local groups are isomorphic under the continuation map (cf.\,the previous \cref{Footnote about continuation map being iso}). So the total contribution of the continuation cascades in $N_{\a}$ (all contributing with weight $T^{\delta_{a}}$) is the identity continuation map times $T^{\delta_{a}}$.

The second claim follows: $\psi_{p^+,p^-}(T^{-\delta_{\a}}x_\a) = x_\a+(T^{>0}$-terms) for $x_\a\in H^*(\F_\a)[-\mu_{p^-}(\F_\a)]$ where $\F_\a$ is  $p$-stable. For the restricted map in the claim,
$\psi_{p^+,p^-}(\sum n_\a T^{-\delta_\a}x_\a)=\sum n_\a (x_\a + T^{>0}$-terms), which is non-zero unless all $n_\a\in \k$ are zero.\footnote{if the minimal order $T$ term arising in the $\psi$-output is $T^{\min}\sum b_\a x_\a$, where $b_\a\in \mathbb{B}$ are in the base field, then this must arise for the subset of the $n_\a\in \k$ of the form $T^{\min}(b_\a + T^{>0}$-terms). Then use that the $x_\a$ are independent over $\mathbb{B}$.}
The final claim follows by rank-nullity: by stability in \eqref{Intro Stability property}, $c_{p^-}^*$ is an isomorphism, and by the second claim: $\|\mathrm{rk}\,\psi_{p^+,p^-}\|\geq \sum_{(p\textrm{-stable }\F_\a )} \|H^*(\F_\a)\|$.
\end{proof}

\begin{rmk}\label{Remark explaining conj about Euler class in intro}
When $\lambda$ passes a ``critical time'' $\tfrac{k}{m}$, where $m$ is a weight of $\F_\a$, the index $\mu_{\lambda}(\F_\a)$ often drops. If $\F_\a$ does not have negative weights $-mb$ for $b\in \N$, then $\F_\a=\min(H|_{\Ymc})\subset \Ymc$ is $m$-minimal (\cref{Definition generic points of Ymc}), and
$\Ymc$ near $\F_\a$ is a complex vector bundle over $\F_\a$ of complex rank $\rk(\Ymc)$. We show that $\mu_{\lambda}(\F_\a)$ drops precisely by  $2 \mathrm{rk}(\Ymc)$ when $\lambda$ passes the critical time. Working within that bundle, there is an $m$-th root $\varphi^{1/m}$ of the action on $\Ymc$; that corresponds to a full rotation in $\Ymc$ so it has an associated class $Q_{\varphi^{1/m}}\in QH^*(\Ymc)$. One expects by \cite{R14} that $Q_{\varphi^{1/m}}$ is the Euler class $e(\Ymc)$, at least up to higher order $T$ terms. This led us to formulate \cref{Introduction conjecture about cup with Euler class}. 
\end{rmk}

\begin{lm}\label{Lemma when min weights are plus one}
    $\lambda_{min}=1 \Leftrightarrow \mu=\mathrm{codim}_{\C}\F_{\min} \Leftrightarrow (\textrm{all nonzero weights of }\F_{\min}\textrm{ are }+1)$.
\end{lm}
\begin{proof}
There are $\dim_{\C} Y - \dim_{\C}\F_{\aa}$ non-zero weights for $\F_{\aa}$, they are positive integers (as $\F_{\aa}=\min H$) and their sum is $\mu$ (\cref{LemmaSec2MaslovIndex}).
So $\mu=\dim_{\C} Y - \dim_{\C}\F_{\aa}\Leftrightarrow $ positive weights $=+1$.
\end{proof}
 
\begin{cor}\label{Cor CSR when unit dies}
Let $Y$ be a weight-2 CSR with $\Fmin=\{\mathrm{point}\},$
or a weight-1 CSR. %

Then $\Fil^{\Fi}_{\lambda}\cap H^*(\F_{\min})=\{0\}$ for $\lambda<1,$ in particular $1\in H^0(Y)$ survives until time-1. 
\end{cor}
\begin{proof}
Given a weight-2 CSR with $\Fmin=\{\mathrm{point}\},$ due to \eqref{NonDegenPairingByOmC} it has only 
+1 weights, thus \cref{Lemma when min weights are plus one} applies and
the claim follows from \cref{Corrolary on until when F_min = Id under continuation} ($\lambda<1$ excludes $\alpha=\min$ as  $\lambda_{\min}=1$). 
For a weight-1 CSR, %
$\F_{\aa}$ is an $\omega_J$-Lagrangian (the ``minimal Lagrangian'' of \cite{vzivanovic2022exact}), thus 
half of its weights are equal to 0, and the other half equal to +1 due to \eqref{NonDegenPairingByOmC},
thus the same argument applies.
\end{proof}
\begin{thm}[{\cite[Sec.1.12]{R14}}]\!\!\label{Theorem neg lb paper summary}\footnote{it is not difficult to verify that the construction and properties of this continue to hold in our setup, by using the maximum principle from \cref{Subsec max principle}, and to carry out \cite[Sec.5.4]{R14} we use the same monotonicity-lemma arguments for the projection to the convex base as in 
\cite{RZ2}, when we show that Floer solutions consume $F$-filtration when they cross large linearity regions.
For this part of the argument, we need the assumption that $B^{\mathrm{out}}$ is geometrically bounded.
}
The $S^1$-action $\Fi_t:Y \to Y$ induces a commutative diagram:
$$
\xymatrix@C=45pt@R=13pt
{ SH^*(Y)  \ar@{<-}_{\varinjlim}[d]
\ar@{->}_-{\sim}^-{\mathcal{R}_{\widetilde{\Fi}}}[rr]  & & SH^{*+2I(\widetilde{\Fi})}(Y)
 \ar@{<-}^{\varinjlim}[d] \\
HF^*(\lambda H)  \ar@{<-}_{c_{\lambda}^*}[d] 
\ar@{->}_-{\sim}^-{\mathcal{S}_{\widetilde{\Fi}}}[r]  & HF^{*+2I(\widetilde{\Fi})}((\lambda-1) H)
\ar@{->}[r]^-{\textrm{continuation}}
 &
 HF^{*+2I(\widetilde{\Fi})}(\lambda H)
 \ar@{<-}^{c_{\lambda}^*}[d]
 \\
 QH^*(Y)
\ar@{->}[rr]^-{r_{\widetilde{\Fi}}} & & QH^{*+2I(\widetilde{\Fi})}(Y)
}
$$
Here $\mathcal{R}_{\widetilde{\Fi}}$ and  $\mathcal{S}_{\widetilde{\Fi}}$ are $\k$-module isomorphisms, and $r_{\widetilde{\Fi}}$ is a $\k$-module homomorphism, given as %
quantum %
product by a (typically non-invertible) Gromov-Witten invariant $Q_{\varphi}:=r_{\widetilde{\Fi}}(1)\in QH^{2I(\widetilde{\Fi})}(Y)$,
and  $\mathcal{R}_{\widetilde{\Fi}}$ is pair-of-pants product by
the invertible element $c^*(Q_{\varphi})\in SH^{2I(\widetilde{\Fi})}(Y)$.

The construction depends on a certain choice of lift $\widetilde{\varphi}$ of the $S^1$-action, and this choice can be made \cite[Sec.3.1 and Sec.7.8]{R14} so that $I(\widetilde{\Fi})=\mu$ is the Maslov index of $\Fi$ from \cref{Subsection discussion of ma and muFalapha}.

It follows that $c^*:QH^*(Y)\to SH^*(Y)$ is a quotient map, inducing $\k$-algebra isomorphisms
$$
SH^*(Y)\cong QH^*(Y)/E_0(Q_{\varphi}) \cong QH^*(Y)_{Q_{\varphi}},
$$
where $E_0(Q_{\varphi})$ is the generalised $0$-eigenspace of quantum product by $Q_{\varphi}$, and $QH^*(Y)_{Q_{\varphi}}$ denotes localisation at $Q_{\varphi}$ of the $\k$-algebra $QH^*(Y).$ For the latter isomorphism, see \cite[Lem.4.3]{R16}.

The association of $\widetilde{\varphi}$ to the above $\k$-homomorphisms respects group multiplication, in particular $Q_{\varphi^N}=Q_{\varphi}^{\star N}$. (We caution that with the above choices of lifts, $\widetilde{\Fi_1}\circ \widetilde{\Fi_2}=T^a \widetilde{\Fi_1\circ\Fi_2}$ typically gives rise to correction factors $T^a\in \k$, some $a\in \R$, due to certain deck transformations \cite[Sec.2B,4D]{R16}).
\end{thm}

When $QH^*(Y)$ is ordinary cohomology or when\footnote{The theorem holds in the weak+ monotone setup \cite[Sec.2.2]{R14}, being cautious that the grading is no longer a $\Z$-grading outside of the $c_1(Y)=0$ setup.} $c_1(Y)=0$, it follows that $SH^*(Y)=0$, because $\mathcal{R}_{\widetilde{\varphi}}$ is a $\k$-module homomorphism of non-zero degree $\mu>0$ on a finite rank $\k$-module $SH^*(Y)$ 
(as it is a quotient of $QH^*(Y)$, which has finite rank). This gives an alternative proof of \cref{PropSec2VanishingTheorem}.

The Theorem implies that the full-rotation continuation maps %
$$
c_{N+\delta}^*: QH^*(Y)\cong HF^*(F_{\delta}) \to HF^*(F_{N+\delta}),
$$
can be identified with quantum product $N$ times by $Q_{\varphi}\in QH^{2\mu}(Y)$:
$$
r_{\widetilde{\Fi}}^N=Q_{\varphi}^{\star N} \star \cdot :QH^*(Y)
\to QH^{*+2N\mu}(Y).
$$
\begin{cor}\label{Filtration via QFi}
For any integer $N>0$, 
$\Fil^{\Fi}_{N} = \ker (Q_{\varphi}^{\star N}\star\cdot),$
thus
$\Fil^{\Fi}_{N}=E_0(Q_{\Fi})$ for large $N$.
In particular, $\Fil^{\Fi}_{N}=QH^*(Y)$ if $Q_{\varphi}^{\star N}=0$. (Compare with \cref{estimates on filtration by drops}) 
\qed
\end{cor}

\begin{lm}\label{Lemma PDFmin nonzero}
Let $Y$ be a symplectic $\C^*$-manifold, and 
$\F_{\min}$
its minimal component (\cref{Lemma Frankel properness}).
\begin{equation} \label{non vanishing equivalence}
    \mathrm{PD}[\F_{\min}]\neq 0 \in H^*(Y) \iff e(U_{\min}) \neq 0 \in H^*(\F_{\min}),
\end{equation}
where $e(U_{\min})$ is the Euler class of the normal bundle of $\F_{\min}$ (see \cref{Lemma Umin is dense}). 

In particular, 
$\dim_{\C}\F_{\min}\geq \tfrac{1}{2}\dim_{\C}Y$ is a necessary condition for the non-vanishing \eqref{non vanishing equivalence}.
\end{lm}
\begin{proof}
The last claim follows for degree reasons: $e(U_{\min})\in H^{2\,\codim_{\C}\, \F_{\min}}(\F_{\min})$.
The identification of $U_{\aa}:=W^{s}_{-\nabla H}(\F_{\min})$ with  the normal bundle of $\F_{\aa}$ follows from 
\cref{Lemma Torsion bundle is a bundle}
(letting $m=1$). 

By {\PL} duality,
$\mathrm{PD}[\F_{\aa}]\neq 0 \in H^*(Y)$ if and only if the locally finite (lf) cycle $[\F_{\aa}]\in H_*
^{lf}(Y)$ has non-trivial intersection number with some cycle $C\in H_*(Y)$.
We can perturb the lf-cycle $\F_{\aa}$ in $U_{\aa}$ as the image of a generic smooth section of the normal bundle to $\F_{\aa}$.
We can perturb the metric in the complement of a small neighbourhood of $\F_{\aa}$ to make the flow of $H$ Morse-Smale, and we can perturb $H$ to make it Morse away from $\F_{\aa}$. Then by Morse theory the cycles in $H_*(Y)$ that do not come from the inclusion $H_*(\F_{\aa})\to H_*(Y)$ can be represented as pseudo-cycles by linear combinations of the unstable manifolds of the critical points of $H$ not in $\F_{\aa}$ (the algebro-geometrical analogue of this is described in \cite[Sec.2.2]{vzivanovic2022exact}). These unstable manifolds correspond to submanifolds\footnote{using the Morse--Smale property, and for the codimension claim we use that the unstable manifold of a Morse critical point $p$ is the Morse index, and that Morse indices will be $\leq 2\dim_{\C}Y-2$ by \cref{Lemma Frankel properness} and non-compactness of $Y$.
} in 
$N^+$ of real codimension at least $2$, where $N^+\subset TY|_{\F_{\min}}$ is the subbundle where $\mathrm{Hess}(H)$ is positive definite.
It follows that we can construct a smooth section $\F_{\aa} \to U_{\aa}$ of the normal bundle that avoids the closures of those pseudo-cycles.
Therefore, we have built an lf-homologous perturbation of $[\F_{\aa}]$ which can only intersect the cycles in $\F_{\aa}$. This implies that $[\F_{\aa}] \neq 0 \in H_*^{lf}(Y)$ if and only if  $[\F_{\aa}] \neq 0 \in H_*^{lf}(U_{\min})$.
Finally, by the proof of \cite[Thm.67]{R14}, PD of $[\F_{\aa}]  \in H_{2\dim_{\C}\F_{\aa}}^{lf}(U_{\aa})$ equals the 
pull-back in $H^{2\,\codim_{\C}\,\F_{\aa}}(U_{\aa})$ of $e(U_{\aa})$.
\end{proof}

\begin{prop}\label{Prop QFi computation first order}
Suppose $Y$ is K\"{a}hler with $c_1(Y)=0$ %
or
non-compact Fano\footnote{see \cite[Lem.1.6]{R16} or 
\cref{Subsection relating weight spaces via Cstar action}
for how to correctly interpret the Maslov index when $c_1(Y)\neq 0$.} 
and
$$
\mu=\mathrm{codim}_{\C}\F_{\min} \;\; \textrm{ and } \;\; \mathrm{PD}[\F_{\min}]\neq 0 \in H^{2\mu}(Y).
$$
Then
$Q_{\varphi}=\mathrm{PD}[\F_{\min}]+($terms with $T^{>0}) \neq 0 \in QH^{2\mu}(Y).$
\end{prop}
\begin{proof}
This will rely on \cite[Lem.1.6]{R16}\footnote{In that lemma,
$D=\mathrm{Fix}(g)$ needs to be connected, otherwise the statement of \cite[Eq.(6)]{R16} needs to be adjusted to: $r(g^{\wedge})=\mathrm{PD}[\F_{\min}]+(\textrm{higher order }t\textrm{ terms}) \in QH^{2(m-d)}(M)$, where $\F_{\min}=\{K_g=0\}\subset D$ is the absolute minimum of the Hamiltonian $K_g$ of the action $g$, as will become clear from our discussion of $T$-weights in the current proof.} (that was stated in the non-compact Fano case, but the written proof for \cite[Eq.(6)]{R16}
works also in the non-compact CY case). The contribution $\F_{\min}$ arises from certain constant sections counted by the GW-interpretation of $Q_{\Fi}$. The additional contributions to $Q_{\Fi}$ come from: (1) other moduli spaces of constant sections that sweep out lf-cycles inside the other fixed components $\F_{\beta}$ (after cutting down the moduli space using obstruction bundle techniques \cite[Sec.8.4-8.6]{R14}); and (2) moduli spaces of non-constant sections.
To justify that (1) and (2) contribute with strictly positive powers of $T$, it is necessary to address the technical issue of the choice of lift $\widetilde{\Fi}$ mentioned in \cref{Corrolary on until when F_min = Id under continuation} \cite[Sec.3.1]{R14}. One actually works with a cover of the free loop space of $Y$, obtained by capping off loops by discs, and identifying two discs if they give rise to a sphere class in $\pi_2(Y)_0$, which are $\pi_2(Y)$ classes on which $\omega$ and $c_1(TY)$ vanish. A lift $\widetilde{\Fi}$ is determined by knowing how it acts on one such capping. We choose $\widetilde{\Fi}$ as in \cite[Sec.7.8]{R14} by declaring that it sends the constant cap at a fixed point in $\F_{\min}$ to itself.
When constructing $Q_{\Fi}$ we ``normalise'' the Hamiltonian of the action: we use $K:=H-H(\F_{\min})$ so that the Hamiltonian $K=0$ on $\F_{\min}$. The same argument as in \cite[first Lemma in Sec.7.8]{R14}\footnote{In the proof of that Lemma, $\phi$ is a function that increases from $0$ to $1$ (there is a sign change compared to the $\phi$ in the proof of \cite[Lm.40]{R14}, due to a sign in the definition of $\widetilde{\Omega}$ over $D^+$).} then applies in our setting: the symplectic form $\widetilde{\Omega}$ on the Hamiltonian fibration used to construct $Q_{\Fi}$ evaluates to zero on the class of the constant section $s_{\min}$ at a point of $\F_{\min}$. On the other hand, running the computation of \cite[first Lemma in Sec.7.8]{R14} for the class of a constant section $s_{\alpha}$ at a point $x_{\a}\in \F_\a\neq \F_{\min}$, we get%
\footnote{
This is consistent with the fact that the induced ``cap'' for the constant loop at a point of $\F_{\a}$ (having chosen the constant ``cap'' at points of $\F_{\min}$) is described by a succession of {\ph} spheres as in \cref{CorollaryGradientTrajBecomeSpheres} starting from $\F_{\min}$ and ending at $\F_\a$, which defines a class $A\in \pi_2(Y)/\pi_2(Y)_0$ with $\omega(A)=H(\F_\a)-H(\F_{\min})=K(\F_\a).$
}
$\widetilde{\Omega}(s_{\alpha})=K(x_{\a})>0=K(\F_{\min})$. Indeed more generally, 
by \cite[Lm.30]{R14} the $\widetilde{\Omega}$-value on any pseudoholomorphic section counted by $Q_{\Fi}$ will be strictly positive except on constant sections at $\F_{\min}.$
We use a simplified Novikov field compared to \cite{R14}: we count sections in class $s=s_{\min}+\gamma$ for $\gamma\in \pi_2(Y)/\pi_2(Y)_0$ (see \cite[Sec.5.1]{R14}) with weight 
\begin{equation}\label{Equation positive omega contributions to Seidel element}
T^{\omega(\gamma)}
=T^{\int\! s_{\min}^*\widetilde{\Omega}}\cdot T^{\omega(\gamma)}
=
T^{\int\! s^*\widetilde{\Omega}}.
\end{equation}
It follows that all sections, except for constants at $\F_{\min}$, are counted with a $T^{>0}$-factor.
\end{proof}

\begin{rmk}\!\!\label{Remark about Seidel element and constant sections}\footnote{This remark proves a result analogous to an observation due to Seidel for closed symplectic manifolds \cite[Lm.3.1]{mcduff2006topological}, which is at the heart of McDuff-Tolman's \cite[Prop.3.3]{mcduff2006topological}, which is analogous to \cref{Prop QFi computation first order} in the closed case. Their semi-freeness condition on $\F_{\max}$ is precisely our condition that the weights are $0,1$ at $\F_{\min}$ (see \cref{Lemma when min weights are plus one}).}
The argument in the previous proof shows the following general observation for the class $Q_{\Fi}$ in the setting of \cref{Theorem neg lb paper summary} (and also in the setting of the spaces considered in \cite{R14,R16}), with the convention that the lift $\widetilde{\Fi}$ is chosen to fix the constant ``cap'' at a point of $\F_{\min}$.

\vspace{0.2cm}
\noindent {\bf Observation.} \emph{The only constant sections that can contribute in class $s_{\min}$ to $Q_{\Fi}$ must lie at $\F_{\min}$ (constant sections at other $\F_{\a}\neq \F_{\min}$ would have to contribute in a class $s_{\min}+\gamma$ with $\omega(\gamma)>0$), and non-constant sections can only contribute in classes $s_{\min}+\gamma$ with $\omega(\gamma)>0$. Thus, when constant sections at $\F_{\min}$ are regular,\footnote{Otherwise, one must compute the Euler class of the relevant obstruction bundle over $\F_{\min}$.} they yield the $T^0$-part to $Q_{\Fi}$, which is $\mathrm{PD}[\F_{\min}]$.}
\vspace{0.2cm}

\noindent \emph{Proof}. Non-constant sections $u$ in class $s=s_{\min}+\gamma$ contribute with energy $u^*\widetilde{\Omega}>0$ (see \cite[Lem.30]{R14}) and $u^*\widetilde{\Omega}=s^*\widetilde{\Omega}=\omega(\gamma)$, so constant sections have $\omega(\gamma)=0$. Finally, for constant sections, the previous proof showed that the constant point lies in $\F_{\min}$ whenever $\omega(\gamma)=0.$ \qed
\vspace{0.2cm}

We mention one setting where higher order $T$ correction terms can be excluded.
The virtual complex dimension of the space of sections in class $s=s_{\min}+\gamma$ being counted by $Q_{\Fi}$ is\footnote{see \cite[Lem.35]{R14} using \cite[Sec.5.1 and Def.25]{R14}.} $$\dim_{\C} Y - \mu(\Fi) + c_1(TM)(\gamma).$$
If $Y$ is monotone\footnote{see \cref{Rmk technical symplectic assumptions on Y}, e.g. non-compact Fano $Y$.} then non-constant sections contribute with $c_1(TM)(\gamma)>0$, so they sweep a locally finite cycle of dimension $>\dim_{\C} Y - \mu(\Fi)$ (whose Poincar\'{e} dual contributes to $Q_{\Fi}$ with a factor $T^{\omega(\gamma)}$). In the special case $\mu(\Fi)=1$, this implies that those lf-cycles are trivial for dimension reasons, so $Q_{\Fi}=\mathrm{PD}[\F_{\min}].$ E.g.\,this applies when $\Fi$ is one of the natural rotations about toric divisors in non-compact Fano toric manifolds. For closed Fano toric manifolds, this was first observed by McDuff--Tolman \cite{mcduff2006topological} which led to their proof of the Batyrev presentation of $QH^*(Y)$.
\end{rmk}

\begin{cor}\label{Cor Qfi nonzero for weight 1 csr}
For any weight-$1$ CSR, $Q_{\varphi}=\mathrm{PD}[\F_{\min}]+ ($terms with $T^{>0}) \neq 0 \in QH^{\dim_{\C}Y}(Y)$, so
$\Fil^{\Fi}_1\neq H^*(Y).$
For any weight-$s$ CSR with $s\geq 2$, $Q_{\Fi}= 0$, so $\Fil^{\Fi}_1=H^*(Y).$
\end{cor}
\begin{proof}
For $s\geq 2$, $Q_{\Fi}\in QH^{2\mu}(Y)=0$ as $2\mu=s\dim_{\C}Y>\dim_{\C}Y$ (by \cref{GradientIsRPlus} and \cref{CohomologyOfACSRProperties}). %
Let $Y$ be a weight-$1$ CSR. 
By \cref{Cor CSR when unit dies},
$\F_{\aa}$ is an $\omega_J$-Lagrangian and $\mu=\mathrm{codim}_{\C} \F_{\aa}$.
As $\F_{\aa}$ is $\om_J$-Lagrangian, its normal and cotangent bundle can be identified, so $e(U_{\min})=-e(\F_{\min})$, and $e(\F_{\min})\neq 0$ as the Euler characteristic of $\F_{\min}$ is non-zero (its cohomology lies in even degree, by \cref{Lemma Frankel properness} and \cref{CohomologyOfACSRProperties}.). 
Thus $Q_{\Fi}\neq 0$ by \eqref{non vanishing equivalence} and \cref{Prop QFi computation first order}.
\end{proof}

An alternative proof of \cref{Cor Qfi nonzero for weight 1 csr} follows by non-degeneracy of the intersection form 
which holds for CSRs (\cref{CSR_definite_intersection_form}).
When the intersection product is trivial, we find that $Q_{\Fi}$ vanishes:

\begin{prop}\label{Prop QFi vanishes if intersection form is trivial}
If $c_1(Y)=0$ and 
$\xymatrix@C=50pt{
    H_{2\mu}(Y)\otimes H_{2\dim_{\C}Y-2\mu}(Y) 
    \ar@{->}[r]^-{\textrm{intersection}} & \k}$ is trivial, $Q_{\Fi}=0.$
\end{prop}
\begin{proof}
    Viewing $Q_{\varphi}$ as a class in $H^{lf}_{2\dim_{\C}Y-2\mu}(Y)$ by {\PL } duality, it is a $\k$-linear combination of lf-cycles of dimension $2\dim_{\C}Y-2\mu$ (since $c_1(Y)=0$, the Novikov parameter is in degree zero). But those lf-cycles are all compact cycles as (by definition) they arise from evaluation maps on compact moduli spaces \cite{R14}. Indeed the lf-cycles are supported close to $\mathrm{Core}(Y)$ as the maximum principle prevents the sections counted by $Q_{\Fi}$ from entering the region at infinity where the maximum principle holds.
The assumption on the triviality of the intersection product implies that $Q_{\varphi}$ has zero intersection product with $H_{2\mu}(Y)$. On the other hand, the intersection product $H_{2\mu}(Y)\otimes H_{2\dim_{\C}Y-2\mu}^{lf}(Y)\to \k$
is non-degenerate by {\PL} duality. Therefore $Q_{\Fi}=0$.
\end{proof}
There is family of spaces to which \cref{Prop QFi vanishes if intersection form is trivial} can be applied. Consider %
the moduli space $\MM_G(d,g)$ of $G$-Higgs bundles of degree $d$ over a Riemann surface %
of genus $g$ (e.g.\,see \cite{Hi87}), 
for $G\in \{GL(n,\C),SL(n,\C)\},$
and $d\geq 0$ coprime to $n$. 
Recall these are {\HKL} manifolds satisfying the weight-1 condition $t\cdot \omega_{\C} = t\omega_{\C}$ for the canonical $\C^*$-action $\Fi$.

\begin{cor}\label{Criterion_for_Q_Fi=0}\label{Higgs_moduli_Q_Fi=0}
    $Q_\Fi=0$ for $\MM_G(d,g)$,
    so in particular $\Fil_1^{\Fi}=H^*(\MM_G(d,g))$.
\end{cor}
\begin{proof} %
Let $d:=\dim_{\C}(\MM)$. We have $c_1(Y)=0$ and $2\mu=d$ by the same proofs as in \cref{LemmaCalabiYau} and \cref{GradientIsRPlus}.
So $Q_{\Fi}=0$ follows by \cref{Prop QFi vanishes if intersection form is trivial} together with the fact
(due to \cite{hausel_1998} for $SL(n,\C)$ and 
to \cite{heinloth2016intersection} for $GL(n,\C)$)
that $H_{d}(\MM)\otimes H_{d}(\MM) \to \k$ is trivial (using $2\mu=d=2d-2\mu$).
\end{proof}

\subsection{$S^1$-equivariant symplectic cohomology}\label{Subsection intro S1 equiv SH 2}
Applying the methods from \cite{McLR18}, $S^1$-equivariant symplectic cohomology $ESH^*(Y,\Fi)$ is a $\k[u]$-module, with a canonical $\k[u]$-module homomorphism $$Ec^*:EH^*(Y)\cong H^*(Y)\otimes_{\k}\mathbb{F} \to ESH^*(Y,\Fi),$$ and $u$ is a formal variable in degree $2$. At chain level each $1$-orbit contributes a copy of the $\k[u]$-module $$\mathbb{F}:=\kuu /u\k [\![u]\!]\cong H_{-*}(\C\P^{\infty}),$$ where we identify $u^{-j}=[\C\P^j]$, and $H^*(\C\P^{\infty})=\k[u]$ acts by the nilpotent cap product action. We recall the notation for locally finite $S^1$-equivariant homology,
$$EH^*(Y):=H_{2\dim_{\C}Y-*}^{\,l\!f,S^1}(Y),$$ 
which in this case becomes $H^*(Y)\otimes_{\k}\mathbb{F}$ as it arises from constant $1$-orbits, and $S^1$ is only acting by $S^1$-reparametrisation on $1$-orbits (not on the space $Y$). \cref{Cor intro filtration} becomes:

\begin{thm}\label{Filtration vs equivariant filtration}
There is an $\R_{\infty}$-ordered filtration by graded $\k[u]$-submodules of $H^*(Y)\otimes_{\k}\mathbb{F}$,
\begin{equation}\label{DefinitionOfFiltration2}
E\FF^{\varphi}_p :=\bigcap_{\mathrm{generic}\,\lambda>p} \left(\ker Ec_\lambda^*:H^*(Y)\otimes_{\k}\mathbb{F}\to EHF^*(H_{\lambda})\right), \qquad E\Fil_{\infty}^{\Fi}:=H^*(Y) \otimes_{\k} \mathbb{F},
\end{equation} 
where $Ec_\lambda^*$ is an equivariant continuation map, a grading-preserving $\k[u]$-linear map.

In general, $\Fil^{\Fi}_\lambda \subset E\Fil^{\Fi}_\lambda$. If $H^*(Y)$ lies in even degrees (e.g.\;CSRs), then $\Fil^{\Fi}_\lambda  = QH^*(Y)\cap E\Fil^{\Fi}_\lambda.$
\end{thm}
\begin{proof}
The first part is analogous to the non-equivariant case, so we will just explain the second part. In general, we have the following commutative diagram, where $in:=\mathrm{id}\otimes_{\k} u^0$ is the inclusion of the $u^0$-part, yielding an injective left-vertical arrow below,
$$
\xymatrix@C=45pt@R=15pt
{H^*(Y)\otimes_{\k} \mathbb{F}
\ar@{->}^-{Ec^*_{\lambda}}[r]  
\ar@{<-}[d]_-{in}
& 
EHF^*(H_{\lambda})
 \ar@{<-}^{in}[d]
 \\
 QH^*(Y)
\ar@{->}[r]^-{c^*_{\lambda}} & 
HF^*(H_{\lambda})
}
$$
Thus $\Fil^{\Fi}_\lambda \subset E\Fil^{\Fi}_\lambda$.
Now consider the Gysin sequence from \cite[Sec.4.5]{McLR18}: 
$$
\cdots \longrightarrow HF^*(H_{\lambda},\Fi) \stackrel{in}{\longrightarrow} EHF^*(H_{\lambda},\Fi)
\stackrel{u\cdot}{\longrightarrow} EHF^{*+2}(H_{\lambda},\Fi)
\stackrel{b}{\longrightarrow} HF^{*+1}(H_{\lambda},\Fi)
\longrightarrow \cdots
$$
(the connecting map $b$ at chain level yields images of higher equivariant differentials $\delta_j$, $j\geq 1$).
So $$QH^*(Y)\cap E\Fil_{\lambda}^{\Fi}=\ker(QH^*(Y)\stackrel{c_{\lambda}^*}{\longrightarrow} HF^*(H_{\lambda})\longrightarrow HF^*(H_{\lambda})/b(EHF^{*+1}(H_{\lambda})).$$
The analogue of \cref{Equation intro HF lambda H splits}, when $H^*(Y)$ lies in even degrees, is
\begin{equation}\label{Equation intro EHF Hlambda}
    EHF^*(H_{\lambda})\cong \oplus H^*(\F_\a)\otimes_{\k}\mathbb{F}[-\mu_\lambda(\F_\a)],
\end{equation}
since the $S^1$-reparametrisation action on the constant orbits in $\F_\a$ is trivial. So \eqref{Equation intro EHF Hlambda} is a free $\mathbb{F}$-module in even degrees, on which $u\cdot$ acts surjectively, so $b=0$ and $in: HF^*(H_{\lambda}) \to EHF^*(H_{\lambda})$ is injective. Thus $\Fil^{\Fi}_\lambda = E\Fil^{\Fi}_\lambda$ follows. Also $HF^*(H_{\lambda})=\ker (u\cdot:EHF^*(H_{\lambda}) \to EHF^*(H_{\lambda}))$ recovers  \eqref{Equation intro HF lambda H splits}.
\end{proof}

\section{Filtration on cohomology for Conical Symplectic Resolutions} \label{NonExactSH}

\subsection{Topological properties of CSRs} \label{CSRs}

We refer to {\BPW} \cite{BPW16} where CSRs were introduced and studied in detail, although these spaces were also previously considered by various authors, most notably Kaledin \cite{Ka06, Ka08, Ka09} and Namikawa \cite{Nam08, Nam11}. In particular,  \cite[Sec.2]{BPW16} lists large families of examples of CSRs, including Nakajima quiver varieties, and hypertoric varieties.
We recall:
\begin{de}\label{DefCSRwithFi}
	A \textbf{Conical Symplectic Resolution (CSR)} is a projective resolution\footnote{Meaning: $\M$ is a smooth variety, and $\pi$ is an isomorphism over the smooth locus of $\M_0$.} $\pi:\M \fun \M_0$ of a normal %
	affine variety $\M_0$, where $(\M,\om_\C)$ is a holomorphic symplectic manifold and $\pi$ is equivariant with respect to $\C^*$-actions on $\M$ and $\M_0$ (both denoted by $\Fi$). These actions satisfy: %
	\begin{enumerate}
		\item[(1)] The complex symplectic form $\om_\C$ has a \textbf{weight} $s\in \N$, so 
  $\Fi_t^* \om_\C = t^s \om_\C$ for all $t\in \C^*$.
		\item[(2)]\label{contracts} The action $\Fi$ contracts $\M_0$ to a single fixed point $x_0$, so $\forall x \in \M_0, \displaystyle \lim_{t\fun 0} t \cdot x=x_0.$ %
		Algebraically, $\C[\M_0]=\bigoplus_{n\geq 0}\C[\M_0]^n$ {and} $\C[\M_0]^0=\C,$ where $\C[\M_0]^n$ denotes the $n$-weight space.\footnote{Explicitly
			$\C[\M_0]^n=\{f \in \C[\M_0] \mid (t \cdot f)(x)=f(t\cdot x)=t^n f(x)\}.$} 
	\end{enumerate}
    We call \textbf{weight-$s$ conical actions} such actions $\Fi$. 
    A CSR may have many conical actions, of possibly different weights $s,$ 
    so we emphasise the choice by writing $(\M,\Fi)$.
    We denote the {\bf core} of $(\M,\Fi)$ by
     $\L:=\{p\in \M \mid \lim_{\C^* \ni t\fun \infty} t\cdot p \text{ exists}\}$ (see \cref{Subsection topology of the core}). 
\end{de}
We remark that our normality assumption on $\M_0$ is %
equivalent\footnote{For the proof of this equivalence see \cite[Lem.3.16]{vzivanovic2022exact}.} to the condition that $\pi:\M\fun \M_0$ is the affinisation 
map\footnote{$\M\fun \Aff(\M):=Spec(H^0(\M,\O_{\M})), \  p \mapsto \{f \mid f(p)=0\}.$} that is given in the original definition in \cite[Sec.2]{BPW16}.

\begin{lm}\label{LemmaCalabiYau}
	Any CSR $\M$ satisfies $c_1(T\M,I)=0,$ where $I$ is its complex structure.
\end{lm}
\begin{proof}
The top exterior power of the complex symplectic form $\omega_{\C}$ trivialises the canonical bundle $\Lambda_\C^{top}T^*\M.$ 
Now recall that $c_1(\Lambda_\C^{top}T^*\M)=c_1(T^*\M)=-c_1(T\M)$.
\end{proof}

In the following Theorem, we summarise some well-known facts about CSRs.%
\footnote{Part (\ref{centralfibrecore}) is immediate (for a proof see e.g.\,\cite[Lem.3.15.]{vzivanovic2022exact});
(\ref{CSRcorehomotopyEquiv}) is due to \cite[proof of Prop.2.5]{BPW16}, %
see also \cref{Prop def retr to core analytic case};
(\ref{FibresOddCohOnly}) is due to Kaledin \cite[Prop.2.12]{Ka06},
and (\ref{isotropicCore},\ref{LagrangianCore}) go back to Nakajima \cite[Thm.5.8]{Nak94a}, 
and one can find the detailed proof in \cite[Lem.3.3]{vzivanovic2022exact}.}
\begin{thm}\label{CSRproperties} 
	Let $\pi:\M\fun\M_0$ be a weight-s CSR.
	\begin{enumerate}
	    \item \label{centralfibrecore} Its core is the central fibre, $\L=\pi^{-1}(0).$
	    \item \label{isotropicCore} %
	    $\L$ is an $\om_\C$-isotropic subvariety (but usually singular). 
     \item \label{LagrangianCore}
     When $s=1,$ $\L$ is a $\frac{1}{2}\dim_\C \M$-equidimensional variety, so $\om_\C$-Lagrangian (usually singular).  
	    \item \label{CSRcorehomotopyEquiv} The inclusion $\L \subset \M$ is a homotopy equivalence.
	    \item \label{FibresOddCohOnly}
	    Any fibre $F$ of $\pi$ has $H^{odd}(F,\mathbb{B})=0,$ for any field $\mathbb{B}$ of characteristic zero. %
	\end{enumerate}
\end{thm}

\begin{cor} \label{CohomologyOfACSRProperties}
    The cohomology of a weight-s CSR $\M$ over characteristic zero %
    fields
	is supported in even degrees, and at most up to 
	the degree $\dim_\C \M.$ 
    Moreover, when $s=1$, %
    $H^{\dim_{\C}\M}(\M)\neq 0$. 
\end{cor}

We prove the non-degeneracy of the intersection form of any CSR.
\begin{prop}\label{CSR_definite_intersection_form}
Any CSR $\M$ has a definite intersection form
$H_{\dim_\C \M}(\M) \times H_{\dim_\C \M}(\M) \fun \Q.$
\end{prop}
\begin{proof}
By \cite[Cor.2.1.14]{deCataldoMigliorini2005} %
the intersection product for fibres above the relevant strata 
of semismall resolutions is definite over $\Q$ coefficients.
Symplectic resolutions are semismall \cite[Prop.1.2]{kaledin2000}, so this applies to $\L=\pi^{-1}(0)$.
Finally, the inclusion $\L=\pi^{-1}(0) \subset \M$ is a homotopy equivalence for any CSR $\pi:\M\fun \M_0,$ (\cref{CSRproperties}(\ref{CSRcorehomotopyEquiv})).
\end{proof}

By Kaledin--Verbitsky \cite[Thm.1.1]{KaVe02} and Namikawa 
\cite
{Nam08}, it is also known that %
any CSR $\M$ has a (topologically trivial) deformation whose base is $H^2(\M,\C)$ and whose generic fibre is an affine algebraic variety. In the latter variety, there are no non-constant $I$-holomorphic spheres\footnote{the affine variety embeds into some affine space $\C^N$, and there are no non-constant holomorphic functions on $\C P^1$.} 
so quantum product equals cup product. 
As the quantum product is preserved under deformations of the complex structure  (working over the Novikov field $\k$ over any base field), we deduce:

\begin{prop} \label{quantumequaltocup} For any CSR, there is a ring isomorphism $QH^*(\M)\iso H^*(\M,\k)$. \qed
\end{prop}

The next lemma describes the basic information on the fixed  locus of the $\C^*$-action on a CSR.
\begin{lm}\label{LemmaMomentMapMorseBott}
	Consider a weight-$s$ CSR $(\M,\Fi).$ We have the following: 
	\begin{enumerate}
	    \item \label{FixedSmooth} Its fixed locus $\F:=\M^{\Fi}$ is a smooth subvariety contained in the core $\L.$ 
        \item \label{finitelyMany} $\F$ is a proper\footnote{Meaning: compact in the analytic topology.} variety which breaks into finitely many connected components 
	$\F=\sqcup_\a \F_\a.$ 
	    \item \label{WeightDecompCSR} Given a fixed point $p\in\F_\a,$ the induced $\C^*$-action on $T_p\M$ has a weight decomposition
	$$T_p \M = \oplus_{k\in\Z} H_k, \ H_k:=\{v\in T_p \M \mid t \cdot v =t^k v\}.$$ %
	    \item \label{NonDegClaimCSR}The weight $s\in\N$ condition $t\cdot \om_\C = t^s \om_\C$ implies the non-degeneracy of the following pairing,  \begin{equation}\label{NonDegenPairingByOmC}\qquad\qquad
     \om_\C: H_k \oplus H_{s-k} \fun \C, \quad \textrm{ for each }k\in\Z.   \qquad (\omega_{\C}\textbf{-duality})
	\end{equation}
  \item \label{fa vs top cohomology weight-1 CSR}  When $s=1$,  we have\;\; $2\dim_{\C}\F_\a + \mu_\a = \dim_{\C} \M$\;\; and \;\;$\# \{\F_\a\}= \mathrm{rk}(H^{\dim_{\C}\M}(\M))$.
	\end{enumerate}
\end{lm}
\begin{proof}
The fixed locus of a reductive group action on a smooth variety is smooth;\footnote{e.g.\,see \cite[Lem.5.11.1]{CGi97}.} in particular, it applies to $\F=\M^{\Fi}.$ 
It is contained in the core, as it is the preimage $\L=\pi^{-1}(0)$ of the only fixed point $0 \in \M_0,$ and $\pi$ is equivariant. Being closed in a proper variety $\L,$ 
the fixed locus is proper itself, thus indeed breaks into finitely many connected components $\F=\sqcup_\a \F_\a.$ 

The $\C^*$-action at a fixed point $p\in\F_\a$ yields a representation $\C^*\dejstvo T_p\M,$ thus the weight decomposition \eqref{WeightDecompCSR} is immediate. 
Considering two homogeneous vectors $v_1\in H_{k_1},$ $v_2\in H_{k_2},$ we have
\begin{equation*}\label{ActionOnSymplForm}
	\om_\C(v_1,v_2)=t^{-s} \om_\C(t\cdot v_1, t\cdot v_2)=t^{-s}\om_\C(t^{k_1}v_1,t^{k_2}v_2)=t^{k_1+k_2-s}\om_\C(v_1,v_2),
\end{equation*} 
thus, $\om_\C(v_1,v_2)=0$ unless $k_1 + k_2=s.$ So \eqref{NonDegClaimCSR} follows, as $\om_\C$ is non-degenerate on $T_p \M=\oplus_k H_k$. 

When $s=1,$ this duality forces 
$H_0\oplus H_-= \oplus_{k\leq 0} H_k \iso H_+,$
thus, abbreviating $|V|:=\dim_{\C}V$, 
$\dim_\C \M = |H_0| + |H_-| + |H_+| 
=2(|H_0|+|H_-|)= 2 \dim_\C \F_\a + \mu_\a.$
Thus, by \eqref{EqnFrankel}, the number of generators in $H^{\dim_{\C}\M}(\M)$ is the number of $\a$.
\end{proof}

\begin{lm}\label{LemaCSRM0nonSingularEqualToAffineSpace}
Given a CSR $\pi:\M\fun\M_0,$ if $\M_0$ is non-singular then $\M_0 \iso \C^{2n}$, for some $n\in\N.$
\end{lm}
\begin{proof}
 If $\M_0$ is non-singular, $\pi:\M_0 \fun \M_0$ is a CSR, with a single fixed point $0.$ Thus by the \BB decomposition theorem for semiprojective varieties and the fact that CSRs are semiprojective
 \cite[Cor.2.7 and Lem.3.15.]{vzivanovic2022exact},
 we deduce that $\M_0$ is an affine bundle over a point, hence an affine space, thus isomorphic to $\C^{2n}$ (with some linear action on it).
\end{proof}

\subsection{Symplectic structures on a CSR}\label{Preliminaries}

We now show that CSRs $(\M,\Fi)$ fit into the framework of Section \ref{SectionSHassociatedToHamS1Action}.
In Corollary \ref{MomentMapExhausting} we build an explicit $I$-compatible $S^1$-invariant Calabi-Yau\footnote{due to \cref{LemmaCalabiYau}.} \KH structure on $(\M,\Fi)$, with an exhausting moment map. Without the exhausting condition, this
also arises from any holomorphic embedding $\iota:\M \hookrightarrow X$ into a \KH manifold $(X,\omega_X)$, by averaging:
\begin{equation}\label{EquationS1invtMetric}\textstyle
\om_I:=\int_{S^1} \Fi_t^*(\iota^* \om_X)\,dt.    
\end{equation}
Any CSR $\M$ admits such an embedding, being projective over the affine variety $\M_0$.\footnote{$\pi:\M \fun \M_0$ being projective means that there is an embedding $i:\M\hookrightarrow \M_0 \times \C P^n$ for some $n$ (\cite[p.103]{Ha77}). As $\M_0$ is affine, $\M_0 \hookrightarrow \C^m$ for some $m,$ hence we can obtain $\M \hookrightarrow \C^m \times \CP^n.$
 }
Abbreviate by $g(\cdot,\cdot):=\om_I(\cdot,I\cdot)$ the induced Riemannian metric ($\om_I$ is $I$-compatible).
Since $H^1(\M)=0$ (\cref{CohomologyOfACSRProperties}), it has a moment map $H.$ 
We show that $\M$ is a symplectic $\C^*$-manifold over a convex base in \cref{PropMapPsi}. Therefore, if $H$ %
is not exhausting, we can modify $\om_I$ so that the new moment map $H$ is proper using \cref{Lemma making H proper}, and we know $H$ is bounded below by \cref{Lemma H bdd below}.

\begin{lm}\label{GradientIsRPlus}
	Any weight-$s$ CSR $(\M,\Fi)$ is a symplectic $\C^*$-manifold with 
        an $S^1$-invariant $I$-compatible \KH structure $(g,I,\om_I).$ 
	The $S^1$-action is Hamiltonian with Maslov index
	$\mu=\frac{1}{2}s\cdot\dim_{\C} \M.$
\end{lm}
\begin{proof}
The $S^1$-action is symplectic as it preserves $\om_I$. By Corollary \ref{CohomologyOfACSRProperties}, $H^1(\M)=0,$ so the action is Hamiltonian. 
The canonical bundle $\Lambda^{\mathrm{top}}_\C T^*\M$ is trivialised by $\om_\C^d$ for $d:=\frac{1}{2}\dim_\C \M$. The weight $s$ condition ($\Fi_t^*\om_\C=t^s \om_\C$) implies $\varphi_t^*(\om_\C^d)=(\varphi_t^*\om_\C)^{d}=(t^s\om_\C)^d=t^{sd}\om_\C^d$, so $\mu=sd$ (see \cref{Subsection discussion of ma and muFalapha}).
\end{proof}

The methods of Section  \ref{SectionSHassociatedToHamS1Action}
are required, as $(\M,\om_I)$ is almost never convex at infinity:

\begin{prop}\label{NonIsoSing} 
Suppose that $0\in\M_0$ is a non-isolated singularity. Then \emph{any} choice of $I$-compatible (real) symplectic form $\omega$ on $\M$ is non-exact at infinity.
\end{prop}
\begin{proof} As $0\in \M_0$ is a symplectic singularity, by \cite[Thm.2.3]{Ka06} there is a finite stratification $\M_0=\sqcup_{a\in A} \M_0^{a}$ by locally closed smooth strata, where $\M_0^{0}=0.$ By assumption, there is at least another non-generic stratum $\M_0^{1}.$ As the $\C^*$-action on $\M_0$ is algebraic, it leaves the strata invariant. As the points in $\M_0^{1}$ have finite isotropy subgroups, an arbitrary $\C^*$-orbit makes the stratum $\M_0^{1}$  non-compact. Thus, there is a sequence of points $(x_i)_{i\in\N} \in \M_0^{1}$ that goes to infinity. %
	Their fibres $\pi^{-1}(x_i)$ are $I$-holomorphic, hence $\omega$-symplectic, projective subvarieties in $\M.$ Thus, integrating their (possibly singular) irreducible components with a suitable power of $\omega$ gives a positive value (such integration is well-defined, see  \cite[p.60]{GrHa78}). 
    If $\omega$ were exact outside of a compact set $K$, then a power of $\omega$ would also be exact in this region.
	But any (possibly singular) irreducible component of a fibre $\pi^{-1}(x_i)$ in this region would have a well-defined fundamental class \cite[p.61]{GrHa78},
    in particular by Stokes's theorem \cite[p.60]{GrHa78} integration of any exact form is zero.  
    Contradiction.
\end{proof}

By \cref{LemaCSRM0nonSingularEqualToAffineSpace}, when $0\in\M_0$ is not a singular point, $\M\iso \C^{2n}$ and so it is Liouville with vanishing symplectic cohomology (see e.g.\,\cite[Sec.3]{OanceaEnsaios}).
Symplectic resolutions $\M\fun \M_0$ (not just CSRs) for which 
$0\in\M_0$ is an isolated singularity are completely classified. In complex dimension 2, they are the minimal resolutions of Du Val singularities (by \cite[Prop.1.3]{Be00} and \cite[Thm.7.5.1]{Ish97}); in higher dimensions, they are the cotangent bundles $T^*\C P^n$ \cite[Thm.8.3]{CMS-B02}.
In the former case, they are convex at infinity and have vanishing symplectic cohomology for $\om_I$ \cite[Lem.42]{R10}. In the latter case, for $n\geq 2$ they are not convex at infinity \cite[Rmk. in Sec.11.1]{R14}. Thus:

\begin{cor}\label{Cor CSR is almost never convex}
A CSR is convex at infinity only if it is isomorphic to $\C^{2n}$ for some $n$, or it is a minimal resolution of a Du Val singularity.
\end{cor}

\subsection{CSRs are symplectic $\C^*$-manifolds globally defined over a convex base}

\begin{prop}\label{PropMapPsi}
Any CSR $(\M,\Fi)$ is a symplectic $\C^*$-manifold globally defined over the convex base $\C^N$, in the sense of Definition \ref{Def:KahlerMfdWithProjection}. Indeed, there is a proper $\C^*$-equivariant holomorphic map $$\Psi=\Theta \circ j \circ \pi: \M \to \C^N,$$
with $\Psi^{-1}(0)=\L$,
where $\C^*$ acts diagonally on $\C^N$ by a certain weight $w>0$. The map $\Theta \circ j:\M_0\to \C^N$ is
a $\C^*$-equivariant holomorphic map, which is a local embedding except at $0\in \M_0$.
\end{prop}
\begin{proof}
By \cref{DefCSRwithFi}, $\M \fja{\pi} \M_0$ is $\C^*$-equivariant, and the coordinate ring of the affine variety $\M_0$ is the graded ring $\C[\M_0]=\bigoplus_{n\geq 0}\C[\M_0]^n,$ whose grading prescribes the weight of the $\C^*$-action. Fix a choice of homogeneous polynomials $f_1,\dots,f_N$ that generate $\C[\M_0]$. As $\C[\M_0]^0=\C$, we may assume that all $f_i$ are non-constant with positive weights $w_i$.
These determine an embedding $$j:\M_0 \to \C^N, \qquad p\mapsto (f_1(p),\dots,f_N(p)),$$ with $j(0)=0$. Let 
$w=\text{lcm}(w_1,\dots,w_N).$ 
The holomorphic map $j\circ \pi:\M \to \C^N$ is $\C^*$-equivariant for the $\C^*$-action $t \cdot (z_1,\dots, z_N)=(t^{w_1}z_1,..,t^{w_N}z_N)$ on $\C^N$. Let $\Theta$ be the holomorphic map
$$\C^N \fja{\Theta} \C^N, \ \ \Theta(z_1,\dots,z_N)=(z_1^{w/w_1},\dots, z_N^{w/w_N}).$$
Then $\Psi$ is proper as $\pi,j,\Theta$ are proper. 
Finally, $\M$ is connected, as it is homotopy equivalent to 
$\pi^{-1}(0)=\L$ (\cref{CSRproperties}(\ref{centralfibrecore},\ref{CSRcorehomotopyEquiv})) and the fibres of $\pi: \M \to \M_0$ are connected \cite[Prop.3.17]{vzivanovic2022exact}.
\end{proof}

\begin{rmk}\label{RmkFunctionPsi}
Explicitly, 
$\Psi = (\pi^*(f_1)^{w/w_1},\ldots,\pi^*(f_N)^{w/w_N}).$
Denoting $\hat{\pi}$ the constant $3.1415\ldots$, the $1$-periodic $S^1$-action $e^{2\hat{\pi} it}$ on $\C^N$ has Hamiltonian $\hat{\pi} (|z_1|^2+\cdots + |z_N|^2)$, and we denote its pull-back by
\begin{equation}\label{Equation Phi function}
    \Phi:=\hat{\pi}\,\Psi^*(|z_1|^2+\cdots + |z_N|^2)=\hat{\pi}\sum \pi^*(|f_i|^{2w/w_i}):\M\to \R,
\end{equation}
so $\L=\Phi^{-1}(0)$. The function $\Phi$ is typically not related to the moment map $H$ on $\M$ in \eqref{Equation intro moment map}. 
\end{rmk}

\subsection{An explicit $S^1$-invariant K\"{a}hler form with exhausting Hamiltonian}\label{Subsection Explicit S1 invt Kahler form for CSR}

\begin{lm}\label{Equivariant embeedding CSR}
Given a CSR $\pi:\M \fun \M_0,$ for some integer $M>0$ there is a $\C^*$-equivariant embedding 
$\M \hookrightarrow \M_0 \times \P^M$ using the diagonal action on the target, which is linear on the $\P^M$-factor. 
\end{lm}
\begin{proof}
The morphism $\pi$ is projective, thus $\pi$ factors through a closed immersion and a projection
$$\M \xhookrightarrow{\iota} \M_0 \times \P^n \fja{\pi_{\M_0}} \M_0,$$ for some integer $n$. Composing that inclusion $\iota$ with projection to the second-factor yields
$$\pi_{\mathbb{P}^n} \circ \iota : \M \hookrightarrow \mathbb{P}^n \times \M_0 \rightarrow \mathbb{P}^n,$$
and thus a pull-back bundle $\mathcal{L}=\iota^*\pi_{\mathbb{P}^n}^*(\mathcal{O}(1)).$ By \cite[\href{https://stacks.math.columbia.edu/tag/01VT}{Lem.01VT}]{stacks-project} that bundle is $\pi$-relatively ample, using that $\M_0$ is affine and $\pi$ is of finite type (as $\pi$ is projective, by \cite[\href{https://stacks.math.columbia.edu/tag/01WC}{Lem.01WC}]{stacks-project} it is proper, and hence of finite type by definition).
Recall that a bundle on $\M$ is $\C^*$-linearisable if it admits a $\C^*$-action linear on the fibres which lifts the action on $\M$. 
As $\M$ is normal, \cite[Thm.2.14]{Brion13} ensures that $\mathcal{L}^{\otimes k}$ is $\C^*$-linearisable for some positive integer $k.$ The same holds for positive tensor powers of $\mathcal{L}^{\otimes k}$.
As $\mathcal{L}$ is $\pi$-relatively ample, so is $\mathcal{L}^{\otimes k},$ this follows from \cite[Ch.II, Prop.4.5.6(i)]{EGA}. %
By \cite[\href{https://stacks.math.columbia.edu/tag/01VU}{Lem.01VU}]{stacks-project}, the quasi-compactness of $\M_0$ (being an affine variety) and the finite type property of $\pi$ ensure that some positive power $L:=(\mathcal{L}^{\otimes k})^{\otimes d}$ 
is $\pi$-relatively very ample. %
By  \cite[\href{https://stacks.math.columbia.edu/tag/02NP}{Lem.02NP}]{stacks-project}, %
as $\M_0$ is affine and $\pi:\M\fun \M_0$ is of finite type, such an $L$ yields an immersion 
$$j: \M \hookrightarrow \mathbb{P}^M \times \M_0$$ for some integer $M>0$, with
$L \cong j^*\pi_{\mathbb{P}^M}^*\mathcal{O}(1).$ It remains to prove that one 
can construct this immersion to be $\C^*$-equivariant, where the action on $\mathbb{P}^M$ is linear. 
That this immersion is $\C^*$-linear follows by construction, by the same argument as in \cite[Prop.1.7]{MuFo82} (using work of Kambayashi and Sumihiro). We remark that the proof above is essentially the same argument as in \cite[Prop.1.7]{MuFo82} for the linear algebraic group $G=\C^*$, except we are working with schemes that are proper over $\M_0=\mathrm{Spec}(R)$ (where $R$ is the coordinate ring of the affine variety $\M_0$) rather than working over $\mathrm{Spec}(k)$.
\end{proof}

\begin{cor}\label{MomentMapExhausting}  
    Any CSR %
    admits an $S^1$-invariant \KH structure %
	with an exhausting moment map. %
\end{cor}
\begin{proof}
 	Combining Lemma \ref{Equivariant embeedding CSR} with $\Theta\circ j$ from Proposition \ref{PropMapPsi} we obtain a $\C^*$-equivariant morphism
	$$
	\Pi: \M \to \M_0 \times \P^M \to \C^N \times \P^M.
	$$
	This morphism is proper, holomorphic, $\C^*$-equivariant (using the rescaled action on $\C^N$ as in \cref{PropMapPsi}), and locally it is a closed topological embedding. Moreover, $\Pi: \M \to \C^N \times \P^M$ is the composite of two closed immersions, so it is a closed immersion.
	To conclude that $\Pi$ is locally a holomorphic embedding, it remains to show that the differential $\Pi_*: T\M \to T\C^N \times T\P^M$ at any point $p$ is injective (so that the implicit function theorem applies). 
	This will follow from the surjectivity of the dual map, viewed as the algebro-geometrical map $\mathfrak{m}/\mathfrak{m}^2 \to \mathfrak{n}/\mathfrak{n}^2$ on cotangent spaces.\footnote{$\mathfrak{n}$ is the maximal ideal of functions vanishing at $p$ in the coordinate ring for an affine patch $\mathrm{Spec}(A)$ around $p\in \M$; $\mathfrak{m}$ are functions vanishing at $\Pi(p)$ in the coordinate ring for an affine patch $\mathrm{Spec}(B)$ around $\Pi(p)\in \C^N \times \P^M$.} As $\Pi$ is a closed immersion, we may assume $\Pi^{\#}:B \to A$ is surjective. By construction, $(\Pi^{\#})^{-1}(\mathfrak{n})= \mathfrak{m}$, therefore $\Pi^{\#}: \mathfrak{m} \to \mathfrak{n}$ is surjective, and thus the induced map $\mathfrak{m}/\mathfrak{m}^2 \to \mathfrak{n}/\mathfrak{n}^2$ is surjective, as required.
	
	The claim follows by pulling back the standard $S^1$-invariant \KH structure $\omega_Y$ from $Y=\C^N \times \P^M$ via $\Pi$. If $H_Y: Y \to \R$ is the Hamiltonian generating the $S^1$-vector field $X_{S^1,Y}$ on $Y$, then $H=H_Y\circ \Pi$ is the Hamiltonian on $\M$ generating the $S^1$-action.\footnote{  
	$
	(\Pi^*\omega_Y)(\cdot,X_{S^1,\M})
	= \omega_Y(\Pi_* \cdot,\Pi_* X_{S^1,\M})
	= \omega_Y(\Pi_* \cdot, X_{S^1,Y})
	= dH_{Y}(\Pi_* \cdot)
	= d(H_Y \circ \Pi)(\cdot),
	$
	as $\Pi$ is $S^1$-equivariant.}
 As $H_Y$ is the sum of the Hamiltonians on the two factors $\C^N$ and $\P^M$, and the Hamiltonian on the $\C^N$ factor grows like a power of the norm on $\C^N$, $H_Y$ is exhausting.
  The properness of $\Pi$ and $H_Y$ imply the properness of $H=H_Y \circ \Pi$, and $H$ is bounded below since $H_Y$ is. Thus $H$ is exhausting.
			\end{proof}

\subsection{Implications from Section \ref{SectionSHassociatedToHamS1Action} and \ref{SectionFiltration}}\label{SHforCSRCorrolaries}

By Section \ref{SectionSHassociatedToHamS1Action}, 
for generic slopes $\lambda>0$, the only $1$-periodic orbits of $\lambda H$ are the constant orbits inside the core $\L$ given by the fixed points $x\in \F:=\M^{\Fi}$ of the $S^1$-action (i.e.\,the critical locus of $H$).
As $c_1(\M)=0$, Proposition \ref{PropSec2VanishingTheorem} implies that
$$
SH^*(\M,\Fi,\om_I)=0.
$$
The fixed locus $\F$ decomposes into connected components  $\F_\a$ which are the {\MB} submanifolds for $H$. 
At $x\in \F$, the tangent space $T_x \M=\oplus_k H_k$ has a unitary decomposition given by the weight $k$ subspaces $H_k$ for the linearised $S^1$-action, where $H_0=T_x\F$.
We defined an even integer grading $$\mu_\lambda(\F_\a)=\dim_\C\, \M - \dim_\C\, \F_\a- \sum_k \dim_{\C}(H_k) \mathbb{W}(\lambda k)$$ 
where $\mathbb{W}(\lambda k)
=2 \lfloor \lambda k \rfloor + 1$ for $k\neq 0$, and $\mathbb{W}(0)=0$ (see \cref{LemmaSec2GradingOfFa}).

By Proposition \ref{quantumequaltocup}, $QH^*(\M)=H^*(\M;\k)$ as $\k$-algebras. Thus, \cref{SectionFiltration} and \cref{DefinitionOfFiltration} yields:
\begin{cor}\label{CorFiltrForCSR}
For any CSR $(\M,\Fi)$, the $\Fi$-filtration ordered by $p\in \R$,
$$
\Fil^{\varphi}_{p}:=\bigcap_{\textrm{generic }\lambda>p} \ker(c_{\lambda}^*:H^*(\M;\k) \fun HF^*(\lambda H)),
$$
is a filtration on the singular cohomology ring $H^*(\M;\k)$
by ideals with respect to the cup product.
\end{cor}

We suspect that these filtrations, up to $\lambda$-reparametrisation, %
do not depend on the choice of \KH form $\om_I,$ %
although we will not try to prove it.\footnote{For small deformations of the K\"{a}hler form, this should follow by the methods developed by Benedetti--Ritter \cite{BeRit20}.}
The filtrations however do depend on the choice of $\Fi$. 

We now abbreviate $H^*$ to mean cohomology with coefficients in $\k$.
The cohomology $H^*(\M)$ is concentrated only in even degrees by Corollary \ref{CohomologyOfACSRProperties}, therefore Corollary \ref{pureHam} yields
$$
\textstyle
HF^*(\H_\lambda)\iso \bigoplus_{\a} H^*(\F_\a)[-\mu_\lambda(\F_\a)],
$$
where $H^*(\F_\a)$ lives in even degrees and $\mu_{\lambda}(\F_\a)$ is even.
Thus the map $c_{\lambda}^*$ in \cref{CorFiltrForCSR} is a grading-preserving $\k$-linear homomorphism between $\k$-modules supported in even degrees:
$$
\textstyle
H^*(\M)\cong\bigoplus_{\a} H^*(\F_\a)[-\mu_\a]
\to 
\bigoplus_{\a} H^*(\F_\a)[-\mu_{\lambda}(\F_\a)],
$$
where $\mu_\a$ is the (even) {\MB } index of $\F_\a$ (and the first isomorphism is \eqref{EqnFrankel}).

\begin{rmk}\label{RemarkWhyFilrationsOnCohomologyInteresting}\textbf{Comparison with the literature.}
	There is an interest in filtrations on the cohomology of CSRs in the representation-theoretic literature. Using a certain contracting $\C^*$-action on the resolution of Slodowy varieties (certain CSRs coming from Lie algebras),
 Bellamy--Schedler %
 construct filtrations on their cohomologies,
whose corresponding grading can be read-off from the Poincaré polynomial,
given in \cite[Cor.2.5]{BeSch18}.\footnote{More precisely, it is a Poincaré polynomial on \textit{Poisson} homology of the Slodowy variety 
$S_{\phi}\cap \NN$, whose \textit{resolution} 
has the same cohomology, up to a grading twist. The variable $x$ presents (co)homological grading, whereas the
variable $y$ comes from their filtration.}
In the example of resolutions of type A singularities %
there is a choice of $\C^*$-action $\Fi$ so that $\Fil^{\Fi}$
yields %
the same grading %
as theirs. 
For the $A_n$ singularity $XY=Z^{n+1},$ this action is the lift of
$t\cdot(X,Y,Z)=(t^n X, t Y, t Z);$ for $n=2$, this is action $(c)$ in \cref{Example running example of intro 2}. 
\end{rmk}
\begin{rmk}\label{Rmk Refinement of the McKay correspondence}\textbf{Refinement of the McKay correspondence.}
	For resolutions of Du Val singularities (and possibly of other holomorphic-symplectic quotient singularities $\C^{2n}/\Gamma$), our filtration yields a refinement of the McKay correspondence \cite{Re92} which states that a graded basis for the cohomology of the resolution is in graded bijection with the conjugacy classes of the given finite group using the \emph{age grading} on conjugacy classes. An example for the resolution $\M$ of $\M_0:=\C^2/(\Z/5)$ is shown in \cite{RZ2} (the spectral sequence for $X_{\Z/5}$). The top cohomology has two filtration levels, which correspond to two pairs of orbits, which all have age grading equal to 1 (the orbits are loops in $\M \setminus \L$, and those are
	labelled naturally by their free homotopy classes, as explained e.g. in \cite[Eq.(1.4)]{McLR18} for the case of isolated quotient singularities, and in the subsequent footnote for the general case).
	Thus, our filtration makes a distinction between the orbits lying above $[e^{2\pi it/5},0], [0,e^{2\pi it/5}]$ and $[e^{4\pi it/5},0], [0,e^{4\pi it/5}],$ corresponding to the conjugacy classes $[\eps^1],[\eps^{-1}]$ and $[\eps^2], [\eps^{-2} ]$	in $\Z/5$ (here $\eps$ is a primitive $5^{th}$-root of unity).

The McKay correspondence involves  crepant resolutions of quotient singularities $\C^n/\Gamma$ for finite subgroups $\Gamma \subset SL(2,\C)$. 
For these to arise as CSRs, we need $\M_0=\C^{2n}/\Gamma$ for $\Gamma \subset \mathrm{Sp}(2n,\C)$ as $\M_0$ has a Poisson structure. Apart from four exceptional examples,
this has a conical symplectic resolution $\M$ 
precisely when $\Gamma$ is of type $G^n \rtimes S_n$, where $S_n$ is the symmetric group and $G\subset SL(2,\C)$ is a finite subgroup, 
\cite{bellamy2016non}.
These $\M$ are in fact all quiver varieties of affine ADE type, \cite{Kuz07}.
These were studied by Kaledin \cite{KaledinMcKay} and Bezrukavnikov--Kaledin \cite{BezKaledinMcKay}.
\end{rmk}
\section{Filtration on Floer complexes separating the periods of orbits}\label{filtrationFloer}
In this Section, $(Y,\omega,\J,\Fi)$ is a symplectic $\C^*$-manifold over a convex base $B$, satisfying \eqref{Equation psi Xs1 is Reeb}. Recall that we can always tweak $\omega$ by \cref{Lemma making H proper} to make $H$ proper in \eqref{Equation intro moment map}, which we assume from now on.
\subsection{Construction of a specific  Hamiltonian $H_\lambda$}\label{SectionConstructionOfHLambda part 1}
We have an $S^1$-equivariant proper holomorphic map $\Psi: Y^{\mathrm{out}} \to B=\Sigma \times [R_0,\infty)$. We abusively write
$
\Psi:Y \fun B
$
but it is understood that all constructions involving $\Psi$ are only defined on $Y^{\mathrm{out}}$. 

\begin{figure}[ht]
	\centering
	{
  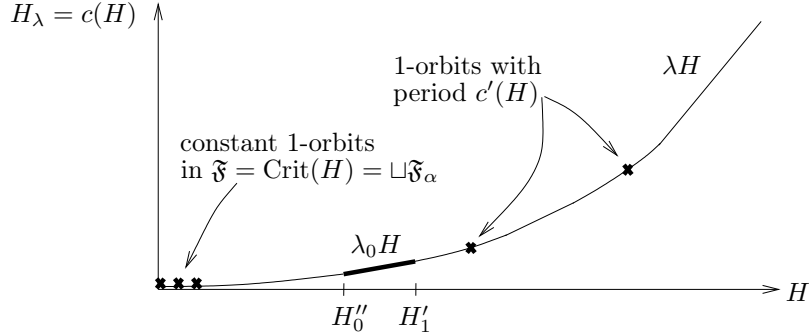
		\caption{Graph of $H_\lambda$}
		\label{H_lambdagraph part 1}
	}
\end{figure}

The Hamiltonian $H_\lambda:=c \circ H$ will be constructed as in Figure \ref{H_lambdagraph part 1} in terms of a function $c$ of $H$.
This ensures that $X_{H_{\lambda}}=c'(H)\cdot X_H$, so $1$-periodic Hamiltonian orbits of $H_{\lambda}$ corresponds precisely to orbits of period $T=c'(H)$ of the flow of $X_{S^1}=X_H$. We construct $c:[\min H,+\infty)\fun \R$ so that
\begin{enumerate}
	\item[(1)]  $c'\geq 0$ \; and \; $c''\geq 0$. 
	\item[(2)]  $c''(H)>0$ whenever $c'(H)$ is a period of an $S^1$-orbit (these will be outer $S^1$-periods).
 \item[(3)]  $c(H)=\lambda_0 H$ on some interval, say for $H\in [H_0'',H_1'],$ for $0<\lambda_0<\min\{\textrm{positive }S^1\textrm{-periods}\}$.
	\item[(4)]  $c(H)=\lambda H$ for all sufficiently large $H.$
\end{enumerate}

We also assume that $c'$ is sufficiently small on $Y^{\mathrm{in}}$ so that there are no non-constant $1$-periodic orbits of $H_{\lambda}$ in $Y^{\mathrm{in}}$ since the potential period values $c'$ are smaller than the minimal positive $S^1$-period.
Thus in $Y^{\mathrm{in}}$ the only $1$-orbits are constants at points of the $\F_\a$ submanifolds. 
More precisely, we may assume $$Y^{\mathrm{in}}:=\{H\leq \ell\}$$ is a sublevel set, where $m$ is large enough so that 
$Y^{\mathrm{in}} \supset H^{-1}(H(\mathrm{Core}(Y)) \supset \mathrm{Core}(Y).$ %
We build $c'$ to be small on $[\mathrm{min }(H),\ell]$, with $c'\neq 0$ except possibly at  
$\mathrm{Crit}(H)$, so $\mathrm{Crit}(H)=\mathrm{Crit}(H_{\lambda})$, and so that $H_{\lambda}$ has the same (constant) $1$-orbits as $H$ in that region. So (1) above is only needed for $Y^{\mathrm{out}}=\{H\geq \ell\}$. 

When $Y=\M$ is a CSR, $\Psi$ is globally defined over $B=\C^N$, and $H_B=\pi w\|z\|^2$ is defined everywhere (the $S^1$-action has weight $w$, see \cref{PropMapPsi}), and one can pick $\M^{\mathrm{in}}=\{H\leq \ell\}$ to be any sublevel set containing the core $\L=\Psi^{-1}(0)$.
We recall that 
by the rescaling trick in Remark \ref{Rmk weight w of Reeb} we can get rid of the factor of $w$ that would appear in \cref{Equation psi Xs1 is Reeb} for CSRs.

\begin{lm}
The level sets $\Phi^{-1}(q)=\Psi^{-1}(\{H_B=q\})\subset Y$ 
for $q\geq R_0$ are closed submanifolds of $Y.$
\end{lm}
\begin{proof}
 $d\Phi=dH_B \circ \Psi_*$ is non-zero on $\nabla H=X_{\R_+}$ as $\Psi_*X_{\R_+}=\nabla H_B$ (by \cref{Lemma:XR+ is nabla h}). 
\end{proof}

We choose $H_0'',H_1'$ so that we have a nesting
$$
\{H\leq H_0''\} \subset
\{\Phi\leq R_0''\} \subset
\{\Phi\leq R_{1}'\} \subset
\{H\leq H_{1}'\}
$$
for some values $0<R_0''<R_1'$.
The nesting condition ensures that $H_\lambda=c(H)=\lambda_0 H + \mathrm{constant}$ 
in a region of $Y$ that covers (via $\Psi$) the region in $B$ where $R\in L_0:=[H_0'',H_1']$, called {\bf linearity region}.
That this nesting can be achieved follows from $\Phi$ and $H$ being proper.

\subsection{Filtration functional on $B$}
\label{SubsectionFiltrationFunctional} 
Choose a smooth cut-off function $\phi:[0,+\infty)\fun \R$ satisfying
\begin{enumerate}
	\item[(1)] $\phi=0$ on $[0,R_0'']$.
	\item[(2)] $\phi'\geq 0$ everywhere.
    \item[(3)] $\phi_0:=\int_{L_0} \phi'(R)\,dR > 0$.
\end{enumerate}

The choice of $\phi$ on $R\geq R_1'$ is not so important. Let us choose $\phi=\phi_0$ to be constant there.
The cut-off function $\phi$ defines the exact 2-form $\eta$ on $B$,
\begin{equation}\label{Equation eta expanded}
\eta := d(\phi(R)\alpha) = \phi(R)\, d\alpha + \phi'(R)\, dR\wedge \alpha,
\end{equation}
and an associated $1$-form $\Omega_{\eta}$ on the free loop space $\mathcal{L}B=C^{\infty}(S^1,B)$, involving the Reeb field $\mathcal{R}_B$,
\begin{equation}\label{Equation filtration one form}
	\Omega_{\eta}: T_x\mathcal{L}B = C^{\infty}(S^1,x^*TB)\fun \R, \ \  \xi \mapsto -{\int} \eta(\xi,\partial_t x - \lambda_0 \mathcal{R}_B)\, dt.
\end{equation}
Define the {\bf filtration functional} $F: \mathcal{L}B \to \R$ on the free loop space by
$$
F(x):=-\int_{S^1} x^*(\phi \alpha) + \lambda_0 \int_{S^1} \phi(R(x(t)))\, dt.
$$

\begin{lm}{\normalfont\cite[Thm.6.2(1)]{McLR18}} \label{Fprimitive} $F$ is a primitive of $\Omega_{\eta}.$ That is, $dF(x)(\xi)=\Omega_{\eta}(x)(\xi).$ \qed
\end{lm}

When $\Psi$ is not globally defined, a loop $y\in \mathcal{L}Y$ may not have a well-defined projection $x=\Psi\circ y$. Nevertheless, it makes sense to talk about $F(x)$ and $\Omega_{\eta}|_x$ because the relevant integrands in the base $B$ will vanish near $\Sigma \times \{R_0\}$ as $\phi=\phi'=0$ near $R=R_0$ (also see {\PartII} for a detailed description of how to define ``pull-backs'' of the functionals to $\mathcal{L}Y$).

\subsection{The filtration inequality for $CF^*(H_\lambda)$}\label{FiltrationOnCFLambda}

\begin{thm}\label{H_lambdaIsOneDirected}
	The Floer chain complex $CF^*(H_\lambda)$ has a filtration given by the value of $F.$ That is, given two $1$-periodic orbits $x_-,x_+$ of $H_\lambda$ and a Floer cylinder for $(H_\lambda,I)$ from $x_-$ to $x_+,$
	\begin{equation}\label{FiltrationByF}
		F(x_-) \geq F(x_+).  %
	\end{equation}
\end{thm}

\begin{proof}
	The Floer cylinder $u:\R \times S^1 \fun Y$ for $H_\lambda$ satisfies
	$\partial_s u + I(\partial_t u -X_{H_\lambda})=0.$
	By \cref{Equation psi Xs1 is Reeb},
	\begin{equation}\label{EqnPhiProjectionNice}
		\Psi_*(X_{H_\lambda})=\Psi_*(c'(H)X_{H})=c'(H)\Psi_*(X_{H})=c'(H) \mathcal{R}_B,
	\end{equation}
	noting that $c'(H)$ depends on the original coordinates in $Y$.
	Projecting $u$ via $\Psi$ defines a map 
	\begin{equation} \label{projectFloerinM}
		v:=\Psi \circ u:\R \times S^1 \fun B, \ \ \ \partial_s v + I_B(\partial_t v -k(s,t)\mathcal{R}_B)=0
	\end{equation}
	that converges to $y_-=\Psi(x_-)$, $y_+=\Psi(x_+)$ at $s=-\infty$, $+\infty,$ respectively, where $$k(s,t):=  c'(H(u(s,t)))$$
	is a domain-dependent function.
 The key observation is that we chose $\phi'(R)=0$ except on the region  $L_0$, and over $L_0$ we chose $k(s,t)=c'(H(u(s,t)))=\lambda_0$.
 Thus, using $d\alpha(\cdot, \mathcal{R}_B)=0$,
 \begin{equation}\label{Equation etas are consisten}
 \eta(\cdot,k(s,t)\mathcal{R}_B) =
   \eta(\cdot,\lambda_0 \mathcal{R}_B)
   \end{equation}
   holds everywhere, and it recovers the integrand used in \eqref{Equation filtration one form}.
 Combining with Lemma \ref{Fprimitive}, 
	\begin{equation}\label{F_DifferenceInTermsOfOmega}
		\begin{array}{rll}
			F(x_-) - F(x_+)  & \; = \;  \textstyle - \int_{\R} dF(v(s,t))(\partial_s v)\ ds  
            &
			\textstyle \; = \;  -\int_{\R} \Omega_\eta(v(s,t))(\partial_s v)\ ds
            \\[1mm]
    & \textstyle \; = \;  \int_{\R\times S^1} \eta(\partial_s v,\partial_t v -  \lambda_0 \mathcal{R}_B)\, dt\, ds 
     & \textstyle \; = \;  \int_{\R\times S^1} \eta(\partial_s v,\partial_t v -  k(s,t)\mathcal{R}_B)\, dt\, ds. 
		\end{array}
	\end{equation}
 By \eqref{projectFloerinM}, the last integrand equals $\eta(\partial_s v,I_B\partial_s v)$.
	Hence, we reduced the problem to the same computation as in the convex setting \cite[Lem.6.1]{McLR18}: abbreviating $\rho=R\circ v$,
		\begin{equation} \label{nonpositivityOfOmegaV}
			\begin{array}{rcl}
				\eta(\partial_s v,I_B\partial_s v)
								& = &
				\phi(\rho) \cdot d\alpha(\partial_s v,I_B\partial_s v) + \phi'(\rho) \cdot (dR \wedge \alpha)(\partial_s v,I_B\partial_s v)
				\\
				& = &
				\textrm{positive}\cdot \textrm{positive} + \textrm{positive}\cdot (dR \wedge \alpha)(\partial_s v,I_B\partial_s v),
			\end{array}
		\end{equation}
	where ``positive'' here means ``non-negative''.
	To estimate the last term, we may assume that $R\geq R_0''$ since $\phi'=0$ otherwise. Thus, 
	we decompose $\partial_s v$ according to an orthogonal decomposition of $T B$:
	\begin{equation}\label{Eqn C in xi part}
	\partial_s v = C\oplus y \mathcal{R}_B \oplus zZ \in \xi \oplus \R \mathcal{R}_B \oplus \R Z,
	\end{equation}
	where $Z=-I_B\mathcal{R}_B=R\partial_R$ is the Liouville vector field and $\xi=\ker \alpha|_{R=1}$. Notice that $\ker \a = \xi \oplus \R Z.$ Thus: $dR(\partial_s v) = Rz$ and $\alpha(I_B\partial_s v) = \alpha(I_BzZ) = \alpha(z \mathcal{R}_B)=z$. The claim then follows from
	\begin{equation}\label{Eqn non xi part}
	(dR \wedge \alpha)(\partial_s v,I_B\partial_s v) =
	dR(\partial_s v) \alpha(I_B\partial_s v) - \alpha(\partial_s v) dR(I_B \partial_s v)
	= Rz^2 + R y^2 \geq 0. \qedhere
    \end{equation}
	\end{proof}
 \begin{rmk}
     How one deals with the issue of transversality, without ruining the filtration, is a rather tricky matter that will be dealt with in {\PartII}.
 \end{rmk}

\subsection{The $F$-filtration values on $1$-orbits}

\begin{cor} \label{FiltrValue} The $F$-filtration values satisfy the following properties:
\begin{enumerate}
\item $F=0$ at the constant orbits, so at each point of  $\F=\sqcup_\a \F_\a$;
\item $F(y)=F(\Psi(x))<0$ for every non-constant $1$-periodic orbit $x$ of $H_{\lambda}$;
\item $F(y)$ only depends on the Reeb period $c'(H(x))$ of the projected orbit $y$, see \eqref{Equation filtration value of orbit};
\item\label{Item filtration is like filtering by minus H} on non-constant orbits, the $F$-filtration is equivalent to filtering by $-H$, or equivalently: filtering by negative $S^1$-period values $-c'(H)$. 
\end{enumerate}
\end{cor}
\begin{proof} Let us calculate the value of the functional $F(y)$ explicitly for the projection $y:=\Psi(x(t))$ of a 1-periodic orbit $x$ of $H_\lambda.$ If $x$ is a fixed point, $y$ lies in the region where $\phi=0$ so $F(y)=0.$ Otherwise, $c'(H(x))=:T$ is the period of some $S^1$-orbit in $Y$ (in particular $T>\lambda_0$), and $\phi(y)=\phi_0$. Thus
\begin{align}
\label{Equation filtration value of orbit}
		F(y(t))=-{\textstyle \int_{S^1}} y^*(\phi \alpha) + \lambda_0 {\textstyle \int_{S^1}} \phi(R(y(t))) \ dt
		= -\phi(y)  T + \lambda_0\phi_0 
		= (\lambda_0-T)  \phi_0<0.
\end{align}
The drop in filtration value for $1-$orbits $y_1,y_2$ arising for successive slopes $T_1<T_2$ is:
\begin{equation}\label{Equation drop in filtration value}
F(y_1) - F(y_2) = 
\phi_0 (T_2-T_1) > 0.
\end{equation}
Thus $F(y(t))<0$ strictly decreases when $c'(H(x))$, hence $H(x),$ increases.
\end{proof}

\subsection{Period-bracketed symplectic cohomology}\label{Subsection Period-bracketed symplectic cohomology}

By convention, $x_-$ appears in the output of the chain differential $\partial(x_+)$ if a Floer trajectory $u$ flows from $x_-$ to $x_+$. 
So $\partial$ ``increases the $F$-filtration'', since $F(x_-)\geq F(x_+)$ by
\cref{H_lambdaIsOneDirected}.
 Also $\partial$ decreases $H$, and decreases the period $T=c'(H)$. Restricting to $1$-orbits with $F\geq A$ defines a subcomplex, $CF^*_{[A,\infty)}(H_{\lambda})$, and a quotient complex 
$$
CF^*_{(A,B]}(H_{\lambda}):=CF^*_{[B,\infty)}(H_{\lambda})/CF^*_{[A,\infty)}(H_{\lambda}).
$$
These fit into a short exact sequence, which induces a long exact sequence on cohomology,
\begin{equation}\label{LES for bracketed HF}
\cdots \to 
HF^*_{[A,\infty)}(H_{\lambda})
\to
HF^*_{[B,\infty)}(H_{\lambda})
\to 
HF^*_{(A,B]}(H_{\lambda})
\to 
HF^{*+1}_{[A,\infty)}(H_{\lambda})
\to
\cdots
\end{equation}
\begin{lm}\label{Lemma Floer continuation maps are ok}
A Floer continuation map $CF^*(H_{\lambda}) \to CF^*(H_{\lambda'})$ for $\lambda \leq \lambda'$, for a homotopy $H_s: = c_s\circ H$, respects the filtration if $H_{\lambda}$, $H_{\lambda'}$, $H_s$ are linear in $H$ over $L_0$, and
$\partial_s c_s'\leq 0$ over $L_0$.
\end{lm}
\begin{proof}
Abbreviate $\rho=R(v(s,t))$,
and $\lambda_{0,s}:=c_s'$ on $L_0$.
We have $\phi'(\rho)k(s,t)=\phi'(\rho) \lambda_{0,s}$ everywhere.
We now use an $s$-dependent 
filtration and $s$-dependent filtration one-form,
$$
F_s(x):=-{\textstyle \int_{S^1}} x^*(\phi \alpha) + \lambda_{0,s} {\textstyle \int_{S^1}} \phi(R(x(t)))\, dt, \qquad \Omega_{\eta}(\xi):= -{\textstyle \int_{S^1}} \eta(\xi,\partial_t x - \lambda_{0,s} \mathcal{R}_B)\, dt.
$$
\cref{Equation etas are consisten} becomes  $\eta(\cdot,k(s,t)\mathcal{R}_B) =
   \eta(\cdot,\lambda_{0,s} \mathcal{R}_B)$.
However, in \eqref{F_DifferenceInTermsOfOmega} a new term appears because
$\partial_s (F_s\circ u) = d_u F_s\cdot \partial_s u + (\partial_s F_s)\circ u$, but it has the sign needed for the argument to work because: $$(\partial_s F_s)(u)= (\partial_s \lambda_{0,s})\cdot {\textstyle \int_{S^1}}  \phi(R(u))\, dt \leq 0,$$
using  $\phi\geq 0$  and the assumption $\partial_s \lambda_{0,s}\leq 0$ (see \cite[Sec.6.5]{McLR18} for the proof in the convex setting).
\end{proof}

\cref{Lemma Floer continuation maps are ok} implies that continuation maps $CF^*(H_{\lambda})\to CF^*(H_{\lambda'})$ can be built for $\lambda\leq \lambda'$ in a way that preserves the filtration (see {\PartII} for details).
Thus, taking direct limits as $\lambda \to \infty$ in \eqref{LES for bracketed HF},
$$
\cdots \to 
SH^*_{[A,\infty)}(Y,\Fi)
\to
SH^*_{[B,\infty)}(Y,\Fi)
\to 
SH^*_{(A,B]}(Y,\Fi)
\to 
SH^{*+1}_{[A,\infty)}(Y,\Fi)
\to
\cdots
$$

\subsection{Positive symplectic cohomology}\label{PositiveSH}

\begin{de}
Let $CF_0^*(H_\lambda):=CF^*_{[0,\infty)}(H_{\lambda}) \subset CF^*(H_\lambda)$ be the subcomplex generated by the fixed locus $\F=\sqcup_\a \F_\a$ (constant orbits have $F=0$).
{\bf Positive Floer cohomology} $HF_+^*(H_\lambda)=H_*(CF_+(H_\lambda))$ is the homology of the 
quotient complex $CF_+(H_\lambda):=CF^*(H_\lambda)/CF_0^*(H_\lambda)$.
The direct limit over continuation maps is \textbf{positive symplectic cohomology}, $SH_+(Y,\Fi,\omI):=\lim_{\lambda\fun \infty} HF_+^*(H_\lambda).$
\end{de}
\begin{rmk}
    When \eqref{Equation psi Xs1 is Reeb 2} holds, the construction of the $F$-filtration is not possible. However one can still define, somewhat unsatisfactorily, $SH^*_+(Y,\Fi):=\,$Cone$(c^*:QH^*(Y)\to SH^*(Y,\Fi))$.
\end{rmk}

\begin{prop}\label{Corollary LES for H SH and SHplus}
	For any symplectic $\C^*$-manifold satisfying \eqref{Equation intro Psi}-\eqref{Equation psi Xs1 is Reeb},
	there is a long exact sequence 
	$$
	\cdots \to QH^*(Y) \stackrel{}{\to} SH^*(Y,\Fi) \to SH^*_+(Y,\Fi) \to QH^{*+1}(Y) \to \cdots
	$$
\end{prop}
\begin{proof}
The condition $F\geq 0$ imposed on generators of $CF^*_0(H_{\lambda})$ means that the generators are precisely the fixed points in $\F$, and that no Floer solution $u$ (with ends on $\F$) can have $v=\Psi\circ u$ exit the region $R\leq R_0''$, otherwise it would enter the region $\phi'>0$ (where also $\phi>0$) causing \eqref{nonpositivityOfOmegaV} to be strictly positive unless $v$ is $s$-independent.
Denote by $H_{\lambda_0}$ the Hamiltonian of slope $\lambda_0$ obtained by modifying $H_{\lambda}$ to be linear in $H$ of slope $\lambda_0$ for $H\geq H_0''$. 
Then the complexes $CF_0^*(H_\lambda)=CF^*(H_{\lambda_0})$ are equal as the Floer differentials count the same solutions, so their cohomology equals and is isomorphic to $QH^*(Y)$ by Proposition \ref{SmallHam} (as $\lambda_0$ is smaller than any non-zero $S^1$-period).
\end{proof}

\subsection{Compatibility of the $F$-filtration with the product}

We assume the reader has a familiarity with the construction of the pair-of-pants product on Floer cohomology (e.g.\,\cite{R13} and \cite[Sec.3.2-3.3]{Abbondandolo-Schwarz}).
The product on $SH^*$ is obtained by taking a direct limit of certain $\k$-linear pair-of-pants product maps $CF^*(H_{\lambda})\otimes CF^*(H_{\lambda'}) \to CF^*(H_{\lambda''})$ for $\lambda''\geq \lambda + \lambda'$, and up to quasi-isomorphism this is independent of the choices made in the construction. One does not need these maps for all such choices of $\lambda,\lambda',\lambda''$, it suffices to have these for three cofinal families of such slopes, see \cite{R13}.
We want to show that choices can be made so that this construction is compatible with the $F$-filtration.
By \cite[Sec.6.5]{McLR18} the $F$-filtration can be made consistent with continuation cylinders for monotone homotopies $H_s$. So we can use two continuation cylinders associated to two monotone homotopies to change $H_{\lambda},H_{\lambda'}$ until they both equal some Hamiltonian $H_{\mu}$, while preserving filtrations. By composing/gluing with these continuation cylinders, it therefore suffices to build, for a cofinal family of $H_{\mu}$, a pair-of-pants map  $CF^*(H_{\mu})\otimes CF^*(H_{\mu}) \to CF^*(2H_{\mu})$ in a way that is compatible with the $F$-filtration.\footnote{we may also glue/compose with a Floer continuation isomorphism $CF^*(2H_{\mu})\to CF^*(H_{2\mu})$ on the output.}

We will use the model of the pair-of-pants $P$ described by Abbondandolo--Schwarz in \cite[Sec.3.2]{Abbondandolo-Schwarz}, which admits a holomorphic $2:1$ branched covering $P \to \R\times S^1$ of the cylinder. One can think of $P=P_-\cup P_{+,1}\cup P_{+,2}$ as the union of three half-cylinders such that the positive
boundary of $P_-$ is the figure eight-loop consisting of the two negative boundaries of $P_{+,1}$, $P_{+,2}$. Under the covering, $P_-$ covers $(-\infty,0]\times S^1$ twice, whereas each $P_{+,i}$ covers $[0,\infty)\times S^1$ once. As explained in \cite[Sec.3.2]{Abbondandolo-Schwarz}, the Floer equation for $u: P \to Y$ corresponds to the usual equation $\partial_s u + I(\partial_t u - X_{H_{\mu}})=0$ on the three half-infinite cylinders except that on $P_-$ the time coordinate is parameterised by $\R/2\Z$ rather than by $S^1=\R/\Z$, so it can be turned into $\partial_s u + I(\partial_t u - X_{2H_{\mu}})=0$ by reparametrising time. We remark that the first term in the integral of the filtration $1$-form $\int \Omega_{\eta}(u)(\partial_s u)\, ds=-\int\!\int \eta(\partial_s u,\partial_t u - X_h)\, dt \wedge ds$ over a cylinder is invariant under conformal rescaling of $z=s+it$. The second term is not, but it will be invariant under simultaneous rescaling of $(s,t)$ if we rescale $h$. So for $P_-$ we either use $F=F_h$ as defined for $h$ but using time interval $t\in [0,2]$, or we use $F=F_{2h}$ as defined for $2h$ with $t\in [0,1]$.

\begin{lm}\label{Lemma F filtration preserved by product}
The $F$-filtration is respected by the pair-of-pants Floer solutions.
\end{lm}
\begin{proof}
Let $x_-,x_{+,1},x_{+,2}$ be the asymptotic orbits at the ends of $u$, and denote $u_-,u_{+,1},u_{+,2}$ the restriction of $u: P \to Y$ to the three half-infinite cylinders $P_-$, $P_{+,1}$, $P_{+,2}$.
Integrating the filtration $1$-form over the projection $v=\Psi\circ u: P \to B$ reduces to computing three separate integrals for Floer solutions $u_-$, $u_{+,1}$ and $u_{+,2}$ for Hamiltonians $2H_{\mu},$ $H_{\mu}$ and $H_{\mu}$ respectively, over three half-infinite cylinders.  
Those three integrals are non-positive by \cref{nonpositivityOfOmegaV}, so
$$
F(x_-) - F(u_-(0,\cdot))\geq 0, \quad
F(u_{+,1}(0,\cdot)) - F(x_{+,1})\geq 0, \quad
F(u_{+,2}(0,\cdot)) - F(x_{+,2})\geq 0.
$$
Considering what happens along the figure-eight in the middle of the pair-of-pants:
$$
F(u_-(0,\cdot)) = F(u_{+,1}(0,\cdot)) + F(u_{+,2}(0,\cdot)).
$$
Combining the two equations yields the required inequality
$F(x_-) \geq  F(x_{+,1}) + F(x_{+,2})$.
\end{proof}

\subsection{Simplification of the construction when $H$ is a function of $R$}\label{Subsection simplification of construction when H is function of R}

Suppose that on $Y^{\mathrm{out}}$ the moment map $H$ can be written as a function of $R\circ \Psi$,
\begin{equation}\label{equation H function of R simplify}
    H=\rho(R\circ \Psi),
\end{equation}
for a function $\rho:Y^{\mathrm{out}}\to \R$. This rarely holds since $\|X_{\R_+}\|$ is typically not constant on a level set of $H$. It does hold for
$T^*X$ for a projective variety $X$ (\cref{Moment map is function of radial coordinate for TCPn}) and for
negative vector bundles \cite[Sec.11.2]{R14}.
One can then considerably simplify the construction as level sets of $H$ map into level sets of $R.$
Note $\rho$ is strictly increasing as $\Psi_*\nabla H=\Psi_*X_{\R_+}=X_{\R_+,B}=\nabla R$. 
Define 
$$
\textstyle
h(R) := \int_0^R c'(\rho(r))\,dr,
$$
so $h'(R)=c'(\rho(R))$,
where $c(H)$ only needs to satisfy (1)-(2) in \cref{SectionConstructionOfHLambda part 1}. 
Redefine $F: \mathcal{L}B \to \R$ as
$$
\textstyle
F(x):=-\int_{S^1} x^*(\phi \alpha) + \lambda_0 \int_{S^1} \int_0^{R(x(t))} \phi'(\tau)h'(\tau)\,d\tau \, dt.
$$
Also, redefine the $1$-form $\Omega_{\eta}$ on $\mathcal{L}B=C^{\infty}(S^1,B)$ using $X_h=h'(R)\mathcal{R}_B$:
\begin{equation}
\textstyle
	\Omega_{\eta}: T_x\mathcal{L}B = C^{\infty}(S^1,x^*TB)\fun \R, \ \  \xi \mapsto -{\int} \eta(\xi,\partial_t x - X_h)\, dt.
\end{equation}
By \cite[Thm.6.2(1)]{McLR18}, we again have $dF(x)(\xi)=\Omega_{\eta}(x)(\xi).$ Abbreviating $v=v(s,t)$, and using $k(s,t) = c'(H\circ u) = c'(\rho(R(v)))=h'(R(v)),$
\eqref{Equation etas are consisten} becomes
\begin{equation}
 \eta(\cdot,k(s,t)\mathcal{R}_B) =
   \eta(\cdot,X_h),
   \end{equation}
irrespective of our choice of $\phi$. Thus, we can pick $\phi$ to just satisfy the two conditions: $\phi=0$ for $R\leq R_0''$, and $\phi'>0$ for $R>R_0''$ (similarly to \cite[Sec.6.3]{McLR18}). This construction (when \eqref{equation H function of R simplify} holds) also applies in the more complicated filtration setting of \cite{RZ2}, where a very complicated $\phi$-function would otherwise need to be constructed. 
\cref{FiltrValue} still holds, but the proof needs to be slightly modified: in \eqref{Equation filtration value of orbit} we have $F(y)=\chi(R(y))$ where
$\textstyle
\chi(R):=-\phi(R)h'(R)+\int_0^{R} \phi'(\tau)h'(\tau)d\tau,
$
and $\chi'(R)=-\phi(R)h''(R)\leq 0$, so $\chi$ strictly decreases when evaluated at $R$-values of projected $1$-orbits.
In addition, \eqref{nonpositivityOfOmegaV}/\eqref{Eqn non xi part} imply that $F(x_-)-F(x_+)<0$ unless $\partial_s v\equiv 0$, or $v$ lies in the region $R\leq R_0''$.

\subsection{Rank estimates}\label{Subsection rank estimates via H of Fa}

We now justify the claims in \cref{Subsection existence of orbits and upper bounds for ranks}.
Fix generic slopes $0<\lambda<\gamma$.
Let $A:=CF^*(H_{\lambda})$, so $CF^*(H_{\gamma})=A\oplus B$ where $B$ is the linear subspace spanned by $1$-orbits of $H_{\gamma}$ that are not $1$-orbits of $H_{\lambda}$ (so $S^1$-orbits with period values $p\in (\lambda,\gamma)$).
By \cref{FiltrValue}.\eqref{Item filtration is like filtering by minus H}, there are filtration values $F_{\lambda},F_{\gamma}$ such that $1$-orbits in $A,B$ respectively satisfy $F\geq F_{\lambda}$, $F\in [F_{\gamma},F_{\lambda})$.
By \cref{Subsection Period-bracketed symplectic cohomology}, $A$ is a subcomplex of $A\oplus B$. Thus, the Floer differential on $A\oplus B$ has the form 
$$d_{\gamma}=
\Big( \begin{smallmatrix}
    d_{\lambda} & \;\nu \\
    0 & \;d_B
\end{smallmatrix} \Big):A\oplus B \to (A\oplus B)[1].$$
In particular, $H^*(A,d_{\lambda})\cong HF^*(H_{\lambda})$, which is $QH^*(M)$ for small $\lambda>0$. Also, by construction, the inclusion $A \to A\oplus B$ induces the continuation map $\psi_{\lambda,\mu}:HF^*(H_{\lambda})\to HF^*(H_{\gamma})$. 

\begin{rmk}
Although we do not use this here,
in \cite{RZ2} we construct the {\MBF} model for $CF^*(H_{\gamma})$ consisting of the critical points of auxiliary Morse functions on the {\MB} manifolds 
$\F_\a$ (of period $p:=0$) and $B_{p,\beta}$ (for period values $p>0$). In particular, $B$ is generated by the
critical points in $B_{p,\beta}$ for period values $p\in (\lambda,\gamma)$.
By \cite{RZ2}, in each degree $k$ the following bounds hold
$$
\textstyle \sum_{\lambda <p<\gamma}\sum_{\beta} |H_{k}(B_{p,\beta})|
\geq 
|H_{k}(B)|.
$$
\end{rmk}

\begin{lm}\label{Theorem about the cone}
$HF^*(H_{\gamma})$ is the homology of the cone of $\nu: B \to A[1]$, yielding the LES\footnote{Remark. For $\lambda=0^+$, $H_*(B,d_B)=:HF^*_+(H_{\gamma})$, whose direct limit is $SH^*_+(Y,\Fi)$ as $\gamma\to \infty$; and the long exact sequence becomes $\cdots \to SH^{*-1}_+(Y,\Fi) \to QH^*(Y)\to SH^*(Y,\Fi)\to SH^*_+(Y,\Fi)\to \cdots$.}
$$
\cdots {\longrightarrow} H_{*-1}(B,d_B) 
\stackrel{\nu_{*-1}}{\longrightarrow}
HF^*(H_{\lambda})
\stackrel{\psi_{\gamma,\lambda}}{\longrightarrow}
HF^*(H_{\gamma})
\stackrel{\textrm{project}}{\longrightarrow}
H_{*}(B,d_B)
\stackrel{\nu_{*}}{\longrightarrow}
\cdots
$$
thus
$
HF^*(H_{\gamma}) \cong  \ker \nu_{*} \oplus \tfrac{HF^*(H_{\lambda})}{\mathrm{im}\, \nu_{*-1}}
$
as $\k$-vector spaces.
\end{lm}
\begin{proof}
    In the notation of \cref{Subsection Period-bracketed symplectic cohomology} the subcomplex $CF^*_{[F_{\lambda},\infty)}(H_{\gamma})\cong A$ of $CF^*(H_{\gamma})$ gives the quotient complex  $CF^*_{[F_{\gamma},F_{\lambda})}(H_{\gamma})\cong B$, so \eqref{LES for bracketed HF} is the LES for the cone of the inclusion $A\subset A\oplus B$. Now exploit exactness and rank-nullity. By properties of cones, $HF^*(H_{\gamma})\cong \mathrm{im}\, (\textrm{project})\oplus \ker (\textrm{project})$, then use the LES, so $HF^*(H_{\lambda})/\ker \psi_{\gamma,\lambda} \cong \mathrm{im}\,\psi_{\gamma,\lambda}$ and $\ker\psi_{\gamma,\lambda}=\mathrm{im}\,\nu_{*-1}$. The result follows.
    \end{proof}

\begin{rmk}
Cycles $(a,b)\in A\oplus B$ satisfy $d_B(b)=0$ and $d_{\lambda}(a)=-\nu(b)$ (and $d_{\lambda} \circ \nu = -\nu \circ d_B$ due to $d^2=0$). 
By the above LES, we identify $\mathrm{coker}\, \psi_{p^-,p^+} \cong \ker \nu_*=\{[b]\in H_*(B): \nu(b) \textrm{ is exact in }A \}$, and we call these classes of $H_*(B)$ the {\bf new classes} of $H_*(A\oplus B)$. 
We refer to $H_*(A)$ as the {\bf old classes}. The kernel of the continuation $H_*(A)\to H_*(A\oplus B)$ corresponds precisely to the old classes {\bf killed} by $H_*(B)$ via $\nu_*$. Finally, $H_*(B)$ {\bf creates} new classes in $H_*(A\oplus B)$ via $\ker \nu_*$.
\end{rmk}

Suppose $H^*(Y)$ lies in even degrees.
Recall the notation \eqref{Equation difference of HF ranks},
which relies on \cref{pureHam},
\begin{equation}\label{Equation proofs section on difference of HF ranks}
\delta_{\lambda}^k:
=
|QH^k(Y)|-|HF^*(H_{\lambda})|
=
\textstyle \sum_{\a} |H^{\e-\mu_\a}(\F_\a)|-|H^{\e-\mu_{\lambda}(\F_\a)}(\F_\a)|
.
\end{equation}
We now show that $B$ kills $|H_{2k-1}(B)|$ independent old classes, and $B$ creates $|H_{2k}(B)|$ independent new classes, and the difference of these numbers is determined by the indices \eqref{Equation proofs section on difference of HF ranks}.
\begin{cor}\label{Corollary finding homology of B complex}
Suppose $c_1(Y)=0$, and $H^*(Y)$ lies in even degrees. 
Then $\nu_{2k-1}$ is injective, and
$$
\textstyle
HF^{2k}(H_{\gamma})=H_{2k}(A\oplus B)\cong H_{2k}(B)\oplus \frac{H_{2k}(A)}{\nu_{2k-1}(H_{2k-1}(B))}$$
as $\k$-vector spaces. 
Moreover, the fixed loci $\F_\a$ and their weights combinatorially determine
\begin{equation}\label{Equation difference of delta indices}
\delta_{\lambda}^{2k}-\delta_{\gamma}^{2k} = 
|H^{2k}(H_{\gamma})|-|HF^{2k}(H_{\lambda})|
=
|H_{2k}(B)|-|H_{2k-1}(B)|.
\end{equation}
When $\lambda=0^+$, then $H_*(B)=HF^*_+(H_{\gamma})$ by definition, and 
$$
\delta_{\gamma}^{2k}=|H_{2k-1}(B)|-|H_{2k}(B)| \qquad \textrm{ and }\qquad \FF^{\Fi}_{\gamma}\cap QH^{2k}(Y)=\nu_{2k-1}(H_{2k-1}(B)).
$$

When $c_1(Y)\neq 0$, the results above hold after replacing $(2k,2k-1)$ by $(\mathrm{even,odd})$.
\end{cor}
\begin{rmk}
    If we denote the total ranks by $\|V_*\|:=\sum |V_n|$, then it also follows that $\|H_{\mathrm{odd}}(B)\|=\|H_{\mathrm{even}}(B)\|$, and\footnote{This follows from 
$\|HF^*(H_{\lambda})\|=\|HF^*(H_{\gamma})\|$ (using \cref{pureHam}). } $\|\ker \psi_{\gamma,\lambda}\|
 = \tfrac{1}{2} \|H_*(B,d_B)\|$ (this also holds when $c_1(Y)$ is non-zero).
\end{rmk}
\begin{proof} \cref{Theorem about the cone} implies the first claim, as $HF^*(H_{\lambda})$,$HF^*(H_{\gamma})$ lie in even degrees by \cref{pureHam}; \eqref{Equation difference of HF ranks} gives the first equality in \eqref{Equation difference of delta indices}, whilst the second expresses the difference $|H^{2k}(A\oplus B)|-|H^{2k}(A)|$ via the first claim.
Finally, $c_{0^+}^*:QH^*(Y)\cong HF^*(H_{0^+})$, so $\delta_{0^+}^{2k}=0$, and $\FF^{\Fi}_{\gamma}=\ker \psi_{\gamma,0^+}$.
\end{proof}

Let $P$ be a {\bf critical slope}: some $B_{P,\beta}$ in \eqref{Equation Bkm slices intro} exists.
Let $B$ be the complex obtained above for $\lambda=0^+$ and $\gamma=P^+$. 
Plot the points $(p,\delta_{p^+}^k)$ in the $xy$-plane, for critical periods $p\in [0,P]$ (e.g.\,$(0,0)$ at $p=0$). Join the points by straight line segments.
Define $\Delta_+^k$ to be the peak $y$-value of the points. To define $\Delta_-^k$, we look at the segments along which the $y$-value drops, and add up these drops. Thus:
$$
\Delta_+^k 
 := \!\!
\displaystyle\max_{\{0<\lambda<P^+:\;
\delta_{\lambda}^k>0\}}\!\delta_{\lambda}^k,
\qquad\quad
\Delta_-^k 
 :=  \!\!\!\!\!\!
\displaystyle \sum_{\{0<p\leq P: \; \delta_{p^-}^k-\delta_{p^+}^k>0\}} \!\!\!\!\!\!\!(\delta_{p^-}^k-\delta_{p^+}^k),\quad\quad
 \textrm{ and }\quad\quad
K^k:=\Delta_-^{k} +\delta_{\gamma}^{k}.
$$
We see\footnote{Loosely: in the above plot, starting from $\delta_{\gamma}^{k}$ one can ``climb back up to'' the maximal $\delta$-value using drop-segments.} that $K^k\geq \Delta_+^k$. Observe that $\Delta_+^k$ is the lower bound for the number of independent classes of $QH^k(Y)$ killed by $B$ obtained in \cref{Cor estimating filtration using degrees}. %
By definition $\Delta_-^k$ adds up all \emph{positive jumps} $\sum_{\a} |H^{\e-\mu_{p^-}(\F_\a)}(\F_\a)|-|H^{\e-\mu_{p^+}(\F_\a)}(\F_\a)|$, which force $B$ to create new classes in view of \eqref{Equation proofs section on difference of HF ranks}.

\begin{thm}\label{Corollary rank estimates via H of Fa}
Assume $c_1(Y)=0$ and that $H^*(Y)$ lies in even degrees. Then, 
$$ %
\begin{array}{rclcrcl}
|H_{2k}(B)|
& \geq & \Delta_-^{2k}  
,
& \qquad\textrm{ and }\qquad &
 |H_{2k-1}(B)|
& \geq &
K^{2k} \geq
\Delta_+^{2k}, 
\\[2mm]
& & & & 
 |H_{2k-1}(B)|
& = &
|HF_+^{2k-1}(H_{\gamma})|
 \geq 
|\FF^{\Fi}_{\gamma}\cap QH^{2k}(Y)|
 \geq 
\Delta_+^{2k}.
\end{array}
$$
The last two inequalities are equalities for large $\gamma$: $HF^{2k-1}_+(H_{\gamma})\cong SH^{2k-1}_+(Y,\Fi)\cong QH^{2k}(Y).$

More precisely, there is an integer $r^{2k}$ depending on $P$, with
\begin{equation}
    \label{Equation the r ambiguity}
    \qquad\quad\,
|H_{2k-1}(B)|
=
K^{2k} + r^{2k}
\qquad \textrm{ and }\qquad
|H_{2k}(B)|
=
\Delta_-^{2k} + r^{2k},
\end{equation}
moreover\qquad
$0\leq r^{2k}\leq 
|QH^{2k}(Y)|-\delta_{P^-}^{2k}-K^{2k} \qquad \textrm{ and }\qquad
|\FF^{\Fi}_{P^-}\cap QH^{2k}(Y)| \geq 
r^{2k}+\delta_{P^-}^{2k}.
$
\end{thm}
\begin{proof}
We first prove \eqref{Equation the r ambiguity}.
Recall $\Delta_-^{2k}$ is a lower bound for the number of independent classes in degree $2k$ that were inductively created for slopes up to $\gamma$ to ensure that, at each previous slope value $\lambda<\gamma$, the rank of $HF^*(H_{\lambda})$ agrees with \eqref{Equation intro HF lambda H splits}.
Thus, the actual number of created classes in degree $2k$ is $\Delta_-^{2k}+r^{2k}$ for some $r^{2k}\geq 0$ (compare with the $r$ in \cref{Example Intro Slodowy S32 two tables}). At each step, the relevant injective map $\nu_{2k-1}$ kills classes,\footnote{no odd classes of $B$ are allowed to survive at any step of the inductive argument (induction on the critical slope $P$), since $HF^*(H_{\lambda})$ lies in even degrees; similarly no even classes can kill earlier classes as those are also in even degree.} and the total surviving classes in $H^{2k}(A\oplus B)$ is therefore
$$
|H^{2k}(Y)|+\Delta_{-}^{2k}+r^{2k} -
|H_{2k-1}(B)| = |H^{2k}(A\oplus B)|=|H^{2k}(Y)|-\delta_{\gamma}^{2k}.
$$
The first equality in \eqref{Equation the r ambiguity} follows, the second follows from \cref{Corollary finding homology of B complex} as $|H_{2k}(B)|=|H_{2k-1}(B)|-\delta_{\gamma}^{2k}$.
For the bound on $r^{2k}$, note that
$H_{2k-1}(B)$ maps injectively into $HF^{2k}(H_{P^-})$ (to kill classes), so 
$$
r^{2k}=
|H_{2k-1}(B)|-K^{2k}
\leq |HF^{2k}(H_{P^-})|-K^{2k}
=
|QH^{2k}(Y)|-\delta_{P^-}^{2k}-K^{2k}.
$$

We now justify the third inequality line: $B$ can kill at most $|H_{2k-1}(B)|$ classes in $QH^{2k}(Y)$, and the actual number of killed classes is $|\FF^{\Fi}_{\gamma}|=|\ker \psi_{0^+,\gamma}|$ by definition, which is at least $\Delta_+^{2k}$ by \cref{Cor estimating filtration using degrees}. 
For the final inequality, $H_{2k-1}(B)$ kills classes of $HF^{2k}(H_{P^-})$ but must allow $|QH^{2k}(Y)|-|\FF^{\Fi}_{P^-}\cap QH^{2k}(Y)|$ independent classes from $QH^{2k}(Y)$ to survive, by definition, so:
$$
r
\leq |HF^{2k}(H_{P^-})|-(|QH^{2k}(Y)|-|\FF^{\Fi}_{P^-}\cap QH^{2k}(Y)|)
=
|\FF^{\Fi}_{P^-}\cap QH^{2k}(Y)|-\delta_{P^-}^{2k}.
$$

The statement about the direct limit follows from \cref{PropSec2VanishingTheorem}: $SH^*(Y,\Fi)=0$, and $HF^*_+(H_\gamma)$ stabilises to $SH^*_+(Y,\Fi)$ for $\gamma\gg k$ because new generators appear in grading $\ll k$.
\end{proof}

\appendix
\section{Grading for Hamiltonian Floer theory}\label{AppendixFloer}\label{RSIndeces}

We follow the conventions in \cite[App.C]{McLR18}, and refer to \cite{Sa97,Gu14} for details and references. 
The Conley-Zehnder index is a $\Z$-valued index defined for certain non-degenerate paths of symplectic matrices, e.g.\;non-degenerate 1-periodic Hamiltonian orbits.\footnote{I.e. 1-periodic orbits $x(t)$ of $X_H$ satisfying $ker((\phi_1^H)_*-Id)_{x(0)}=0,$ where $\phi_t^H$ is the Hamiltonian flow of $H.$ } For Hamiltonians with degenerate orbits (e.g.\;autonomous Hamiltonians), one uses a generalisation, the \textbf{{\RS } index} \cite{RS95}. This a $\tfrac{1}{2}\Z$-valued index for \emph{any} continuous path $[0,1] \rightarrow Sp(\C^{n},\Omega_0)$ of real symplectic matrices in $\C^{n}$ with the standard real symplectic structure $\Omega_0$. We now list the main properties. Call $\psi_1\diamond\psi_2:\C^{n} \oplus \C^{m} \fun \C^{n} \oplus \C^{m}$ the direct sum of two symplectic matrices $\psi_1:\C^{n}\fun \C^{n}$ and $\psi_2:\C^{m}\fun \C^{m}.$

\begin{thm}\label{RSProperties} The {\RS } index satisfies the following properties:\footnote{We follow \cite[App.C]{McLR18} but abbreviated $\W(x)=W(2\pi x)$ compared to the function $W$ in that paper.}
\begin{enumerate}[(1)]
	\item \label{RShomotopiesproperty}$\mu_{RS}$ is invariant under homotopies with fixed endpoints.
	\item \label{catenation}$\mu_{RS}$ is additive under concatenation of paths.
	\item \label{RSofSUM} 
	$\mu_{RS}$ is compatible with sums: $\mu_{RS}(\psi_1\diamond\psi_2) = \mu_{RS}(\psi_1)+\mu_{RS}(\psi_2).$ 
	\item \label{prop:invariance}
	$
	\mu_{RS}(\phi\psi\phi^{-1})=\mu_{RS}(\psi)
	$ for any continuous paths of symplectic matrices $\psi , \phi :  [0,1] \rightarrow Sp(\R^{2n},\Omega_0).$	
	\item \label{RSofRotationInC}
	$\mu_{RS}((e^{ 2\pi is})_{s \in [0,x]})=\W(x),$ where
	\begin{equation}\label{WfunctionForRSAppendix}
	\W :\R \to \Z, \quad
	\W(x) := \left\{
	\begin{array}{ll}
	2 \lfloor x \rfloor + 1 & \text{if} \ x \notin \Z \\
	2x & \text{if}  \ x \in  \Z.
	\end{array}
	\right. 
	\end{equation}
	\item \label{shear}The {\RS } index of the \textit{symplectic shear}
	$\Big[\begin{smallmatrix}
		1&0\\
		b(t)&1
	\end{smallmatrix}\Big]$ is equal to $\tfrac{1}{2}(\mathrm{sign}(b(1))-\mathrm{sign}(b(0))).$
\end{enumerate}	
\end{thm}

We remark the following basic properties of the function $\W$:
\begin{equation}
\label{Equation Appendix W properties}
\W(0)=0, \;\;\quad \W(-x)=-\W(x),\;\;\quad \W(x)\textrm{ is odd except at }\Z, \;\;\quad 2x\geq \W(x)-1\geq 2x-2.
\end{equation}

For any 1-orbit $x$ of a Hamiltonian $H$ in a symplectic manifold $M$ of dimension $2n$, let $\phi_t$ be the Hamiltonian flow and consider its linearisation 
$(\phi_t)_*: T_{x(0)}M \fun T_{x(t)}M$.
We pick a symplectic trivialisation $\Phi: x^* TM \fun \C^n \times S^1,$ 
of the tangent bundle above the orbit $x$ to get a path of symplectic matrices $\psi(t)=\Phi_t \circ (\phi_t)_* \circ \Phi^{-1}_0: \C^n \fun \C^n.$
Then define 
\begin{equation}\label{RSForAnOrbit}
	RS(x,H):=\mu_{RS}(\psi).
\end{equation}

This may depend on the choice of trivialisation $\Phi$. When $c_1(M)=0,$ one can trivialise the canonical bundle $\Lambda^{n,0}(T^*M),$ and then choose a trivialisation $\Phi$ compatible with it. %
If in addition $H^1(M)=0,$ then all trivialisations of $\Lambda^{n,0}(T^*M)$ are equivalent, so the indices  \eqref{RSForAnOrbit} are canonical  \cite[Sec.(3a)]{Sei08}. 

We define the \textbf{grading} of a 1-orbit $x$ of a Hamiltonian $H:M\fun \R$ by
\begin{equation}\label{DefinitionGrading}
	|x|:=\dim_\C M - RS(x,H).
\end{equation}
\begin{lm}\label{ZGradingsOnComplSympl}
	Let $(M,\om)$ be a symplectic manifold with $c_1(TM,I)=0$ for an an $\om$-compatible almost complex structure $I$.
  Then Floer cohomology $HF^*(H)$ is $\Z$-graded, and canonically so if 
 $H^1(M)=0$.
\end{lm}
\section{Cotangent bundles and negative vector bundles}\label{Appendix Cotangent bundles and vector bundles}
\subsection{The moment map is a function of the radial coordinate for $T^*X$}
\label{Moment map is function of radial coordinate for TCPn}
Recall that $T^*\CP^{n-1}$ can be seen as the \HK reduction of the flat space 
$M:=\C^n \oplus \C^n$ with the action 
$G:=U(1) \dejstvo M$ given by 
$g \cdot (z,\xi) = (z g^{-1}, g \xi).$ 
The \HK moment map $\mu=(\mu_\R,\mu_\C)$ has real part
$\mu_\R=\xi\xi^*-z^*z$
and complex part
$\mu_\C=\xi z$ (viewing $\xi$ as a row vector).
Taking $\zeta_\R>0$, the \HK reduction gives 
$$\M:= \mu^{-1}(-\zeta_\R \Id,0)/G \iso T^*\CP^{n-1}.$$
There is an $S^1$-action induced by $t \cdot (z,\xi)= (z, t \xi)$ on $M,$ 
with moment map 
$H=\tr(\xi \xi^*).$
The projection 
$$\Psi: \M \fun \sl_n, \ [(z, \xi)] \mapsto z \xi$$
is an example of a Springer resolution. 
This map, together with the $S^1$-invariant \KH structure on $\M$ (induced from the \KH structure on $M$)
makes $\M$ a symplectic $\C^*$-manifold globally defined over the convex base $B=\sl_n.$
The pull-back of the radial coordinate on $B$ is
thus equal to
\begin{equation}\label{Radial_Coordinate_TPN}
\Phi=\tr(z \xi (z\xi)^*)=\tr(z \xi \xi^* z^*)=\tr(z^*z \xi \xi^*).
\end{equation}
Substituting the moment map equation
$\mu_\R=\xi\xi^*-z^*z=-\zeta_\R \Id$
into \eqref{Radial_Coordinate_TPN}
we get
\begin{align}\label{Radial_coordinate_2}
   \Phi=\tr((\xi\xi^*+\zeta_\R \Id)\xi\xi^*)
=\tr(\xi\xi^* \xi\xi^*)+ \zeta_\R \tr(\xi\xi^*) %
= \tr(\xi\xi^* \xi\xi^*) + \zeta_\R H.
\end{align}
Now, noticing that $A:=\xi \xi^*$ is actually a $1 \times 1$-matrix, the last term in 
\eqref{Radial_coordinate_2} 
becomes
\begin{equation}\label{key moment of proof TCPn radial=moment^2}
   \Phi=\tr(A^2) + \zeta_\R H = \tr(A)^2 + \zeta_\R H=H^2+\zeta_\R H, 
\end{equation}
thus\footnote{Since $H\geq 0,$ the other solution of the quadratic equation is invalid.} 
$H=-\frac{1}{2}\zeta_\R + \sqrt{\Phi+\frac{1}{4}\zeta_\R^2}$
is indeed a function of $\Phi.$

We remark that %
the same conclusion does not hold for cotangent bundles of other flag varieties of type A, with the symplectic structure from the analogous \HK reduction.\footnote{Constructed e.g. in \cite[Sec.7]{Nak94a}.}
For example, to get the cotangent bundle of a Grassmannian, $T^*Gr(k,n),$ one has to change $z$ and $\xi$ to $(n \times k)$ and $(k \times n)$-matrices. Then $A=\xi \xi^*$ is a $(k \times k)$-matrix, so the $\tr(A^2)=\tr(A)^2$ needed in \eqref{key moment of proof TCPn radial=moment^2} fails.
\subsection{$T^*\C\P^n$ and negative vector bundles}

A vector bundle $E\to X$ over a compact manifold is \emph{negative} in the sense of \cite[Lem.70]{R14} if
there is a (real) symplectic form on $E$ of type
$\omega=\pi^*\omega_X + \Omega,$
determined by a Hermitian connection on $E$, where $\Omega|_{\textrm{fibre}}=\textrm{(area form)}/\pi$ in a unitary frame, and $\Omega$ on horizontal vectors is determined by the curvature form of $E \to X$.\footnote{Moreover $\Omega(TX,\cdot)=0$, and $\Omega(v,h)=0$ if $v$ is vertical and $h$ is horizontal.}
By \cite[Lem.70]{R14}, it suffices\footnote{
Key observation: $\omega:=\pi^*\omega_{X} + \Omega$ on the total space $\mathrm{Tot}(E \to X)$ is symplectic because of the positivity of the form%
$$
\tfrac{1}{2\pi i} w^{\dagger}\mathcal{F}^E_{(\pi_*h,I \pi_* h)}w=\tfrac{r^2}{2\pi i} \mathcal{F}_{(h,Ih)}^{L}>0, \qquad \textrm{ for }w\in E\setminus\{0\},
$$
for horizontal $h\neq 0 \in T_w E$ (a horizontal vector of the projection $\P(E)\to X$), where $r=|w|$ is the radial coordinate for fibres of $E$. By \cite[Sec.11.2]{R14} the ``radial coordinate'' for $L$ as a convex symplectic manifold is $R^L=(1+r^2)/2$ (its moment map generates the flow $e^{\pi i t}$, not $e^{2\pi i t}$), and Hamiltonians used for $E$ are functions $c(2R^L)$ linear in $R^L$ at $\infty$.}
that the natural complex line bundle $L:=L(E)\to \mathbb{P}(E)$ over the complex projectivisation is a negative line bundle.\footnote{A complex line bundle $L\to B$ over a symplectic manifold is \emph{negative} if $c_1(L)\in \R_{<0}\cdot [\omega_B]$ \cite[Sec.7]{R14}. Equivalently, $\exists$ Hermitian metric and a compatible Hermitian connection, whose curvature induces a form
$-\tfrac{i}{2\pi}\mathcal{F}^{L}$ (representing  $-c_1(L)$) which
is positive on complex lines for any $\omega_B$-compatible almost complex structure. For holomorphic line bundles, it means a Hermitian metric yields a Chern connection whose curvature gives a K\"{a}hler form on $B$ via $-\tfrac{i}{2\pi}\mathcal{F}^{L}$.}
If $E$ is Griffiths weak-negative \cite{GriffithsComputation}, this holds for $L$ \cite[Prop.9.1]{GriffithsComputation} so $E$ is negative.\\
{\bf Examples.} $E=T^*\C\P^n$ is Griffiths weak-negative\footnote{weak-negativity of $T^*\C P^n$ follows from the ampleness of $T\C P^n$ in \cite[Sec.10.(iii)]{GriffithsComputation}, however ampleness of $TX$ for a projective variety only holds for $X=\C P^n$ by Mori \cite{Mori}.} for the Fubini--Study metric \cite[Sec.10.(iii)]{GriffithsComputation}, so it is negative. For $E$ a complex vector bundle, and $\mathcal{L}$ a negative line bundle, $E\otimes \mathcal{L}^{\otimes N}$ is negative for all large enough $N$
(\cite[Cor.9]{R14},
\cite[Prop.4.3]{GriffithsComputation}).
Duals of ample vector bundles are negative  \cite[Prop.7.3]{GriffithsComputation}, e.g.\;conormal\footnote{the dual of the normal bundle, so sections are cotangent vectors of the ambient that vanish on $TV$.} bundles to projective algebraic varieties $V\subset \C\P^n$ \cite[Cor.\,in Sec.10]{GriffithsComputation}. 

\subsection{The filtration for negative vector bundles}\label{Subsection The filtration for negative vector bundles}
We now prove \cref{Proposition filtration for negative line bundles}.
By \cite[Thm.72]{R14}, $Q_{\Fi}=c:=\pi^*c_{\mathrm{rank}_{\C}E}(E)$ (with a possible rescaling by an invertible element $1+T^{>0}\textrm{-terms}\in \k$ in the Calabi--Yau setting); this holds more generally if $E$ is weak+ monotone.
The $S^1$-periods are integers, so by stability \eqref{Intro Stability property} the filtration can only jump at full-rotations  (\cref{Subsection Computing the continuation map for full rotations: the QFi invariant}).
From \cref{Prop vanishing of SH}: $E_0 = \ker c^m$, where $x^m$ is the highest factor of $x$ in the minimal polynomial of quantum product by $c$ on $QH^*(Y)$.
By linear algebra, kernels of powers of an endomorphism on a finite dimensional vector space will strictly increase until they stabilise. Thus $\FF_{\lambda}^{\Fi}$ strictly increases in size precisely at $\lambda=1,2,\ldots,m$; in the Calabi-Yau setting $\FF_{m}^{\Fi}=QH^*(Y)$ as $c$ is nilpotent ($QH^*(Y)_c=SH^*(Y)=0$).

For $n_0\in H^*(Y;\mathbb{B})$, let $q=n_0 + T^{>0}\textrm{-terms}\in QH^*(Y)$ so $n_0=\mathrm{ini}(q)$ (\cref{Subsection specialisation intro}). 
Then $c^{\star j} \star q=0$ implies the lowest order $T$-term is $0=\mathrm{ini}(c^{\star j})\cup \mathrm{ini}(q)=c^j\cup n_0$, so $\FF_{\mathbb{B},j}^{\Fi} = \mathrm{ini} \FF_j^{\Fi}=\mathrm{ini}(\ker (c^{\star j}\star))\subset \ker (c^j\cup)\subset H^*(Y;\mathbb{B})$.
Assuming $\mathrm{Im}(c^{\star j}\star)\subset \mathrm{Im}(c^j\cup)$, we now show $\ker (c^j\cup) \subset \mathrm{ini}(\ker c^{\star j}\star)$, hence equality. By the assumption, $c^{\star j}\star n_0 - c^j\,\cup n_0 = c^j \, \cup T^{a_1}(n_1 +T^{>0}\textrm{-terms})$ for $n_1\in H^*(Y;\mathbb{B})$,  $a_1>0$. If $n_0\in \ker (c^j\cup)$, then $c^{\star j}\star (n_0 - n_1 T^{a_1})=T^{>a_1}$-terms. Thus inductively we can build $q=n_0-n_1T^{a_1}-n_2T^{a_2}-\cdots\in \ker (c^{\star j}\star)$ with $0<a_1<a_2<\cdots\to \infty$, satisfying $\mathrm{ini}(q)=n_0$. $\qed$
\bibliography{FZ}

@article{Abbondandolo-Schwarz,
    AUTHOR = {A. Abbondandolo and M. Schwarz},
     TITLE = {Floer homology of cotangent bundles and the loop product},
   JOURNAL = {Geom. Topol.},
  FJOURNAL = {Geometry \& Topology},
    VOLUME = {14},
      YEAR = {2010},
    NUMBER = {3},
     PAGES = {1569--1722}
}

@article{albers2017local,
  title={{Local systems on the free loop space and finiteness of the Hofer-Zehnder capacity}},
  author={Albers, P. and Frauenfelder, U. and Oancea, A.},
  journal={Math. Ann.},
  volume={367},
  pages={1403--1428},
  year={2017},
  publisher={Springer}
}

@article{At82,
  title={{Convexity and commuting Hamiltonians}},
  author={Atiyah, M. F.},
  journal={Bull. London Math. Soc.},
  volume={14},
  number={1},
  pages={1--15},
  year={1982},
  publisher={Citeseer}
}

@article{AtB83,
author = {Atiyah, M. F. and Bott, R.},
title = {{The Yang-Mills equations over Riemann surfaces}},
journal = {Philos. Trans. Roy. Soc. London Ser. A},
year = {1983},
volume = {308},
number = {1505},
pages = {523-615},
}

@article{BH13,
author = {A. Banyaga and D. E. Hurtubise},
title = {{Cascades and perturbed Morse--Bott functions}},
journal = {Algebr. Geom. Topol.},
year = {2013},
volume = {13},
number = {1},
pages = {237-275},
}

@article{Batyrev,
  title={{Non-Archimedean integrals and stringy Euler numbers of log-terminal pairs}},
  author={Batyrev, V. V.},
  journal={J. Eur. Math. Soc.},
  volume={1},
  pages={5--33},
  year={1999},
  publisher={Springer}
}

@article{Be00,
author = {A. Beauville},
title = {{Symplectic singularities}},
journal = {Invent. Math},
year = {2000},
volume = {139},
number = {3},
pages = {541-549},
}

@article{BeSch18,
author = {G. Bellamy and T. Schedler},
title = {{Filtrations on Springer fiber cohomology and Kostka polynomials}},
journal = {Lett. Math. Phys.},
year = {2018},
volume = {108},
number = {3},
pages = {679-698},
}

@article{bellamy2016non,
  title={On the (non) existence of symplectic resolutions of linear quotients},
  author={Bellamy, G. and Schedler, T.},
  journal={Math. Res. Lett.},
  volume={23},
  number={6},
  pages={1537--1564},
  year={2016},
  publisher={International Press of Boston}
}

@article {BeRit20,
author = {Benedetti, G. and Ritter, A. F.},
title = {Invariance of symplectic cohomology and twisted cotangent bundles over surfaces},
journal = {Internat. J. Math.},
year = {2020},
volume = {31},
number = {9},
pages = {2050070, 56},
}

@article {BezKaledinMcKay,
    AUTHOR = {Bezrukavnikov, R. V. and Kaledin, D. B.},
     TITLE = {Mc{K}ay equivalence for symplectic resolutions of quotient singularities},
   JOURNAL = {Tr. Mat. Inst. Steklova},
  FJOURNAL = {Trudy Matematicheskogo Instituta Imeni V. A. Steklova},
    VOLUME = {246},
      YEAR = {2004},
    NUMBER = {Algebr. Geom. Metody, Svyazi i Prilozh.},
     PAGES = {20--42}
}

@article{BPW16,
author = {T. Braden and N. Proudfoot and B. Webster},
title = {{Quantizations of conical symplectic resolutions I: local and global structure}},
journal = {Ast{\'e}risque No},
year = {2016},
volume = {384},
pages = {1-73},
}

@book{Bre72,
author = {G. E. Bredon},
title = {Introduction to compact transformation groups},
year = {1972},
address = {},
publisher = {Academic press},
}

@article {Brion13,
    AUTHOR = {Brion, M.},
     TITLE = {On linearization of line bundles},
   JOURNAL = {J. Math. Sci. Univ. Tokyo},
  FJOURNAL = {The University of Tokyo. Journal of Mathematical Sciences},
    VOLUME = {22},
      YEAR = {2015},
    NUMBER = {1},
     PAGES = {113--147},
}

@inproceedings{deCataldoMigliorini2005,
  title={{The Hodge theory of algebraic maps}},
  author={de Cataldo, M. A. A. and Migliorini, L.},
  booktitle={Ann. Sci. Éc. Norm. Supér},
  volume={38},
  number={5},
  pages={693--750},
  year={2005}
}

@article{de2012topology,
  title={{Topology of Hitchin systems and Hodge theory of character varieties: the case $A_1$}},
  author={de Cataldo, M. A. A. and Hausel, T. and Migliorini, L.},
  journal={Ann. of Math.},
  pages={1329--1407},
  year={2012},
  publisher={JSTOR}
}

@incollection{CMS-B02,
  title={Characterizations of projective space and applications to complex symplectic manifolds},
  author={Cho, K. and Miyaoka, Y. and Shepherd-Barron, N. I.},
  booktitle={Higher dimensional birational geometry},
  pages={1--88},
  year={2002},
  publisher={Math. Soc. of Japan}
}

@book{CGi97,
author = {N. Chriss and V. Ginzburg},
title = {{Representation Theory and Complex Geometry}},
year = {1997},
address = {Boston},
publisher = {Birkh{\"a}user},
}

@ARTICLE{EGA,
    AUTHOR ={J. Dieudonn{\'e} and A. Grothendieck},
    TITLE = {\'{E}l\'ements de g\'eom\'etrie alg\'ebrique},
    JOURNAL ={Inst. Hautes \'Etudes Sci. Publ. Math.},
    VOLUME = {4, 8, 11, 17, 20, 24, 28, 32},
    YEAR = {1961--1967},
    PAGES={}
}

@book{DK00,
author = {J. J. Duistermaat and J. A. C. Kolk},
title = {Lie groups},
year = {2000},
address = {Springer},
publisher = {Universitext},
}

@article{Frankel59,
 ISSN = {0003486X},
 URL = {http://www.jstor.org/stable/1969889},
 author={T. Frankel},
 journal = {Ann. of Math.},
 number = {1},
 pages = {1--8},
 publisher = {Ann. of Math.},
 title={{Fixed Points and Torsion on K{\"a}hler Manifolds}},
 volume = {70},
 year = {1959}
}

@article {Giesecke,
    AUTHOR = {Giesecke, B.},
     TITLE = {Simpliziale {Z}erlegung abz\"{a}hlbarer analytischer {R}\"{a}ume},
   JOURNAL = {Math. Z.},
  FJOURNAL = {Mathematische Zeitschrift},
    VOLUME = {83},
      YEAR = {1964},
     PAGES = {177--213}
}

@article {Goresky,
    AUTHOR = {Goresky, R. M.},
     TITLE = {Triangulation of stratified objects},
   JOURNAL = {Proc. Amer. Math. Soc.},
  FJOURNAL = {Proceedings of the American Mathematical Society},
    VOLUME = {72},
      YEAR = {1978},
    NUMBER = {1},
     PAGES = {193--200}
}

@article{GriffithsComputation,
 author = {Griffiths, P. A.},
 journal = {Journal of Mathematics and Mechanics},
 number = {1},
 pages = {117--140},
 publisher = {Indiana University Mathematics Department},
 title = {Hermitian Differential Geometry and the Theory of Positive and Ample Holomorphic Vector Bundles},
 urldate = {2023-04-12},
 volume = {14},
 year = {1965}
}

@article{Gr15,
author = {Y. Groman},
title = {Floer theory and reduced cohomology on open manifolds},
year = {2015},
journal = {arXiv:1510.04265},
pages = {1--103}
}

@book{GrHa78,
author = {P. Griffiths and J. Harris},
title = {{Principles of algebraic geometry. Pure and Applied Mathematics}},
year = {1978},
address = {New York},
publisher = {John Wiley \& Sons},
}

@article{gutt2018symplectic,
  title={{Symplectic capacities from positive $S^1$--equivariant symplectic homology}},
  author={Gutt, J. and Hutchings, M.},
  journal={Algebr. Geom. Topol.}, 
  volume={18},
  number={6},
  pages={3537--3600},
  year={2018},
  publisher={Mathematical Sciences Publishers}
}

@article{Gu14,
AUTHOR = {Gutt, J.},
     TITLE = {Generalized {C}onley-{Z}ehnder index},
   JOURNAL = {Ann. Fac. Sci. Toulouse Math. (6)},
  FJOURNAL = {Annales de la Facult\'{e} des Sciences de Toulouse. Math\'{e}matiques.
              S\'{e}rie 6},
    VOLUME = {23},
      YEAR = {2014},
    NUMBER = {4},
     PAGES = {907--932},
}

@book {Hatcher,
    AUTHOR = {Hatcher, A.},
     TITLE = {Algebraic topology},
 PUBLISHER = {Cambridge University Press, Cambridge},
      YEAR = {2002},
     PAGES = {xii+544}
}

@book{Ha77,
author = {R. Hartshorne},
title = {Algebraic Geometry},
year = {1977},
address = {New York, Springer},
publisher = {Grad. Texts in Math. 52},
}

@article{HS02,
author = {T. Hausel and B. Sturmfels},
title = {Toric hyperk{\"a}hler varieties},
journal = {Doc. Math},
year = {2002},
volume = {7},
pages = {495-534},
}

@article{hausel2022p,
  title={${P}= {W} $ via $\mathcal{H}_2$},
  author={Hausel, T. and Mellit, A. and Minets, A. and Schiffmann, O.},
  journal={arXiv:2209.05429},
  year={2022},
  pages={1--54}
}

@article {Hausel-Villegas,
    AUTHOR = {Hausel, T. and Rodriguez-Villegas, F.},
     TITLE = {Cohomology of large semiprojective hyperk\"{a}hler varieties},
   JOURNAL = {Ast\'{e}risque},
  FJOURNAL = {Ast\'{e}risque},
    NUMBER = {370},
      YEAR = {2015},
     PAGES = {113--156}
}

@article{hausel_1998,
    author = {Hausel, T.},
    title = {Vanishing of intersection numbers on the moduli space of Higgs bundles},
    journal = {Adv. Theor. Math. Phys.},
    volume = {2},
    number = {5},
    year = {1998},
    pages = {1011--1040},
    doi = {10.4310/ATMP.1998.v2.n5.a3},
    issn = {1095-0761},
    publisher = {International Press of Boston}
}

@article{heinloth2016intersection,
  title={{The intersection form on moduli spaces of twisted $PGL_n$-Higgs bundles vanishes}},
  author={J. Heinloth},
  journal={Math. Ann.},
  volume={365},
  pages={1499--1526},
  year={2016},
  publisher={Springer}
}

@inproceedings{Hironaka75,
    AUTHOR = {Hironaka, H.},
     TITLE = {Triangulations of algebraic sets},
 BOOKTITLE = {Algebraic geometry ({P}roc. {S}ympos. {P}ure {M}ath., {V}ol.
              29, {H}umboldt {S}tate {U}niv., {A}rcata, {C}alif., 1974)},
     PAGES = {165--185},
 PUBLISHER = {Amer. Math. Soc., Providence, R.I.},
      YEAR = {1975}
}

@article{Hi87,
  title={{The self-duality equations on a Riemann surface}},
  author={Hitchin, N. J.},
  journal={Proc. Lond. Math. Soc.},
  volume={3},
  number={1},
  pages={59--126},
  year={1987},
  publisher={Wiley Online Library}
}

@incollection{HS95,
title = {{Floer homology and Novikov rings}},
author = {H. Hofer and D. Salamon},
booktitle={The {F}loer memorial volume},
    SERIES = {Progr. Math.},
    VOLUME = {133},
  pages={483--524},
  year={1995},
PUBLISHER = {Birkh\"{a}user, Basel},
}

@article{hutchings2011quantitative,
  title={Quantitative embedded contact homology},
  author={Hutchings, M.},
  journal={J. Differential Geom.},
  volume={88},
  number={2},
  pages={231--266},
  year={2011}
}

@article{irie2014hofer,
  title={{Hofer--Zehnder capacity of unit disk cotangent bundles and the loop product}},
  author={Irie, K.},
  journal={J. Eur.
Math. Soc. (JEMS)},
  volume={16},
  number={11},
  pages={2477--2497},
  year={2014}
}

@book{Ish97,
author = {S. Ishii},
title = {Introduction to singularities},
year = {1997},
address = {Tokyo},
publisher = {Springer-Verlag},
}

@article {KaledinMcKay,
    AUTHOR = {Kaledin, D.},
     TITLE = {Mc{K}ay correspondence for symplectic quotient singularities},
   JOURNAL = {Invent. Math.},
  FJOURNAL = {Inventiones Mathematicae},
    VOLUME = {148},
      YEAR = {2002},
    NUMBER = {1},
     PAGES = {151--175}
}

@article{kaledin2000,
  title={Symplectic resolutions: deformations and birational maps},
  author={D. Kaledin},
  journal={arXiv:math/0012008},
  year={2000},
  pages={1--35}
}

@article{Ka06,
author = {D. Kaledin},
title = {{Symplectic singularities from the Poisson point of view}},
journal = {J. Reine Angew. Math},
year = {2006},
volume = {600},
pages = {135-156},
}

@article{Ka08,
author = {D. Kaledin},
title = {Derived equivalences by quantization},
journal = {Geom. Funct. Anal},
year = {2008},
volume = {17},
number = {6},
pages = {1968-2004},
}

@incollection {Ka09,
 author = {D. Kaledin},
     TITLE = {Geometry and topology of symplectic resolutions},
 BOOKTITLE = {Algebraic geometry---{S}eattle 2005. {P}art 2},
    SERIES = {Proc. Sympos. Pure Math.},
    VOLUME = {80},
     PAGES = {595--628},
 PUBLISHER = {Amer. Math. Soc., Providence, RI},
      YEAR = {2009},
       DOI = {10.1090/pspum/080.2/2483948},
       URL = {https://doi.org/10.1090/pspum/080.2/2483948},
}

@article{KaVe02,
author = {D. Kaledin and M. Verbitsky},
title = {Period map for non-compact holomorphically symplectic manifolds},
journal = {Geom. Funct. Anal.},
year = {2002},
volume = {12},
number = {6},
pages = {1265--1295},
}

@book{Ki84,
author = {F. C. Kirwan},
title = {Cohomology of quotients in symplectic and algebraic geometry},
year = {1984},
address = {Princeton Univ. Press, Princeton},
publisher = {Math. Notes 31},
}

@article{Ko58,
author = {Kobayashi, S.},
title = {Fixed points of isometries},
journal = {Nagoya Math. J.},
year = {1958},
volume = {13},
pages = {63--68},
}

@article{Kuz07,
  title={{Quiver varieties and Hilbert schemes}},
  author={Kuznetsov, A. G.},
  journal={Mosc. Math. J.},
  volume={7},
  number={4},
  pages={673--697},
  year={2007},
  publisher={Независимый Московский университет--МЦНМО}
}

@article {Laudenbach,
    AUTHOR = {Laudenbach, F.},
     TITLE = {On the {T}hom-{S}male complex},
      NOTE = {Appendix to Bismut-Zhang, An extension of a theorem by {C}heeger and {M}\"{u}ller, Ast\'{e}risque no. 205 (1992), 235 pp.},
   JOURNAL = {Ast\'{e}risque},
  FJOURNAL = {Ast\'{e}risque},
    NUMBER = {205},
      YEAR = {1992},
     PAGES = {219--233}
}

@incollection{lee2013smooth,
  title={Smooth manifolds},
  author={Lee, J. M.},
  booktitle={Introduction to smooth manifolds},
  pages={1--31},
  year={2013},
  publisher={Springer}
}

@article{lerman1995symplectic,
  title={Symplectic cuts},
  author={Lerman, E.},
  journal={Mathematical Research Letters},
  volume={2},
  number={3},
  pages={247--258},
  year={1995},
  publisher={International Press of Boston}
}

@article {Lojasiewicz,
    AUTHOR = {Lojasiewicz, S.},
     TITLE = {Triangulation of semi-analytic sets},
   JOURNAL = {Ann. Scuola Norm. Sup. Pisa Cl. Sci. (3)},
  FJOURNAL = {Annali della Scuola Normale Superiore di Pisa. Classe di
              Scienze. Serie III},
    VOLUME = {18},
      YEAR = {1964},
     PAGES = {449--474}
}

@article{maulik2022p,
  title={{The $P= W$ conjecture for $\mathrm{GL}_n$}},
  author={Maulik, D. and Shen, J.},
  journal={arXiv:2209.02568},
  year={2022},
  pages={1--23}
}

@book{May,
AUTHOR = {May, J. P.},
     TITLE = {A concise course in algebraic topology},
    SERIES = {Chicago Lectures in Mathematics},
 PUBLISHER = {University of Chicago Press, Chicago, IL},
      YEAR = {1999},
     PAGES = {x+243},
}

@article{McLR18,
author={McLean, M. and Ritter, A. F.},
title = {{The McKay correspondence for isolated singularities via Floer theory}},
volume = {124},
journal = {J. Differential Geom.},
number = {1},
pages = {113 -- 168},
year = {2023}
}

@article{Mori,
 author = {S. Mori},
 journal = {Ann. of Math.},
 number = {3},
 pages = {593--606},
 publisher = {Annals of Mathematics},
 title = {Projective Manifolds with Ample Tangent Bundles},
 urldate = {2023-10-02},
 volume = {110},
 year = {1979}
}

@book{MuFo82,
    author={Mumford, D. and Fogarty, J.},
    title={Geometric invariant theory},
    series={Ergebnisse der Mathematik und ihrer Grenzgebiete},
    VOLUME = {34},
   EDITION = {Second},
 PUBLISHER = {Springer-Verlag, Berlin},
      YEAR = {1982},
     PAGES = {xii+220},
}

@article{Nak94a,
  title={{Instantons on ALE spaces, quiver varieties, and Kac-Moody algebras}},
  author={H. Nakajima},
  journal={Duke Math. J.},
  volume={76},
  number={2},
  pages={365--416},
  year={1994},
  publisher={Duke University Press}
}

@book{Nak99,
  title={{Lectures on Hilbert schemes of points on surfaces}},
  author={H. Nakajima},
  number={18},
  year={1999},
  publisher={American Mathematical Soc.}
}

@article{Nam08,
author = {Y. Namikawa},
title = {{Flops and Poisson deformations of symplectic varieties}},
journal = {Publ. Res. Inst. Math. Sci},
year = {2008},
volume = {44},
number = {2},
pages = {259-314},
}

@article{Nam11,
author = {Y. Namikawa},
title = {Poisson deformations of affine symplectic varieties},
  journal={Duke Math. J.},
  volume={156},
  number={1},
  pages={51--85},
  year={2011},
  publisher={Duke University Press}
}

@book {Nicolaescu,
    AUTHOR = {Nicolaescu, L.},
     TITLE = {An invitation to {M}orse theory},
    SERIES = {Universitext},
   EDITION = {Second},
 PUBLISHER = {Springer, New York},
      YEAR = {2011},
     PAGES = {xvi+353},
}

@article{OanceaEnsaios,
  TITLE = {{A survey of Floer homology for manifolds with contact type boundary or symplectic homology}},
  AUTHOR = {Oancea, A.},
  URL = {https://hal.science/hal-00141468},
  JOURNAL = {{Ensaios Matem{\'a}ticos}},
  PUBLISHER = {{Brazilian Mathematical Society}},
  VOLUME = {7},
  PAGES = {51-91},
  YEAR = {2004},
  MONTH = Sep,
  HAL_ID = {hal-00141468},
  HAL_VERSION = {v1},
}

@article{oh1997symplectic,
  title={{Symplectic topology as the geometry of action functional. I. Relative Floer theory on the cotangent bundle}},
  author={Oh, Y.-G.},
  journal={J. Differential Geom.},
  volume={46},
  number={3},
  pages={499--577},
  year={1997}
}

@book{PSS,
author = {S. Piunikhin and D. Salamon and M. Schwarz},
title = {{Symplectic Floer-Donaldson theory and quantum cohomology}},
year = {1996},
address = {Publ. Newton Inst. 8, Cambridge University Press, 171--200},
publisher = {in Contact and symplectic geometry},
}

@book{Polterovich,
    AUTHOR = {L. Polterovich and D. Rosen 
              and K. Samvelyan and  J. Zhang},
     TITLE = {Topological persistence in geometry and analysis},
    SERIES = {University Lecture Series},
    VOLUME = {74},
 PUBLISHER = {American Mathematical Society, Providence, RI},
      YEAR = {2020},
     PAGES = {xi+128}
}

@book{Re92,
author = {M. Reid},
title = {{The McKay correspondence and the physicists' Euler number conjecture}},
year = {1992},
address = {Sep/Nov},
publisher = {Lectures at the University of Utah and at MSRI},
}

@article{ritter2009novikov,
  title={Novikov-symplectic cohomology and exact Lagrangian embeddings},
  author={A. F. Ritter},
  journal={Geom. Topol.},
  volume={13},
  number={2},
  pages={943--978},
  year={2009},
  publisher={Mathematical Sciences Publishers}
}

@article{R10,
author = {A. F. Ritter},
title = {{Deformations of symplectic cohomology and exact Lagrangians in ALE spaces}},
journal = {Geom. Funct. Anal.},
year = {2010},
volume = {20},
number = {3},
pages = {779-816},
}

@article{R13,
author = {A. F. Ritter},
title = {Topological quantum field theory structure on symplectic cohomology},
journal = {J. Topol.},
year = {2013},
volume = {6},
number = {2},
pages = {391-489},
}

@article{R14,
author = {A. F. Ritter},
title = {{Floer theory for negative line bundles via Gromov-Witten invariants}},
journal = {Adv. Math.},
year = {2014},
volume = {262},
pages = {1035-1106},
}

@article {R16,
author = {A. F. Ritter},
title = {{Circle actions, quantum cohomology, and the {F}ukaya category of {F}ano toric varieties}},
journal = {Geom. Topol.},
year = {2016},
volume = {20},
number = {4},
pages = {1941--2052},
}

@article{ritter2017monotone,
  title={{The monotone wrapped Fukaya category and the open-closed string map}},
  author={A. F. Ritter and I. Smith},
  journal={Selecta Math. (N.S.)},
  volume={23},
  pages={533--642},
  year={2017},
  publisher={Springer}
}

@article{RZ2,
author = {A. F. Ritter and F. {\v{Z}}ivanović},
title = {{Filtrations on quantum cohomology via Morse--Bott--Floer Spectral Sequences}},
journal={arXiv:2304.14384},
year = {2023},
pages={1--90}
}

@article{RZ3,
author = {A. F. Ritter and F. {\v{Z}}ivanović},
title = {{Filtrations on equivariant quantum cohomology and Hilbert-Poincar\'{e} series}},
journal={arXiv:2410.17237},
year = {2024},
pages={1--72}
}

@article{RZ4,
author = {A. F. Ritter and F. {\v{Z}}ivanović},
title = {{Quantum cohomology and Floer invariants of semiprojective toric manifolds}},
journal={arXiv:2501.09011},
year = {2025},
pages={1--38}
}

@book{RS95,
author = {J. Robbin and D. Salamon},
title = {{The spectral flow and the Maslov index}},
year = {1995},
address = {1--33},
publisher = {Bull. London Math. Soc. 27},
}

@article{Sa97,
author = {D. Salamon},
title = {{Lectures on Floer homology. Symplectic geometry and topology}},
journal = {IAS/Park City Math. Ser., 7},
year = {1997},
volume = {},
number = {},
pages = {143-229},
}

@book {McDuff-Salamon,
    AUTHOR = {D. McDuff and D. Salamon},
     TITLE = {{$J$}-holomorphic curves and symplectic topology},
    SERIES = {American Mathematical Society Colloquium Publications},
    VOLUME = {52},
 PUBLISHER = {American Mathematical Society, Providence, RI},
      YEAR = {2004},
     PAGES = {xii+669},
       DOI = {10.1090/coll/052},
}

@article{mcduff2006topological,
  title={{Topological properties of Hamiltonian circle actions}},
  author={McDuff, D. and Tolman, S.},
  journal={Int. Math. Res. Not.},
  volume={2006},
  pages={72826},
  year={2006},
  publisher={Hindawi Publishing Corporation}
}

@article{schwarz2000action,
  title={{On the action spectrum for closed symplectically aspherical manifolds}},
  author={M. Schwarz},
  journal={Pac. J. Math.},
  volume={193},
  number={2},
  pages={419--461},
  year={2000}
}

@article{Sei97,
    AUTHOR = {P. Seidel},
     TITLE = {{$\pi_1$} of symplectic automorphism groups and invertibles in
              quantum homology rings},
   JOURNAL = {Geom. Funct. Anal.},
    VOLUME = {7},
      YEAR = {1997},
    NUMBER = {6},
     PAGES = {1046--1095},
}

@incollection{Sei08,
author = {P. Seidel},
     TITLE = {A biased view of symplectic cohomology},
 BOOKTITLE = {Current developments in mathematics, 2006},
     PAGES = {211--253},
 PUBLISHER = {Int. Press, Somerville, MA},
      YEAR = {2008},
}

@book{Spanier,
author = {Spanier, E. H. },
title = {Algebraic topology},
year = {1966},
address = {},
publisher = {McGraw-Hill Book Co New York-Toronto Ont.-London},
}

@misc{stacks-project,
  author       = {The {Stacks project authors}},
  title        = {The Stacks project},
  howpublished = {\url{https://stacks.math.columbia.edu}},
  year         = {2022},
}

@article{venkatesh2021quantitative,
AUTHOR = {Venkatesh, S.},
     TITLE = {The quantitative nature of reduced {F}loer theory},
   JOURNAL = {Adv. Math.},
  FJOURNAL = {Advances in Mathematics},
    VOLUME = {383},
     PAGES = {Paper No. 107682, 80},
     YEAR = {2021},
}

@article{Vi96,
author = {C. Viterbo},
title = {{Functors and Computations in Floer homology with Applications Part II}},
journal = {arXiv:1805.01316},
year = {1996},
pages = {1--26}
}

@article{Vi99,
author = {C. Viterbo},
title = {{Functors and Computations in Floer homology with Applications, I}},
journal = {Geom. Funct. Anal},
year = {1999},
volume = {9},
number = {5},
pages = {985--1033},
}

@article{viterbo1992symplectic,
  title={Symplectic topology as the geometry of generating functions},
  author={C. Viterbo},
  journal={Math. Ann.},
  volume={292},
  number={1},
  pages={685--710},
  year={1992}
}

@article{FZ20,
author = {F. {\v{Z}}ivanović},
title = {{Symplectic geometry of conical symplectic resolutions}},
journal = {DPhil thesis, University of Oxford},
year = {2020},
pages = {1--263}
}

@article{vzivanovic2022exact,
  title={{Exact Lagrangians from contracting $\mathbb{C}^*$-actions}},
  author={F. {\v{Z}}ivanović},
  journal={arXiv:2206.06361},
  year={2022},
  pages={1--46}
}
\bibliographystyle{amsalpha}
\end{document}